\documentclass[a4paper,12pt]{amsart}
\usepackage{amsmath,amsfonts,amssymb,amsthm}
\usepackage{mathrsfs}
\usepackage{mathtools}
\usepackage{graphics}
\usepackage[table]{xcolor}
\usepackage{epsfig}
\usepackage{esint}
\usepackage[pagebackref=true]{hyperref}
\usepackage{enumerate}
\usepackage{multirow}

\newcommand{\ignore}[1]{}

\newtheorem{definition}{Definition}
\newtheorem{proposition}{Proposition}
\newtheorem{theorem}{Theorem}
\newtheorem{remark}{Remark}
\newtheorem{lemma}{Lemma}
\newtheorem{corollary}{Corollary}

\newcommand{\N}{\mathbb{N}}

\newcommand{\R}{\mathbb{R}}

\newcommand{\pua}{\partial_1^2\mathsf{u}}

\newcommand{\pva}{\partial_1^2 \va}
\newcommand{\vva}{\va \diam \partial_1^2 \va}
\newcommand{\pwa}{\partial_1^2 \wa}
\newcommand{\xva}{\xa \diam\partial_1^2 \va}

\newcommand{\wva}{\wa \diam \partial_1^2 \va}
\newcommand{\vwa}{\va \diam \partial_1^2 \wa}
\newcommand{\na}{\mathsf{1}}
\newcommand{\va}{\mathsf{v}_{\alpha}}
\newcommand{\xa}{\mathsf{v}_{1}}

\newcommand{\wa}{\mathsf{w}_{2\alpha}}
\newcommand{\ta}{\mathsf{v}}
\newcommand{\tta}{\ta}

\newcommand{\auf}{a \diam \puf}
\newcommand{\aufi}{a \diamond \puf}
\newcommand{\puf}{\partial_1^2 u}

\newcommand{\pvf}{\partial_1^2 \vf}
\newcommand{\vvf}{\vf \diam \partial_1^2 \vf}
\newcommand{\pwf}{\partial_1^2 \wwf}
\newcommand{\xvf}{\xf \partial_1^2 \vf}

\newcommand{\wvf}{\wwf \diam \partial_1^2 \vf}
\newcommand{\vwf}{\vf \diam \partial_1^2 \wwf}
\newcommand{\vf}{v_{\alpha}}
\newcommand{\wwf}{w_{2\alpha}}
\newcommand{\of}{\omega}
\newcommand{\xf}{v_{1}}
\newcommand{\nf}{1}

\newcommand{\tf}{v}

\newcommand{\cut}[1]{C_{#1}}

\newcommand{\Ta}[1]{\mathsf{T}_{#1}}
\newcommand{\oTa}[1]{\overline{\mathsf{T}}_{#1}}
\newcommand{\tTa}[1]{{\mathsf{T}}_{#1}}
\newcommand{\Tplus}{\mathsf{T}_{+}}
\newcommand{\Tminus}{\mathsf{T}_{-}}
\newcommand{\Tplust}{\mathsf{T}_{+}}

\newcommand{\Tminust}{\mathsf{T}_{-}}
\newcommand{\Aa}[1]{\mathsf{A}_{#1}}
\newcommand{\oAa}[1]{\overline{\mathsf{A}}_{#1}}
\newcommand{\tAa}[1]{\mathsf{A}_{#1}}
\newcommand{\Ga}[2]{\Gamma_{#1}^{#2}}
\newcommand{\tGa}[2]{\Gamma_{#1}^{#2}}
\newcommand{\oGa}[2]{\overline\Gamma_{#1}^{#2}}

\newcommand{\tbeta}{\beta}
\newcommand{\tgamma}{\gamma}

\newcommand{\betap}{\beta_{+}}
\newcommand{\betam}{\beta_{-}}
\newcommand{\gammap}{\gamma_{+}}
\newcommand{\gammam}{\gamma_{-}}
\newcommand{\epsp}{\eps_{+}}
\newcommand{\epsm}{\eps_{-}}
\newcommand{\betau}{\underline{\beta}}
\newcommand{\betao}{\overline{\beta}}
\newcommand{\betapu}{\underline{\beta}_{+}}

\newcommand{\betamu}{{\underline{\beta}_{-}}}
\newcommand{\betamo}{\overline{\beta}_{-}}
\newcommand{\alphao}{\overline{\alpha}}
\newcommand{\alphapo}{\overline{\alpha}_{+}}
\newcommand{\alphamo}{\overline{\alpha}_{-}}

\newcommand{\aop}{{A_+}}
\newcommand{\bop}{{A_-}}
\newcommand{\taop}{{B_+}}
\newcommand{\tbop}{{B_-}}

\newcommand{\vp}[1]{\ensuremath{\left(#1\right)}}
\newcommand{\ap}[1]{\ensuremath{\left\langle #1\right\rangle}}

\newcommand{\cald}{{\mathcal D}}

\newcommand{\call}{{\mathcal L}}

\newcommand{\calr}{{\mathcal R}}
\newcommand{\cals}{{\mathcal S}}

\newcommand{\set}[1]{\ensuremath{\{#1\}}}
\newcommand{\setc}[2]{\ensuremath{\{#1 :\ #2\}}}
\newcommand{\interval}[1]{\ensuremath{I^{#1}}}

\newcommand{\dd}{\,\text{d}}
\newcommand{\eps}{\varepsilon}

\newcommand{\ti}{\tilde}

\newcommand{\mix}{\otimes}
\newcommand{\id}{{\rm id}}
\newcommand{\thresh}{\eta}
\newcommand{\Tc}[1]{(T^{\frac{1}{4}})^{#1}}

\newcommand{\bd}[3]{\|#1\|_{#2;#3}}
\newcommand{\Mpb}{M_{+}^{b}}
\newcommand{\Mpc}{M_{+}^{c}}
\newcommand{\Mmb}{M_{-}^{b}}
\newcommand{\Mmc}{M_{-}^{c}}

\usepackage{color}
\definecolor{darkred}{rgb}{0.9,0.1,0.1}
\definecolor{darkblue}{rgb}{0,0,0.7}
\definecolor{darkgreen}{rgb}{0,0.5,0}

\newcommand{\diam}{\hspace{-0.4ex}\diamond\hspace{-0.4ex}}

\hypersetup{
pdfauthor={Felix Otto, Jonas Sauer, Scott Smith, Hendrik Weber},
pdftitle={Parabolic equations with rough coefficients and singular forcing},
breaklinks=true,
colorlinks=true,
linkcolor=blue,
citecolor=blue,
urlcolor=blue,
filecolor=blue,
}
\title[Parabolic equations with rough coefficients]{Parabolic equations with rough coefficients and singular forcing} 
\author{Felix Otto, Jonas Sauer, Scott Smith and Hendrik Weber
}
\begin{document}

\begin{abstract}
This article focuses on parabolic equations with rough diffusion coefficients which are ill-posed in the classical sense of distributions due to the presence of a singular forcing.  Inspired by the philosophy of rough paths and regularity structures, we introduce a notion of modelled distribution which is suitable in this context.  We prove two general tools for reconstruction and integration, as well as a product lemma which is tailor made for the reconstruction of the rough diffusion operator.  This yields a partially automated deterministic theory, which we apply to obtain an existence and uniqueness theory for parabolic equations with rough diffusion coefficients and a singular forcing in the negative parabolic H\"{o}lder space of order larger than $-\frac{3}{2}$.\\
\end{abstract}

\maketitle

\tableofcontents

\section{Introduction}
The present article is devoted to linear parabolic partial differential equations of the form
\begin{equation}\label{eq:PDEintro}
\partial_{2}u-P(a\partial_{1}^{2}u)=Pf,
\end{equation}
where we seek to build a $[0,1)^{2}$-periodic solution $u=u(x)$, where $x=(x_{1},x_{2}) \in \R^{2}$ and $P$ denotes the projection onto mean-free functions. The coefficient field is a periodic uniformly elliptic function $a \in C^{\alpha}$, while the forcing is merely a distribution $f \in C^{\alpha-2}$ for some $\alpha \in (0,1)$; see Section \ref{S:pre} for the definition of the H\"older spaces.  Here we restrict our attention to scalar equations and a single space dimension for notational convenience. The framework of solutions that are periodic not just in the space-like variable $x_1$, but also the time-like variable $x_2$ is a more substantial modification: this more ``elliptic'' treatment of a parabolic equation necessitates the projection $P$ and motivates the non-standard labeling of the variables. The initial-value problem will be treated in subsequent work. Solutions to \eqref{eq:PDEintro} are expected to be generically at most $C^{\alpha}$ when measured on the H\"older scale, so the product $a\partial_{1}^{2}u$ does not carry a canonical meaning in the sense of distributions, since $\alpha+(\alpha-2)<0$.  Note that \eqref{eq:PDEintro} arises naturally as a linear counterpart to non-linear stochastic PDEs where $a$ itself depends on the solution $u$, and $f$ is a generic realization of a random forcing, the most canonical example being the case of space-time white noise, which however requires $\alpha=\frac{1}{2}-$, while in this paper we treat the range of $\alpha\in(\frac12,\frac23)$, the range of $(\frac23,1)$ having dealt with in this framework in \cite{OtW16}.

\medskip

This work is part of an increased effort to understand singular stochastic PDE
following Hairer's discovery of regularity structures \cite{hairer2014theory} 
and the work by Gubinelli-Imkeller-Perkowski on paracontrolled distributions 
\cite{gubinelli2015paracontrolled}. Both of these theories gave a way to interpret 
and construct solutions to a large class of semi-linear equations
\begin{equation}\label{e:intro:2}
\mathcal{L}u = F(u, f),
\end{equation}
where $\mathcal{L}$ is a linear operator  (often the diffusion operator $\partial_2 - \partial_1^2$, also in multiple space dimensions)
and the right hand side involves an irregular noise term $f$ (e.g. space-time white noise).
 Among the problems which were treated were  the KPZ equation \cite{hairer2011solving,friz2014course,gubinelli2017kpz}, the parabolic Anderson model and more general multiplicative stochastic heat equations \cite{hairer2014theory,gubinelli2015paracontrolled,hairer2015wong},  as well as  the dynamic $\Phi^4_3$ model \cite{hairer2011solving,catellier2013paracontrolled,mourrat2017dynamic}.
 
 \medskip

The main issue addressed in all of these works is that certain non-linear operations which are part of 
the right hand side $F$ are not defined in the regularity class one expects the solution $u$ to live in. 
Because of this lack of regularity, the definition of these operations  involves a renormalization 
procedure which often amounts to adding formally infinite counter-terms. 
Let us stress that the definition of the leading order operator $\mathcal{L}$
does not pose a problem in any of these works.  In both \cite{hairer2014theory}  and 
\cite{gubinelli2015paracontrolled} the semi-linear structure is used to formulate \eqref{e:intro:2}  in the 
so-called mild form using the variation of constants formula. In both approaches explicit kernel representations 
of the operator $\mathcal{L}^{-1}$ are used to infer the regularity for $u$.

\medskip

The rigorous implementation of the renormalization procedure for \eqref{e:intro:2}
is split into two steps: a stochastic step and an analytic step.  In the stochastic step, explicit 
approximate solutions to \eqref{e:intro:2} are constructed off-line, as we like to say, the structure of which is given by a formal perturbation expansion of \eqref{e:intro:2} in small $f$.
These terms possess the 
same low regularity as one expects for $u$, but nonetheless the renormalization procedure
can be performed  more easily for these objects than for $u$, because they are known explicitly, typically 
as polynomials in the underlying random noise $f$. In the analytic step, going back to the original
problem  \eqref{e:intro:2}, the non-linear operations for $u$ are defined based on the assumption
that $u$ is well-approximated by the explicit terms in the perturbation expansion, and \eqref{e:intro:2} is solved in this framework. In this paper, we only focus on the second, analytic step in case of \eqref{eq:PDEintro}.

\medskip

The analytic part of this approach works for a very large class of equations which 
satisfy a certain scaling property, called sub-criticality in \cite{hairer2014theory}. In the
context of equation \eqref{eq:PDEintro} this condition translates to the assumption  $\alpha>0$. 
However, when approaching the threshold of criticality,
more and more terms are required in the perturbation expansion with more and more complicated 
relations among themselves. In the framework of regularity structures these relations have now 
been completely analyzed \cite{bruned2016algebraic,chandra2016analytic} systematically. Such higher-order expansions have yet to be analyzed in the framework of paracontrolled distributions; see 
however \cite{bailleul2016higher} for recent progress in this direction.

\medskip

In \cite{OtW16}, the first and last authors studied the variable coefficient equation \eqref{eq:PDEintro} in the regime $\alpha \in (\frac{2}{3},1)$ en route to a small data theory for the corresponding quasilinear equation (for suitably randomized $f$), where $a=A(u)$ for a sufficiently smooth, uniformly elliptic function $A$ (see 
also \cite{furlan2016paracontrolled}, \cite{bailleul2016quasilinear} for related work in the
context of paracontrolled distributions).  They introduce a parametric model based on a family $\{ v_{\alpha}(\cdot,a_{0}) \}_{a_0 \in I} \subset C^{\alpha}$ of periodic solutions to the linear equation
\begin{equation}
	(\partial_{2}-a_{0}\partial_{1}^{2})v_{\alpha}(\cdot,a_{0})=Pf,
\end{equation}
where $I=[\lambda,\lambda^{-1}]$ denotes the ellipticity interval.  They consider a sub-class of functions $u$ which are ``modelled'' after $v_{\alpha}$ in the sense that for some $\nu \in C^{2\alpha-1}$ it holds
\begin{align}\label{eq:MROW}
&\big |u(y)-u(x)-\big (v_{\alpha}(y,a(x))-v_{\alpha}(x,a(x)) \big )-\nu(x)(y-x)_{1} \big | \\
& \quad \leq M d^{2\alpha}(x,y), \nonumber
\end{align}
where $M$ is a suitable constant and $d(x,y)$ is the parabolic distance (see \eqref{1.11}). We remark that this approach was also adopted by \cite{furlan2016paracontrolled}. Relation \eqref{eq:MROW} is a higher-dimensional version of Gubinelli's \cite{GUBINELLI200486} notion of a controlled rough path and also closely related to Hairer's notion of a modelled distribution.
However \eqref{eq:MROW}  is nonlinear, since the function $a$ is required to guide the ellipticity parameter. 

\medskip

The present work builds on the ideas developed in \cite{OtW16}, with the broader goal of reaching a wider class of equations, including more singular noise and more sophisticated non-linearities.  To build a solution theory for \eqref{eq:PDEintro} with more singular forcing than the regime $\alpha \in (\frac{2}{3},1)$, one must build a larger model, involving several  functions which are higher-order analogues of $v_{\alpha}$, having the same regularity but a higher homogeneity in $f$.  As a result, writing an explicit local description of $u$ analogous to \eqref{eq:MROW} becomes increasingly tedious, since the role of increments of $v_{\alpha}$ in the relation \eqref{eq:MROW} is replaced by more complicated expressions reflecting their own modelledness.  Moreover, these objects naturally depend on several ellipticity arguments $(a_{0},a_{0}',\ldots)$, so in the analysis of the linear problem \eqref{eq:PDEintro} one requires additional rough functions playing the role of $a$ in \eqref{eq:MROW} above.  In fact, these various instances of $a$ must satisfy their own analogues of \eqref{eq:MROW}, leading to a hierarchy of controlled rough path relations.

\medskip

Due to this complexity, it is natural to seek a more abstract framework that keeps track of these intricacies through an efficient book keeping.  In the spirit of regularity structures \cite{hairer2014theory}, we introduce abstract spaces the elements of which act as placeholders for the various rough functions and singular distributions arising in the analysis of \eqref{eq:PDEintro}.  To incorporate the parametric dependence of the model, we use abstract spaces which are infinite dimensional, consisting of functions of various instances $(a_0,a_0',\ldots)$ of the ellipticity parameter.  This should be contrasted with the abstract spaces encountered in the semi-linear context, where finite dimensional-spaces suffice for the local description of the solution.  Moreover, in order to bundle the hierarchy of controlled rough path relations into a single condition, we switch to the perspective that the local description of $u$ is given by a family of linear forms on the direct sum of the abstract spaces modulated by $x$ (thus an element of the cotangent-bundle) satisfying a simple continuity condition (see \eqref{prod5}).

\medskip


\subsection{Structure of the Paper}
The first part of the paper, namely Sections \ref{S:MD} and \ref{S:int}, provides abstract tools which are vital in the treatment of \eqref{eq:PDEintro}. The main results in this abstract part are Propositions \ref{rec_lem} and \ref{int lem}, which we refer to as reconstruction and integration respectively.  These can be viewed as generalizations of two of the core ideas in \cite{OtW16}, the first being a wavelet-free approach to the reconstruction theorem of Hairer, cf.\@ Theorem 3.10 in  \cite{hairer2014theory}, and the second being a kernel-free approach to Schauder estimates.   In comparison to \cite{OtW16}, their novelty is more in their statement and greater generality.  In particular, they are freed from the particular details of the underlying model, so they may eventually be applied to analyze \eqref{eq:PDEintro} within the full regime $\alpha \in (0,1)$.   We also note that these results apply equally well in semi-linear contexts (\emph{i.e.}, without parametric dependence), and could be useful for the reader looking for additional clarity on the reconstruction and integration results in Sections 3 and 5 of Hairer \cite{hairer2014theory}.

\medskip

Proposition \ref{rec_lem} is well aligned with the more abstract perspective taken in the present article.  Namely, we view the off-line inputs (a \emph{model}) as a distribution with values in an abstract space and use a family of $x$ dependent semi-norms given in terms of a \emph{skeleton} to quantify the required bounds (see \eqref{cw04}).  In turn, these semi-norms determine a certain continuity and boundedness condition (see \eqref{prod4} and \eqref{prod5}) required for a family of linear forms (a \emph{modelled distribution}) to have a unique reconstruction relative to the given off-line inputs.  This result is certainly very close to that of Hairer \cite{hairer2014theory}, but is packaged in a different way and relies on a different set of tools in the proof.

\medskip

Our key notions to formulate these results, namely model, skeleton and modelled distribution, are closely
related to Hairer's notion of model and modelled distribution \cite{hairer2014theory}, but we prefer to
present them with a slightly different focus. For Hairer a model consists of families of distributions
$\Pi_x \tau$ indexed by the base-point $x$ and the symbol $\tau$
as well as operators $\Gamma_{xy}$ that translate these distributions from one base point to another
(see \cite[Definition 2.17]{hairer2014theory}). In concrete situations (see e.g. \cite[Remark 2.25, Section 8.2]{hairer2014theory} or
\cite[Section 3.3]{hairer2015wong}), such a model is typically constructed from  distributions $\mathbf{\Pi} \tau$
which do not depend on the base point $x$ and operators $F_x$  that
``recenter" the distributions $\mathbf{\Pi} \tau$ around a specific point $x$ and yield the $\Pi_x \tau$.
We view the ``uncentered" distributions (\emph{i.e.}, Hairer's $\mathbf{\Pi}$) as the more fundamental object
and this is what we call model. Our skeleton $\Ga{}{}$ then corresponds exactly to the family of $F_x$, cf.\@ Remark \ref{wg01},
and we stress that our $x$-dependent continuity assumption \eqref{cw04}
corresponds exactly to Hairer's ``order condition"  \cite[Equation (2.15)]{hairer2014theory}.
Both, the model and the skeleton are part of the off-line data and in applications to stochastic
PDE we expect to construct them together. However, they play a somewhat different role in the analysis
and we prefer  to separate them to emphasize their difference.
Similarly, we define modelled distribution in terms of the skeleton; note that our ``form-continuity"
condition  \eqref{prod5} corresponds exactly to Hairer's continuity condition $\|\Gamma_{xy}f(y) - f(x) \|_{\beta}
\lesssim d^{\alphao-\beta}(y,x)$, see Remark \ref{wg01}.

\medskip

Proposition \ref{int lem} can be viewed as a splitting method, where in contrast to a typical decomposition into a rough piece and a more regular piece, the remainder does not need to satisfy an explicit equation in order to be shown of higher order, only to satisfy a certain bound \eqref{KS1}, referred to as the local splitting condition.  This idea is certainly implicit in the theory of regularity structures and paracontrolled calculus, but the formulation here seems new.  It extends Safonov's kernel-free approach to Schauder theory to singular equations.

\medskip

Another important result in Section \ref{S:MD} is Lemma \ref{lem:Multiplication}, which together with Corollary \ref{lem:ExMult} yields an automated procedure for defining the rough diffusion operator $a \diamond \partial_{1}^{2}u$ as a distribution, given as inputs modelled distributions $a$ and $\partial_{1}^{2}u$ and assuming the necessary off-line products have already been defined.  This is ultimately a corollary of the abstract reconstruction result Proposition \ref{rec_lem}, but we must first build a suitable family of linear forms to represent $a \diamond \partial_{1}^{2}u$ at the abstract level.  This turns out to almost be the form given by a naive tensorization, but it must be extended further to include an additional abstract placeholder for the distribution $\partial_{1}^{2}u$ (see \eqref{MS1}).  We remark that since our main objective is to give an enhanced description of the commutator $(\auf)_{T}- a(\puf)_{T}$, the power counting for the order of this form is slightly better than in the typical multiplication theorem of Hairer, cf.\@ Theorem 4.7 in \cite{hairer2014theory}, in that one can neglect the contribution of the lowest homogeneity in the description of $a$.

\medskip

Our larger goal, for which the present work gives partial progress, is to develop a complete well-posedness theory for \eqref{eq:PDEintro} under the hypothesis that the rough diffusion coefficient $a$ has a suitable description as a modelled distribution.  Given such a description, specified in terms of appropriate abstract spaces and a corresponding model, one then attempts to build the abstract framework sufficient to describe $u$ together with the description of $a$. The construction of such a joint framework has to anticipate the nonlinear relationship $a\mapsto u$ given through \eqref{eq:PDEintro}. Finally, one seeks an estimate for the norm of $u$ as a modelled distribution in terms of the corresponding norm for $a$.
We do not put forth a general theory of this type in the present article.  However, as a proof of concept, illustrating how to apply the results in Sections \ref{S:MD} and \ref{S:int}, we carry out this approach to \eqref{eq:PDEintro} in a concrete setting motivated by quasi-linear SPDEs with additive forcing in the regime $\alpha \in (\frac{1}{2},\frac{2}{3})$ corresponding to a marginally smoother covariance structure than space-time white noise.

\medskip

Namely, in Section \ref{S:App} we study \eqref{eq:PDEintro} under the simplifying hypothesis that the modelled distribution describing $a$ resembles a linear transformation of the solution $u$.  In Sections \ref{fspc} and \ref{subsec:theModel}, we specify the abstract spaces (functions of several parameters) along with the functions and distributions required in an off-line step for the local description of $a, u$, and $a \diamond \partial_{1}^{2}u$.  In Section \ref{SS:algo}, we present the ``algebraic'' algorithm for constructing the modelled distribution describing $u$ from the modelled distribution describing $a$.  The main result, Theorem \ref{Prop:aPriori}, establishes existence and uniqueness including an \emph{a priori} bound for the norm of $u$ as a modelled distribution in terms of the norm of $a$ as a modelled distribution, under a smallness condition on the H\"{o}lder norm of $a$
.  For concreteness, we carry this analysis out on a very particular example of a model for $a$ (and $u$) suitable for the quasilinear problem $a=a(u)=u+1$ and additive noise. In Corollary \ref{cor:wg}, we point out that our \emph{a priori} estimate is sufficient for the self-mapping property in a fixed point argument in the quasilinear case. The complementing contraction property would follow from a stability analysis for \eqref{eq:PDEintro} including Lipschitz dependence on the modelled distribution describing $a$, which is currently in preparation. Moreover, we expect that our result remains true with roughly the same proof if the model describing $a$ were larger in order to be suitable for a general (smooth) $a(u)$ or multiplicative noise $\sigma(u)f$, as long as the relevant stochastic off-line inputs can be constructed, and $a$ is described to a sufficiently high order.  In contrast, carrying out this procedure with rougher noise would require additional ideas.  Indeed, in Section \ref{S:concr} we find that in contrast to the semi-linear setting, the algebraic and the analytic aspects of the problem are more intertwined.  In particular, the differential structure in the parameter plays a crucial role in the proof of Lemma \ref{int_lem_concr}, the application of our general integration lemma, Proposition \ref{int lem}. This is the main analytic ingredient which must still be done by hand, cf.\@ the proof of Lemma \ref{cor:prod1} and Step \ref{Intprf_000} in the proof of Lemma \ref{int_lem_concr}, and would need to be automated in order to provide a black box of analytic tools for solving \eqref{eq:PDEintro} with arbitrary $\alpha \in (0,1)$.

\medskip

While completing this article, we became aware of a recent work by Hairer and Gerencs{\'e}r \cite{gerencser2017quasi}
on a closely related subject. These authors developed a short-time solution theory for the initial value problem of 
 stochastic PDEs of the type
\begin{equation}\label{e:Intro:4}
(\partial_2- a(u) \partial_1^2 )u = b(u) (\partial_1 u)^2 + \sigma(u) f.
\end{equation}
Here $f$ is a suitable random distribution of low regularity. Their results are optimal in the sense that the analytic part of their construction 
works up to the threshold of sub-criticality, \emph{i.e.}, for all $f$ of regularity $\alpha-2$ for $\alpha>0$. In order to construct the 
terms in the stochastic expansion they impose a stronger regularity assumption which corresponds to $\alpha>\frac14$
and this condition may also be close to optimal. However, in order to analyze the renormalization procedure they  
impose an even stronger regularity assumption on $f$ which corresponds exactly to our condition $\alpha>\frac12$.
Their analysis is based on a transformation which turns  \eqref{e:Intro:4} into a rather complicated non-local integral equation.
Despite its complexity this equation does fall into the class of problems that can be treated within the sophisticated regularity 
structure machinery as developed in \cite{bruned2016algebraic,chandra2016analytic,hairer2014theory}. 
The reason for their restriction $\alpha >\frac12$ stems from the fact that in their framework it is not obvious that the
terms produced in the renormalization depend locally on the solution $u$ and they only checked this explicitly 
in the case $\alpha>\frac12$. 

\medskip
In the present work we do not discuss the construction of the off-line stochastic terms or their effect on renormalization.
However, we do not expect any problem connected to non-locality as this 
only arises in the transformation to an integral equation and can probably be avoided by dealing directly
with the more natural equation \eqref{e:Intro:4}, as we do.
\subsection{Semigroup Convolution and H\"older Scale}\label{S:pre}
We will use similar notation as in \cite{OtW16} whenever possible. Let us recall the basic notations used there, in particular the parabolic metric and the convolution kernel.

The parabolic operator $\partial_2-a_0\partial_1^2$ and its
mapping properties on the scale of H\"older spaces (\emph{i.e.}, Schauder theory) imposes its intrinsic 
(Carnot-Carath\'eodory) metric, which is given by
\begin{align}\label{1.11}
d(x,y)=|x_1-y_1|+\sqrt{|x_2-y_2|}.
\end{align}
In order to define negative norms of distributions in an intrinsic way, 
it is convenient to have a family $\{(\cdot)_T\}_{T>0}$ of mollification operators $(\cdot)_T$ consistent 
with the relative scaling $(x_1,x_2)=(\ell\hat x_1,\ell^2\hat x_2)$ of the two variables
dictated by \eqref{1.11}.  
It will turn out to be extremely convenient to have in addition the semi-group property
\begin{align}\label{1.10}
(\cdot)_T(\cdot)_t=(\cdot)_{T+t}.
\end{align}
All the above is achieved by convolution with the kernel of the semigroup $\exp(-T(\partial_1^4-\partial_2^2))$
of the elliptic operator $\partial_1^4-\partial_2^2$, which is the simplest positive operator
displaying the same relative scaling between the variables as $\partial_2-\partial_1^2$ and being symmetric
in $x_2$ next to $x_1$. We note that the corresponding convolution kernel $\psi_T$ is easily
characterized by its Fourier transform $\hat\psi_T(k)=\exp(-T(k_1^4+k_2^2))$; since
the latter is a Schwartz function, also $\psi_T$ is a Schwartz function.  For later reference we note that
\begin{align}\label{wg50}
\partial_T\psi_T=-(\partial_1^4-\partial_2^2)\psi_T.
\end{align} 
The only three (minor) inconveniences
are that 1) the $x_1$-scale is played by $T^\frac{1}{4}$ (in line with \eqref{1.11} the $x_2$-scale
is played by $T^\frac{1}{2}$) since we have
$\psi_T(x_1,x_2)=\frac{1}{T^\frac{3}{4}}\psi_1(\frac{x_1}{T^\frac{1}{4}},\frac{x_2}{T^\frac{1}{2}})$, 
that 2) $\psi_1$ (and thus $\psi_T$) does not have a sign, and that 3) $\psi_1$ is not compactly supported. The only properties of the kernel we need
are moments of derivatives:
\begin{align}\label{1.13}
\begin{array}{rcl}
\int |\partial_1^k\psi_T(x-y)|d^\alpha(x,y) \dd y&\lesssim&(T^\frac{1}{4})^{-k+\alpha}\quad\mbox{and}\quad\\
\int |\partial_2^k\psi_T(x-y)|d^\alpha(x,y) \dd y&\lesssim&(T^\frac{1}{4})^{-2k+\alpha}
\end{array}
\end{align}
for all orders of derivative $k=0,1,\cdots$ and moment exponents $\alpha\ge 0$. Here, the symbol $\lesssim$ is used whenever the inequality holds up to a constant that depends only on $\alpha$ and $k$. Estimates \eqref{1.13} follow immediately
from the scaling and the fact that $\psi_1$ is a Schwartz function.  In addition, we use that
\begin{align}\label{he1}
\int \psi x_{1}\dd x = 0.
\end{align}

Given a Banach space $\Ta{}$, we define $c\Ta{}$ for any $c>0$ to be $\Ta{}$ endowed with the norm
 \begin{align}\label{jj4}
  \|\cdot\|_{c\Ta{}}:=c\|\cdot\|_{\Ta{}}.
 \end{align}
 Moreover, we denote by $C^{0}(\R^{2}; \Ta{})$ the space of strongly continuous functions $v: \R^{2} \to \mathsf{T}$.  We will also use the parabolic H\"{o}lder space $C^{\alpha}(\R^{2}; \mathsf{T})$, defined for $\alpha \in (0,1)$ via the semi-norm
\begin{align}\label{j1101}
[v]_{C^{\alpha}(\R^{2}; \mathsf{T})}:=\sup_{x \neq y}d^{-\alpha}(y,x)\|v(y)-v(x)\|_{\mathsf{T}},
\end{align}
where $v \in C^{0}(\R^{2}; \mathsf{T})$. If $\Ta{}$ happens to be the scalar field, we will simply write 
\begin{align*}
 [\cdot]_{\alpha}:=[v]_{C^{\alpha}(\R^{2}; \mathsf{T})}. 
\end{align*}
We will also need a space of tempered distributions $C^{\alpha-2}(\R^{2}; \mathsf{T})$, which carries the norm
\begin{align}\label{j1001}
\|g\|_{C^{\alpha-2}(\R; \mathsf{T})}:=\sup_{T\le 1}\Tc{2-\alpha}\sup_{x \in \R^{2}}\|g_{T}(x)\|_{\mathsf{T}},
\end{align}
where $g \in \cals'(\R^{2}; \mathsf{T})$ and for $x \in \R^{2}$ we have $g_{T}(x)=\langle g, \psi_{T}(x- \cdot) \rangle_{\R^{2}} \in \mathsf{T}$.


\section{Modelled Distributions}\label{S:MD}

In this section, we introduce our notion of a modelled distribution. As in Chapter 3 of \cite{hairer2014theory}, this abstract object should be thought of as a local description of a function or a distribution to a given order.  Modelled distributions are defined relative to an abstract space and a skeleton.  Elements of the abstract spaces should be thought of as placeholders for functions and distributions, while the skeleton should be thought of as the primitive object from which modelled distributions are built.
\begin{definition}
An abstract space is a pair $(\Aa{},\Ta{})$, where $\Aa{} \subset \R$ is finite and $\Ta{}$ is a Banach space graded according to $\Aa{}$, meaning that
\begin{align}\label{j1002}
\Ta{}=\bigoplus_{\beta \in \Aa{}} \Ta{\beta},
\end{align}
for some collection of Banach spaces $\{\Ta{\beta}\}_{\beta \in \Aa{}}$.  The smallest and largest elements of $\Aa{}$ are denoted $\betau$ and $\betao$ respectively, that is
\begin{align}\label{j1000}
 \betau:=\min\Aa{}, \qquad \betao:=\max\Aa{}.
\end{align}
\end{definition}
Given an abstract space $(\Aa{},\Ta{})$, real numbers in $\Aa{}$ are referred to as homogeneities and typically denoted by $\beta$, $\gamma$, $\eps$ etc., while elements of $\Ta{}$ are denoted by $\ta=(\ta_\beta)_{\beta\in\Aa{}}$.

\medskip

Given two Banach spaces $\mathsf{T}_{1},\mathsf{T}_{2}$, we denote by $\call(\mathsf{T}_{1},\mathsf{T}_{2})$ the space of bounded linear operators from $\mathsf{T}_{1}$ to $\mathsf{T}_{2}$. In particular, for a Banach space $\mathsf{T}$ with a typical element $\ta$ and $V \in C^{0}(\R^{2}; \Ta{}^*)$ with $\Ta{}^*:=\call(\mathsf{T}, \R)$, we denote by $V(x).\ta$ the application of $V(x) \in \Ta{}^*$ to an element $\ta\in \mathsf{T}$.

\begin{definition}\label{def:Skel}
Given an abstract space $(\Aa{},\Ta{})$, an operator-valued function $\Ga{}{} \in C^{0}(\R^{2}; \mathcal{L}(\Ta{}, \Ta{}))$ is called a skeleton provided that 
\begin{enumerate}

\item Triangular representation: for each $x\in\R^2$ and $\ta=(\ta_\gamma)_{\gamma\in\Aa{}} \in \Ta{}$, we write $\Ga{}{}(x)\ta= \vp{\sum_{\gamma\in\Aa{}}\Ga{\beta}{\gamma}(x)\ta_{\gamma}}_{\beta \in \Aa{}}$, where in particular $\Ga{\beta}{\gamma} \in C^{0}(\R^{2};\mathcal{L}(\Ta{\gamma},\Ta{\beta}))$, and assume
\begin{align}\label{j551}
 \Ga{\beta}{}(x)\ta:=(\Ga{}{}(x)\ta)_\beta =\ta_\beta + \sum_{\gamma < \beta}\Ga{\beta}{\gamma}(x)\ta_{\gamma}.
\end{align}
In other words, \eqref{j551} means that for all $\beta \in \Aa{}$ and $x \in \R^{2}$, we impose $\Ga{\beta}{\gamma}(x)=0$ for $\gamma>\beta$ and $\Ga{\beta}{\beta}(x)=\id_{\Ta{\beta}}$.  

\item Continuity property: for some constant $C\ge 0$ and all $\beta\in\Aa{}$, $\ta\in\Ta{}$, and $x,y \in \R^{2}$ with $d(y,x)\le 1$, it holds
\begin{align}
\|(\Ga{\beta}{}(y)-\Ga{\beta}{}(x))\ta\|_{\Ta{\beta}}  \leq C \sum_{\gamma<\beta}d^{\beta-\gamma}(y,x)\|\Ga{\gamma}{}(x)\ta\|_{\Ta{\gamma}}.\label{ssss2}
\end{align}
The minimal constant $C$ is denoted by $\|\Gamma\|_{sk}$.
\end{enumerate}
\end{definition}

\medskip

\begin{definition}\label{prodDef1}
Given an abstract space $(\Aa{},\Ta{})$ and a skeleton $\Ga{}{}$, we call $V \in C^{0}(\R^{2}; \Ta{}^*)$ a modelled distribution of order $\overline{\alpha}$, if there is a family of non-negative constants $\{M^b_\beta\}_{\beta\in\mathsf{A}}$ and a constant $M^c\ge 0$ such that the following hold:
\begin{enumerate}
\item Form boundedness: for every $\ta\in\Ta{}$ and $x\in \R^2$
\begin{equation}
|V(x).\ta| \leq \sum_{\beta\in\Aa{}}M_\beta^b \|\Ga{\beta}{}(x)\ta\|_{\Ta{\beta}}.\label{prod4}
\end{equation}
\item Form continuity: for all $\ta\in\Ta{}$ and $x, y \in \R^{2}$ with $d(y,x)\le 1$,
\begin{equation}
|(V(y)-V(x)).\ta| \leq M^c\sum_{\beta \in \Aa{}} d^{(\overline{\alpha}-\beta)\vee 0 }(y,x)\|\Ga{\beta}{}(x)\ta\|_{\Ta{\beta}}.\label{prod5}
\end{equation}
\end{enumerate}
The space of modelled distributions is denoted $\cald^{\alphao}(\Ta{}; \Ga{}{})$.  The smallest constant $M^c$ in \eqref{prod5} is denoted by 
$[V]_{{\mathcal D}^{\overline{\alpha}}(\mathsf{T};\Gamma)}$, defining a semi-norm.
Given a parameter $N\in(0,1]$ and a family of non-negative exponents 
$\{\langle\beta\rangle\}_{\beta\in\mathsf{A}}$, we denote the smallest value of 
$\max_{\beta\in\mathsf{A}}N^{-\langle\beta\rangle}M_\beta^b$
over all $\{M^b_\beta\}_{\beta\in\mathsf{A}}\in[0,\infty)$ with \eqref{prod4}
by $\|V\|_{\mathsf{T};\Gamma}$; this defines a norm.
\end{definition}
A remark on the definition of (semi-)norms $[\cdot]_{{\mathcal D}^{\overline{\alpha}}(\mathsf{T};\Gamma)}$
and $\|\cdot\|_{\mathsf{T};\Gamma}$ is in place. 
In principle, both \eqref{prod4} and \eqref{prod5} implicitly define two families of semi-norms
indexed by $\mathsf{A}$. In order to present our results in a compact form, we consider
weighted sums of these semi-norms.
In our application, for the more important \eqref{prod5}, we absorb the suitable weights by
adjusting the norms $\|\cdot\|_{\mathsf{T}_\beta}$, cf.\@ \eqref{S8030}, which also has the advantage of making all contributions to
$\|\Gamma\|_{sk}$ to be of the same order unity. Hence in this abstract section,
we are acquitted of introducing
$\{M_\beta^c\}_{\beta\in\mathsf{A}}$. However, we have to allow for weights when aggregating
$\{M_\beta^b\}_{\beta\in\mathsf{A}}$; without loss of generality, we parameterize these weights through 
$\{N^{-\langle\beta\rangle}\}_{\beta\in\mathsf{A}}$. 
This specific parameterization of weights is motivated
by the application, where $N\le 1$ is a scale parameter measuring the size of the r.~h.~s.~$f$ 
and $\langle\beta\rangle$ is {\it another} (integer-valued) homogeneity parameter, 
essentially indicating the homogeneity of the model on level $\beta$ in the r.~h.~s. $f$, cf.~Subsections
\ref{fspc} and \ref{subsec:theModel}. 
For later reference we note that \eqref{prod4} takes the form
\begin{align}\label{wg16}
|V(x).\mathsf{v}|\le\|V\|_{\mathsf{T};\Gamma}
\sum_{\beta\in\mathsf{A}}N^{\langle\beta\rangle}\|\Gamma_\beta(x)\mathsf{v}\|_{\mathsf{T}_\beta}.
\end{align}
\begin{remark}\label{wg01}
The form boundedness \eqref{prod4} highlights that there exists $\tilde V\in C^0(\mathbb{R}^2;\mathsf{T}^*)$ such that
$V$ is of the form 
\begin{align}\label{cw92}
V.\mathsf{v}=\tilde V.\Gamma\mathsf{v},
\end{align}
and it follows easily (see Subsection \ref{subsec_remarkPf} for a short proof) that form continuity \eqref{prod5} is then equivalent to
\begin{align}\label{cw90}
\big\|\big(\Gamma^{-*}(x)\Gamma^*(y)\tilde V(y)-\tilde V(x)\big)_\beta\big\|_{\mathsf{T}_\beta^*}
\le [V]_{{\mathcal D}^{\overline{\alpha}}(\mathsf{T};\Gamma)}d^{(\overline{\alpha}-\beta)\vee 0}(y,x)
\end{align}
for all $\beta\in\mathsf{A}$ and $x,y\in\mathbb{R}^2$ with $d(y,x)\le 1$. Here $\Gamma^{-1}$
$\in C^0(\mathbb{R}^2;{\mathcal L}(\mathsf{T},\mathsf{T}))$
denotes the (pointwise in $x$) inverse of $\Gamma$, which exists thanks to the triangular structure \eqref{j551};
and $\Gamma^*,\Gamma^{-*}\in C^0(\mathbb{R}^2;{\mathcal L}(\mathsf{T}^*,\mathsf{T}^*))$ stand for
the (pointwise) adjoints
of $\Gamma,\Gamma^{-1}$. This equivalent formulation \eqref{cw90} coincides 
with Hairer's notion of continuity
\cite[Definition 3.1]{hairer2014theory}, with $\Gamma^{-*}(x)\Gamma^*(y)$ playing the role of $\Gamma_{xy}$
\cite[Definition 2.17]{hairer2014theory}.
\end{remark}
\subsection{Reconstruction}\label{S:rec}
Before discussing the next result, a word on the notation is in order. Throughout the paper we will denote typical elements in $\Ta{}$ by a bold face $\ta$, while $\Ta{}$-valued functions or distributions are typically denoted by a symbol $\tf$.  The following proposition states under which assumptions on a local description $V$ one can define a distribution $\calr V$ that satisfies the local description. In analogy to Theorem 3.10 in \cite{hairer2014theory}, we call this procedure \emph{reconstruction}. 
We observe that $\|\Ga{\beta}{}(x)\ta\|_{\Ta{\beta}}$ acts as a (semi)-norm on $\Ta{\beta}$
and thus puts a topology on the tangent bundle $\R^2\times \Ta{}$.  The idea is that if the topology on the one hand is weak enough that a given $\Ta{}$-valued distribution $\tf$ (a model) satisfies a continuity condition \eqref{cw04} reminiscent of a negative H\"older condition $\tf_{\beta}\in C^\beta$, cf.\@ \eqref{j1001}, and on the other hand is strong enough that a given $\Ta{}^*$-valued function $V$ (a modelled distribution) satisfies a continuity condition \eqref{prod5} reminiscent of a positive H\"older condition, cf.\@ \eqref{j1101}, then this balance causes that $V.\tf$ defines a distribution $\calr V$. The boundedness \eqref{prod4} of $V$ is only needed in a qualitative way.  

\begin{proposition}\label{rec_lem} (Reconstruction)
Let $(\mathsf{A},\mathsf{T})$ be an abstract space with $\underline{\beta}\in(-4,0)$, cf.~\eqref{j1000},
and skeleton $\Gamma$. Assume we are given $v\in\cals'(\mathbb{R}^2;\Ta{})$ such that for all $\beta\in \Aa{}$, $x\in\mathbb{R}^2$, $T\le 1$ it holds
\begin{align}
\|\Gamma_\beta(x) v_T(x)\|_{\mathsf{T}_\beta}&\le (T^\frac{1}{4})^\beta.\label{cw04}
\end{align}
%
%
%
Then for each $V \in \mathcal{D}^{\overline{\alpha}}(\Ta{};\Ga{}{})$ with $\overline{\alpha}>0$, there exists a unique ${\mathcal R}V\in\cals'(\mathbb{R}^2)$, the ``reconstruction'',  such that
for all $T\le 1$
\begin{align}\label{cw40}
\|({\mathcal R}V)_T-V.v_T\|\lesssim (1+\|\Gamma\|_{sk})^{|\Aa{}|-1} [V]_{{\mathcal D}^{\overline{\alpha}}(\mathsf{T};\Gamma)}
(T^\frac{1}{4})^{\overline{\alpha}},
\end{align}
where $|\Aa{}|$ denotes the number of elements in $\Aa{}$ 
and $\lesssim$ denotes $\le C$ up to a constant only depending on $\Aa{}$
and $\overline{\alpha}$.
\end{proposition}

\subsection{Reduction of Modelled Distributions}\label{S:red}
In our analysis, it is often convenient to pass from a modelled distribution $V$ on an abstract space $(\Aa{},\Ta{})$ to a reduced modelled distribution (effectively) acting on a smaller abstract space.  This is made precise through the cutting operation introduced below.

\medskip

Let $(\Aa{},\Ta{})$ be an abstract space.  For each $\eta$ satisfying $\underline{\beta}< \eta \leq \overline{\beta}$, cf.\@ \eqref{j1000}, we may define a new abstract space $(\Aa{<\eta},\Ta{<\eta})$ where the objects of homogeneity $\eta$ and higher have been removed:
\begin{align*}
\Aa{<\eta}:=\{\beta \in \Aa{} : \beta<\eta \}, \qquad \quad \Ta{<\eta}:=\bigoplus_{\beta<\eta}\Ta{\beta}.
\end{align*}
Moreover, given a skeleton $\Gamma$ on $(\Aa{},\Ta{})$, the skeleton $\Gamma_{<\eta}$ on $(\Aa{<\eta},\Ta{<\eta})$ is defined by restriction
\begin{align}\label{j961}
\vp{\Ga{<\eta}{}\ta'}_{\beta}:=\ta_{\beta}'+\sum_{\gamma<\beta} \Ga{\beta}{\gamma}\ta_{\gamma}' \quad \text{for } \beta<\eta.
\end{align}
It is straightforward to see that $\Ga{<\eta}{}$ constitutes a skeleton with respect to the abstract space $(\Aa{<\eta},\Ta{<\eta})$ and that 
\begin{equation}\label{skel1}
\|\Ga{<\eta}{}\|_{sk} \leq \|\Gamma \|_{sk}.
\end{equation}
We will now define a reduction operator $C_\eta$ on modelled distributions $V$ over $(\mathsf{A},\mathsf{T})$.
It is convenient to think of the reduced modelled distribution to be both an element of 
$C^0(\mathbb{R}^2;\mathsf{T}_{<\eta}^*)$
and of the original $C^0(\mathbb{R}^2;\mathsf{T}^*)$. Since the base point $x\in\mathbb{R}^2$ is a mere
parameter in this operation, we suppress it in the notation.
The definition on the level of $V$
proceeds recursively, starting from cutting at the highest homogeneity $\overline{\beta}$ via
\begin{align}\label{cut2}
C_{\overline{\beta}}V.\mathsf{v}:=V.({\rm id}_{\mathsf{T}}-\Gamma_{\overline{\beta}})\mathsf{v}.
\end{align}
By the triangular structure \eqref{j551} of $\Gamma$, $C_{\overline{\beta}}V.\mathsf{v}$ does not depend on 
$\mathsf{v}_{\overline{\beta}}$, so that \eqref{cut2} indeed defines an element of $C^0(\mathbb{R}^2;\mathsf{T}_{<\overline{\beta}}^*)$.
The definition of the reduction operator is then extended to all $\eta$ satisfying $\underline{\beta}< \eta \leq \overline{\beta}$ by successively cutting each of the homogeneities in $\Aa{}$ larger than $\eta$ (see Section \ref{S:cutPf} for more details).  The following lemma establishes that the reduced form is again a modelled distribution. Since the only non-trivial reduction procedure occurs in $V$, we will suppress the subscripts in $\Aa{<\eta}$, $\Ta{<\eta}$ and $\Ga{<\eta}{}$ for notational clarity.
\begin{lemma}\label{cutting_lemma}
Let $(\Aa{},\Ta{})$ be an abstract space with skeleton $\Ga{}{}$ and $V\in \cald^{\alphao}(\Ta{};\Ga{}{})$.  For each $\eta$ satisfying $\underline{\beta}<\eta \leq \overline{\beta}$, we have $\cut{\eta}V\in \cald^{\alphao \wedge \eta}(\Ta{};\Ga{}{})$ and
\begin{align}
\|\cut{\eta}V\|_{\Ta{};\Ga{}{}}&\le \|V\|_{\Ta{};\Ga{}{}}, \label{S70a}\\
[\cut{\eta}V]_{\mathcal{D}^{\overline{\alpha} \wedge\eta}(\Ta{}; \Ga{}{})} &\leq [V]_{\mathcal{D}^{\overline{\alpha}}(\Ta{}; \Ga{}{})}+(|\Aa{}|-1)N^{\min_{\beta\ge \eta}\ap{\beta}}\|\Ga{}{}\|_{sk}\|V\|_{\Ta{};\Ga{}{}},\label{eq:cutLemmaBd1}
\end{align}  
where we recall that $|\Aa{}|$ denotes the number of elements in $\Aa{}$. In addition, for each $\kappa> \eta$, $\ta\in\Ta{}$ and $x\in\R^2$ it holds
\begin{align}
\left |(C_{\kappa}-C_{\eta})V(x).\mathsf{v} \right | &\leq \|V\|_{\Ta{};\Ga{}{}} \sum_{\beta \in [\eta,\kappa) \cap \mathsf{A}} N^{\ap{\beta}}\|\Ga{\beta}{}(x)\mathsf{v}\|_{\Ta{\beta}}, \label{S8300}
\end{align}
in particular
\begin{align}
 |(\id-C_\eta) V(x).\mathsf{v}|&\le \|V\|_{\Ta{};\Ga{}{}}\sum_{\beta\ge\eta} N^{\ap{\beta}}\|\Gamma_\beta(x)\mathsf{v}\|_{\mathsf{T}_\beta}.\label{new}
\end{align}
\end{lemma}

\begin{remark}\label{wg04}
While the continuity condition \eqref{prod5} is more natural in terms of $V$ than $\tilde V$, cf.~\eqref{cw90},
the reduction procedure is more natural in terms of $\tilde V$, cf.~\eqref{cw92}, in which it becomes a simple restriction:
\begin{align}\label{cw94}
\widetilde{C_\eta V}.\mathsf{v}=\sum_{\beta<\eta}\tilde V_\beta.\mathsf{v}_\beta.
\end{align}
In particular we learn that $\{C_\eta\}_{\eta\in(\underline{\beta},\overline{\beta}]}$ acts
as an increasing family of projections on ${\mathcal D}^{\overline{\alpha}}(\mathsf{T};\Gamma)$ 
in the sense that for all levels $\underline{\beta}<\eta\le\eta'\le\overline{\beta}$
\begin{align}\label{wg06}
C_\eta C_{\eta'}=C_{\eta'}C_\eta=C_\eta.
\end{align}
\end{remark}

We display the simple computation for \eqref{cw94} in Step \ref{Cutprf_v} of 
the proof of Lemma \ref{cutting_lemma}.

\medskip

For Corollary \ref{lem:ExMult}, and ultimately for our main result Theorem \ref{Prop:aPriori},
we need the following more precise version of Proposition \ref{rec_lem},
where one of the homogeneities plays a special role. The proof of Remark \ref{lem:ff}
is relegated to Step \ref{Cutprf_vi} in the proof of Lemma \ref{cutting_lemma}.
%
\begin{remark}\label{lem:ff}
Suppose that there is a $\beta_*\in \Aa{}$ such that
%
\begin{align}
|V(x).\mathsf{v}|
&\le M_*^bN^{\ap{\beta_*}}\|\Gamma_{\beta_*}(x)\mathsf{v}\|_{\Ta{\beta_*}}
+M^b\sum_{\beta\not=\beta_*}N^{\ap{\beta}}\|\Gamma_\beta(x)\mathsf{v}\|_{\Ta{\beta}}\label{cw550}
\end{align}
for some constants $M_*^b$ and $M^b$ and for all $\mathsf{v}\in\mathsf{T}$ and $x\in\mathbb{R}^2$.
Then provided $\eta>\beta_*$ the outcome \eqref{eq:cutLemmaBd1} takes the more specific form
\begin{align}\label{cw51}
[\cut{\eta}V]_{\mathcal{D}^{\overline{\alpha} \wedge\eta}(\Ta{}; \Ga{}{})} &\leq [V]_{\mathcal{D}^{\overline{\alpha}}(\Ta{}; \Ga{}{})}+(|\Aa{}|-1)N^{\min_{\beta\ge \eta}\ap{\beta}}\|\Ga{}{}\|_{sk}M^b.
\end{align}
\end{remark}
\subsection{Multiplication}\label{S:mult}
In this subsection, we develop an abstract multiplication result which is tailored to the product $a \,\diam \, \partial_{1}^{2}u$.  At the abstract level, we think of $a$ as being described by a modelled distribution acting only on placeholders for functions, while $\partial_{1}^{2}u$ is described by a modelled distribution acting only on placeholders for distributions.  We work in slightly more generality, and consider a modelled distribution $V_{u_{+}}$ describing a function $u_{+}$ together with a modelled distribution $V_{u_{-}}$ describing a distribution $u_{-}$.

\medskip

We are motivated by the following result concerning classical multiplication of a function and a distribution.  If $u_{+} \in C^{\alphapo}$ and $u_{-} \in C^{ \alphamo}$, where $\alphapo \in (0,1)$, $\alphamo \in (-1,0)$ and $\alphapo+\alphamo>0$, then there exists a unique distribution $u_{+}u_{-}\in C^{ \alphamo}$ such that for all $T \leq 1$ it holds
\begin{align*}
\|(u_{+}u_{-})_{T}-u_{+}(u_{-})_{T}\| \lesssim \Tc{\alphapo+\alphamo} [u_{+}]_{{\alphapo}}\|u_{-}\|_{C^{\alphamo}}.
\end{align*}
In this section, we generalize this result, cf.\@ \eqref{S870}, by building an appropriate linear form $V_{u_{+} \diamond u_{-}}$ and applying the abstract reconstruction result Proposition \ref{rec_lem}.  

\medskip

The first step is to build an appropriate abstract space on which the form $V_{u_{+} \diamond u_{-}}$ acts. That is, given two abstract spaces $(\Aa{+},\Ta{+})$ and $(\Aa{-},\Ta{-})$ with corresponding skeletons $\Ga{+}{}$ and $\Ga{-}{}$, respectively, we construct an abstract product space $(\Aa{},\Ta{})$ and a product skeleton $\Ga{}{}$ in the following way.
The set of homogeneities $\Aa{}$ is defined as the algebraic sum
\begin{equation}
{\tAa{}}:=\Aa{+}+\Aa{-}=\{\betap+\betam : \, \, \betap \in \Aa{+}, \betam \in \Aa{-} \}\label{Prod6}.
\end{equation}
The Banach space $\tTa{}$ is defined to be the topological tensor space
\begin{align*}
 \tTa{}=\Tplus\mix\Tminus,
\end{align*}
equipped with an appropriate crossnorm $\|\cdot\|_{\tTa{}}$ (we refer to the monographs \cite{Hac12} and\cite{Rya02} for details on topological tensor spaces). 
Observe that in view of \eqref{j1002}, $\Ta{}$ is graded with respect to $\Aa{}$ via
\begin{align}
\tTa{}=\bigoplus_{\tbeta\in\tAa{}} \tTa{\tbeta}=\bigoplus_{\tbeta\in\tAa{}}\bigoplus_{\betap+ \betam=\tbeta} \Ta{\betap}\mix \Ta{\betam}, \label{Prod8}
\end{align}
where by a slight abuse of notation $\Ta{\betap}:=(\Tplus)_{\betap}$ and $\Ta{\betam}:=(\Tminus)_{\betam}$. Hence, a typical element $\ta\in\tTa{}$ decomposes into
\begin{align}\label{j1009}
  \ta= (\ta_\tbeta)_{\tbeta\in\Aa{}}=\Big(\bigoplus_{\betap+\betam=\beta} \ta_{\betap,\betam}\Big)_{\tbeta\in\Aa{}}.
 \end{align}
We assume that the crossnorm $\|\cdot\|_{\tTa{}}$ has good properties which we will verify in Lemma \ref{tensor_lemma} in the case of spaces of continuously differentiable functions. Namely, we assume that it is a \emph{uniform crossnorm}, in particular that for any indices $\betap,\gammap,\epsp,\betam,\gammam,\epsm$ there is a bounded bilinear map 
\begin{align*}
 \call(\Ta{\betap},\Ta{\gammap})\times \call(\Ta{\betam},\Ta{\gammam})&\to \call(\Ta{\betap}\mix\Ta{\betam},\Ta{\gammap}\mix\Ta{\gammam}), \\
 (\aop,\bop)&\mapsto \aop\mix\bop,
\end{align*}
characterized by
\begin{align}\label{j509}
 (\aop\otimes \bop)(\ta_{+}\otimes \ta_{-}):=(\aop \ta_{+})\otimes (\bop \ta_{-}), \ \ta_{+}\in \Ta{\betap}, \ta_{-}\in \Ta{\betam},
\end{align}
and that for linear operators $\taop\in\call(\Ta{\betap}, \Ta{\epsp})$ and $\tbop\in\call(\Ta{\betam}, \Ta{\epsm})$ with
  \begin{align}
   \|\aop \ta_{+}\|_{\Ta{\gammap}} &\leq \|\taop \ta_{+}\|_{\Ta{\epsp}}, \quad \ta_{+}\in \Ta{\betap}, \label{j525}\\
   \|\bop \ta_{-}\|_{\Ta{\gammam}} &\leq \|\tbop \ta_{-}\|_{\Ta{\epsm}}, \quad \ta_{-}\in \Ta{\betam}, \label{j526}
  \end{align}
one has 
  \begin{align}\label{j522}
   \|(\aop\otimes \bop)\tta\|_{\Ta{\gammap}\mix\Ta{\gammam}} \leq \|(\taop\otimes \tbop)\tta\|_{\Ta{\epsp}\mix\Ta{\epsm}}, \quad \tta\in \Ta{\betap}\mix\Ta{\betam}.
  \end{align}
We will also use this for $\Ta{\epsp}$ or $\Ta{\epsm}$ replaced by a finite-dimensional Euclidean space.
In addition to the abstract product space $(\Aa{},\tTa{})$, we construct an associated product skeleton in the following lemma. In order to establish that the tensorization $V_+\otimes V_-\in C^0(\mathbb{R}^2;\mathsf{T})$
of two forms $V_+\in C^0(\mathbb{R}^2;\mathsf{T}_+)$ and $V_-\in C^0(\mathbb{R}^2;\mathsf{T}_-)$
is well-behaved in terms of the boundedness norms, cf.~Definition \ref{prodDef1} and \eqref{wg16}, 
we need the compatibility of $\langle\cdot\rangle$ that associates exponents to homogeneities 
in the sense of 
\begin{align}\label{wg10}
\langle\beta_+\rangle_++\langle\beta_-\rangle_- \ge \langle\beta_++\beta_-\rangle\quad\mbox{for all}\;\beta_\pm\in\mathsf{A}_\pm,
\end{align}
where we think of $\ap{\cdot}_+$ given on $\Aa{+}$, $\ap{\cdot}_-$ given on $\Aa{-}$ and $\ap{\cdot}$ given on $\Aa{}$. If there is no danger of confusion, we will write $\ap{\betap}$ instead of $\ap{\betap}_+$ and similarly for negative homogeneities.
\begin{lemma}\label{lem:prod_skel}
For $x\in\R^2$ we define $\Ga{}{}(x):= \Ga{+}{}(x)\otimes \Ga{-}{}(x)$. Then $\Ga{}{}$ is a skeleton on $(\Aa{},\Ta{})$ and we have 
\begin{align}\label{j1200}
  \|\Ga{}{}\|_{sk}\le \|\Ga{+}{}\|_{sk}+\|\Ga{-}{}\|_{sk} + \|\Ga{+}{}\|_{sk}\|\Ga{-}{}\|_{sk}.
 \end{align}
 Moreover, for all $x\in \R^2$, $\tbeta,\tgamma\in\Aa{}$ and $\ta_{\tgamma} \in \tTa{\tgamma}$ we have
\begin{align}
&\Ga{\tbeta}{\tgamma}(x)\ta_\tgamma=\bigoplus_{\betap+\betam=\beta}  \sum_{\gammap+\gammam=\gamma} \Ga{\betap}{\gammap}(x)\otimes\Ga{\betam}{\gammam}(x)\ta_{\gammap,\gammam},
\label{Prod9}
\end{align}
cf.\@ \eqref{j1009}. In addition, for $V_+\in C^0(\mathbb{R}^2;\mathsf{T}_+)$ and $V_-\in C^0(\mathbb{R}^2;\mathsf{T}_-)$,
we have
\begin{align}
\|V_+\otimes V_-\|_{\mathsf{T};\Gamma}&\le\|V_+\|_{\mathsf{T}_+;\Gamma_+}\|V_-\|_{\mathsf{T}_-;\Gamma_-},
\label{neededlater1}
\end{align}
and for any pair of exponents $\overline{\alpha}_\pm$
\begin{align}
\lefteqn{[V_+\otimes V_-]_{{\mathcal D}^{(\alphapo+\betamu)
\wedge(\alphamo+\betapu)}(\mathsf{T};\Gamma)}
\le[V_+]_{{\mathcal D}^{\overline{\alpha}_-}(\mathsf{T}_+;\Gamma_+)}
[V_-]_{{\mathcal D}^{\overline{\alpha}_+}(\mathsf{T}_-;\Gamma_-)}}\nonumber\\
&+
[V_+]_{{\mathcal D}^{\overline{\alpha}_-}(\mathsf{T}_+;\Gamma_+)}
\|V_-\|_{\mathsf{T}_-;\Gamma_-}
+\|V_+\|_{\mathsf{T}_+;\Gamma_+}
[V_-]_{{\mathcal D}^{\overline{\alpha}_+}(\mathsf{T}_-;\Gamma_-)}.\label{wg09}
\end{align}
\end{lemma}
We will not give a proof for \eqref{wg09} since it is not needed and the argument is contained
in the one for \eqref{MS3} in Lemma \ref{lem:Multiplication}.
In fact, we now want to build a form $V_{u_{+} \diamond u_{-}}$ using two modelled distributions $V_{u_{+}} \in \cald^{\alphapo}(\Tplus; \Gamma_{+})$ and $V_{u_{-}}\in \cald^{\alphamo}(\Tminus; \Gamma_{-})$ for some $\alphapo>0$ and $\alphamo<0$. In view of \eqref{neededlater1} and \eqref{wg09}, one option would be to consider the tensored form $V_{u_{+}}\otimes V_{u_{-}}$ on the product space $(\Aa{},\Ta{})$. However, it turns out that we can increase the order (as a modelled distribution) of $V_{u_{+} \diamond u_{-}}$ when the object described by the $V_{u_{+}}$ is an honest function $u_{+}$ rather than a distribution. Towards this end, we further assume that $\Aa{+}$ is of the form
\begin{align}\label{Prod1}
\Aa{+}=\{0\} \cup \Aa{+}', \quad \alpha:=\min\Aa{+}'>0.
\end{align}
The special role of the homogeneity $0$ is underlined by demanding
\begin{align}\label{cw97}
 \Ta{0}:=(\Tplus)_0=\R, \quad \na:=(\ta_{+})_{0},
\end{align}
\emph{i.e.}, by writing $\na$ instead of $(\ta_{+})_{0}$ for a typical element in $\Ta{0}=\R$. We also assume that $\Aa{-}$ consists solely of negative homogeneities, that is
\begin{align*}
 \betamo\stackrel{\eqref{j1000}}{=}\max \Aa{-}<0.
\end{align*}
Note that by \eqref{Prod1}, the reduction $C_\alpha V_{u_+}$ of a form $V_{u_+}$ on $(\mathsf{A}_+,\mathsf{T}_+)$
is an element of $C^0(\mathbb{R}^2;\mathsf{T}_0^*)$ 
and thus by \eqref{cw97} can canonically be identified with
an element of $C^0(\mathbb{R}^2)$, which we denote by $u_+$:
\begin{align}\label{cw98}
u_+ \mathsf{1}=C_\alpha V_{u_+}.\mathsf{1}.
\end{align}
Also the complementary part of this projection will play an important role, motivating the notation
\begin{align}\label{j535}
\overline{V}_{u_+}:=V_{u_+}-C_\alpha V_{u_+}.
\end{align}
 
\begin{remark}\label{wg03}
Let $v_+\in C^0(\mathbb{R}^2;\mathsf{T}_+)$ be characterized by
\begin{align}\label{j700}
\Gamma_\beta v_+\equiv0\;\;\mbox{for}\;\;\beta\in \Aa{+}'\quad\mbox{and}\quad v_0:=(v_+)_0\equiv 1.
\end{align}
Then $u_+$, cf.~\eqref{cw98}, may also be recovered as
\begin{align}\label{j1003}
u_+=C_\eta V_{u_+}.v_+=V_{u_+}.v_+
\end{align}
for all $\eta\in(\underline{\beta}_+,\overline{\beta}_+]$.
Moreover, \eqref{prod5} for $\beta=0$ turns into 
\begin{align}\label{cw91}
\lefteqn{\big|(u_+(y)-u_+(x))}\nonumber\\
&-\overline{V}_{u_+}(x).(v_+(y)-v_+(x))\big|
\le [V]_{{\mathcal D}^{\overline{\alpha}_+}(\mathsf{T}_+;\Gamma_+)}d^{\overline{\alpha}_{+}}(y,x)
\end{align}
for all $x,y\in\mathbb{R}^2$ with $d(y,x)\le 1$. Inequality \eqref{cw91} states that the increments of 
$u_+\in C^0(\mathbb{R}^2)$ follow the increments of $v_+\in C^0(\mathbb{R}^2;\mathsf{T}_+)$, 
modulated by $\overline{V}_{u_+}$. Therefore, $v_+$ should be interpreted as a model, and then \eqref{cw91}
is an instance of Gubinelli's controlled rough path condition \cite{GUBINELLI200486}.
For later purpose, we note that we obtain immediately from Definition \ref{prodDef1} and
\eqref{cw98}
\begin{align}\label{wg31}
\|u_+\|= N^{\langle 0\rangle}\|C_\alpha V_{u_+}\|_{\mathsf{T}_+;\Gamma_+},\quad
[u_+]_\alpha=[C_\alpha V_{u_+}]_{{\mathcal D}^{\alpha}(\mathsf{T}_+;\Gamma_+)}.
\end{align}
\end{remark}
We display the simple computation that derives \eqref{cw91} from \eqref{cw90} at the beginning of Section \ref{6}.

\medskip

%
The projections \eqref{cw98} and \eqref{j535} will prove to be crucial in order to obtain a higher order for the modelled distribution $V_{u_{+}\diamond u_{-}}$ in Lemma \ref{lem:Multiplication}, and suggests to use the original modelled distribution $V_{u_{-}}$ to construct an extended space $(\oAa{},\oTa{})$ together with an extended skeleton $\oGa{}{}$: 
Let
\begin{align}\label{j982}
\kappa :=(\alphapo+\betamu) \wedge (\alpha+\alphamo),
\end{align}
and assume
\begin{align}\label{jj3}
\kappa-\alpha\notin \Aa{} \quad \mathrm{ and } \quad 
\kappa-\alpha>\betamo. 
\end{align}
Define the extended product space $(\oAa{},\oTa{})$ via
\begin{align}\label{j1102}
 \oAa{}&:=\set{\kappa-\alpha}\cup \Aa{}, \quad \oTa{}:= \oTa{\kappa-\alpha}\oplus \Ta{} \quad \text{with } \oTa{\kappa-\alpha}:=\frac{1}{N_*}\R,
\end{align}
where it is convenient to have a parameter $N_*>0$ in the norm of $\oTa{\kappa-\alpha}$, cf.\@ \eqref{jj4}. Moreover, we extend the skeleton $\oGa{}{}$ by
\begin{align}\label{j980}
\begin{split}
 \oGa{\beta}{}\begin{pmatrix}\mathsf{u}_{-} \\ \tta \end{pmatrix}&:=\Ga{\beta}{}\tta,  \quad \beta\neq \kappa-\alpha, \\
 \oGa{\kappa-\alpha}{}\begin{pmatrix}\mathsf{u}_{-} \\ \tta \end{pmatrix}&:=\mathsf{u}_{-}- V_{u_{-}}.\ta_{-},
\end{split}
\end{align}
where $\ta_{-}:=(\tta_{0,\betam})_{\betam\in \Aa{-}}$, cf.\@ \eqref{j1009}. We will show in Lemma \ref{lem:Multiplication} that this does indeed define a skeleton on $(\oAa{},\oTa{})$.

Now we define $V_{u_{+} \, \diam \, u_{-}} \in C^{0}(\R^{2};\oTa{}^*)$ by
\begin{align}
V_{u_{+} \, \diam \, u_{-}}.
\begin{pmatrix}
\mathsf{u}_{-}\\
\tta
\end{pmatrix}
&:=u_{+}\mathsf{u}_{-}+\overline{V}_{u_{+}}\otimes V_{u_{-}}.\tta,\label{MS1}
\end{align}
where $u_+$ and $\overline{V}_{u_{+}}$ are defined by \eqref{cw98} and \eqref{j535}, so that we have the explicit formula
\begin{align*}
 V_{u_{+} \, \diam \, u_{-}}.
\begin{pmatrix}
\mathsf{u}_{-}\\
\tta
\end{pmatrix}
&:=(C_\alpha V_{u_+}.\nf)\mathsf{u}_{-}+(V_{u_+}-C_\alpha V_{u_+})\otimes V_{u_{-}}.\tta.
\end{align*}
%

We are now in the position to state the abstract multiplication lemma.
\begin{lemma}\label{lem:Multiplication}
Let $V_{u_{+}}\in \cald^{\alphapo}(\Tplus; \Gamma_{+})$, $V_{u_{-}}\in \cald^{\alphamo}(\Tminus; \Gamma_{-})$, $N_*>0$ and assume $\kappa$ defined by \eqref{j982} satisfies \eqref{jj3}. Then $\oGa{}{}$ defined via \eqref{j980} is a skeleton on $(\oAa{},\oTa{})$ with
\begin{align}\label{j1004}
 \|\oGa{}{}\|_{sk}\le \|\Ga{}{}\|_{sk}+ \frac{1}{N_*}[V_{u_{-}}]_{\cald^{\alphamo}(\Tminus; \Gamma_{-})}.
\end{align}
In addition, $V_{u_{+}\, \diam \, u_{-}}$ defined by \eqref{MS1} belongs to $\cald^{\kappa}(\oTa{};\oGa{}{})$ and more precisely, 
%
the following hold
\begin{enumerate}
\item form boundedness: for all $x \in \R^{2}$ and $(\mathsf{u}_{-},\tta) \in \oTa{}$ it holds
\begin{align}
\quad \Big| &V_{u_{+} \diamond u_{-}}(x) .
\begin{pmatrix}
\mathsf{u}_{-} \\
\tta\\ 
\end{pmatrix}
\Big| \nonumber \\
& \leq N_*\|u_{+}\| \Big\|\oGa{\kappa-\alpha}{}(x)
\begin{pmatrix}
\mathsf{u}_{-} \\
\tta\\ 
\end{pmatrix}\Big\|_{\oTa{\kappa-\alpha}}
+ M_{u_{+} \diamond u_{-}}^b \sum_{\tbeta\in\Aa{}} N^{\ap{\beta}}\|\tGa{\tbeta}{}(x)\tta \|_{\Ta{\beta}}\label{MS2}.
\end{align}
\item form continuity: for all $x,y \in \R^{2}$ with $d(y,x) \leq 1$ and all $(\mathsf{u}_{-}, \tta)\in \oTa{}$ it holds
\begin{align}
\Big |(&V_{u_{+} \, \diam \, u_{-}}(y) - V_{u_{+} \, \diam \, u_{-}}(x)).
\begin{pmatrix}
\mathsf{u}_{-} \\
\tta
\end{pmatrix}
\Big |\nonumber\\
&\leq N_*[u_{+}]_{\alpha}d^{\alpha}(y,x) \Big\|\oGa{\kappa-\alpha}{}(x)
\begin{pmatrix}
\mathsf{u}_{-} \\
\tta\\ 
\end{pmatrix}\Big\|_{\Ta{\kappa-\alpha}}
\nonumber\\
&\quad 
+M_{u_{+} \diamond u_{-}}^c\sum_{\tbeta\in\Aa{}}d^{(\kappa-\tbeta)\vee 0}(y,x)\|\tGa{\tbeta}{}(x)\tta\|_{\tTa{\tbeta} },\label{MS3}
\end{align}
\end{enumerate}

where the constants $M_{u_{+} \diamond u_{-}}^b$ and $M_{u_{+} \diamond u_{-}}^c$ are given by
\begin{align}
 M_{u_{+}\, \diam \, u_{-}}^b&:=
 \bd{V_{u_{+}}}{\Tplus}{\Ga{+}{}} \bd{V_{u_{-}}}{\Tminus}{\Ga{-}{}}, \label{S300} \\
 M_{u_{+}\, \diam \, u_{-}}^c&:= 
 \big(N^{\min_{\beta \in \Aa{+}'}\ap{\beta}}\|\overline{V}_{u_+}\|_{\mathsf{T}_+;\Gamma_+}+[u_+]_\alpha+
[V_{u_+}]_{{\mathcal D}^{\overline{\alpha}_+}(\mathsf{T}_+;\Gamma_+)}\big)
[V_{u_-}]_{{\mathcal D}^{\overline{\alpha}_-}(\mathsf{T}_-;\Gamma_-)}\nonumber\\
&\quad +N^{\min_{\beta \in \Aa{-}}\ap{\beta}}[V_{u_+}]_{{\mathcal D}^{\overline{\alpha}_+}(\mathsf{T}_+;\Gamma_+)}
\|V_{u_-}\|_{\mathsf{T}_-;\Gamma_-}. \label{S301}
\end{align}
\end{lemma}



Combining the reconstruction result in Proposition \ref{rec_lem} and the multiplication result in Lemma \ref{lem:Multiplication}, we obtain the following corollary.
\begin{corollary}\label{lem:ExMult}
 Under the hypotheses and with the notation of Lemma \ref{lem:Multiplication}, assume $\kappa>0$. Suppose there exists $\tf_{+} \diamond \tf_{-} \in \cals'(\R^{2}; \Ta{<\kappa})$ such that for each $\beta<\kappa$, $x \in \R^{2}$ and $T \leq 1$
\begin{align}\label{j770}
\left\|\Ga{\beta}{}(x)  (\tf_{+} \diamond \tf_{-})_{T}(x) \right\|_{\Ta{\beta}} \leq \Tc{\beta}.
\end{align}
Following \eqref{j1009}, we denote $((\tf_{+} \diamond \tf_{-})_{0,\betam})_{\betam\in\Aa{-}}$ by $\tf_{-}\in \cals'(\R^2; \Tminus)$, which is consistent notation in view of $0\in \Aa{+}$ and $(\tf_{+})_0 \equiv \nf$.
Suppose there is a distribution $u_{-}\in \cals'(\R^2)$ such that for all $x \in \R^{2}$ and $T \leq 1$
\begin{align}\label{j771}
&\left\| (u_{-})_{T}(x)-V_{u_{-}}(x).(\tf_{-})_{T}(x) \right \|_{\oTa{\kappa-\alpha}}  \leq (T^{\frac{1}{4}})^{\kappa-\alpha}.
\end{align}

Then there exists a unique distribution $u_{+} \diamond u_{-}\in \cals'(\R^2)$ such that for all $T\leq 1$
\begin{align}\label{j772}
\Big \| (u_{+}\diam u_{-})_T- \cut{\kappa}V_{u_{+}\diamond u_{-}}.
 \begin{pmatrix}
  (u_{-})_{T} \\
  (\tf_{+} \diamond \tf_{-})_{T}
 \end{pmatrix}
 \Big\| \lesssim M_{u_+\diamond u_-, \kappa}\Tc{\kappa},
\end{align}
where
\begin{align*}
 &M_{u_+\diamond u_-, \kappa} := \\
 &(1+\|\oGa{}{}\|_{sk})^{|\Aa{}|-1} \big(N_*[u_+]_{\alpha}+M_{u_{+}\diamond u_{-}}^c   + N^{\min_{\beta\ge \kappa}\ap{\beta}}\|\oGa{}{}\|_{sk}M_{u_{+}\diamond u_{-}}^b\big).
\end{align*}
Moreover, for each $\eta<\kappa$ the following suboptimal bound holds
\begin{align}\label{S870}
\Big \| &(u_{+}\diam u_{-})_T- u_+(u_-)_{T} - \cut{\eta}(\overline{V}_{u_+}\otimes V_{u_{-}}).
  (\tf_{+} \diamond \tf_{-})_{T}
 \Big\| \nonumber \\
 & \lesssim M_{u_+\diamond u_-, \eta}\Tc{\eta},
\end{align}
where
\begin{align}\label{S452}
 M_{u_+\diamond u_-, \eta}:= M_{u_+\diamond u_-, \kappa} + N^{\min_{\beta\ge \eta}\ap{\beta}} \|\overline{V}_{u_{+}}\|_{\Ta{+};\Ga{+}{}}\|V_{u_{-}}\|_{\Ta{-};\Ga{-}{}}.
\end{align}
\end{corollary}
\begin{remark}
In \cite{OtW16}, it was emphasized that the quantitative output of the reconstruction theorem, cf.\@ \eqref{cw40} could be formulated in terms of commutators of mollification by $(\cdot)_{T}$ and multiplication with a function in the positive model.  It should be noted that under the assumptions of Remark \ref{wg03}, it is still true that the output \eqref{S870} of Corollary \ref{lem:ExMult} can be interpreted as a commutator estimate. Note that
\begin{align*}
\overline{V}_{u_{+}} \otimes V_{u_{-}}.(v_{+} \diamond v_{-})_{T}&\stackrel{\eqref{j535}}{=}V_{u_{+}} \otimes V_{u_{-}}.(v_{+} \diamond v_{-})_{T}-u_{+}V_{u_{-}}.(v_{-})_{T}\\
&\stackrel{\eqref{j1003}}{=}V_{u_{+}} \otimes V_{u_{-}}.\big ( (v_{+} \diamond v_{-})_{T}-v_{+} \otimes (v_{-})_{T} \big).
\end{align*}
Thus, \eqref{S870} may be re-cast as 
\begin{align*}
\Big \| &(u_{+}\diam u_{-})_T- u_+(u_-)_{T} - \cut{\eta}\big ( V_{u_{+}} \otimes V_{u_{-}} \big ).\big ( (v_{+} \diamond v_{-})_{T}-v_{+} \otimes (v_{-})_{T} \big)
 \Big\| \nonumber \\
 & \lesssim M_{u_+\diamond u_-, \eta}\Tc{\eta},
 \end{align*}
 which emphasizes that the commutator $(u_{+}\diam u_{-})_T- u_+(u_-)_{T}$ is locally described by the commutator $(v_{+} \diamond v_{-})_{T}-v_{+} \otimes (v_{-})_{T}$, upon modulation by the modelled distribution $V_{u_{+}} \otimes V_{u_{-}}$.
\end{remark}
\section{Integration}\label{S:int}
We now state a general tool for obtaining \emph{a priori} estimates on the solution of a rough PDE.  Recall that we write $\interval{}=[\lambda,\lambda^{-1}]$ for the interval of ellipticity.
\begin{proposition}[Integration]\label{int lem}
Let $\kappa\in (1,2)$ and $\Aa{}\subset [0,\kappa)$ finite. Let $U$ be a bounded function of $(x,y)\in\mathbb{R}^2\times\mathbb{R}^2$ that is periodic and continuous
in the $y$-variable.
Assume that there exists an $I$-valued function $a$ of $x\in\mathbb{R}^2$ and a constant $\overline{M}$ such that for all base points $x$ and length scales $L\le 1$, $T^{\frac14}\le 1$ it holds that
\begin{equation}\label{KS1}
\Tc{2}\inf_{c\in\R}\|(\partial_{2}-a(x)\partial_{1}^{2})U_{T}(x,\cdot)-c\|_{B_{L}(x)}\le \overline{M}\sum_{\beta\in\Aa{}}\Tc{\kappa}L^{\kappa-\beta},
\end{equation} 
where $\|\cdot\|_{B_{L}(x)}$ denotes the supremum norm on the ball $B_L(x)$ with radius $L$ around the point $x$ with respect to the metric $d$.

\medskip

Assume there also exists $\gamma \in C^{0}(\R^{2}\times \R^{2})$ and a constant $\overline{\overline{M}}$ such that for any $x,y,z \in \R^{2}$ there holds the ``three-point continuity''
\begin{align}\label{KS2}
\big |U&(x,z)-U(x,y)-U(y,z)+U(y,y)-\gamma(x,y)(z-y)_{1} \big | \nonumber\\
&\le \overline{\overline{M}}\sum_{\beta \in \Aa{}} d^{\beta}(y,x)d^{\kappa-\beta}(z,y).
\end{align}

\medskip

Then there exists a function $\nu$ of $x\in\mathbb{R}^2$ such that
\begin{align}\label{KS3}
M_U&:=\sup_{x\not=y}d^{-\kappa}(y,x)|U(x,y)-U(x,x)-\nu(x)(y-x)_1|\nonumber\\
&\lesssim \overline{M}+\overline{\overline{M}}.
\end{align}
Moreover, we have
\begin{align}
\|\nu\|&\le M_U\stackrel{\eqref{KS3}}{\lesssim}\overline{M}+\overline{\overline{M}},\label{wg52}\\
\sup_{x\not=y}d^{-(\kappa-1)}(y,x)|\nu(y)-\nu(x)+\gamma(x,y)|&\lesssim M_U+\overline{\overline{M}}
\stackrel{\eqref{KS3}}{\lesssim}\overline{M}+\overline{\overline{M}}.\label{KS4}
\end{align}
Here, $\lesssim$ means up to a multiplicative constant that only
depends on $\lambda$, $\kappa$ and $\Aa{}$.
\end{proposition}
\medskip

A couple of useful observations on the function $\nu$ are collected in the next remark,
the proof of which is the last step in the proof of Proposition \ref{int lem}.

\begin{remark}\label{rem-wg53}
Under the assumptions and with the notation of Proposition \ref{int lem}, the function $\nu$ is uniquely
determined through $U$. Moreover, if $U$ is also periodic in the $x$-variable, so is $\nu$. Finally,
if $0\notin \Aa{}$, then $\nu$ is H\"older continuous with exponent $(\kappa-1)\wedge\min \Aa{}$.  

\medskip

 In the proof of Proposition \ref{int lem}, cf. Step \ref{intLemQH}, we show that defining $\nu^{\tau}(x):=\frac{\partial}{\partial y_2}U_\tau(x,x)$, where $U_{\tau}$ denotes convolution in the second argument only, \eqref{KS3} and \eqref{wg52} hold uniformly in $\tau$ in the sense that
\begin{align}
&\sup_{\tau \leq \frac{1}{2}}\sup_{x\not=y}d^{-\kappa}(y,x)|U_{\tau}(x,y)-U_{\tau}(x,x)-\nu^{\tau}(x)(y-x)_1|\nonumber\\
&\quad \lesssim \overline{M}+\overline{\overline{M}},\label{KS6}\\
&\sup_{\tau \leq \frac{1}{2}}\|\nu^{\tau}\|\lesssim \overline{M}+\overline{\overline{M}}.\label{KS7}
\end{align} 
\end{remark}

Let us comment a bit more on Proposition \ref{int lem} and its application. For every $x$, we are given
a function $U(x,\cdot)\colon y\mapsto U(x,y)$, which one should think of as the remainder of a generalized Taylor's expansion
to order $\kappa$. On the one hand, we assume that $U(x,\cdot)$ satisfies $(\partial_2-a(x)\partial_1^2)U(x,\cdot)$
$\approx 0$, cf.~\eqref{KS1}. On the other hand, up to ``affine functions'' (\emph{i.e.}, functions of the form 
$y\mapsto c+\nu y_1$ which are exactly the ``caloric'' polynomials of order less than $\kappa\in(1,2)$ in our
parabolic metric), $U(x,\cdot)$
depends continuously on $x$, cf.~\eqref{KS2}. As a result, again up to affine functions, $U(x,y)$ is indeed small
of order $\kappa$ in $d(y,x)$, cf.\@ \eqref{KS3}. Moreover, the two affine functions, namely $\gamma$ appearing in
\eqref{KS2} and $\nu$ appearing in \eqref{KS3} are related to order $\kappa-1$, cf.~\eqref{KS4}. 

\medskip

Notice that both assumptions express the order $\kappa$ in a very general way: Assumption \eqref{KS1}
measures $(\partial_2-a(x)\partial_1^2)U(x,\cdot)$ in a ``graded'' negative H\"older norm, cf.~\eqref{j1001}, 
where the order $\kappa-\beta$ of vanishing in the size $L$ of the neighborhood of $x$
has to compensate the (negative) exponent $\kappa-2$ of the H\"older norm.
Assumption \eqref{KS2} measures the increment $U(x,\cdot)-U(y,\cdot)$ in a graded positive H\"older norm,
where the order $\beta$ of vanishing in the distance $d(y,x)$ has to compensate the positive but possibly
low H\"older exponent $\kappa-\beta$. It will be important that \eqref{KS2} does not involve four points in general position,
but only three positions, thus we call it the ``three-point continuity''.

\medskip

Proposition \ref{int lem} is the only PDE ingredient of our main result, \emph{i.e.}, Theorem \ref{Prop:aPriori},
which estimates the form $V_u$ describing the solution $u$ in terms of the form $V_a$ describing the coefficients $a$.
It only enters in the proof of Lemma \ref{int_lem_concr} below in the following way: The form $V_u$ can be
``algebraically'' constructed from $V_a$ by a formula \eqref{s300bis}. However,
this formula for $V_u$ involves the functions $u$ and $\nu$ (the latter plays the role of $\partial_1 u$
in the regular setting) next to the $V_a$. Since these functions are \emph{a priori} uncontrolled, the form continuity
of $V_u$ (as opposed to $V_{\puf}$, cf.\@ \eqref{algebraicRel}) does not entirely follow from that of $V_a$ -- two ingredients are missing with respect to the hierarchy
of controlled rough path conditions encoded in \eqref{prod5}. In Lemma \ref{int_lem_concr}, 
these two ingredients are recovered from the two estimates \eqref{KS3} and \eqref{KS4}.
\section{Existence and Uniqueness}\label{S:App}
\subsection{Statements of the Concrete Results}
With Theorem \ref{Prop:aPriori} we now state the main concrete result of this paper.  
Its statement refers to definitions and assumptions of two later sections.
Namely, in Section \ref{fspc} we specify the abstract space $(\Aa{+}=\{0,\alpha,1,2\alpha\},\Ta{+})$ and its skeleton $\Ga{+}{}$,
from which $(\Aa{-}=\{\alpha-2,2\alpha-2\},\Ta{-},\Ga{-}{})$ will be obtained functorially (i.~e.~by formally applying $\partial_1^2$),
and from which $(\mathsf{A},\mathsf{T},\Ga{}{})$ will be obtained by multiplication, cf.~Section \ref{S:mult}.
In Section \ref{subsec:theModel}, we specify the appropriate models
$v_+\in C^0(\mathbb{R}^2,\mathsf{T}_+)$ and $v\in C^0(\mathbb{R},\mathsf{T}_{<0})$ and quantify 
their amplitude through a single parameter $N$, which is mostly hidden in the norm we endow $\mathsf{T}_+$ with.  

\medskip

Equipped with $(\Aa{+},\Ta{+},\Ga{+}{})$, we have at hand spaces of modelled distributions of varying order, 
along with associated norms and semi-norms, cf.\@ Definition \ref{prodDef1}, which will be used to formulate the \emph{a priori} bounds below.  
In the following, we will use a subscript to indicate the function or distribution the modelled distribution describes,
and a superscript to denote its order.  This allows us to ease notation by omitting the subscript in the corresponding norms 
and semi-norms, writing for instance $[V_{a}^{\eta}]$ and $\|V_{a}^{\eta}\|$ instead of 
$\|V_{a}^{\eta}\|_{\mathcal{D}^{\eta}(\Ta{+};\Ga{+}{})}$ and $\|V_{a}^{\eta}\|_{\Ta{+};\Ga{+}{}}$.  

\medskip

\begin{theorem}\label{Prop:aPriori}
Let 
\begin{align}\label{j1220}
 \alpha\in \big( \frac12,\frac23 \big), \quad \mathrm{} \quad \thresh\in \vp{2-\alpha,\alpha+1},
\end{align}
and $\lambda >0$.  Define $I:=[\lambda,\lambda^{-1}]$ and assume we are given periodic model inputs 
\begin{align*}
f &\in \cals'(\R^2), &&\vvf \in \cals'(\R^2, C^{2,2}(\interval{2})), \\
\wvf &\in \cals'(\R^2, C^{2,1,2}(\interval{3})), &&\vwf\in \cals'(\R^2, C^{2,2,1}(\interval{3}))
\end{align*}
satisfying the off-line assumptions \eqref{f29}, \eqref{off1}, and \eqref{f101} with respect to a parameter $N \ll 1$, where here and
in the sequel $\ll$ means $\le c$ for some sufficiently small positive constant $c$ only depending on $\alpha$, $\eta$, and $\lambda$.  

\medskip

Let $V_{a}^{\thresh}\in \mathcal{D}^{\thresh}(\Tplust; \Ga{+}{})$ be periodic in the sense of \eqref{cw14} and assume that $a:=C_{\alpha}V_{a}^{\eta}$ is $I$-valued with $[a]_{\alpha} \ll 1$.  Then there exists a unique $V_{u}^{\eta} \in \mathcal{D}^{\eta}(\Tplust; \Ga{+}{})$ with the following three properties.
\begin{enumerate}
\item The modelled distribution $V_{\partial_{1}^{2}u}^{\eta-2} \in C^0(\mathbb{R}^2;\mathsf{T}_-^*)$ defined functorially from $V_{u}^{\eta}$ by
\begin{equation}\label{functDef}
V_{\partial_1^2u}^{\eta-2}.\ta_{-}=V_{u}^\thresh.(0,\pva,0,\pwa),
\end{equation}
for $\ta_{-}=(\pva,\pwa) \in \Ta{-}$, is algebraically determined by $V_{a}^{\eta}$ via
\begin{equation}
V_{\partial_{1}^{2}u}^{\eta-2}.\mathsf{v}_-
=\delta_{a}.\partial_1^2\mathsf{v}_{\alpha}+\overline{V}_{a}^{\eta-\alpha} \otimes \delta_{a}.
\begin{pmatrix}
\partial_{a_0}\partial_1^2\mathsf{v}_{\alpha} \\
\partial_1^2\mathsf{w}_{2\alpha}
\end{pmatrix}
,\label{algebraicRel}
\end{equation}
where $V_{a}^{\eta-\alpha}:=C_{\eta-\alpha}V_a^\eta$ and $\overline{V}_a^{\eta-\alpha}:=(\id - C_{\alpha})V_{a}^{\eta-\alpha}$.
\item There exists a (periodic) distribution $a \diamond \partial_{1}^{2}u \in \mathcal{S}'(\R^{2})$ such that for some $\delta>0$,
\begin{align}
\lim_{T \to 0}(T^{\frac{1}{4}})^{-\delta}\|(\auf)_T -a(\partial_{1}^{2}u)_{T}-C_{\eta+\alpha-2}\big ( \overline{V}_{a}^{\eta} \otimes V_{\partial_{1}^{2}u}^{\eta-2} \big ).v_{T}\|=0,\label{MT34}
\end{align}
where $u:=C_{\alpha}V_{u}^{\eta}$, $\overline{V}_{a}^{\eta}:=(\id - C_{\alpha})V_{a}^{\eta}$, and $v$ is the abbreviation
for the vector appearing in \eqref{j531}.
\item The function $u$ is mean free and satisfies
\begin{equation}
\partial_2 u- P(a\diam\partial_1^2 u)=Pf \quad \mathrm{distributionally}.\label{MT33}
\end{equation}
\end{enumerate}

\medskip

Moreover, the (non-linear) mapping from $V_{a}^{\eta}$ to $V_{u}^{\eta}$ is bounded in the sense that the following estimates hold:
\begin{align}
[V_{u}^{\eta}]&\lesssim N\big ([a]_{\alpha}([a]_{\alpha}+N\|\overline{V}_{a}^{\eta-\alpha}\|)+[V_{a}^{\eta-\alpha}] \big) \big (1+N\|\overline{V}_{a}^{\eta}\| \big)\nonumber\\
&+N[V_{a}^{\eta}]\big(1 + \|\overline{V}_{a}^{\eta-\alpha}\|+[V_{a}^{\eta-\alpha}] \big )\nonumber\\
&+N^{2}\left ([a]_{\alpha}+N^{}\|\overline{V}_{a}^{\eta}\|+N^{2}\|V_{a}^{\eta}\|\right)(1 + \|\overline{V}_{a}^{\eta-\alpha}\|),\label{S306}\\
\|V_{u}^{\eta} \|&\lesssim N^{-1}\big ([a]_{\alpha}([a]_{\alpha}+N\|\overline{V}_{a}^{\eta-\alpha}\|)+[V_{a}^{\eta-\alpha}] \big) \big (1+N\|\overline{V}_{a}^{\eta}\| \big)\nonumber\\
&+N^{-1}[V_{a}^{\eta}]\big(1 + \|\overline{V}_{a}^{\eta-\alpha}\|+[V_{a}^{\eta-\alpha}] \big )\nonumber\\
&+\left (1+N^{}\|\overline{V}_{a}^{\eta}\|+N^{2}\|V_{a}^{\eta}\|\right)(1 + \|\overline{V}_{a}^{\eta-\alpha}\|),\label{SFin4}
\end{align}
where here and in the sequel $\lesssim$ denotes inequality up to a universal constant depending only on $\alpha$, $\eta$, and $\lambda$.
\end{theorem}

Item (1) in the above definition of a solution $V_u^\eta$ might appear strange: The form
$V_{\partial_1^2u}^{\eta-2}$ on $\mathsf{T}_-$ describing the distribution $\partial_1^2u$,
which is naturally (or ``functorially'') derived from $V_u^\eta$ through \eqref{functDef},
is completely determined by $V_a^\eta$ through identity \eqref{algebraicRel}, which we call
algebraic since it is pointwise in $x$. We motivate in Subsection \eqref{subsec:Heuristics} why this is natural.

\medskip

We also note that compared to a classical reading, equation \eqref{MT33} involves a {\it renormalization}, which becomes apparent by considering
what \eqref{MT33} turns into under the approximation argument for existence, cf.~Step \ref{MTprf_2} of the Theorem \ref{Prop:aPriori}. 
The approximation argument proceeds via (semi-group) convolution $(\cdot)_\tau$ of $f$ and the model $v$. The counter term 
$C_{\eta+\alpha-2}(\overline{V}_a^\eta\otimes V_{\partial_1^2u}^{\eta-2}).v_\tau$ in \eqref{MT17}
depends algebraically, and thus in particular locally, on $V_a^\eta$ through \eqref{algebraicRel}.

\medskip

The following corollary shows that the detailed \emph{a priori} estimates \eqref{S306} and \eqref{SFin4} of Theorem \ref{Prop:aPriori} are sufficient to
establish the self-mapping property in a fixed-point argument for a quasi-linear setting. Here we have in mind the simplest
non-trivial situation of $a=u+1$, which is consistent with choosing $V_a^\eta.\mathsf{v}=V_u^\eta.\mathsf{v}+\mathsf{1}$ so that
$\overline{V}_a^\eta=\overline{V}_u^\eta$ and $[V_{a}^{\eta}]=[V_{u}^{\eta}]$.

\begin{corollary}\label{cor:wg}
Suppose that for some $\epsilon\leq 1$ we have
\begin{align}\label{wg32}
N\le\epsilon^{3},\;[V_a^\eta]\le \epsilon^{-4}N^3,\;\|V_a^\eta\|\le\epsilon^{-1},\; \|\overline{V}_{a}^{\eta-\alpha}\|\le 1.
\end{align}
Then under the assumptions of Theorem \ref{Prop:aPriori}, we have
\begin{align}\label{wg34}
[V_u^\eta]\lesssim \epsilon^{-2} N^3,\;\|V_u^\eta\|\lesssim 1, \|\overline{V}_u^{\eta-\alpha}\|\leq 1.
\end{align}
\end{corollary}

Theorem \ref{Prop:aPriori} essentially relies
\begin{itemize}
\item on Lemma \ref{cor:prod1}, which provides \emph{a priori} bounds on $V_{\partial_1^2u}^{\eta-2}$
in terms of $V_a^\eta$, based on the algebraic relation \eqref{algebraicRel}, 
\item on the concrete reconstruction result, Lemma \ref{rec_lem_concr}, which provides a characterization of 
$a\diamond\partial_1^2u$ in terms of the product form $V_{a\diamond\partial_1^2u}^{\eta+\alpha-2}$
of $V_a^\eta$ and $V_{\partial_1^2u}^{\eta-2}$ in the sense of \eqref{MS1}, with estimates in terms of
$V_{\partial_1^2u}^{\eta-2}$, of $V_a^\eta$, and a bit of $V_u^\eta$,
\item on the concrete integration result, Lemma \ref{int_lem_concr}, which provides the \emph{a priori} estimates
on $V_{u}^\eta$ in terms of $V_{\partial_1^2u}^{\eta-2}$ and of (a suboptimal version of)
the estimate on $a\diamond\partial_1^2u$ from Lemma \ref{rec_lem_concr}.
\end{itemize}
This loop in the two last items reflects that the proof of the \emph{a priori} estimates in Theorem \ref{Prop:aPriori}
requires a {\it buckling} argument, which relies on $[a]_\alpha\ll1$.

\medskip

The following lemma states the \emph{a priori} bound \eqref{S565} on $V_{\partial_1^2u}^{\eta-2}$, on the level
of both form boundedness and form continuity, in terms of $V_a^\eta$, or rather $V_a^{\eta-\alpha}$. It relies
on a bound on $V_{a\diamond\partial_1^2u}^{\eta-2}$ following from the abstract multiplication result Lemma 
\ref{lem:Multiplication}.

\begin{lemma}\label{cor:prod1}
With the above notation and under the off-line assumptions of Theorem \ref{Prop:aPriori}, we are given
$V_a^\eta\in{\mathcal D}^\eta(\mathsf{T}_+;\Ga{+}{})$. Then
$V_{\partial_1^2u}^{\eta-2}$ defined through \eqref{algebraicRel} satisfies
\begin{equation}
\|V_{\partial_{1}^{2}u}^{\eta-2}\| \leq 1 + \|\overline{V}_{a}^{\eta-\alpha}\|, \quad 
[V_{\partial_{1}^{2}u}^{\eta-2}] \leq N[a]_{\alpha}^{2}+M_{a,\eta-\alpha},\label{S565}
\end{equation}
where
\begin{equation}\label{IOR30}
M_{a,\eta-\alpha}:=\big( [a]_{\alpha}+N\|\overline{V}_{a}^{\eta-\alpha}\|
+[V_{a}^{\eta-\alpha}] \big )[V_{\partial_{1}^{2}u}^{\eta-\alpha-2}] 
+N[V_{a}^{\eta-\alpha}]\|V_{\partial_{1}^{2}u}^{\eta-\alpha-2}\|,
\end{equation}
and where $V_{\partial_1^2u}^{\eta-\alpha-2}$ $:=C_{\eta-\alpha-2}V_{\partial_1^2u}^{\eta-2}$ satisfies
\begin{equation}
\|V_{\partial_{1}^{2}u}^{\eta-\alpha-2}\| \leq 1, \quad [V_{\partial_{1}^{2}u}^{\eta-\alpha-2}] \leq N[a]_{\alpha} . \label{SFin1}
\end{equation}  
%
%
%
%
%
%
%
%
%
\end{lemma}

The primary ingredient for the following lemma is the abstract integration result Proposition \ref{int lem}. 
Note that it is obvious from \eqref{algebraicRel} and \eqref{s300bis} that $V_{\partial_1^2u}^{\eta-2}$
is functorially related to $V_u^\eta$, cf.~\eqref{functDef}.

\begin{lemma}\label{int_lem_concr}
With the above notation and under the off-line assumptions of Theorem \ref{Prop:aPriori} and those related to $V_a^\eta$,
let $u\in C^{0}(\R^{2})$, $a\,\diam\,\partial_1^2 u \in \cals'(\R^{2})$ be periodic and related by
\begin{align} \label{ss11}
\partial_2 u- P(a\diam\partial_1^2 u)=Pf \quad \mathrm{distributionally.}
\end{align}
Define the product form $V_{a \diamond \partial_{1}^{2}u}^{\eta-2}$ from $V_{a}^{\eta-\alpha}$ and $V_{\partial_{1}^{2}u}^{\eta-\alpha-2}$ via \eqref{MS1} and assume there exists a constant $M_{\aufi,\eta-2}$ such that for all $T \leq 1$ 
\begin{align}
\left\| (\auf)_T -V_{\aufi}^{\thresh-2}.
\begin{pmatrix}
(\puf)_{T}\\
(\pvf)_{T}\\
(\vvf)_{T}
\end{pmatrix} \right\|
&\lesssim M_{\aufi,\eta-2}\Tc{\eta-2}. \label{s30}
\end{align}
Then there exists a periodic function $\nu \in C^{0}(\R^{2})$ such that $V_{u}^{\thresh}$ defined 
for $\mathsf{v}_+=(\mathsf{1},\mathsf{v}_\alpha,\mathsf{v}_1,\mathsf{w}_{2\alpha})\in\mathsf{T}_+$ by
\begin{align}\label{s300bis}
V_u^\eta.\mathsf{v}_+&:=u\mathsf{1}+\delta_a.(\mathsf{v}_\alpha-v_{\alpha}\mathsf{1}) +\nu(\mathsf{v}_1-v_1\mathsf{1}) \nonumber\\
&\quad +\overline{V}_a^{\eta-\alpha}\otimes\delta_a.{\partial_{a_0}\mathsf{v}_\alpha-\partial_{a_0}v_\alpha\mathsf{1}\choose
\mathsf{w}_{2\alpha}-w_{2\alpha}\mathsf{1}}
\end{align}
belongs to $\mathcal{D}^{\thresh}(\Tplust; \Ga{+}{})$ and the following estimates hold
\begin{align}
[V_{u}^{\thresh}] &\lesssim M_{\aufi,\eta-2} +[a]_{\alpha}[V_{u}^{\eta-\alpha}]+M_{a,\eta-\alpha}+[V_{\partial_{1}^{2}u}^{\eta-2}], \label{s16}\\ 
\|\nu \| &\lesssim [V_{u}^{\thresh}], \label{SFin3}\\
\|V_{u}^{\eta}\| &\lesssim \|u \|+N^{-2}\|\nu\|+\|V_{\partial_{1}^{2}u}^{\eta-2}\|,\label{SFin5}
\end{align}
where $M_{a,\eta-\alpha}$ is defined by \eqref{IOR30} and $V_{u}^{\eta-\alpha}:=C_{\eta-\alpha}V_u^\eta$.
\end{lemma}
 
The following lemma is a consequence of Corollary \ref{lem:ExMult} to the abstract reconstruction result Proposition \ref{rec_lem}.
 
\begin{lemma}\label{rec_lem_concr}
With the above notations and under the off-line assumptions of Theorem \ref{Prop:aPriori}, 
let $V_a^\eta, V_u^\eta\in{\mathcal D}^\eta(\mathsf{T}_+;\Ga{+}{})$ be given, and assume both are periodic in the sense of \eqref{cw14}.  Consider also the product form $V_{\aufi}^{\thresh+\alpha-2}$, cf.~\eqref{MS1}, of $V_a^\eta$ and $V_{\partial_1^2u}^{\eta-2}$.
Then there exists a unique periodic $a \diamond \puf \in \cals'(\R^2)$ such that for all $T \leq 1$ it holds
\begin{align}
&\left\|(\auf)_T -C_{\eta+\alpha-2}V_{\aufi}^{\thresh+\alpha-2}.
\begin{pmatrix}
(\puf)_{T}\\
(\pvf)_{T}\\
(\pwf)_{T}\\
(\vvf)_{T}\\
(\xvf)_{T}\\
(\wvf)_{T}\\
(\vwf)_{T}
\end{pmatrix} \right\| \nonumber\\
&\quad\lesssim M_{a \diamond \partial_{1}^{2}u,\eta+\alpha-2}\Tc{\eta+\alpha-2},
\label{eq:QualRO}
\end{align}
where
\begin{align}
M_{a \diamond \partial_{1}^{2}u,\eta+\alpha-2}&:=[a]_{\alpha}[V_{u}^{\eta}]+M_{a,\eta}+N^{4}\|V_{a}^{\eta}\| \|V_{\partial_{1}^{2}u}^{\eta-2}\|\label{S600},\\
M_{a,\eta}&:= \big( [a]_{\alpha}+N\|\overline{V}_{a}^{\eta}\|+[V_{a}^{\eta}] \big ) [V_{\partial_{1}^{2}u}^{\eta-2}]  +N[V_{a}^{\eta}]\|V_{\partial_{1}^{2}u}^{\eta-2}\|
.\label{SFin35}
\end{align}
Moreover, the sub-optimal estimate \eqref{s30} holds with
\begin{align}
M_{a \diamond \partial_{1}^{2}u,\eta-2}&:=M_{a \diamond \partial_{1}^{2}u,\eta+\alpha-2}+N^{3}\|\overline{V}_{a}^{\eta} \| \|V_{\partial_{1}^{2}u}^{\eta-2}\|\label{SFin6}.
\end{align}
\end{lemma}

\subsection{Heuristics for the Concrete Results}\label{subsec:Heuristics}

One crucial aspect of our result is that, given a point $x\in\mathbb{R}^2$, 
we may construct the form $V_{\partial_1^2u}^{\eta-2}(x)\in \mathsf{T}_-^*$, 
supposed to describe $\partial_1^2u$ near $x$ to order $\eta-2$,
purely ``algebraically'' on the basis of the form $V_a^\eta(x)\in \mathsf{T}_+^*$, 
which describes the coefficient $a$ near a point $x$ to order $\eta$, see \eqref{algebraicRel}.
Here $\eta$ is such that $\eta+\alpha-2$ $>0$, cf.~\eqref{j1200}, so that our multiplication result 
Lemma \ref{lem:Multiplication} yields a form $V_{\aufi}^{\eta+\alpha-2}$ of positive order
and thus determines a distribution $a\diam\partial_1^2u$.
We now explain, at first on a general level, why this algebraic relation holds.
We make the assumptions of the abstract subsection \ref{S:mult} on $\mathsf{A}_+$ and $\mathsf{T}_+$; 
we are given a model $v_+\in C^0(\mathbb{R}^2;\mathsf{T}_+)$. In addition, we assume that
also homogeneity $1\in\mathsf{A}_+$ plays a special role in the sense that 
$\mathsf{T}_1=\mathbb{R}$ and that $v_1=x_1$. In line with this,
we functorially impose that $\mathsf{A}_-= (\Aa{+} \setminus \{0,1\} )-2$ and $\mathsf{T}_-\supset\mathsf{T}_{\beta}$
$\cong\mathsf{T}_{\beta+2}\subset\mathsf{T}_+$ for $\beta\in\mathsf{A}_-$. 
Finally, we are given a model $v\in\cals'(\mathbb{R}^2;\mathsf{T})$ where $\mathsf{T}=\mathsf{T}_+\otimes\mathsf{T}_-$.

\medskip

In view of our abstract integration result Proposition \ref{int lem},
$V_{\partial_1^2u}^\eta(x)$ should be such that $U(x,\cdot):=u-V_{\partial_1^2u}^{\eta-2}(x).v_+$ is a good approximate solution to
$(\partial_2-a(x)\partial_1^2)U(x,\cdot)=0$, cf.~\eqref{KS1}, where we recall $a=C_\alpha V_a^\eta$. 
(Note that because of the canonical identification of $\mathsf{T}_-$ with a subspace
of $\mathsf{T}_+$, $V_{\partial_1^2u}^{\eta-2}(x)$ can be identified with an element of $\mathsf{T}_+^*$). 
More precisely, in the notation of our abstract multiplication result, Lemma \ref{lem:Multiplication},
we seek the identity
\begin{align}\label{cw74a}
(\partial_2-a(x)\partial_1^2)U_T(x,\cdot)=(a\diam\partial_1^2u)_T-V_{\aufi}^{\eta-2}(x).{\partial_1^2u_T\choose v_T},
\end{align}
which in view of $\partial_2u-a\diam\partial_1^2u=f$ is an immediate consequence of
\begin{align}\label{cw75}
(\partial_2-a(x)\partial_1^2)V_{\partial_1^2u}^{\eta-2}(x).v_+
=(\overline{V}_a^{\eta-\alpha}\otimes V_{\partial_1^2 u}^{\eta-\alpha-2})(x).v-f,
\end{align}
where we think of $\partial_2-a(x)\partial_1^2$ as acting on the implicit $y$-variable, and interpret \eqref{cw75}
in a distributional way.
Note that the r.~h.~s.~of \eqref{cw74a} is precisely such that we expect to control it (after possible reduction) thanks to Corollary \ref{lem:ExMult}, the outcome of our abstract reconstruction result, 
and thanks to the form continuity of $V_{a\diamond\partial_1^2u}^{\eta-2}$.
We now note that by telescoping, \eqref{cw75} can be separated into the identities for $n=1,2,\cdots$
\begin{align}\label{wg76}
\lefteqn{(\partial_2-a(x)\partial_1^2)(V_{\partial_1^2u}^{\eta-(n-1)\alpha-2}-V_{\partial_1^2u}^{\eta-n\alpha-2})(x).v_+}\nonumber\\
&=(\overline{V}_a^{\eta-n\alpha}\otimes V_{\partial_1^2 u}^{\eta-n\alpha-2}
-\overline{V}_a^{\eta-(n+1)\alpha}\otimes V_{\partial_1^2 u}^{\eta-(n+1)\alpha-2})(x).v,
\end{align}
with the understanding that for $\eta-(n+1)\alpha-2<\alpha<\eta-n\alpha-2$, \eqref{wg76} is to be read as
\begin{align}\label{wg77}
(\partial_2-a(x)\partial_1^2)&(V_{\partial_1^2u}^{\eta-(n-1)\alpha-2}-V_{\partial_1^2u}^{\eta-n\alpha-2})(x).v_+\
\nonumber\\
&=(\overline{V}_a^{\eta-n\alpha}\otimes V_{\partial_1^2 u}^{\eta-n\alpha-2})(x).v,
\end{align}
and for $\eta-n\alpha-2<\alpha$, \eqref{wg76} is replaced by
\begin{align}\label{wg78}
(\partial_2-a(x)\partial_1^2)V_{\partial_1^2u}^{\eta-(n-1)\alpha-2}(x).v_+=f.
\end{align}
Here, we think of $V_a^{\eta-n\alpha}$ as the reduction $C_{\eta-n\alpha}V_a$. If it is at all possible to construct
these forms, and provided $v_+$ is -- loosely speaking -- non-degenerate, 
$V_{\partial_1^2u}^\eta(x)$ must be a multi-linear expression in $V_a^\eta(x)$.

\medskip

The more delicate issue is the construction of $\mathsf{T_+}$ and $v_+$ (alongside with conditions on $v$) that make
\eqref{wg78}, \eqref{wg77}, and \eqref{wg76} solvable. In this paper, we restrict ourselves to the range of $\alpha\in(\frac{1}{2},\frac{2}{3})$,
cf.~\eqref{j1220}, in which case this set of linear equations in forms (substantially) reduces to
\begin{align}
(\partial_2-a(x)\partial_1^2)(V_{\partial_1^2u}^{\eta-2}-V_{\partial_1^2u}^{\eta-\alpha-2})(x).v_+
&=(\overline{V}_a^{\eta-\alpha}\otimes V_{\partial_1^2 u}^{\eta-\alpha-2})(x).v,\label{cw79}\\
(\partial_2-a(x)\partial_1^2)V_{\partial_1^2u}^{\eta-\alpha-2}(x).v_+&=f.\label{cw80a}
\end{align}
Equation \eqref{cw80a} suggests to define $\mathsf{T}_\alpha\subset\mathsf{T}_+$ as the space of functions $\mathsf{v}_\alpha$
of a placeholder $a_0\in I$ (we postpone the definition of
the norm to Subsection \ref{fspc}) and to characterize $v_\alpha(\cdot,a_0)$ by
\begin{align}\label{wg82}
(\partial_2-a_0\partial_1^2)v_\alpha(\cdot,a_0)=f,
\end{align}
where we ignore the issue of (periodic) boundary conditions in this heuristic discussion. 
Then \eqref{cw80a} is solved by
\begin{align}\label{wg83}
V_{\partial_1^2u}^{\eta-\alpha-2}(x).\mathsf{v}_-=\delta_{a(x)}.\partial_1^2\mathsf{v}_\alpha,
\end{align}
where we denote by $\partial_1^2\mathsf{v}_\alpha$ the generic element of 
$\mathsf{T}_-\supset\mathsf{T}_{\alpha-2}$ $\cong\mathsf{T}_{\alpha}$,
a function of $a_0\in I$.

\medskip

Under the additional assumption of $\eta<\alpha+1$, cf.~\eqref{j1220}, 
we may assume that the reduction $V_a^{\eta-\alpha}(x).\mathsf{v}_+$ depends 
on $\mathsf{v}_+$ only through $(\mathsf{1},\mathsf{v}_\alpha)$ (and not on $\mathsf{v}_1$). 
We note that by definition of the tensor product, the space 
$\mathsf{T}_\alpha\otimes \mathsf{T}_{\alpha-2}$ $\subset\mathsf{T}$ consists of functions 
of $(a_0',a_0)\in I^2$, and which we give the generic (but suggestive) name of
$\mathsf{v}_\alpha\diam\partial_1^2\mathsf{v}_\alpha$.
This suggests the choice of $\mathsf{T}_+\supset
\mathsf{T}_{2\alpha}\cong\mathsf{T}_\alpha\otimes \mathsf{T}_{\alpha-2}$, thus also consisting of functions
of $(a_0',a_0)$ which we give the generic name of $\mathsf{w}_{2\alpha}$.
Moreover, for $(a_0',a_0)$ we characterize $w_{2\alpha}(\cdot,a_0',a_0)$ by
\begin{align*}
(\partial_2-a_0\partial_1^2)w_{2\alpha}(\cdot,a_0',a_0)=(v_\alpha\diam\partial_1^2v_\alpha)(\cdot,a_0',a_0),
\end{align*}
where $v_\alpha\diam\partial_1^2v_\alpha$ (suggestively) denotes the $(\mathsf{T}_\alpha\otimes \mathsf{T}_{\alpha-2})$-component of 
$v\in\cals'(\mathbb{R}^2,\mathsf{T})$.
Since \eqref{wg82} implies $(\partial_2-a_0\partial_1^2)\partial_{a_0}v_\alpha$ $=\partial_1^2v_\alpha$, \eqref{cw79} is solved by
\begin{align}\label{wg83bis}
(V_{\partial_1^2u}^{\eta-2}-V_{\partial_1^2u}^{\eta-\alpha-2})(x).\mathsf{v}_-
&=(\overline{V}_a^{\eta-\alpha}\otimes V_{\partial_1^2 u}^{\eta-\alpha-2})(x).
{\partial_{a_0}\partial_1^2\mathsf{v}_\alpha\choose\mathsf{v}_\alpha\diam\partial_1^2\mathsf{v}_\alpha},
\end{align}
where we used that thanks to \eqref{wg83} we have the commutation relation
$(\partial_2-a(x)\partial_1^2)(\overline{V}_a^{\eta-\alpha}\otimes V_{\partial_1^2 u}^{\eta-\alpha-2})(x)$
$=(\overline{V}_a^{\eta-\alpha}\otimes V_{\partial_1^2 u}^{\eta-\alpha-2})(x)(\partial_2-a_0\partial_1^2)$.
The combination of \eqref{wg83} and \eqref{wg83bis} yields \eqref{algebraicRel}.
\subsection{Function Spaces and Skeletons}\label{fspc}

In our application, the roles of the abstract spaces $\Ta{\beta}$ are played by the spaces $C^{k}(\interval{})$ of $k$ times continuously differentiable functions equipped with the usual $C^k(\interval{})$ norm, and their tensor products  $C^{k,l}(\interval{2})$ and $C^{k,l,m}(\interval{3})$ of two or three arguments and mixed order of differentiability,
cf.\@ Lemma \ref{tensor_lemma}.  We recall that $\interval{}=[\lambda,\lambda^{-1}]$ denotes the ellipticity interval for some fixed $\lambda\in (0,1)$.
More precisely, we define the norms
\begin{align}\label{j552}
\begin{split}
 \|u\|_{C^{k',k}}&:=\max_{\substack{(a_0',a_0)\in\interval{2} \\ i'\le k', i\le k}}
 \left|\partial^{i'}_{a_0'}\partial^{i}_{a_0}u(a_0',a_0)\right|, \\
 \|u\|_{C^{k'',k',k}}&:=\max_{\substack{(a_0'',a_0'a_0)\in\interval{3} \\ i''\le k'', i'\le k', i\le k}}
 \left|\partial^{i''}_{a_0''}\partial^{i'}_{a_0'}\partial^{i}_{a_0}u(a_0'',a_0',a_0)\right|,
\end{split}
\end{align}
and note that for $g\in C^{k'}(\interval{})$ and $h\in C^{k}(\interval{})$, the tensor product can be identified with the $2$-variate function $g\otimes h\in C^{k',k}(\interval{2})$ via
\begin{align*}
 (g\otimes h)(a_0',a_0):=g(a_0')h(a_0), \quad (a_0',a_0)\in \interval{2}.
\end{align*}
%
A similar identification holds in the case of $C^{k'',k',k}$.
In the following lemma we observe that these tensor spaces fulfill the properties assumed in Subsection \ref{S:mult}. 
\begin{lemma}\label{tensor_lemma} 
  The norm $\|\cdot\|_{C^{k',k}}$ defined in \eqref{j552} is a crossnorm in the sense
\begin{align}\label{j501a}
   \|g\otimes h\|_{C^{k',k}} = \|g\|_{C^{k'}}\|h\|_{C^{k}}, \quad g\in C^{k'}(\interval{}), h\in C^{k}(\interval{}).
  \end{align}
  The completion of $C^{k'}(\interval{})\mix C^{k}(\interval{})$ under this norm can be identified with $C^{k',k}(\interval{2})$.
Moreover, $\|\cdot\|_{C^{k',k}}$ is a uniform crossnorm with respect to $\|\cdot\|_{C^{k'}}$ 
and $\|\cdot\|_{C^k}$ in the sense \eqref{j509}-\eqref{j522}. Similarly, $\|\cdot\|_{C^{k'',k',k}}$ 
is a uniform crossnorm with respect to $\|\cdot\|_{C^{k'',k'}}$ and $\|\cdot\|_{C^k}$.
\end{lemma}

We begin by defining an abstract space $(\Aa{+},\Ta{+})$ describing the relevant objects of non-negative homogeneity.  
The index set is defined by
\begin{equation}\label{S1}
  \Aa{+}:=\set{0,\alpha,1,2\alpha
  },
\end{equation}
(note that this notation is consistent with \eqref{Prod1})
and to each homogeneity in $\Aa{+}$ we associate a Banach space via
 \begin{align}\label{S8030}
\Ta{\alpha}:
=\frac{1}{N}C^{2} (\interval{} ),
\quad \Ta{0}:=\Ta{1}:=\R, \quad \Ta{2\alpha}:=\frac{1}{N^2}C^{2,1} (\interval{2}),
 \end{align}
 where it will turn out to be convenient to scale the norms of $\Ta{\alpha}$ and $\Ta{2\alpha}$ with $N$ and $N^2$ respectively, cf.\@ \eqref{jj4}, in order to ensure that the skeleton $\Ga{}{}$ defined below has a norm of order one, cf.\@ \eqref{j790}. 
The Banach space $\Tplust$ consists of
\begin{equation}\label{S2}
\ta_+ \in \Ta{+}:=\Ta{0}\oplus \Ta{\alpha}\oplus \Ta{1}\oplus \Ta{2\alpha} \ni (\na,\va,\xa,\wa),
 \end{equation}
where as before, we write $\na$ for $\mathsf{v}_{0}$.  As an input, we will associate to $\Tplust$ a $\Tplust$-valued function
\begin{equation}\label{j532}
\tf_+:=(1,\vf,\xf,\wwf
) \in C^0(\R^{2};\Tplust),
\end{equation}
where the functions $v_{0}(x):=1$ and 
\begin{equation}
v_{1}(x):=x_{1} \label{SFin23}
\end{equation}
are the ones that span the kernel of $\partial_{1}^{2}$, while $\vf$ and $\wwf$ will be defined in Section \ref{subsec:theModel}. In order to work with the boundedness norm of a modelled distribution, cf.~Definition \ref{prodDef1}
and \eqref{wg16}, we need to specify $\langle\cdot\rangle$ for $\mathsf{A}_+$:
\begin{align}\label{wg12}
\langle 0\rangle:=0,\;\langle\alpha\rangle:=1,\;\langle 1\rangle:=2,\;\langle 2\alpha\rangle:=2.
\end{align}
Note that with the exception of $\ap{1}$, the value of $\ap{\beta}$ coincides with the homogeneity of 
$\tf_{\beta}$ in the right-hand side $f$.  The information on $(\Aa{+},\Tplust)$ is summarized in the following table.  
 \begin{center}
   \begin{tabular}{| l | | l | l | l | l | l|}
    \hline
    $\beta$ & $\ta_{\beta}$ & $\Ta{\beta}$ & $\tf_{\beta}$ & $\ap{\beta}$ \\ \hline
    $0$ & $\na$ & $\R$ & $\nf$ & $0$\\ 
    $\alpha$ & $\va$ & $\frac{1}{N}C^2(\interval{})$ & $\vf$ & $1$\\ 
    $1$ & $\xa$ & $\R$ & $\xf$ & $2$\\ 
    $2\alpha$ & $\wa$ & $\frac{1}{N^2}C^{2,1}(\interval{2})$ & $\wwf$ & $2$ \\ \hline
   \end{tabular}
 \end{center}
We will also build a skeleton $\Ga{+}{}$ for $(\Aa{+},\Tplust)$ by defining $\Ga{\beta_{+}}{\gamma_{+}}(x)$
$\in{\mathcal L}(\Ta{\gamma_{+}},\Ta{\beta_{+}})$ for all $x\in \R^2$ and $\gamma_{+},\beta_{+}\in \Aa{+}$ 
with $\gamma_{+}<\beta_{+}$. The discussion in Subsection \ref{S:mult} on abstract multiplication and 
in particular Remark \ref{wg03} motivates to postulate that $\Ga{+}{}$ satisfies relation \eqref{j700}, 
which implies that $\Ga{\beta_{+}}{0}(x)$ will be determined as soon as we have fixed $\Ga{\beta_{+}}{\gamma_{+}}(x)$ 
for all positive $\gamma_{+}$. Thus, it suffices to make the choices
\begin{align*}
 \Ga{1}{\alpha}(x)\va:=0, \, \Ga{2\alpha}{\alpha}(x)\va:=-\vf(x)\otimes\partial_{a_0}\va, \,  \Ga{2\alpha}{1}(x)\xa:=-\of(x)\xa,
\end{align*}
where $\omega \in C^{0}(\R^{2};\Ta{2\alpha})$ will be constructed in Lemma \ref{est_wwf_lemma}, 
which also motivates this definition of $\Ga{2\alpha}{\alpha}$ and $\Ga{2\alpha}{1}$, whereas $\Ga{1}{\alpha}=0$ 
is motivated by the explanation for level $\alpha-1$ in Section \ref{subsec:theModel} below.  This choice
determines the operator $\Ga{+}{} \in C^{0}(\R^{2}; \mathcal{L}(\Tplust,\Tplust))$ via the prescription 
\eqref{j551}, which can be summarized in the block form
\begin{align}\label{cw12}
\Ga{+}{}=\left(\begin{array}{lccc}
{\rm id}_{\R}                                                   &0                             &0             &0\\
-v_\alpha                                                       &{\rm id}_{C^2}   &0             &0\\
-v_1                                                            &0                             &{\rm id}_{\R} &0\\
-w_{2\alpha}+v_\alpha\otimes\partial_{a_0}v_\alpha+\omega v_1   &-v_\alpha\otimes\partial_{a_0}&-\omega       &{\rm id}_{C^{2,1}}
\end{array}\right).
\end{align}
We note that it is the presence of the parameter derivative $\partial_{a_0}$ that requires us to reduce the differentiability with respect to the second argument for functions in $\Ta{2\alpha}$.

\medskip

Next we define an abstract space $(\Aa{-},\Ta{-})$ and its skeleton $\Ga{-}{}$ describing the relevant objects 
of negative homogeneity. Loosely speaking, it naturally or ``functorially'' arises from $(\Aa{+},\Ta{+})$ and 
$\Ga{+}{}$ by applying the operator $\partial_1^2$. More precisely, the index set is defined by
  \begin{equation}\label{S4}
  \Aa{-}:=\big (\Aa{+} \setminus \{0,1\} \big )-2=\{\alpha-2,2\alpha-2\},
 \end{equation}
and to each homogeneity in $\Aa{-}$ we associate a Banach space via
\begin{align}\label{S8001}
\Ta{\alpha-2}:=\Ta{\alpha}=\frac{1}{N}C^2(\interval{}), 
\quad \Ta{2\alpha-2}:= \Ta{2\alpha} = \frac{1}{N^2}C^{2,1}(\interval{2}).
\end{align}
The Banach space $\Tminust$ consists of
\begin{equation}\label{S5}
\ta_{-} \in \Tminust:= \Ta{\alpha-2}\oplus\Ta{2\alpha-2} \ni (\pva,\pwa),
\end{equation}
which act as placeholders for the $\Ta{-}$-valued distribution
\begin{equation}\label{S6}
\tf_{-}:=(\pvf,\pwf) \in {\mathcal S}'(\R^{2}; \Tminust)
\end{equation}
obtained from applying $\partial_{1}^{2}$ to $v_{+}$. We specify $\ap{\cdot}$ on $\Aa{-}$ through
\begin{equation}
\langle\alpha-2\rangle:=\langle\alpha\rangle=1,\;\langle 2\alpha-2\rangle:=\langle 2\alpha\rangle=2.\label{wg90}
\end{equation}
This information is summarized in the following table.  
 \begin{center}
   \begin{tabular}{| l | | l | l | l | l | l |}
    \hline
    $\beta$ & $\ta_{\beta}$ & $\Ta{\beta}$ & $\tf_{\beta}$ & $\ap{\beta}$ \\ \hline
    $\alpha-2$ & $\pva$ & $\frac{1}{N}C^2(\interval{})$ & $\partial_{1}^{2}\vf$ & $1$\\ 
    $2\alpha-2$ & $\partial_{1}^{2}\wa$ & $\frac{1}{N^2}C^{2,1}(\interval{2})$ & $\partial_{1}^{2}\wwf$ & $2$ \\ \hline
   \end{tabular}
 \end{center}

In line with these definitions we define the associated skeleton $\Ga{-}{}$ as resulting from $\Ga{+}{}$ via
%
$\Ga{\beta}{\gamma}:=\Ga{\beta+2}{\gamma+2}$ for $\beta,\gamma\in\Aa{-}$,
%
so that the triangular structure \eqref{j551} of $\Ga{+}{}$ is passed to $\Ga{-}{}$. For later reference, we record
\begin{align}\label{cw12a}
\Gamma_-=\left(\begin{array}{lccc}
{\rm id}_{C^2}                      &0\\
-v_\alpha\otimes\partial_{a_0}&{\rm id}_{C^{2,1}}
\end{array}\right).
\end{align}
%

\medskip

With these objects defined, we can define the product space $(\tAa{},\tTa{})$ and the associated skeleton $\tGa{}{}$ in virtue of Section \ref{S:mult}. The elements of $\Aa{}$ are ordered according to $\alpha-2<2\alpha-2<\alpha-1<3\alpha-2<0<2\alpha-1<4\alpha-2$, and we define $\ap{\cdot}$ via
\begin{align}\label{wg13}
\begin{array}{c}
\langle\alpha-1\rangle:=3,\;
\langle 3\alpha-2\rangle:=3,\;
\langle2\alpha-1\rangle:=4,\;\langle 4\alpha-2\rangle:=4.
\end{array}
\end{align}
Is is easily checked by inspection that $\langle\cdot\rangle$ satisfies the consistency
condition \eqref{wg10}, which we recall was crucial for multiplication. Following \eqref{Prod8} 
and removing objects of nonnegative homogeneity in virtue of Section \ref{S:red}, we obtain the following tensor spaces.  
Homogeneities $\alpha-2$ and $\alpha-1$ are each attached to a single space via
\begin{align*}
\Ta{0}\otimes\Ta{\alpha-2}\cong \Ta{\alpha-2}, \quad \Ta{1}\otimes\Ta{\alpha-2}=:\Ta{\alpha-1}.
\end{align*}
Homogeneity $2\alpha-2$ is attached to two spaces denoted by
\begin{align*}
\quad \Ta{0}\otimes\Ta{2\alpha-2}=:\Ta{2\alpha-2,1}, \quad \Ta{\alpha}\otimes\Ta{\alpha-2}=:\Ta{2\alpha-2,2}.
\end{align*}
Homogeneity $3\alpha-2$ is also attached to two spaces denoted by
\begin{align*}
\Ta{2\alpha}\otimes\Ta{\alpha-2}=:\Ta{3\alpha-2,1}, \quad \Ta{\alpha}\otimes \Ta{2\alpha-2}=:\Ta{3\alpha-2,2}.
\end{align*}
By Lemma \ref{tensor_lemma}, these spaces are identified with function spaces of mixed order of 
differentiability the norms of which carry certain powers of $N$ as stated in the table below.
Hence, we obtain the direct sum decomposition
\begin{equation}\label{S7}
\Ta{<0}=\Ta{\alpha-2}\oplus \Ta{2\alpha-2,1} \oplus \Ta{2\alpha-2,2} \oplus \Ta{\alpha-1} \oplus \Ta{3\alpha-2,1} \oplus \Ta{3\alpha-2,2}
\end{equation}
and write a generic element in $\tTa{<0}$ as
\begin{equation}
\ta=(\pva,\pwa,\vva,\xva,\wva, \vwa),
\end{equation}
which acts as a placeholder for the $\Ta{<0}$-valued distribution
\begin{equation}\label{j531}
\tf:=(\pvf,\pwf,\vvf,\xvf,\wvf, \vwf)
\end{equation}
belonging to ${\mathcal S}'(\R^{2}; \Ta{<0})$. This information is summarized in the following table.
 \begin{center}
   \begin{tabular}{| l | | l | l | l | l | l |}
    \hline 
    $\beta$ & $\ta_{\beta}$ & $\Ta{\beta}$ & $\tf_{\beta}$ & $\ap{\beta}$\\ \hline 
    $\alpha-2$ & $\pva$ & $\frac{1}{N}C^2(\interval{})$ & $\pvf$ & $1$\\ 
    $2\alpha-2$ & $\pwa$ & $\frac{1}{N^2}C^{2,1}(\interval{2})$ & $\pwf$ & $2$ \\
    $2\alpha-2$ & $\vva$ & $\frac{1}{N^2}C^{2,2}(\interval{2})$ & $\vvf$ & $2$ \\ 
    $\alpha-1$ & $\xva$ & $\frac{1}{N}C^2(\interval{})$ & $\xvf$  & $3$ \\
    $3\alpha-2$ & $\wva$ & $\frac{1}{N^3}C^{2,1,2}(\interval{3})$ & $\wvf$ & $3$ \\
    $3\alpha-2$ & $\vwa$ & $\frac{1}{N^3}C^{2,2,1}(\interval{3})$ & $\vwf$ & $3$ \\ \hline
  \end{tabular}
 \end{center}

\subsection{The Model}\label{subsec:theModel} 
We now define the concrete objects $v_\alpha$, $w_{2\alpha}$, $v_\alpha\diamond\partial_1^2v_\alpha$, 
$v_1\partial_1^2v_\alpha$, $w_{2\alpha}\diamond\partial_1^2v_\alpha$, and $v_\alpha\diamond\partial_1^2w_{2\alpha}$. 
We also construct $\omega$ and thus complete the construction of the skeletons. 
In doing so, we make sure (either by an argument or by an ``off-line'' assumption) that the so defined models
$v_+\in C^0(\mathbb{R}^2;\mathsf{T}_+)$ and $v\in {\mathcal S}'(\mathbb{R}^2;\mathsf{T}_{<0})$ satisfy the conditions
\begin{align}
\|\Gamma_\beta(x)v_+(y)\|_{\mathsf{T}_\beta}&\lesssim d^\beta(y,x)\quad\mbox{for all}\;\beta\in\mathsf{A}_+
,\;x,y\in\mathbb{R}^2,\label{lo02}\\
\|\Gamma_\beta(x)v_T(x)\|_{\mathsf{T}_\beta}&\lesssim (T^\frac{1}{4})^\beta\quad\mbox{for all}\;\beta\in\mathsf{A}_{<0}
,\;x\in\mathbb{R}^2,T\in(0,1],\label{lo01}
\end{align}
recalling that \eqref{lo01} corresponds to the assumption \eqref{j770} of the abstract Corollary \ref{lem:ExMult},
and noting that \eqref{lo02} contains the consistency condition \eqref{j700} (for $y=x$).
We also make sure that the skeletons satisfy the continuity condition \eqref{ssss2}. We proceed homogeneity level by level,
noting that \eqref{lo02} is trivially satisfied for $\beta=0,1$.

\medskip

\underline{Level $\alpha-2$}:
Our standing assumption is that there exists an $\R$-valued periodic distribution $f$ such that
\begin{align}\label{f29}
\|f\|_{C^{\alpha-2}(\R^{2})}\le N,
\end{align}
which sets both $\alpha$ and $N$.

\medskip

\underline{Level $\alpha$}:
We now define $\vf \in C^{\alpha}(\R^{2}; \Ta{\alpha})$ by the condition that for each fixed $a_{0} \in \interval{}$, $\vf(\cdot,a_{0})$
is the mean-free periodic solution to
\begin{align}\label{f21}
(\partial_2-a_0\partial_1^2)\vf(\cdot,a_0)=Pf.
\end{align}
In addition, it will be useful to note that for all $k \in \N$
\begin{align}\label{se40}
(\partial_2-a_0\partial_1^2)\partial_{a_0}^{k}\vf(\cdot,a_0)=k\partial_1^2\partial_{a_0}^{k-1}\vf(\cdot,a_0).
\end{align}
Assumption \eqref{f29} together with the classical Schauder estimates therefore imply \eqref{lo02} for $\beta=\alpha$,
that is, for all $x,y\in \R^2$
\begin{align}
\|\Ga{\alpha}{}(x)\tf_{+}(y)\|_{\Ta{\alpha}}&\stackrel{\eqref{j532},\eqref{cw12},\eqref{S8030}}{=}
\|\vf(y)-\vf(x)\|_{\frac{1}{N}C^2}\nonumber\\
&\le \|\vf(y)-\vf(x)\|_{\frac{1}{N}C^3}\lesssim d^{\alpha}(y,x).\label{s203}
\end{align}
%
%
%
This in turn yields \eqref{lo01} for $\beta=\alpha-2$, that is, for all $T\in(0,1]$
%
\begin{align}
\|\Gamma_{\alpha-2}v_{T}\|_{C^0(\mathbb{R}^2;\Ta{\alpha-2})}&\stackrel{\eqref{j531},\eqref{cw12a},\eqref{S8001}}{=}
\|(\partial_{1}^{2}v_{\alpha})_{T}\|_{C^0(\mathbb{R}^2;\frac{1}{N}C^2)}\nonumber\\
&\stackrel{\mathclap{\eqref{1.13}}}{\lesssim}
\|v_{\alpha}\|_{C^\alpha(\mathbb{R}^2;\frac{1}{N}C^2)}(T^{\frac{1}{4}})^{\alpha-2} 
\stackrel{\eqref{s203}}{\lesssim} (T^{\frac{1}{4}})^{\alpha-2}\label{p13}.
\end{align}

\medskip

\underline{Level $2\alpha-2,2$}:
We now make our first off-line assumption and postulate that there exists a distribution $v_{\alpha}\diamond \partial_{1}^{2}v_{\alpha} \in C^{\alpha-2}(\R^{2};C^{2,2}(\interval{2}))$ satisfying \eqref{j770}. We observe that
$
\Ga{2\alpha-2,2}{\alpha-2}\stackrel{\eqref{Prod9} }{=}\Ga{\alpha}{0}\otimes \Ga{\alpha-2}{\alpha-2}\stackrel{\eqref{cw12}, \eqref{cw12a}}{=}-\vf \otimes \id,
$
so that
\begin{align*}
\Ga{2\alpha-2,2}{}v_{T}&\stackrel{\eqref{j551}}{=}\,(v_{2\alpha-2,2})_{T}+\Ga{2\alpha-2,2}{\alpha-2}(v_{\alpha-2})_{T}\\
&\stackrel{\eqref{j531}}{=}(v_{\alpha} \diamond \partial_{1}^{2}v_{\alpha})_{T}+\Gamma_{2\alpha-2,2}^{\alpha-2}(\partial_{1}^{2}v_{\alpha})_{T}\\
&\, \, =\, \, \,(v_{\alpha}\diamond \partial_{1}^{2}v_{\alpha})_{T}-v_{\alpha}\otimes(\partial_{1}^{2}v_{\alpha})_{T}.
\end{align*}
Hence, our assumption \eqref{lo01} takes the natural form
\begin{align}
\|&\Ga{2\alpha-2,2}{}\tf_T \|_{C^0(\R^2;\Ta{2\alpha-2,2})} \nonumber \\
&=\left\|(\vvf)_T-\vf\otimes(\pvf)_T\right \|_{C^0(\R^2;\frac{1}{N^2}C^{2,2})}\nonumber\\ 
&\le\left\|(\vvf)_T-\vf\otimes(\pvf)_T\right \|_{C^0(\R^2;\frac{1}{N^2}C^{3,2})}\leq (T^{\frac{1}{4}})^{2\alpha-2}.\label{off1}
\end{align}
We note that the factor $N^{2}$ in the definition of the space $\Ta{2\alpha-2,2}$ is justified, since $(v_{\alpha}\diamond \partial_{1}^{2}v_{\alpha})_{T}-v_{\alpha}\otimes(\partial_{1}^{2}v_{\alpha})_{T}$ is (formally) a quadratic expression in $v_{\alpha}$, and hence \eqref{off1} would follow from \eqref{s203} if $\alpha$ were larger than $1$.

\medskip

\underline{Level $\alpha-1$}:
The product $v_{1}\partial_{1}^{2}v_{\alpha}$ of a smooth function and a distribution is classically defined.  
By the triangular structure of $\Gamma$, we have $\Ga{\alpha-1}{}v_{T}$ $=(v_{\alpha-1})_{T}$ 
$+\Ga{\alpha-1}{\alpha-2}(v_{\alpha-2})_{T}$ $+\Ga{\alpha-1}{2\alpha-2}(v_{2\alpha-1})_{T}$.
We further note that by \eqref{Prod9} and the definitions \eqref{cw12} of $\Ga{+}{}$ and \eqref{cw12a} of $\Ga{-}{}$, it holds $\Ga{\alpha-1}{\alpha-2}=\Ga{1}{0}\otimes\Ga{\alpha-2}{\alpha-2}=-\xf\otimes \id_{\Ta{\alpha-2}}$, while $\Ga{\alpha-1}{2\alpha-2}=\Ga{1}{\alpha}\otimes\Ga{\alpha-2}{\alpha-2}=0$ vanishes.
Together with \eqref{j531}, these observations combine to
$\Ga{\alpha-1}{}v_{T}=(v_{1}\partial_{1}^{2}v_{\alpha})_{T}-v_{1}\otimes (\partial_{1}^{2}v_{\alpha})_{T}$.
Hence \eqref{lo01} is satisfied:
\begin{align}\label{off20}
\|&\Gamma_{\alpha-1}v_{T}\|_{C^{0}(\R^{2};\Ta{\alpha-1})} \nonumber\\
&=\|(v_{1}\partial_{1}^{2}v_{\alpha})_{T}
-v_{1}\otimes (\partial_{1}^{2}v_{\alpha})_{T}\|_{C^{0}(\R^{2};\frac{1}{N}C^2)}
\lesssim ( T^{\frac{1}{4}})^{\alpha-1},
\end{align}
which follows from \eqref{p13} by Lemma 10 in \cite{OtW16}.
%


%
%

\medskip 

\underline{Level $2\alpha$}:
%
%
%
%
%
Next we define $\wwf \in C^{\alpha}(\R^{2}; \Ta{2\alpha})$ to be the mean-free periodic solution to
\begin{align}\label{f22}
(\partial_2-a_0\partial_1^2)\wwf(a_0',a_0)=P(\vvf)(a_0',a_0).
\end{align}
Since by assumption $\vvf\in C^{\alpha-2}(\R^{2};C^{3,2}(\interval{2}))$, for each $j\le 3$, $k \leq 2$
\begin{align}
(\partial_2-a_0\partial_1^2)\partial_{a_0}^{k}\partial_{a_0'}^{j}\wwf(a_0',a_0)&=\partial_{a_0}^{k}\partial_{a_0'}^{j}P(\vvf)(a_0',a_0) \nonumber \\
&+k\partial_{a_0}^{k-1}\partial_{a_0'}^{j}\partial_{1}^{2}\wwf(a_0',a_0), \label{jj2}
\end{align}
which implies $\wwf\in C^{\alpha}(\R^{2};C^{3,2}(\interval{2}))$ by repeatedly applying standard Schauder theory.  
Moreover, $\wwf$ behaves as an object of order $2\alpha$ in the sense that \eqref{lo01} holds for $\beta=2\alpha$,
as made precise by the following lemma, which is another  consequence of our abstract integration result, 
Proposition \ref{int lem}.
\begin{lemma}\label{est_wwf_lemma}
There exists a periodic function $\of \in C^{2\alpha-1}(\R^{2}; \Ta{2\alpha})$ with
\begin{align}\label{j781}
\|\omega\|_{C^{2\alpha-1}(\R^{2};\Ta{2\alpha})} \lesssim 1
\end{align}
and such that for all $x,y \in \R^{2}$ 
\begin{align}
\big\|\Ga{2\alpha}{}(x)\tf_{+}(y)\big\|_{\Ta{2\alpha}} 
&\stackrel{\mathclap{\eqref{j532}, \eqref{cw12}}}{=}\quad \big\|\wwf(y)-\wwf(x) \nonumber \\ 
&\qquad \qquad -\vf(x)\otimes(\partial_{a_0}\vf(y)-\partial_{a_0}\vf(x)) \nonumber \\
&\qquad \qquad - \of(x)(\xf(y)-\xf(x))\big\|_{\frac{1}{N^2}C^{2,1}} \nonumber \\
& \lesssim d^{2\alpha}(y,x) \label{s2011}.
\end{align}
Moreover, $\Ga{+}{}$ and $\Ga{}{}$ satisfy
\begin{align}\label{j790}
 \|\Ga{+}{}\|_{sk}\lesssim 1 \quad \text{ and } \quad \|\Ga{}{}\|_{sk}\lesssim 1.
\end{align}
\end{lemma}

\medskip

\underline{Level $2\alpha-2,1$}: 
We return to level $2\alpha-2$, this time to the contribution of the first component of the space.  We note that by \eqref{Prod9}, \eqref{cw12}, and  \eqref{cw12a}
$\Ga{2\alpha-2,1}{\alpha-2}=\Ga{0}{0}\otimes \Ga{2\alpha-2}{\alpha-2}
=-v_{\alpha}\otimes \partial_{a_0}$,
 so that 
$
\Ga{2\alpha-2,1}{}v_{T}\stackrel{\eqref{j551}}{=}(v_{2\alpha-2,1})_{T}+\Ga{2\alpha-2,1}{\alpha-2}(v_{\alpha-2})_{T}
\stackrel{\eqref{j531}}{=}(\partial_{1}^{2}\wwf)_{T}-\vf \otimes \partial_{a_0}(\partial_{1}^{2}v_{\alpha})_{T}
$.  
Hence  \eqref{lo01} is satisfied:
\begin{align}
\|&\Gamma_{2\alpha-2,1}\tf_{T}\|_{C^{0}(\R^{2};\Ta{2\alpha-2,1})} \nonumber \\
&= \|(\partial_{1}^{2}w_{2\alpha})_{T}-v_{\alpha}\otimes \partial_{a_0}(\partial_{1}^{2}v_{\alpha})_{T}
\|_{C^{0}(\R^{2};\frac{1}{N^2}C^{2,1})} \ \stackrel{\mathclap{\eqref{1.13},\eqref{s2011}}}{\lesssim} \ (T^{\frac{1}{4}})^{2\alpha-2}.\label{off2}
\end{align}

%
%
%
%
%
%
%
%

\medskip

\underline{Level $3\alpha-2$}:
The last off-line assumption concerns the distribution 
$$
(\vwf, \wvf) \in C^{\alpha-2}(\R^{2}; C^{2,2,1}(\interval{3})\oplus C^{2,1,2}(\interval{3})).
$$
Namely, we assume that
we have
\begin{equation}
\left\|\Ga{3\alpha-2}{}\tf_T\right \|_{C^0(\R^2;\Ta{3\alpha-2}) } \leq (T^{\frac{1}{4}})^{3\alpha-2}.\label{f101}
\end{equation}
We do not spell out the (rather lengthy) expression $\Ga{3\alpha-2}{}$ determined through \eqref{cw12}, \eqref{cw12a}, and \eqref{Prod9}, as it will be needed only as an input in \eqref{j770} when proving Lemma \ref{rec_lem_concr}.  Analogously to level $2\alpha-2,2$ we note that the factor $N^{3}$ in the definition of $\Ta{3\alpha-2}$ is consistent, since the involved distributions are formally cubic in $v_{\alpha}$, and the estimate would follow from our other assumptions if $\alpha$ were larger than $\frac{2}{3}$.
\subsection{Relation between the forms}\label{SS:algo}
We now collect further relations between the forms $V_a^{\eta-\alpha}$, $V_a^{\eta}$,
$V_u^{\eta-\alpha}$, $V_{u}^{\eta}$, $V_{\partial_1^2u}^{\eta-\alpha-2}$, $V_{\partial_1^2u}^{\eta-2}$, 
$V_{a\diamond\partial_1^2u}^{\eta-2}$, and $V_{a\diamond\partial_1^2u}^{\eta+\alpha-2}$.  
This also serves the purpose to reveal the iterative structure of the identity \eqref{s300bis}.

\medskip

\newcounter{FormConst} 
\refstepcounter{FormConst} 
{\sc Step} \arabic{FormConst}.\refstepcounter{FormConst}[Differentiation and Multiplication]
We note that \eqref{s300bis} implies for the reduction $V_{u}^{\eta-\alpha}=C_{\eta-\alpha}V_u^{\eta}$ that
\begin{align}
V_{u}^{\eta-\alpha}.
\begin{pmatrix}
\na \\
\va
\end{pmatrix}
= u\na+\delta_{a}.\big (\va-\vf\na \big ). \label{j41}
\end{align}
We also note that the functorial relation between the forms for $u$ and those for $\partial_1^2u$ commutes with reduction:
%
\begin{equation}
V_{\partial_{1}^{2}u}^{\eta-\alpha-2}.\partial_{1}^{2}\mathsf{v}_{\alpha}=V_{u}^{\eta-\alpha}.
\begin{pmatrix}
0\\
\partial_{1}^{2}\va
\end{pmatrix}\stackrel{\eqref{j41}}{=}\delta_a.\partial_1^2\mathsf{v}_\alpha.
\label{s115}
\end{equation}
Following \eqref{MS1}, the product of the forms $V_{a}^{\eta-\alpha}$ and 
$V_{\partial_{1}^{2}u}^{\eta-\alpha-2}$ assumes the form
\begin{equation}\label{s116}
V_{a \diamond \partial_{1}^{2}u}^{\eta-2}.
\begin{pmatrix}
\pua\\
\pva \\
\va \diamond \pva
\end{pmatrix}
=a\pua + \overline{V}_{a}^{\eta-\alpha}\otimes V_{\pua}^{\eta-\alpha-2}.
\begin{pmatrix}
\pva \\
\va \diamond \pva
\end{pmatrix}.
\end{equation}

\medskip

{\sc Step} \arabic{FormConst}.\refstepcounter{FormConst}[Integration]
From \eqref{s300bis}, \eqref{s115}, and \eqref{s116} we obtain
\begin{align}
V_{u}^{\thresh}.\begin{pmatrix}
\na \\
\va \\
\xa\\
\wa
\end{pmatrix}
&=
V_{u}^{\thresh-\alpha}.\begin{pmatrix}
\na \\
\va
\end{pmatrix}
+\nu(\xa-\xf\na)
\label{s300} \\
&-a\partial_1^2\mathsf{u}+V_{a\diamond\partial_1^2u}^{\eta-2}. 
\begin{pmatrix}
\partial_1^2\mathsf{u}\\
\partial_{a_0}\va-\partial_{a_0}\vf\na\\
\wa-\wwf\na
\end{pmatrix}. \nonumber
\end{align}
%

\medskip

{\sc Step} \arabic{FormConst}.\refstepcounter{FormConst}[Reconstruction]
Following \eqref{MS1}, the product of the forms $V_a^\eta$ and $V_{\partial_1^2u}^{\eta-2}$
assumes the form 
%
\begin{align}
&V_{a \, \diamond \, \puf}^{\thresh+\alpha-2}.
\begin{pmatrix}
& \, &\pva & \pwa\\
&\, &\vva & \vwa \\
&\pua \, &\xva & \xa \diam \pwa \\
& \, &\wva & \wa \diam \pwa
\end{pmatrix}
\nonumber\\
&=a\pua+
C_{\eta+\alpha-2}\big ( \overline{V}_{a}^{\thresh} \otimes V_{\puf}^{\thresh-2} \big ).
\begin{pmatrix}
&\pva & \pwa\\
&\vva & \vwa\\
&\xva & \xa \diam \pwa\\
&\wva & \wa \diam \pwa
\end{pmatrix}.
\label{S880}
\end{align}
\subsection{Periodicity}\label{ss:periodicity}
While Section \ref{S:MD} and in particular the reconstruction result Proposition \ref{rec_lem}
are oblivious to periodicity, it is convenient for the integration result Proposition \ref{int lem},
and boundary conditions are unavoidable for our uniqueness statement. On the level of a modelled
distribution $V$, the notion of periodicity is not the naive one, because the model $v$ itself is not periodic
due to the presence of $v_1(x)=x_1$ on the level of $\mathsf{T}_+$,
and of $v_1\partial_1^2v_\alpha$, $v_1\partial_1^2w_{2\alpha}$ on the level of $\mathsf{T}$. 
In this subsection, we introduce 
such a notion and argue that the forms constructed in Subsection \ref{SS:algo} are indeed
periodic in this sense, based on the assumption that $V_a$ is periodic in this sense and
that the functions $u$ and $\nu$ are plain periodic, and on the off-line assumption that the distributions
$f$ and $v_\alpha\diam\partial_1^2v_\alpha$ are (plain) periodic.
The so ensured (extended notion of) periodicity 
of $V_{a\diamond\partial_1^2u}^{\eta+\alpha-2}$ will be argued to imply 
the periodicity of the distribution $a\diam\partial_1^2u$ in Lemma \ref{rec_lem_concr},
which in turn is the input for Lemma \ref{int_lem_concr}.

\medskip

The additive group $\mathbb{Z}^2\ni k$ acts on space-time $\mathbb{R}^d\ni x$ via translation $k+x$.
Our model $v_+\in C^0(\mathbb{R}^2;\mathsf{T}_+)$, cf.~\eqref{j532}, 
is equi-variant under this action, provided
we let $\mathbb{Z}^2$ act on $\mathsf{T}_+$ via
\begin{align}\label{cw10}
(k+\mathsf{v}_+)_\beta=\left\{\begin{array}{rl}k_1\mathsf{v}_0+\mathsf{v}_\beta&\mbox{for}\;\beta=1\\
                                                 \mathsf{v}_\beta&\mbox{else}\end{array}\right\}.
\end{align}
Equi-variance means
\begin{align}\label{cw13}
v_+(k+x)=k+v_+(x),
\end{align}
and is a consequence of the (plain) periodicity of $v_\alpha$ and $w_{2\alpha}$, which in view of
their definitions \eqref{f21} and \eqref{f22} itself is a consequence of the periodicity of the distributions $f$ and 
$v_\alpha\diam\partial_1^2v_\alpha$, respectively. 

\medskip

We also note that our skeleton $\Ga{+}{}$, cf.~\eqref{cw12},
is covariant under the action \eqref{cw10}, meaning that
\begin{align}\label{cw11}
\Gamma_+(k+x)(k+\mathsf{v}_+)=\Gamma_+(x)\mathsf{v}_+,
\end{align}
a property we call periodicity.
Let us check \eqref{cw11} for the two interesting components $\beta=1,2\alpha$: For $\beta=1$
we have $\Gamma_1(k+x)(k+\mathsf{v})$ $\stackrel{\eqref{cw12}}{=}(k+\mathsf{v})_1-(k_1+v_1(x))(k+\mathsf{v})_0$
$\stackrel{\eqref{cw10}}{=}(k_1\mathsf{v}_0+\mathsf{v}_1)-(k_1+v_1(x))\mathsf{v}_0$ 
$\stackrel{\eqref{cw12}}{=}\Gamma_1(x)\mathsf{v}$. For $\beta=2\alpha$ we have
$\Gamma_{2\alpha}(k+x)(k+\mathsf{v})$ $\stackrel{\eqref{cw12},\eqref{cw10}}{=}
\mathsf{v}_{2\alpha}-\omega(x)(k_1\mathsf{v_0}+\mathsf{v}_1)
-v_\alpha(x)\partial_{a_0}\mathsf{v}_\alpha-(w_{2\alpha}(x)-v_\alpha(x)\otimes\partial_{a_0}v_\alpha(x)
-\omega(x)(k_1+v_1(x)))\mathsf{v}_0$
$\stackrel{\eqref{cw12}}{=}\Gamma_{2\alpha}(x)\mathsf{v}$. Here we used in addition the periodicity of $\omega$,
cf.~Lemma \ref{est_wwf_lemma}.

\medskip

We now turn to a modelled distribution $V_+$ on $(\mathsf{A}_+,\mathsf{T}_+)$, and 
call it periodic provided we have
\begin{align}\label{cw14}
V_+(k+x).(k+\mathsf{v}_+)=V_+(x).\mathsf{v}_+.
\end{align}
We claim that periodicity is preserved under {\it reduction}, cf.~Subsection \ref{S:red}.
Next to \eqref{cw11}, for preservation of periodicity under reduction we crucially need the
following property of the action \eqref{cw10}
\begin{align}\label{cw17}
(k+\mathsf{v}_+)+\mathsf{v}_+'=k+(\mathsf{v}_++\mathsf{v}'_+)\quad\mbox{provided}\;\mathsf{v}_0'=0.
\end{align}
By induction it is sufficient to check preservation of \eqref{cw14} under the one-step reduction 
\eqref{cut2} with $\overline{\beta}>0$.
Periodicity of $C_{\overline{\beta}}V_+$ is now easily seen:
$C_{\overline{\beta}}V_+(k+x).(k+\mathsf{v}_+)$
$\stackrel{\eqref{cut2}}{=}V_+(k+x).((k+\mathsf{v}_+)-\Gamma_{\overline{\beta}}(k+x)(k+\mathsf{v}_+))$ 
$\stackrel{\eqref{cw11}}{=}V_+(k+x).((k+\mathsf{v}_+)-\Gamma_{\overline{\beta}}(x)\mathsf{v}_+)$
$\stackrel{\eqref{cw17}}{=}V_+(k+x).(k+(\mathsf{v}_+-\Gamma_{\overline{\beta}}(x)\mathsf{v}_+))$
$\stackrel{\eqref{cw14}}{=}V_+(x).(\mathsf{v}_+-\Gamma_{\overline{\beta}}(x)\mathsf{v}_+)$
$\stackrel{\eqref{cut2}}{=}C_{\overline{\beta}}V_+(x).\mathsf{v}_+$.

\medskip

Applying this observation on periodicity and reduction
to the forms from Subsection \ref{SS:algo}, we obtain from our assumption of periodicity of
$V_a^\eta$ that also the two reductions $V_a^{\eta-\alpha}$ and $a=V_a^{\alpha}$, 
and thus also $\overline{V}_a^{\eta-\alpha}$, are periodic.
Hence with $u$ also the $V_u^{\eta-\alpha}$ defined in \eqref{j41} is periodic.
As a consequence, with $u$, $\nu$, $v_\alpha$, and $w_{2\alpha}$, also $V_u^\eta$ defined
in \eqref{s300bis} is periodic.
\medskip

Recall from Subsection \ref{fspc} that the structure $(\mathsf{A}_-,\mathsf{T}_-,\Gamma_-,v_-)$ arises from
$(\mathsf{A}_+,\mathsf{T}_+,\Gamma_+,v_+)$ functorially
by formal {\it differentiation} (i.~e.~applying $\partial_1^2$), which shifts homogeneities
by $-2$ while annihilating the homogeneities $0,1$.
In particular, $\Gamma_-$ is given by \eqref{cw12a},
and extending \eqref{cw10} to $\beta=\alpha-2,2\alpha-2$, we obtain \eqref{cw13} and \eqref{cw11} with
$v_+$ and $\Gamma_+$ replaced by $v_-$ and $\Gamma_-$, respectively. 
As a consequence, with $V_u^\eta$ also $V_{\partial_1^2u}^{\eta-2}$ is periodic, and then also
$V_{\partial_1^2u}^{\eta-\alpha-2}=C_{\eta-\alpha-2}V_{\partial_1^2u}^{\eta-2}$.

\medskip

We now turn to the {\it multiplication} of our structures $(\mathsf{A}_+,\mathsf{T}_+,\Gamma_+)$ and $(\mathsf{A}_-,\mathsf{T}_-,\Gamma_-)$
resulting in the product structure $(\mathsf{A},\mathsf{T},\Gamma)$,
cf.~Subsection \ref{S:mult}. The (continuous) action of $\mathbb{Z}^2$ on $\mathsf{T}_\pm$ naturally extends to
an action on the (topological) tensor product $\mathsf{T}=\mathsf{T}_+\otimes\mathsf{T}_-$, cf.~\eqref{j509}, 
and is easily seen to be given by
\begin{align}\label{cw18}
(k+\mathsf{v})_\beta=\left\{\begin{array}{rl}k_1\mathsf{v}_{ \alpha-2}+\mathsf{v}_\beta&\mbox{for}\;\beta= \alpha-1\\
                                             k_1\mathsf{v}_{2\alpha-2}+\mathsf{v}_\beta&\mbox{for}\;\beta=2\alpha-1\\
                                                \mathsf{v}_\beta&\mbox{else}\end{array}\right\}.
\end{align}
As it should, our model $v$, cf.~\eqref{j531}, is equi-variant:
\begin{align}\label{cw21}
v(k+\cdot)=k+v\quad\mbox{in the sense of distributions},
\end{align}
which is to be compared to \eqref{cw13}, and is
a consequence of the fact that the products $v_1\partial_1^2v_\alpha$,
$v_1\partial_1^2w_{2\alpha}$ of the smooth function $v_1$ and a distribution are
classically defined.
Since $\Gamma(x)=\Gamma_+(x)\otimes\Gamma_-(x)$, cf.~Lemma \ref{lem:prod_skel}, the periodicity
property \eqref{cw11} and the similar property for $\Gamma_-$ extends to $\Gamma$:
\begin{align}\label{cw15}
\Gamma(k+x)(k+\mathsf{v})=\Gamma(x)\mathsf{v}.
\end{align}
For the same reason, if a form $V_+$ on $(\mathsf{A}_+,\mathsf{T}_+)$ is periodic in the sense of \eqref{cw14}, 
and thus by the above remark on preservation of periodicity also 
$\overline{V}_+=V_+-C_\alpha V_+$, cf.~\eqref{j535}, 
and if a form $V_-$ on $(\mathsf{A}_-,\mathsf{T}_-)$ is periodic in the same sense, then 
$\overline{V}_+\otimes V_-$ is also periodic. 
%
%

\medskip

Considering now the {\it extended product structure} $(\oAa{},\oTa{},\oGa{}{})$, 
cf.~\eqref{j1102} and \eqref{j980}, and extending \eqref{cw18} also to $\kappa-\alpha$ $
=\eta-\alpha-2\not=\alpha-2,2\alpha-2$ as a definition, 
we learn from the periodicity of $V_-$ that also $\overline{\Gamma}$
is periodic. Moreover, $V.{\mathsf{u}_-\choose\mathsf{v}}$ 
$=(C_\alpha V).\mathsf{u}_-+\overline{V}_+\otimes V_-.\mathsf{v}$, cf.~\eqref{MS1}, is periodic. 

\medskip

Returning to {\it reduction}, this time for the extended product structure, we note that the analogue
of \eqref{cw17} reads  
\begin{align}\nonumber
(k+\mathsf{v})+\mathsf{v}'=k+(\mathsf{v}+\mathsf{v}')\quad\mbox{provided}\;\mathsf{v}_{\alpha-2}'=\mathsf{v}_{2\alpha-2}'=0.
\end{align}
Hence by the argument used above for reduction of a form $V_+$ on $(\mathsf{A}_+,\mathsf{T}_+)$, 
the reduction of a form $V$ on $(\oAa{},\oTa{})$ to levels above $2\alpha-2$ 
preserves periodicity. We will only use that for $\cut{\thresh-\alpha-2}$,
which is in this range.

\medskip

We finally argue that $\cut{\thresh-\alpha-2}V_{a\diamond\partial_1^2u}^{\eta+\alpha-2}$ functorially
defined in \eqref{S880} is periodic. As argued above in the context of differentiation, 
the above deduced periodicity of $V_u^{\eta}$ entails the one of  
$V_{\partial_1^2u}^{\eta-2}$ via \eqref{functDef}. As argued above in the context of the extended
product, the assumed periodicity of 
$V_a^\eta$ and the just established periodicity of $V_{\partial_1^2u}^{\eta-2}$ entails the one
of $V_{a\diamond \partial_1^2u}^{\eta+\alpha-2}$ defined in \eqref{S880}. As argued in
the context of reduction for $\oTa{}$, this 
in turn implies the periodicity of $\cut{\thresh-\alpha-2} V_{a\diamond\partial_1^2u}^{\eta+\alpha-2}$.

\medskip

Finally, we note that the function $\cut{\thresh-\alpha-2} V_{a\diamond\partial_1^2u}^{\eta+\alpha-2}.v_T$ appearing
in \eqref{eq:QualRO} in Lemma \ref{rec_lem_concr} is (plain) periodic. Indeed, \eqref{cw21}
implies $v_T(k+x)=k+v_T(x)$ in the classical sense (with values in $\mathsf{T}$), and 
which by the assumed periodicity of $u$ and thus $(\partial_1^2u)_T$ extends to $\oTa{}$.
Combining this with the periodicity of $\cut{\thresh-\alpha-2} V_{a\diamond\partial_1^2u}^{\eta+\alpha-2}$ in form of
$\cut{\thresh-\alpha-2} V_{a\diamond\partial_1^2u}^{\eta+\alpha-2}(k+x).(k+\mathsf{v})$ 
$=\cut{\thresh-\alpha-2} V_{a\diamond\partial_1^2u}^{\eta+\alpha-2}(x).\mathsf{v}$, we obtain the desired
$\cut{\thresh-\alpha-2} V_{a\diamond\partial_1^2u}^{\eta+\alpha-2}(k+x).v_T(k+x)$ 
$=\cut{\thresh-\alpha-2} V_{a\diamond\partial_1^2u}^{\eta+\alpha-2}(x).v_T(x)$. By \eqref{eq:QualRO} and $\eta+\alpha-2>0$,
this (plain) periodicity of $\cut{\thresh-\alpha-2} V_{a\diamond\partial_1^2u}^{\eta+\alpha-2}.v_T$ extends to $a\diam\partial_1^2u$.

\section{Proofs of Abstract Results}\label{6}
\subsection{Proof of Remarks \ref{wg01} and \ref{wg03}}\label{subsec_remarkPf}

\begin{proof}
We start with Remark \ref{wg01}. We fix $x,y$ and note that \eqref{prod5} is equivalent to
\begin{align*}
|(V(y)-V(x)).\Gamma^{-1}(x)\widetilde{\mathsf{v}}|\le [V]_{{\mathcal D}^{\overline{\alpha}}(\mathsf{T};\Gamma)}
d^{(\overline{\alpha}-\beta)\vee0}(y,x)\|\widetilde{\mathsf{v}}\|_{\mathsf{T}_\beta}
\end{align*}
for all $\beta\in\mathsf{A}$ and 
$\widetilde{\mathsf{v}}\in\mathsf{T}_\beta$. By \eqref{cw92}, the l.~h.~s.~may be rewritten as
$|\tilde V(y).\Gamma(y)\Gamma^{-1}(x)\widetilde{\mathsf{v}}-\tilde V(x).\widetilde{\mathsf{v}}|$.
Hence by duality, this is equivalent to \eqref{cw90}.

\medskip

We now turn to Remark \ref{wg03} and start by giving the argument for \eqref{j1003}.
%
%
By the second item in \eqref{j700}, the first identity for $\eta=\alpha$ follows from the definition
\eqref{cw98}. Hence by definition of $C_\eta$, cf.~Subsection \ref{S:red}, it is enough 
to check that for $\eta\in\mathsf{A}_+'$, the one-step-reduction \eqref{cut2} preserves 
$C_\eta V_+.v_+$. This in turn follows from $({\rm id}_{\mathsf{T}_+}-\Gamma_\eta)v_+=v_+$, which is a
consequence of the first item in \eqref{j700}. 

\medskip

By \eqref{j1003}, the l.~h.~s.~of \eqref{cw91} simplifies:
\begin{align*}
\lefteqn{(u_+(y)-u_+(x))-\overline{V}_+(x).(v_+(y)-v_+(x))}\nonumber\\
&\stackrel{\eqref{j535},\eqref{j1003}}{=}(V_+(y)-V_+(x)).v_+(y)-C_\alpha V_+(x).(v_+(y)-v_+(x))\nonumber\\
&\stackrel{\eqref{j700}}{=}(V_+(y)-V_+(x)).v_+(y),
\end{align*}
where the last identity uses the fact that $C_\alpha V_+.\mathsf{v}_+$ depends on $\mathsf{v}_+$
only through $\mathsf{1}$.
Hence by form continuity \eqref{prod5}, with the roles of $x$ and $y$ exchanged, we obtain
as desired
\begin{align*}
\lefteqn{|(u_+(y)-u_+(x))-\overline{V_+}(x).(v_+(y)-v_+(x))|}\nonumber\\
&\stackrel{\eqref{prod5}}{\le}[V_+]_{{\mathcal D}^{\overline{\alpha}}(\mathsf{T}_+;\Gamma_+)}
\sum_{\beta\in\mathsf{A}_+}d^{(\overline{\alpha}-\beta)\vee0}(x,y)
\|\Gamma_\beta(y)v_+(y)\|_{\mathsf{T}_\beta}\nonumber\\
&\stackrel{\eqref{j700}}{=}[V_+]_{{\mathcal D}^{\overline{\alpha}}(\mathsf{T}_+;\Gamma_+)}
d^{\overline{\alpha}}(x,y).
\end{align*}
\end{proof}
\subsection{Proposition \ref{rec_lem}}
\begin{proof}[Proof of Proposition \ref{rec_lem}]
\newcounter{RecP} 
\refstepcounter{RecP} 
 The strategy of the proof is to define $\mathcal{R}V$ as a limit in $C^{\betau}(\R^2)\subset \cals'(\R^2)$ via 
 \begin{align}\label{j1006}
 \mathcal{R}V:=\lim_{\tau \to 0}V.v_\tau.
 \end{align}
 The defining inequality \eqref{cw40} is obtained in the limit $\tau \to 0$ from an analogous inequality where $(\mathcal{R}V)_{T}$ is replaced by $(V.\tf_{\tau})_{T-\tau}$.  This leads us to analyze the quantity
 \begin{equation*}
 (V.v_{\tau})_{T-\tau}-V.\tf_{T},
 \end{equation*}
 which, by virtue of the semi-group property \eqref{1.10} may be written as a commutator 
 \begin{equation}
 [(\cdot)_{T-\tau},V].\tf_{\tau}. \label{se10}
 \end{equation}
 Here, and in the whole proof of Proposition \ref{rec_lem}, we use the commutator notation $[(\cdot),V]v:=(V.v)_T-V.v_T$ for a generic $v\in \mathcal{D}'(\R^{2};\Ta{})$. Steps \ref{rec_dya_St} to \ref{rec_fullest_St} are devoted to estimating \eqref{se10} with $\tau>0$ fixed under a slightly more restrictive assumption on the modelled distribution, while Step \ref{rec_conc_St} concludes with a compactness argument. This compactness result itself relies on an independent Step \ref{rec_comp_St}, while in step \ref{rec_large_dist}, we argue that the more restrictive assumption used in the first steps is justified.

 In terms of the properties of the modelled distribution $V$, Steps \ref{rec_dya_St} to \ref{rec_fullest_St} require only the form continuity condition \eqref{prod5}.  However, the boundedness condition \eqref{prod4} is required qualitatively for Step \ref{rec_conc_St}.

 \medskip
 
{\sc Step} \arabic{RecP}.\label{rec_dya_St}\refstepcounter{RecP} We claim that for all $\tau<T$, where $T$ is a dyadic multiple of $\tau$, the following identity holds
 \begin{align}\label{p05}
  [(\cdot)_{T-\tau},V].\tf_{\tau}=\sum_{\tau\le t< T}  \vp{[(\cdot)_t,V].\tf_{t}}_{T-2t},
 \end{align}
 where the sum runs over all dyadic multiples $t=\frac{T}{2}, \frac{T}{4},\cdots,\tau$.
 In fact, by the semi-group property \eqref{1.10} of the convolution kernel
 \begin{align*}
 \vp{[(\cdot)_t,V].\tf_{t}}_{T-2t}&=\vp{V.\tf_{t}}_{T-t}- \vp{V.v_{2t}}_{T-2t},
 \end{align*}
 which turns the right hand side of \eqref{p05} into a telescoping sum.
 
 \medskip
 
{\sc Step} \arabic{RecP}.\label{rec_est_St}\refstepcounter{RecP} We claim that for all $s\le t \le 1$ it holds
 \begin{align}\label{p04}
   \|[(\cdot)_s,V].\tf_{t}\|\lesssim [V]_{\cald^{\alphao}(\Ta{};\Ga{}{})} \sum_{\beta \in \Aa{}} (t^\frac14)^{\alphao},
 \end{align}
 where we recall that $\|\cdot\|$ denotes the supremum norm on $\R^2$.
 For clarity of exposition, we will give the proof of \eqref{p04} under the more restrictive assumption that the continuity property \eqref{prod5} holds for all $x,y \in \R^{2}$ (dropping the assumption that $d(x,y) \leq 1$).  In Step \ref{rec_large_dist}, we will give the (slightly more involved) argument without this additional hypothesis.
 
 \medskip
 
 Note that for each $x\in\R^2$
 \begin{align}\label{p03}
  [(\cdot)_s,V].\tf_{t}(x)=\int \psi_s(x-y) (V(y)-V(x)).\tf_{t}(y)  \dd y.
 \end{align}
Indeed, by the definition of the convolution $[(\cdot)_s,V].\tf_{t}(x)$ equals
 \begin{align*}
  \int \psi_s&(x-y) V(y).\tf_{t}(y)  \dd y - V(x).\big(\int \psi_s(x-y) v_{t}(y)  \dd y\big) \\
  &=\int \psi_s(x-y) V(y).\tf_{t}(y)  \dd y - \int \psi_s(x-y) V(x). \tf_{t}(y)  \dd y,
 \end{align*}
 which is \eqref{p03}.  Here, we used that for each $x$, the linear form $V(x)$ is continuous on $\Ta{}$ in order to interchange integration in $y$ with application of $V(x)$ in the above Bochner integral.
 
 \medskip
 
 To estimate \eqref{p03}, we will appeal to our simplifying assumption, that the form continuity condition \eqref{prod5} holds for all $x,y \in \R^{2}$.  We combine this with \eqref{cw04} to find
 \begin{align}
 |(&V(y)-V(x)).\tf_{t}(y)|\nonumber\\
 &\stackrel{\eqref{prod5}}{\lesssim} [V]_{\cald^{\alphao}(\Ta{};\Ga{}{})}\sum_{\beta\in \Aa{}} d^{\vp{\alphao-\beta} \vee 0}(y,x)\|\Ga{\beta}{}(y)\tf_{t}(y)\|_{\Ta{\beta}}\nonumber\\ 
 &\stackrel{\eqref{cw04}}{\le} [V]_{\cald^{\alphao}(\Ta{};\Ga{}{})} \sum_{\beta\in \Aa{}} d^{\vp{\alphao-\beta} \vee 0}(y,x)(t^{\frac14})^{\beta}\label{se2}.  
 \end{align}
 Indeed, note that the left hand side of \eqref{prod5} is symmetric in $x,y$, so applying \eqref{prod5} with $x,y$ interchanged and $\ta=\tf_{t}(y)$ gives the first inequality.
 Now we combine \eqref{p03}, \eqref{se2}, and the moment bounds \eqref{1.13} to the effect of
 \begin{align*}
  \int &|(V(y)-V(x)).\tf_{t}(y)| |\psi_s(x-y)| \dd y \\
   &\stackrel{\mathclap{\eqref{se2}}}{\le} \ [V]_{\cald^{\alphao}(\Ta{};\Ga{}{})} \sum_{\beta\in \Aa{}} \int d^{\vp{\alphao-\beta} \vee 0}(y,x)(t^{\frac14})^{\beta}  |\psi_s(x-y)|\dd y \\
   &\stackrel{\mathclap{\eqref{1.13}}}{\lesssim} [V]_{\cald^{\alphao}(\Ta{};\Ga{}{})} \sum_{\beta\in \Aa{}}  (s^\frac14)^{\vp{\alphao-\beta} \vee 0} (t^\frac14)^{\beta},
 \end{align*}
 which implies \eqref{p04} in light of $s \leq t\le 1$ and $\vp{\vp{\alphao-\beta} \vee 0}+\beta\ge \alphao$.
\medskip

{\sc Step} \arabic{RecP}.\label{rec_fullest_St}\refstepcounter{RecP}  We now claim that if $\tau<T$, then
 \begin{align}\label{p06}
  \|[(\cdot)_{T-\tau},V].\tf_{\tau}\|\lesssim [V]_{\cald^{\alphao}(\Ta{};\Ga{}{})} (T^\frac14)^{\alphao} .
 \end{align}
 Let us first assume that $T$ is a dyadic multiple of $\tau$.
 Observe that for a generic $g$
 \begin{align}\label{j760}
  \|g_{T}\|\lesssim \|g_t\| \quad \text{ for all } T\ge t.
 \end{align}
 Hence,
 \begin{align*}
  \|[(\cdot)_{T-\tau},V].\tf_{\tau}\|&\stackrel{\eqref{p05}}{\leq} \sum_{\tau\le t< T} \| \vp{[(\cdot)_t,V].\tf_{t}}_{T-2t} \| \\
  &\stackrel{\eqref{j760}}{\lesssim} \sum_{\tau\le t< T} \|[(\cdot)_t,V].\tf_{t} \| \stackrel{\eqref{p04}}{\lesssim} [V]_{\cald^{\alphao}(\Ta{};\Ga{}{})} \sum_{\tau\le t< T} (t^\frac14)^{\alphao},
 \end{align*}
 where in the last inequality, we have used \eqref{p04} applied with $s=t\le 1$.
 Since $\alphao>0$, the geometric sum converges and we obtain \eqref{p06}. If now $0<\tau<T$ is arbitrary, we find a $\ti T\in [\frac{T}{2},T)$ which is a dyadic multiple of $\tau$ (including the case $\tau=\ti T$). Since
 \begin{align*}
  [(\cdot)_{T-\tau},V].\tf_{\tau}=\big([(\cdot)_{\ti T-\tau},V].\tf_{\tau}\big)_{T-\ti T}+ [(\cdot)_{T-\ti T},V].\tf_{\ti T},
 \end{align*}
 we have by \eqref{j760}
 \begin{align*}
  \|[(\cdot)_{T-\tau},V].\tf_{\tau}\| \lesssim \|[(\cdot)_{\ti T-\tau},V].\tf_{\tau}\| + \|[(\cdot)_{T-\ti T},V].\tf_{\ti T}\|.
 \end{align*}
 Since $\ti T$ is a dyadic multiple of $\tau$, we can estimate the first contribution of the right-hand side by
 \begin{align*}
  \|[(\cdot)_{\ti T-\tau},V].\tf_{\tau}\|\lesssim [V]_{\cald^{\alphao}(\Ta{};\Ga{}{})} (\ti T^\frac14)^{\alphao} \lesssim [V]_{\cald^{\alphao}(\Ta{};\Ga{}{})} (T^\frac14)^{\alphao}.
 \end{align*}
 For the second contribution, we apply \eqref{p04} with $s=T-\ti T$ and $t=\ti T$, which is applicable since $T-\ti T\le \ti T$ in virtue of $\ti T\in [\frac{T}{2},T)$, to obtain
 \begin{align*}
   \|[(\cdot)_{T-\ti T},V].\tf_{\ti T}\|&\lesssim [V]_{\cald^{\alphao}(\Ta{};\Ga{}{})} (\ti T^\frac14)^{\alphao} 
   \lesssim [V]_{\cald^{\alphao}(\Ta{};\Ga{}{})} (T^\frac14)^{\alphao}.
 \end{align*}

 \medskip
 
{\sc Step} \arabic{RecP}.\label{rec_conc_St}\refstepcounter{RecP} In the final step, we carry out a compactness argument to complete the proof.  We begin by defining $\mathcal{R}V^\tau:= V.\tf_{\tau}$ and claim that for each $\tau>0$ and all $T \leq 1$

\begin{align}
\|(\mathcal{R}V^\tau)_{T}\|\lesssim \|V\|_{\cald^{\alphao}(\Ta{};\Ga{}{})}(T^\frac14)^{\betau} \label{se6}. 
\end{align}

To prove the claim, fix a $\tau$ and 
appeal to \eqref{p06}, which by the semi-group property takes the form
 \begin{align}\label{p07}
  \|(\mathcal{R}V^\tau)_{T-\tau}-V.\tf_{T}\|\lesssim [V]_{\cald^{\alphao}(\Ta{};\Ga{}{})} (T^\frac14)^{\alphao}.
 \end{align}
 Next we observe that using the form boundedness property \eqref{prod4}, $N\le 1$ and definition \eqref{cw04}, it follows that
 \begin{align}
 \|V.\tf_{T}\|&\stackrel{\eqref{prod4}}{\le} \bd{V}{\Ta{}}{\Ga{}{}}\sum_{\beta\in \Aa{}}\sup_{x}\|\Ga{\beta}{}(x) \tf_{T}(x)\|_{\Ta{\beta}} \nonumber\\
 &\stackrel{\eqref{cw04}}{\le}   \bd{V,\mu}{\Ta{}}{\Ga{}{}} \sum_{\beta \in \Aa{}} (T^{\frac{1}{4}})^{\beta}\lesssim \bd{V,\mu}{\Ta{}}{\Ga{}{}} (T^{\frac{1}{4}})^{\underline \beta} \label{se7},
 \end{align}  
 where we used that $T \leq 1$ and $\beta\ge \betau$ for all $\beta\in \Aa{}$ in the final step.  Hence, noting that $\|(\mathcal{R}V^\tau)_T\|\lesssim \|(\mathcal{R}V^\tau)_{T-\tau}\|$ in virtue of \eqref{j760}, the estimate \eqref{se6} now follows from \eqref{p07} and \eqref{se7} via the triangle inequality and the fact that $\alphao\ge \betau$.

 \medskip

 Hence, by relying on the independent Step \ref{rec_comp_St} on weak compactness below, we find a subsequence $\tau_n\to 0$ as $n\to\infty$ and a distribution $\mathcal{R}V \in C^{\betau}$  such that $\mathcal{R}V^{\tau_n} \rightharpoonup \mathcal{R}V$ in the sense of distributions.  Thus, it only remains to obtain the inequality \eqref{cw40}, which we claim follows from passing to the limit $\tau_n \to 0$ in \eqref{p07}.  Indeed, note that for $\tau\le \frac{T}{2}$ it holds
 \begin{align*}
 \| (\calr V^\tau)_{T-\tau}-(\calr V)_{T}\|&= \| \vp{\calr V^\tau-(\calr V)_{\tau}}_{T-\tau}\|\\
 &\stackrel{\mathclap{\eqref{j760}}}{\lesssim} \| \vp{\calr V^\tau-(\calr V)_{\tau}}_{\frac{T}{2}}\| \\
 &\le \| \vp{\calr V^\tau-\calr V}_{\frac{T}{2}}\| + \| ((\calr V)_{\frac{T}{2}})_\tau-(\calr V)_{\frac{T}{2}}\|,
 \end{align*}
 which tends to zero along the subsequence $\tau_n \to 0$.  By the triangle inequality, it follows that \eqref{p07} implies \eqref{cw40},
 which completes the proof.

\medskip
 
{\sc Step} \arabic{RecP}.\label{rec_comp_St}\refstepcounter{RecP}[Weak compactness] We claim that for any sequence $\{f_n\}_{n\uparrow\infty}
\subset{\mathcal S'}(\mathbb{R}^2)$ with bounded 
$\sup_{T\le 1}(T^\frac{1}{4})^{-\underline{\beta}}\|(f_n)_T\|$, there exists a subsequence
that converges in the sense of distributions. This will turn out to be a consequence of the estimate
\begin{align}\label{cw30}
\|(1+\partial_1^4-\partial_2^2)^{-1}f\|\lesssim \sup_{T\le 1}(T^\frac{1}{4})^{-\underline{\beta}}\|f_T\|.
\end{align}
Note that the Fourier symbol of $(1+\partial_1^4-\partial_2^2)^{-1}$ is $(1+k_1^4+k_2^2)^{-1}$
and thus has a well-decaying and moderately regular kernel; in particular, it is integrable,
so that it acts on Schwartz functions. As a consequence, $(1+\partial_1^4-\partial_2^2)^{-1}f$ makes sense
as a distribution, and the finiteness of the l.h.s. of \eqref{cw30} is to be understood in the sense
that this distribution is represented by an $L^\infty(\mathbb{R}^2)$ function.

\medskip

We first argue that \eqref{cw30} yields the desired weak compactness result.
Indeed, under our assumptions, \eqref{cw30} implies that $u_n:=(1+\partial_1^4-\partial_2^2)^{-1}f_n$ 
is bounded in $L^\infty(\mathbb{R}^2)$ so that there exists a subsequence $u_{n'}$ and a
$u\in L^\infty(\mathbb{R}^2)$ such that $u_{n'}$ converges to $u$ in the
weak-$*$ topology. Since $L^\infty(\mathbb{R}^2)\subset L_{loc}^1(\mathbb{R}^2)$, this implies
that $f_{n'}$ converges to $f:=(1+\partial_1^4-\partial_1^2)u$ in ${\mathcal D}(\mathbb{R}^2)$.

\medskip

We now turn to the proof of \eqref{cw30}. By definition of the convolution $(\cdot)_t$, cf.~Subsection \ref{S:pre},
we have $\partial_t f_t=-(\partial_1^4-\partial_2^2)f_t$, so that $\partial_t(1+\partial_1^4-\partial_2^2)^{-1}f_t$
$=-(1+\partial_1^4-\partial_2^2)^{-1}(\partial_1^4-\partial_2^2)f_t$. Since
$(1+\partial_1^4-\partial_2^2)^{-1}(\partial_1^4-\partial_2^2)$ 
$={\rm id}-(1+\partial_1^4-\partial_2^2)^{-1}$ and the operator $(1+\partial_1^4-\partial_2^2)^{-1}$
has integrable kernel (see above), this yields the estimate
\begin{align*}
\|\partial_t(1+\partial_1^4-\partial_2^2)^{-1}f_t\|\lesssim\|f_t\|.
\end{align*}
Since $\underline{\beta}>-4$, we obtain by integration over $t\in(\tau,1)$
\begin{align*}
\|(1+\partial_1^4-\partial_2^2)^{-1}f_\tau\|\lesssim\|(1+\partial_1^4-\partial_2^2)^{-1}f_1\|
+\sup_{t\le 1}(t^\frac{1}{4})^{-\underline{\beta}}\|f_t\|.
\end{align*}
Using once more the fact that $(1+\partial_1^4-\partial_2^2)^{-1}$ has integrable kernel to absorb
the first r.h.s.~term into the second,
this implies \eqref{cw30} in the limit $\tau\downarrow 0$, appealing to the lower semi-continuity of
the $\|\cdot\|$-norm, interpreted as an essential supremum, under distributional convergence.

\medskip

{\sc Step} \arabic{RecP}.\label{rec_large_dist}\refstepcounter{RecP}[Proof in the General Case] In this step, we address the modifications of Step \ref{rec_est_St} made necessary
by the fact that the continuity condition \eqref{prod5} on the modelled distribution
$V$ only holds for pairs of points $(x,y)$ with $d(y,x)\le 1$. In fact, we will establish that
for all $x,y\in\mathbb{R}^2$ and $\mathsf{v}\in\mathsf{T}$
\begin{align}\label{cw01}
|(V(y)-V(x)).\mathsf{v}|&\lesssim [V]_{{\mathcal D}^{\overline{\alpha}}(\mathsf{T};\Gamma)} (1+\|\Gamma\|_{sk})^{|\Aa{}|-1}
d^{2|\Aa{}|}(y,x)\sum_{\beta}\|\Gamma_\beta(y)\mathsf{v}\|_{\mathsf{T}_\beta}\\[-2ex]
&\mbox{provided}\;d(y,x)\ge 1,\nonumber
\end{align}
where we recall that $|\Aa{}|$ denotes the number of elements in $\Aa{}$
and $\lesssim$ stands for $\le C$ up to a constant only depending on $\Aa{}$ (and on $\overline{\alpha}$, which
is irrelevant here). 
In order to do so, we also
have to overcome the fact that the continuity condition \eqref{ssss2} on the skeleton $\Gamma$ only
holds for pairs of points $(x,y)$ with $d(y,x)\le 1$. In fact, we will establish that
for all $\beta\in \Aa{}$,  $x,y\in\mathbb{R}^2$, and $\mathsf{v}\in\mathsf{T}$
\begin{align}\label{cw02a}
\|\Gamma_\beta(x)\mathsf{v}\|_{\mathsf{T}_\beta}&\lesssim \big((1+\|\Gamma\|_{sk})d^2(y,x)\big)^{|\Aa{}|-1}
\sum_{\gamma\le\beta}\|\Gamma_\gamma(y)\mathsf{v}\|_{\mathsf{T}_\gamma}\\[-2ex]
&\mbox{provided}\;d(y,x)\ge 1.\nonumber
\end{align}
Before establishing \eqref{cw01} and \eqref{cw02a}, we first argue how \eqref{cw01} yields the outcome of Step \ref{rec_est_St},
which in this more general context takes the form of
\begin{align}\label{cw03}
\|[(\cdot)_s,V].v_t\|\lesssim (1+\|\Gamma\|_{sk})^{|\Aa{}|-1} [V]_{{\mathcal D}^{\overline{\alpha}}(\mathsf{T};\Gamma)}
(t^\frac{1}{4})^{\overline{\alpha}}.
\end{align}
Indeed, inserting our assumption \eqref{cw04} into \eqref{cw01} yields for all points $x$, $y$ with $d(y,x)\ge 1$
and all $t\le 1$
\begin{align}\label{cw46}
\lefteqn{|(V(y)-V(x)).v_t(y)|}\nonumber\\
&\lesssim (1+\|\Gamma\|_{sk})^{|\Aa{}|-1}  [V]_{{\mathcal D}^{\overline{\alpha}}(\mathsf{T};\Gamma)}
d^{2|\Aa{}|}(y,x)\sum_{\beta}(t^\frac{1}{4})^{\beta},
\end{align}
and thus trivially, recalling that $\underline{\beta}<0$ denotes the most negative exponent in $\Aa{}$,
\begin{align*}
\lefteqn{|(V(y)-V(x)).v_t(y)|}\nonumber\\
&\lesssim (1+\|\Gamma\|_{sk})^{|\Aa{}|-1}  [V]_{{\mathcal D}^{\overline{\alpha}}(\mathsf{T};\Gamma)}
d^{-\underline{\beta}+2|\Aa{}|+\overline{\alpha}}(y,x)(t^\frac{1}{4})^{\underline{\beta}-2|\Aa{}|},
\end{align*}
so that by \eqref{1.13}
\begin{align*}
\lefteqn{\int_{d(y,x)\ge 1}|(V(y)-V(x)).v_t(y)||\psi_s(x-y)|{\rm d}y}\\
&\lesssim (1+\|\Gamma\|_{sk})^{|\Aa{}|-1} [V]_{{\mathcal D}^{\overline{\alpha}}(\mathsf{T};\Gamma)}
(s^\frac{1}{4})^{-\underline{\beta}+2|\Aa{}|+\overline{\alpha}}
(t^\frac{1}{4})^{\underline{\beta}-2|\Aa{}|}\\
&\le(1+\|\Gamma\|_{sk})^{|\Aa{}|-1} [V]_{{\mathcal D}^{\overline{\alpha}}(\mathsf{T};\Gamma)}
(t^\frac{1}{4})^{\overline{\alpha}}.
\end{align*}
By the same argument as in \eqref{se2} of Step \ref{rec_est_St} we have
\begin{align*}
\int_{d(y,x)\le 1}|(V(y)-V(x)).v_t(y)||\psi_s(x-y)|{\rm d}y\lesssim [V]_{{\mathcal D}^{\overline{\alpha}}(\mathsf{T};\Gamma)}
(t^\frac{1}{4})^{\overline{\alpha}},
\end{align*}
so that by the representation \eqref{p03}, we obtain \eqref{cw03}.

\medskip

It is easy to see how \eqref{cw01} follows from \eqref{cw02a} and our assumption \eqref{prod5}:
Let $K\in\{2,3,\cdots\}$ be such that $K-1<d(y,x)\le K$ and 
divide the segment between $x$ and $y$ into $K^2$ intervals delimited by the
points $x=z_0,\cdots,z_{K^2}=y$. By definition, the $K^2+1$ adjacent points have distance $d(z_{k+1},z_k)\le 1$
(note that we need $K^2$ intervals instead of $K$ because of the square root in the definition of
the metric $d$, cf.~\eqref{1.11}), so that we may apply \eqref{prod5} on those. We have
\begin{align*}
\lefteqn{|(V(y)-V(x)).\mathsf{v}|\le\sum_{k=0}^{K^2-1}|(V(z_{k+1})-V(z_k)).\mathsf{v}|}\nonumber\\
&\stackrel{\eqref{prod5}}{\le}[V]_{{\mathcal D}^{\overline{\alpha}}(\mathsf{T};\Gamma)}
\sum_{k=0}^{K^2-1}\sum_\beta\|\Gamma_\beta(z_k)\mathsf{v}\|_{\mathsf{T}_\beta}\nonumber\\
&\stackrel{\eqref{cw02a}}{\lesssim}(1+\|\Gamma\|_{sk})^{|\Aa{}|-1} [V]_{{\mathcal D}^{\overline{\alpha}}(\mathsf{T};\Gamma)}
\sum_{k=0}^{K^2-1}d^{2(|\Aa{}|-1)}(y,z_k)\sum_{\beta}\sum_{\gamma\le\beta}
\|\Gamma_\gamma(y)\mathsf{v}\|_{\mathsf{T}_\gamma},
\end{align*}
which yields \eqref{cw01} because of $K^2\le (d(y,x)+1)^2\le 4d^2(y,x)$ and $d(y,z_k)\le d(y,x)$.

\medskip

We finally turn to the argument for \eqref{cw02a}. It is the nilpotent structure of the continuity condition
\eqref{ssss2} that saves the day. Appealing to the triangle inequality in $\mathsf{T}_\beta$, we use in the simpler form of
\begin{align*}
\|\Gamma_\beta(z_{k})\mathsf{v}\|_{\mathsf{T}_\beta}
\le\|\Gamma_\beta(z_{k+1})\mathsf{v}\|_{\mathsf{T}_\beta}+\|\Gamma\|_{sk}\sum_{\gamma<\beta}
\|\Gamma_\gamma(z_{k+1})\mathsf{v}\|_{\mathsf{T}_\gamma},
\end{align*}
applied to the previously constructed sequence $x=z_0,\cdots,z_{K^2}=y$ of $K^2+1$ points, the adjacent
of which have distance at most one. By iterating this inequality we obtain
\begin{align*}
\|\Gamma_\beta(x)\mathsf{v}\|_{\mathsf{T}_\beta}
\le\sum_{\gamma=\epsilon_0\le\cdots\le\epsilon_{K^2}=\beta}\|\Gamma\|_{sk}^{|\Aa{}(\gamma,\beta)|-1}
\|\Gamma_\gamma(y)\mathsf{v}\|_{\mathsf{T}_\gamma},
\end{align*}
where $|\Aa{}(\gamma,\beta)|$ denotes the number of elements of $\Aa{}\cap [\gamma,\beta]$. Note that an increasing sequence 
$\gamma=\epsilon_0\le\cdots\le\epsilon_{K^2}=\beta$ with values in $\Aa{}\cap[\gamma,\beta]$ 
is parameterized by the (not-strictly)
increasing sequence of length $|\Aa{}(\gamma,\beta)|-1$ of positions where it jumps (by one increment).
Since there are $K^2$ possible positions for these jumps (namely parameterized by the $K^2$ intervals), 
the number of such sequences is estimated by $(K^2)^{|\Aa{}(\gamma,\beta)|-1}$. Since trivially,
$|\Aa{}(\gamma,\beta)|\le|\Aa{}|$ and $K^2\le 4d^2(y,x)$, we obtain \eqref{cw02a}.
\end{proof}
\subsection{Lemma \ref{cutting_lemma}}\label{S:cutPf}
\begin{proof}[Proof of Lemma \ref{cutting_lemma}]
Before turning to the proof, we complete the definition of the reduction $C_\eta V$.
To this purpose, we enumerate the homogeneities $\{\beta_j\}_{j=1}^J$ in increasing order, so that $\beta_1=\underline{\beta}$
and $\beta_J=\overline{\beta}$. Note that $C_{\beta_J}V$ has already been defined in \eqref{cut2}. The reductions 
$C_{\beta_j}V$ for $j\in\{1,\cdots,J-1\}$ are recursively defined through
\begin{align}\label{cutind}
C_{\beta_j}V.\mathsf{v}:=C_{\beta_{j+1}}V.({\rm id}_\mathsf{T}-\Gamma_{\beta_{j}})\mathsf{v},
\end{align}
where we view the operator $\Gamma_{\beta_{j}}$ as an element of ${\mathcal L}(\mathsf{T},\mathsf{T})$ with the understanding that $(\Ga{\beta_{j}}{}(x)\ta)_{\gamma}=0$ for $\gamma \neq \beta_{j}$.  For intermediate values $\eta\in(\beta_j,\beta_{j+1}]$ we just set
\begin{align}\label{eq:cutExt}
C_\eta V:=C_{\beta_{j+1}}V.
\end{align}

\medskip

\newcounter{Cutprf} 
\refstepcounter{Cutprf} 
{\sc Step} \arabic{Cutprf}.\label{Cutprf_i}\refstepcounter{Cutprf} Form boundedness \eqref{S70a} in case of $\eta=\overline{\beta}$. We start by noting that
the triangular structure \eqref{j551} implies
\begin{align}\label{cw78}
\Gamma_\beta({\rm id}-\Gamma_\eta)=\left\{\begin{array}{cc}
\Gamma_\beta&\mbox{for}\;\beta<\eta\\
0&\mbox{for}\;\beta=\eta
\end{array}\right\},
\end{align}
where we write for abbreviation ${\rm id}={\rm id}_{\mathsf{T}}$.  We use this in the string of inequalities
\begin{align*}
|C_\eta V.\mathsf{v}|&\stackrel{\mathclap{\eqref{cut2}}}{=}|V.({\rm id}-\Gamma_\eta)\mathsf{v}| 
\stackrel{\mathclap{\eqref{wg16}}}{\le}
\|V\|_{\Ta{};\Ga{}{}}\sum_{\beta}N^{\ap{\beta}}\|\Gamma_\beta({\rm id}-\Gamma_\eta)\mathsf{v}\|_{\Ta{\beta}} \\ &\stackrel{\mathclap{\eqref{cw78}}}{=} \,\|V\|_{\mathsf{T};\Gamma}\sum_{\beta<\eta}N^{\ap{\beta}} \|\Gamma_\beta\mathsf{v}\|_{\Ta{\beta}}
\le\|V\|_{\Ta{};\Ga{}{}}\sum_{\beta}N^{\ap{\beta}}\|\Gamma_\beta\mathsf{v}\|_{\Ta{\beta}},
\end{align*}
which gives \eqref{S70a} in view of \eqref{wg16}.

\medskip

{\sc Step} \arabic{Cutprf}.\label{Cutprf_ii}\refstepcounter{Cutprf} Form continuity \eqref{eq:cutLemmaBd1} in case of $\eta=\overline{\beta}$. More precisely, we claim
\begin{align}\label{cw80}
[C_\eta V]_{{\mathcal D}^{\overline{\alpha}\wedge\eta}(\mathsf{T};\Gamma)}
\le[V]_{{\mathcal D}^{\overline{\alpha}}(\mathsf{T};\Gamma)}+N^{\ap{\eta}}\|\Gamma\|_{sk}
\|V\|_{\mathsf{T};\Gamma}.
\end{align}
To this purpose, we give ourselves two points $x,y$ with $d(y,x)\le 1$ and write
\begin{align}\label{cw81}
\lefteqn{(C_\eta V(y)-C_\eta V(x)).\mathsf{v}}\nonumber\\
&\stackrel{\eqref{cut2}}{=}(V(y)-V(x)).({\rm id}-\Gamma_\eta(x))\mathsf{v}\,-\,V(y).(\Gamma_\eta(y)-\Gamma_\eta(x))\mathsf{v},
\end{align}
so that
\begin{align*}
\lefteqn{|(C_\eta V(y)-C_\eta V(x)).\mathsf{v}|}\nonumber\\
&\stackrel{\eqref{prod5},\eqref{wg16}}{\le}
[V]_{{\mathcal D}^{\overline{\alpha}}(\mathsf{T};\Gamma)}\sum_{\beta}d^{(\overline{\alpha}-\beta)\vee0}(y,x)
\|\Gamma_\beta(x)({\rm id}-\Gamma_\eta(x))\mathsf{v}\|_{\mathsf{T}_\beta}\nonumber\\
&+\|V\|_{\mathsf{T};\Gamma}\sum_{\beta}N^{\ap{\beta}}\|\Gamma_\beta(y)(\Gamma_\eta(y)-\Gamma_\eta(x))\mathsf{v}\|_{\mathsf{T}_\beta}\\
&\stackrel{\eqref{cw78},\eqref{j551}}{=}
[V]_{{\mathcal D}^{\overline{\alpha}}(\mathsf{T};\Gamma)}\sum_{\beta<\eta}d^{(\overline{\alpha}-\beta)\vee0}(y,x)
\|\Gamma_\beta(x)\mathsf{v}\|_{\mathsf{T}_\beta}\nonumber\\
&+N^{\ap{\eta}}\|V\|_{\mathsf{T};\Gamma}\|(\Gamma_\eta(y)-\Gamma_\eta(x))\mathsf{v}\|_{\mathsf{T}_\eta}\\
&\stackrel{\eqref{ssss2}}{\le}
[V]_{{\mathcal D}^{\overline{\alpha}}(\mathsf{T};\Gamma)}\sum_{\beta<\eta}d^{(\overline{\alpha}-\beta)\vee0}(y,x)
\|\Gamma_\beta(x)\mathsf{v}\|_{\mathsf{T}_\beta}\nonumber\\
&+N^{\ap{\eta}}\|\Gamma\|_{sk}\|V\|_{\mathsf{T};\Gamma}
\sum_{\beta<\eta}d^{\eta-\beta}(y,x)\|\Gamma_\beta(x)\mathsf{v}\|_{\mathsf{T}_\beta}.
\end{align*}
Since $(\overline{\alpha}-\beta)\vee 0,\eta-\beta\ge(\overline{\alpha}\wedge\eta-\beta)\vee 0$ and $d(y,x)\le 1$
we obtain
\begin{align*}
\lefteqn{|(C_\eta V(y)-C_\eta V(x)).\mathsf{v}|}\nonumber\\
&{\le}
([V]_{{\mathcal D}^{\overline{\alpha}}(\mathsf{T};\Gamma)}+N^{\ap{\eta}}\|\Gamma\|_{sk}\|V\|_{\mathsf{T};\Gamma})
\sum_{\beta<\eta}d^{(\overline{\alpha}\wedge\eta-\beta)\vee 0}(y,x)\|\Gamma_\beta(x)\mathsf{v}\|_{\mathsf{T}_\beta},
\end{align*}
which is the detailed version of \eqref{cw80}.

\medskip

{\sc Step} \arabic{Cutprf}.\label{Cutprf_iii}\refstepcounter{Cutprf} General $\eta$. We first treat the case of $\eta=\beta_j$ for some $j\in\{1,\cdots,J-1\}$.
By Step \ref{Cutprf_i} and Step \ref{Cutprf_ii} applied with $V$ replaced by $C_{\beta_{j+1}}V$ (and $\overline{\alpha}$ replaced by
$\overline{\alpha}\wedge\beta_{j+1}$) we obtain by our inductive
definition \eqref{cutind}
\begin{align}
\|C_{\beta_j}V\|_{\Ta{};\Ga{}{}}&\le \|C_{\beta_{j+1}}V\|_{\Ta{};\Ga{}{}},\label{S77a}\\
[C_{\beta_j}V]_{{\mathcal D}^{\overline{\alpha}\wedge\beta_j}(\mathsf{T};\Gamma)}&\le
[C_{\beta_{j+1}}V]_{{\mathcal D}^{\overline{\alpha}\wedge\beta_{j+1}}(\mathsf{T};\Gamma)} \nonumber \\
&\quad +N^{\ap{\beta_j}}\|\Gamma\|_{sk}\|C_{\beta_{j+1}}V\|_{\mathsf{T};\Gamma}.\label{eq:cpit}
\end{align}
Iterating these two inequalities and using $N\le 1$ yields \eqref{S70a} \& \eqref{eq:cutLemmaBd1} for all $\eta\in \Aa{}$.
The extension to arbitrary $\eta\in[\underline{\beta},\overline{\beta}]$, say $\eta\in(\beta_j,\beta_{j+1}]$
for some $j\in\{1,\cdots,J-1\}$, uses definition \eqref{eq:cutExt}, with \eqref{S70a} being immediate and
\eqref{eq:cutLemmaBd1} following as a result of the monotonicity
of the semi-norm in the exponent, i.~e.
$[\cdot]_{{\mathcal D}^{\overline{\alpha}\wedge\eta}(\mathsf{T};\Gamma)}$
$\le [\cdot]_{{\mathcal D}^{\overline{\alpha}\wedge\beta_{j+1}}(\mathsf{T};\Gamma)}$.

\medskip

{\sc Step} \arabic{Cutprf}.\label{Cutprf_iv}\refstepcounter{Cutprf} Proof of \eqref{S8300} and \eqref{new}.
Recalling \eqref{cutind}, arguing by induction (the base case following as in Step \ref{Cutprf_i}), we find that for all $x \in \R^{2}$ and $\mathsf{v} \in \Ta{}$ it holds for each $\beta_{j} \in \mathsf{A}$
\begin{align*}
\left |C_{\beta_{j}}V(x).\mathsf{v} \right | \leq \|V\|_{\Ta{};\Ga{}{}}\sum_{\beta \in [\underline{\beta},\beta_{j}) \cap \mathsf{A}} N^{\ap{\beta}}\|\Ga{\beta}{}(x)\mathsf{v} \|_{\Ta{\beta}}.
\end{align*}
Re-writing \eqref{cutind} as
\begin{align*}
\left (C_{\beta_{j+1}}-C_{\beta_{j}} \right)V(x).\mathsf{v}&=C_{\beta_{j+1}}V(x).\Ga{\beta_{j}}{}(x)\mathsf{v} 
\end{align*}
we find that
\begin{align*}
| \left (C_{\beta_{j+1}}-C_{\beta_{j}} \right)V(x).\mathsf{v} | &\stackrel{\mathclap{\eqref{S70a}}}{\leq} \|V\|_{\Ta{};\Ga{}{}} \sum_{\beta <\beta_{j+1}} N^{\ap{\beta}} \|\Ga{\beta}{}(x)\Ga{\beta_{j}}{}(x)\mathsf{v} \|_{\Ta{\beta}}\\
&= \|V\|_{\Ta{};\Ga{}{}} N^{\ap{\beta_j}} \|\Ga{\beta_j}{}(x)\mathsf{v} \|_{\Ta{\beta_j}},
\end{align*}
since $\Ga{\beta}{}\Ga{\beta_{j}}{}$ vanishes for $\beta < \beta_{j}$.  Let $\eta< \kappa$ and choose natural numbers $k,m$ such that $\eta \in (\beta_{k-1},\beta_{k}]$, $\kappa \in (\beta_{k+m-1},\beta_{k+m}]$.  By telescoping and the triangle inequality we find 
\begin{align*}
 |&(C_{\kappa}-C_{\eta})V(x).\mathsf{v}  | =\left |(C_{\beta_{k+m}}-C_{\beta_{k}})V(x).\mathsf{v} \right | \\
&\le \sum_{j=k}^{k+m-1}\left |\left (C_{\beta_{j+1}}-C_{\beta_{j}} \right)V(x).\mathsf{v} \right |  \leq \|V\|_{\Ta{};\Ga{}{}}\sum_{j=k}^{k+m-1} N^{\ap{\beta_j}}\|\Ga{\beta_{j}}{}(x)\mathsf{v} \|_{\Ta{\beta_j}}  \\
&= \|V\|_{\Ta{};\Ga{}{}}\sum_{\beta \in [\eta,\kappa) \cap \mathsf{A}} N^{\ap{\beta}} \|\Ga{\beta}{}(x)\mathsf{v} \|_{\Ta{\beta}},
\end{align*}
which corresponds to \eqref{S8300}. 

\medskip

For the proof of \eqref{new} we observe that $(\id-\cut{\eta})V.\ta=(\id-\cut{\betao})V.\ta+(\cut{\betao}-\cut{\eta})V.\ta$. For the first term, we appeal to \eqref{cut2} to estimate
\begin{align*}
 |(\id-\cut{\betao})V(x).\ta|&=|V(x).\Ga{\betao}{}(x)\ta| \le \|V\|_{\Ta{};\Ga{}{}}\sum_{\beta\in\Aa{}} N^{\ap{\beta}}\|\Ga{\beta}{}(x)\Ga{\betao}{}(x)\ta\|_{\Ta{\beta}} \\
 &\stackrel{\mathclap{\eqref{j551}}}{=}\|V\|_{\Ta{};\Ga{}{}} N^{\ap{\betao}}\|\Ga{\betao}{}(x)\ta\|_{\Ta{\betao}}.
\end{align*}
Using \eqref{S8300} with $\betao$ playing the role of $\kappa$, we estimate the second term by
\begin{align*}
 |(\cut{\betao}-\cut{\eta})V(x).\ta|\le \|V\|_{\Ta{};\Ga{}{}} \sum_{\beta \in [\eta,\betao) \cap \mathsf{A}} N^{\ap{\beta}} \|\Ga{\beta}{}(x)\mathsf{v} \|_{\Ta{\beta}},
\end{align*}
and summing the two contributions gives \eqref{new}.

\medskip

{\sc Step} \arabic{Cutprf}.\label{Cutprf_vi}\refstepcounter{Cutprf} Proof of Remark \ref{lem:ff}. In this step, we indicate the changes that 
under the assumption \eqref{cw550} lead to the improved outcome \eqref{cw51}.
We first point out the modifications in Step \ref{Cutprf_i}. We claim that
\begin{align}\label{jj20}
|C_\eta V.\mathsf{v}|\le M_*^b N^{\ap{\beta_*}} \|\Ga{\beta_*}{}\ta\|_{\Ta{\beta_*}} + M^b\sum_{\beta<\eta} N^{\ap{\beta}}\|\Gamma_\beta\mathsf{v}\|_{\Ta{\beta}}.
\end{align}
Indeed, appealing to \eqref{cw78} and the assumption \eqref{cw550}, we obtain
\begin{align*}
|C_\eta V.\mathsf{v}|&\stackrel{\mathclap{\eqref{cut2}}}{=}|V.({\rm id}-\Gamma_\eta)\mathsf{v}| \\
&\stackrel{\mathclap{\eqref{cw550}}}{\le}
M_*^b N^{\ap{\beta_*}} \|\Ga{\beta_*}{}({\rm id}-\Gamma_\eta)\ta\|_{\Ta{\beta_*}} + M^b\sum_{\beta\ne \beta_*}N^{\ap{\beta}}\|\Gamma_\beta({\rm id}-\Gamma_\eta)\mathsf{v}\|_{\Ta{\beta}} \\ &\stackrel{\mathclap{\eqref{cw78}}}{=} \, M_*^b N^{\ap{\beta_*}} \|\Ga{\beta_*}{}\ta\|_{\Ta{\beta_*}} + M^b\sum_{\beta<\eta,\beta\ne \beta_*}N^{\ap{\beta}}\|\Gamma_\beta\mathsf{v}\|_{\Ta{\beta}},
\end{align*}
which gives \eqref{jj20}.

\medskip

We now turn to the changes in Step \ref{Cutprf_ii}.
%
Appealing to \eqref{cw550}, the second term in \eqref{cw81} is now controlled as follows
\begin{align*}
\left|V(y).(\Gamma_\eta(y)-\Gamma_\eta(x))\mathsf{v}\right| \le M^b N^{\ap{\eta}}\|\Gamma\|_{sk}
\sum_{\beta<\eta}d^{\eta-\beta}(y,x)\|\Gamma_\beta(x)\mathsf{v}\|_{T_\beta},
\end{align*}
which directly leads to the improved version of \eqref{cw80}
\begin{align}\label{cw56}
[C_\eta V]_{{\mathcal D}^{\overline{\alpha}\wedge\eta}(\mathsf{T};\Gamma)}
\le[V]_{{\mathcal D}^{\overline{\alpha}}(\mathsf{T};\Gamma)}+N^{\ap{\eta}}\|\Gamma\|_{sk}
M^b.
\end{align}
%

\medskip

Finally, we turn to the changes in Step \ref{Cutprf_iii}, i.~e.~the treatment of general cutting level $\eta$. According to \eqref{jj20} and \eqref{cw56}
we have
\begin{align*}
[C_{\beta_j} V]_{{\mathcal D}^{\overline{\alpha}\wedge\beta_j}(\mathsf{T};\Gamma)}
\le[\cut{\beta_{j+1}}V]_{{\mathcal D}^{\overline{\alpha}\wedge\beta_{j+1}}(\mathsf{T};\Gamma)}+N^{\ap{\beta_j}}\|\Gamma\|_{sk}
M^b.
\end{align*}
Hence we iteratively obtain \eqref{cw51} for
$\eta=\beta_j$ in view of $N\le 1$. For $\eta\in(\beta_{j-1},\beta_j]$ the inequality \eqref{cw51} is unaffected thanks
to $d(y,x)\le 1$ once again.

\medskip

{\sc Step} \arabic{Cutprf}.\label{Cutprf_v}\refstepcounter{Cutprf} Proof of Remark \ref{wg04}. By the above inductive argument, we may reduce to the case of
$\eta=\overline{\beta}_+$. Rewriting the r.~h.~s.~of \eqref{cw94} as 
$\tilde V.\mathsf{v}-\tilde V.\mathsf{v}_\eta$, and substituting $\mathsf{v}$ by $\Gamma\mathsf{v}$,
we note that by definition \eqref{cw92} we have to show $C_\eta V.\mathsf{v}$ 
$=V.\mathsf{v}-\tilde V.\Gamma_\eta\mathsf{v}$. By definition \eqref{cut2} the l.~h.~s.~turns into
$V.({\rm id}_{\mathsf{T}_+}-\Gamma_\eta)\mathsf{v}$, so that we have to show
$V.\Gamma_\eta \mathsf{v}$ $=\tilde V.\Gamma_\eta\mathsf{v}$. Again by definition \eqref{cw92} 
this assumes the form $\tilde V.\Gamma\Gamma_\eta \mathsf{v}$ $=\tilde V.\Gamma_\eta\mathsf{v}$. This
in turn follows
from $\Gamma\Gamma_\eta=\Gamma_\eta$, which is a consequence of the maximality of $\eta=\overline{\beta}_+$
and of the triangular structure \eqref{j551}.
\end{proof}
\subsection{Lemmas \ref{lem:prod_skel} \& \ref{lem:Multiplication} and Corollary \ref{lem:ExMult}}\label{S:multPr}

In this section, we will prove the assertions of Section \ref{S:mult}. 

\begin{proof}[Proof of Lemma \ref{lem:prod_skel}]
Identity \eqref{Prod9} is an immediate consequence of \eqref{j551}.
Bilinearity and property \eqref{j522} of the operator norm on tensor spaces ensure continuity of $\tGa{}{}$ with respect to $x$, that is $\tGa{}{}\in C^0(\R^2;\call(\tTa{};\tTa{}))$. We now observe that the triangular structure of the skeletons $\Ga{+}{}$ and $\Ga{-}{}$ transmits to $\tGa{}{}$, that is, it holds $\tGa{\tbeta}{\tgamma}=0$ for $\tbeta<\tgamma$ and $\tGa{\tbeta}{\tbeta}=\id_{\tTa{\tbeta}}$. Indeed, if $\tbeta\in\tAa{}$ and $(\betap,\betam)\neq (\gammap,\gammam)$ satisfy $\betap+\betam=\gammap+\gammam=\tbeta$, then $\Ga{\betap}{\gammap}\otimes \Ga{\betam}{\gammam}=0$, since $\betap<\gammap$ or $\betam<\gammam$.
Let us show that the product skeleton $\Ga{}{}$ satisfies the continuity property \eqref{ssss2}. We decompose $\ta=(\ta_{\gammap,\gammam})_{\gammap,\gammam}\in\tTa{}$ according to \eqref{j1009} and fix $\tbeta\in\tAa{}$ and $x,y\in\R^2$ with $d(y,x)\le 1$.
Writing $[\Ga{}{}]:=\tGa{}{}(y)-\tGa{}{}(x)$ and similarly for $\Ga{\betap}{\gammap}$ and $\Ga{\betam}{\gammam}$, we have
\begin{align}
\|&([\tGa{}{}]\tta)_{\tbeta}\|_{\tTa{\tbeta}}  = \Big\|\sum_{\tgamma}[\tGa{\tbeta}{\tgamma}] \tta_{\tgamma}\Big\|_{\tTa{\tbeta}}\nonumber\\
&\stackrel{\eqref{Prod9}}{=}\sum_{\betap+\betam=\beta}\Big\|\sum_{\gammap,\gammam} \big([\Ga{\betap}{\gammap}]\otimes [\Ga{\betam}{\gammam}]  + [\Ga{\betap}{\gammap}]\otimes \Ga{\betam}{\gammam}(x) \nonumber \\
& \qquad \qquad \qquad \qquad \qquad \qquad + \Ga{\betap}{\gammap}(x)\otimes [\Ga{\betam}{\gammam}]\big) \tta_{\gammap,\gammam}\Big\|_{\Ta{\betap}\otimes\Ta{\betam}} \nonumber \nonumber \\
&\le \sum_{\betap+\betam=\beta}\Big\|\sum_{\gammap,\gammam} \big([\Ga{\betap}{\gammap}]\otimes [\Ga{\betam}{\gammam}]\big)\tta_{\gammap,\gammam}\Big\|_{\Ta{\betap}\otimes\Ta{\betam}} \nonumber \\
&\quad + \sum_{\betap+\betam=\beta}\Big\|\sum_{\gammap,\gammam} \big([\Ga{\betap}{\gammap}]\otimes \Ga{\betam}{\gammam}(x)\big)\tta_{\gammap,\gammam}\Big\|_{\Ta{\betap}\otimes\Ta{\betam}} \nonumber \\
&\quad + \sum_{\betap+\betam=\beta}\Big\|\sum_{\gammap,\gammam} \big(\Ga{\betap}{\gammap}(x)\otimes [\Ga{\betam}{\gammam}]\big) \tta_{\gammap,\gammam}\Big\|_{\Ta{\betap}\otimes\Ta{\betam}}. \label{j1007}
\end{align}
In order to estimate the first term on the right-hand side, we apply \eqref{j522} for fixed $(\betap,\betam)$ with
\begin{align}
\begin{split}
 \aop \ta_{+}&:=\sum_{\gammap<\betap}(\Ga{\betap}{\gammap}(y)-\Ga{\betap}{\gammap}(x))\ta_{\gammap}, \\
 \bop \ta_{-}&:=\sum_{\gammam<\betam}(\Ga{\betam}{\gammam}(y)-\Ga{\betam}{\gammam}(x))\ta_{\gammam}, \\
 \taop \ta_{+}&:= \|\Ga{+}{}\|_{sk}\Big(d^{\betap-\gammap}(y,x)\Ga{\gammap}{}(x)\ta_{+}\Big)_{\gammap<\betap}, \\
 \tbop \ta_{-}&:= \|\Ga{-}{}\|_{sk}\Big(d^{\betam-\gammam}(y,x)\Ga{\gammam}{}(x)\ta_{-}\Big)_{\gammam<\betam},
\end{split}\label{j1501}
\end{align}
and spaces $\Ta{\betap}\mix \Ta{\betam}$ and $\big(\bigoplus_{\gammap<\betap} \Ta{\gammap}\big)\mix \big(\bigoplus_{\gammam<\betam}\Ta{\gammam}\big)$, which by \eqref{ssss2} satisfy \eqref{j525} and \eqref{j526}, to the effect of
\begin{align*}
\sum_{\betap+\betam=\beta}&\Big\|\sum_{\gammap,\gammam}\big([\Ga{\betap}{\gammap}]\otimes [\Ga{\betam}{\gammam}]\big) \tta_{\gammap,\gammam}\Big\|_{\Ta{\betap}\otimes\Ta{\betam}} \\
&\stackrel{\eqref{j522}}{\le} 
  \|\Ga{+}{}\|_{sk}\|\Ga{-}{}\|_{sk}\sum_{\betap+\betam=\beta}\sum_{\gammap<\betap,\gammam< \betam}d^{\betap-\gammap}(y,x) d^{\betam-\gammam}(y,x) \\
  &\qquad \qquad \qquad \times \Big\|\sum_{\epsp, \epsm}\big(\Ga{\gammap}{\epsp}(x)\otimes \Ga{\gammam}{\epsm}(x)\big)\tta_{\epsp,\epsm}\Big\|_{\Ta{\gammap}\otimes\Ta{\gammam}}\\
&\stackrel{\eqref{Prod9}}{\le}  \|\Ga{+}{}\|_{sk}\|\Ga{-}{}\|_{sk}\sum_{\tgamma<\tbeta}d^{\tbeta-\tgamma}(y,x)  \|\Ga{\tgamma}{}(x)\tta\|_{\Ta{\gamma}}.
\end{align*}
It remains to estimate the last two lines in \eqref{j1007}. Appealing to symmetry, we restrict ourselves to the last line and apply \eqref{j522} for fixed $(\betap,\betam)$ with $\bop$ and $\tbop$ as in \eqref{j1501},
\begin{align*}
 \aop \ta_{+}&:=\taop \ta_{+}:=\sum_{\gammap}\Ga{\betap}{\gammap}(x)\ta_{\gammap},
\end{align*}
and spaces $\Ta{\betap}\mix \Ta{\betam}$ and $\bigoplus_{\gammam<\betam} \Ta{\betap}\mix \Ta{\gammam}$, which trivially satisfies \eqref{j525} and as before by \eqref{ssss2} satisfies \eqref{j526}, to the effect of
\begin{align}
\sum_{\betap+\betam=\beta}&\Big\|\sum_{\gammap,\gammam}\big(\Ga{\betap}{\gammap}(x)\otimes [\Ga{\betam}{\gammam}]\big) \tta_{\gammap,\gammam}\Big\|_{\Ta{\betap}\otimes\Ta{\betam}} \nonumber \\
&\stackrel{\mathclap{\eqref{j522}}}{\le} 
  \|\Ga{-}{}\|_{sk}\sum_{\betap+\betam=\beta}\sum_{\gammam< \betam}d^{\betam-\gammam}(y,x) \nonumber \\
  &\qquad \qquad \qquad \times \Big\|\sum_{\gammap, \epsm}\big(\Ga{\betap}{\gammap}(x)\otimes \Ga{\gammam}{\epsm}(x)\big)\tta_{\gammap,\epsm}\Big\|_{\Ta{\betap}\otimes\Ta{\gammam}}. \nonumber
\end{align}
Writing $\betam-\gammam=\tbeta-(\betap+\gammam)$ and observing that trivially it holds
\begin{align*}
&\setc{(\betap,\gammam)\in \Aa{+}\times \Aa{-}}{\exists \betam\in\Aa{-} \text{ with } \betap+\betam=\beta \text{ and } \gammam<\betam} \\
&\quad  \subset \setc{(\betap,\gammam)\in \Aa{+}\times \Aa{-}}{\exists \gamma<\beta \text{ with } \betap+\gammam=\gamma},
\end{align*}
we may continue the string of inequalities with
\begin{align}
&\le  \|\Ga{-}{}\|_{sk}\sum_{\tgamma< \tbeta}\sum_{\betap+\gammam=\tgamma}d^{\tbeta-\tgamma}(y,x) \nonumber \\
  &\qquad \qquad \qquad \times \Big\|\sum_{\gammap, \epsm}\big(\Ga{\betap}{\gammap}(x)\otimes \Ga{\gammam}{\epsm}(x)\big)\tta_{\gammap,\epsm}\Big\|_{\Ta{\betap}\otimes\Ta{\gammam}}\nonumber \\
&\stackrel{\mathclap{\eqref{Prod9}}}{=} \|\Ga{-}{}\|_{sk}\sum_{\tgamma<\tbeta} d^{\tbeta-\tgamma}(y,x)\|\tGa{\tgamma}{}(x))\tta\|_{\tTa{\tgamma}}. \nonumber 
\end{align}
Combining the above estimates, we obtain that $\Ga{}{}$ is a skeleton on $(\Aa{},\Ta{})$ satisfying \eqref{j1200}.

\medskip

Finally, we turn to the proof of \eqref{neededlater1}. In view of \eqref{wg16}, it suffices to show for $\tta\in\Ta{}$
\begin{align}
 |V_{{+}}(x) \otimes & V_{{-}}(x).\tta  | \leq \|V_+\|_{\Ta{+};\Ga{+}{}}\|V_-\|_{\Ta{-};\Ga{-}{}}\sum_{\tbeta} N^{\ap{\beta}}\|\tGa{\beta}{}(x)\tta\|_{\Ta{\beta}}. \label{M1}
\end{align}
This is achieved by property \eqref{j522} of the tensor-space norm: Set 
\begin{align*}
 \aop \ta_{+}&:=V_{{+}}(x).\ta_{+}, \quad \taop \ta_{+}:=  \|V_+\|_{\Ta{+};\Ga{+}{}}\vp{N^{\ap{\betap}_+}\Ga{\betap}{}(x)\ta_{+}}_{\betap\in\Aa{+}}, \\
 \bop \ta_{-}&:=V_{{-}}(x).\ta_{-}, \quad \tbop \ta_{-}:=  \|V_-\|_{\Ta{-};\Ga{-}{}}\vp{N^{\ap{\betam}_-}\Ga{\betam}{}(x)\ta_{-}}_{\betam\in\Aa{-}}.
\end{align*}
Then \eqref{j525} and \eqref{j526} are satisfied by the boundedness of $V_{{+}}$ and $V_{{-}}$, respectively, and the outcome \eqref{j522} turns into \eqref{M1} by definition \eqref{Prod9} and property \eqref{wg10} combined with $N\le 1$.
\end{proof}

%
%
%
%
%
\begin{proof}[Proof of Lemma \ref{lem:Multiplication}]
First, we argue that $\oGa{}{}$ is indeed a skeleton on $(\oAa{},\oTa{})$: The triangular representation is fulfilled in view of $\kappa-\alpha>\betamo$ and the continuity property \eqref{ssss2} is fulfilled on level $\kappa-\alpha$ in virtue of $V_{u_{-}}\in \cald^{\alphamo}(\Tminus; \Gamma_{-})$, so that \eqref{prod5} multiplied by $1/N_*$ turns into \eqref{ssss2} by using again $\kappa-\alpha>\betamo$ to argue that the maximum appearing in \eqref{prod5} is not effective. Clearly, this argument also shows the estimate \eqref{j1004}.

Now, in order to show $V_{u_+\diamond u_-}\in \cald^{\kappa}(\oTa{};\oGa{}{})$ and for the sake of readability, let us introduce the abbreviations
\begin{align*}
 \Mpb&:=N^{\min_{\beta \in \Aa{+}'}\ap{\beta}}\bd{\overline{V}_{u_{+}}}{\Tplus}{\Ga{+}{}}, && \Mmb:=N^{\min_{\beta \in \Aa{-}}\ap{\beta}}\bd{V_{u_{-}}}{\Tminus}{\Ga{-}{}}, \\
 \Mpc&:=[V_{u_{+}}]_{\cald^{\alphapo}(\Tplus;\Ga{+}{})}, && \Mmc:=[V_{u_{-}}]_{\cald^{\alphamo}(\Tminus;\Ga{-}{})}.
\end{align*}
We first verify the form boundedness, cf.\@ \eqref{prod4}, of $V_{u_{+}\diamond u_{-}}$.
Note that by definition \eqref{MS1} we have for any $x\in\R^2$ and $(\mathsf{u}_{-},\tta)\in\oTa{}$
\begin{align*}
 V_{u_{+}\diamond u_{-}}(x).
 \begin{pmatrix}
  \mathsf{u}_{-} \\
  \tta
 \end{pmatrix}
  =u_{+}(x)(\mathsf{u}_{-}-V_{u_{-}}(x).\tta_{-}) + V_{u_{+}}(x)\otimes V_{u_{-}}(x).\tta,
\end{align*}
where $\ta_{-}:=(\ta_{0,\betam})_{\betam\in\Aa{-}}$, cf.\@ \eqref{j1009}. 
Thus, \eqref{MS2} follows in view of \eqref{neededlater1}, \eqref{j980}, $\oTa{\kappa-\alpha}=\frac{1}{N_*}\R$ and $\ap{\kappa-\alpha}=0$.

\medskip

We now turn to the form continuity statement in \eqref{MS3}. Towards this end, we fix points $x,y$ with $d(y,x) \le 1$ and use the definition \eqref{MS1} together with the discrete product rule to find the identity
\begin{align}
&\big (V_{u_{+} \diamond u_{-}}(y)-V_{u_{+} \diamond u_{-}}(x) \big).
\begin{pmatrix}
\mathsf{u_{-}}\\
\tta
\end{pmatrix}
\nonumber\\
&\quad =\big (u_{+}(y)-u_{+}(x)\big)\big (\mathsf{u}_{-}-V_{u_{-}}(x).\ta_{-} \big) \nonumber \\
&\quad -\big (u_{+}(y)-u_{+}(x)\big)\big (V_{u_{-}}(y)-V_{u_{-}}(x)\big).\ta_{-} \nonumber \\
&\quad + \overline{V}_{u_{+}}(x) \otimes (V_{u_{-}}(y)-V_{u_{-}}(x)).\tta \nonumber \\ 
&\quad + (V_{u_{+}}(y)-V_{u_{+}}(x)) \otimes V_{u_{-}}(y) .\tta\label{M5}.
\end{align}
The contribution in the first line of the right hand side is estimated by
\begin{align*}
N_*[u_{+}]_\alpha d^{\alpha}&(y,x)\frac{1}{N_*}\left |\mathsf{u}_{-}-V_{u_{-}}(x).\ta_{-}\right |,
\end{align*}
which falls by $\Ta{\kappa-\alpha}=\frac{1}{N_*}\R$ and \eqref{j980} under the right-hand side of \eqref{MS3}. For the second line on the right hand side of \eqref{M5} we find
\begin{align*}
[u_{+}]_\alpha d^{\alpha}(y,x)|(V_{u_{-}}(y)-V_{u_{-}}(x)).\ta_{-}|.
\end{align*}
 Hence, by definition \eqref{prod5} of form continuity of $V_{u_-}$ and our abbreviations $\ta_{-}=(\ta_{0,\betam})_{\betam\in\Aa{-}}\in \Tminus$, we estimate this by
\begin{align*}
[u_{+}]_\alpha & \Mmc\sum_{\betam}d^{\alpha+(\alphamo-\betam) \vee 0}(y,x)\|\Ga{\betam}{}(x)\ta_{-}\|_{\Ta{\betam}} \\
&\le [u_{+}]_\alpha \Mmc\sum_{\betam}d^{\vp{\kappa-\betam}\vee 0}(y,x)\|\Ga{\betam}{}(x)\ta_{-}\|_{\Ta{\betam}} \\
&\stackrel{\mathclap{\eqref{Prod9}}}{\le} [u_{+}]_\alpha \Mmc\sum_{\beta}d^{\vp{\kappa-\beta}\vee 0}(y,x)\|\Ga{\tbeta}{}(x)\tta\|_{\Ta{\beta}},
\end{align*}
where in the first inequality, we have used $\alpha+(\alphamo-\betam) \vee 0 \stackrel{\alpha\ge 0}{\ge} (\alpha+\alphamo-\betam) \vee 0 \stackrel{\eqref{j982}}{\ge} (\kappa-\betam) \vee 0$. Thus, in view of the definition \eqref{S301} of $M_{u_+\diamond u_-}^c$, this term falls under the right-hand side of \eqref{MS3}.
We now estimate the contribution from the third line of the right hand side in \eqref{M5} by applying \eqref{j522} to
\begin{align*}
 \aop \ta_{+}&:=\overline{V}_{u_{+}}(x).\ta_{+}, \\
 \bop \ta_{-}&:=(V_{u_{-}}(y)-V_{u_{-}}(x)).\ta_{-}, \\
 \taop \ta_{+}&:= \Mpb\vp{\Ga{\betap}{}(x)\ta_{+}}_{\betap\in\Aa{+}'}, \\
 \tbop \ta_{-}&:= \Mmc\vp{d^{\vp{\alphamo-\betam} \vee 0}(y,x)\Ga{\betam}{}(x)\ta_{-}}_{\betam\in\Aa{-}}.
\end{align*}
 The reduction assertion \eqref{new} with $\alpha$ playing the role of $\eta$ together with $N\le 1$ ensure that the input condition \eqref{j525} is satisfied with $\oplus_{\betap\in\Aa{+}}\Ta{\betap}$ playing the role of $\Ta{\epsp}$, whereas \eqref{j526} with $\oplus_{\betam\in\Aa{-}}\Ta{\betam}$ playing the role of $\Ta{\epsm}$ follows from the form continuity of $V_{u_{-}}$. Hence, for $\ta=(\ta_{\gammap,\gammam})_{\gammap,\gamma}\in\Ta{}$, cf.\@ \eqref{j1009},
 we learn from \eqref{j522}
\begin{align*}
\big| \overline{V}_{u_{+}}&(x) \otimes (V_{u_{-}}(y)-V_{u_{-}}(x)).\tta \big|\\
& \leq \Mpb \Mmc \sum_{\betap>0,\beta_{-}} d^{\vp{\overline{\alpha}_{-}-\betam} \vee 0}(y,x) \\
&\qquad \qquad \times \Big\|\sum_{\gammap,\gammam}\big(\Ga{\beta_{+}}{\gamma_{+}}(x)\otimes \Ga{\beta_{-}}{\gamma_{-}}(x)\big)\tta_{\gamma_{+},\gamma_{-}}\Big\|_{\Ta{\betap}\mix\Ta{\betam}} \\
& \leq \Mpb \Mmc \sum_{\tbeta}\sum_{\substack{\betap,\betam \\ \beta_{+}+\beta_{-}=\tbeta}} d^{\vp{\kappa-\tbeta} \vee 0}(y,x) \\
&\qquad \qquad \times \Big\|\sum_{\gammap,\gammam}\big(\Ga{\beta_{+}}{\gamma_{+}}(x)\otimes \Ga{\beta_{-}}{\gamma_{-}}(x)\big)\tta_{\gamma_{+},\gamma_{-}}\Big\|_{\Ta{\betap}\mix\Ta{\betam}} \\
& \stackrel{\eqref{Prod9}}{=} \Mpb \Mmc \sum_{\tbeta} d^{\vp{\kappa-\tbeta} \vee 0}(y,x) \|\tGa{\tbeta}{}(x)\tta\|_{\tTa{\tbeta}},
\end{align*}
where for the second inequality, we have used $-\betam= \betap-\tbeta \stackrel{\eqref{Prod1}}{\ge} \alpha-\beta$ and $(\alphamo+\alpha-\tbeta)\vee 0\stackrel{\eqref{j982}}{\ge} (\kappa-\beta)\vee 0$. Again, the right-hand side falls under \eqref{MS3}. The estimate for the last line in \eqref{M5} follows analogously:
Namely, we choose
\begin{align*}
 \aop \ta_{+}&:=(V_{u_{+}}(x)-V_{u_{+}}(y)).\ta_{+}, \\
 \bop \ta_{-}&:=V_{u_-}(y).\ta_{-}, \\
 \taop \ta_{+}&:= \Mpc\vp{d^{\vp{\overline{\alpha}_{+}-\beta_{+}} \vee 0}(y,x)\Ga{\betap}{}(x)\ta_{+}}_{\betap\in\Aa{+}}, \\
 \tbop \ta_{-}&:=(\Mmb+\Mmc)\vp{\Ga{\betam}{}(x)\ta_{-}}_{\betam\in\Aa{-}}.
\end{align*}
Then the input condition \eqref{j525} is satisfied by the continuity of $V_{u_{+}}$. Due to $N\le 1$ and  $d(y,x)\le 1$, the input condition \eqref{j526} is fulfilled in virtue of
\begin{align*}
 |V_{u_-}(y).\ta_{-}|&\le |V_{u_-}(x).\ta_{-}| + |(V_{u_-}(x)-V_{u_-}(x)).\ta_{-}| \\
 &\le (\Mmb+\Mmc)\sum_{\betam\in\Aa{-}}\|\Ga{\betam}{}(x)\ta_{-}\|_{\Ta{\betam}}.
\end{align*}
Thus, from \eqref{j522} we obtain
\begin{align*}
\big| (V_{u_{+}}&(y)-V_{u_{+}}(x)) \otimes V_{u_-}(y).\tta\big |\\
& \leq \Mpc (\Mmb+\Mmc) \sum_{\beta_{+},\beta_{-}} d^{\vp{\overline{\alpha}_{+}-\betap} \vee 0}(y,x) \\
&\qquad \qquad \times \Big\|\sum_{\gammap,\gammam}\big(\Ga{\beta_{+}}{\gamma_{+}}(x)\otimes \Ga{\beta_{-}}{\gamma_{-}}(x)\big)\tta_{\gamma_{+},\gamma_{-}}\Big\|_{\Ta{\betap}\mix\Ta{\betam}} \\
& \leq \Mpc (\Mmb+\Mmc)\sum_{\tbeta}\sum_{\substack{\betap,\betam \\ \beta_{+}+\beta_{-}=\tbeta}} d^{\vp{\kappa-\tbeta} \vee 0}(y,x) \\
&\qquad \qquad \times \Big\|\sum_{\gammap,\gammam}\big(\Ga{\beta_{+}}{\gamma_{+}}(x)\otimes \Ga{\beta_{-}}{\gamma_{-}}(x)\big)\tta_{\gamma_{+},\gamma_{-}}\Big\|_{\Ta{\betap}\mix\Ta{\betam}} \\
& \stackrel{\mathclap{\eqref{Prod9}}}{=} \Mpc (\Mmb+\Mmc)\sum_{\tbeta} d^{\vp{\kappa-\tbeta} \vee 0}(y,x) \|\sum_{\gamma}\tGa{\tbeta}{\tgamma}(x)\tta_{\tgamma}\|_{\tTa{\tbeta}},
\end{align*}
where in the second inequality, we have used $-\betap=\betam-\tbeta \stackrel{\eqref{j1000}}{\ge} \betamu - \tbeta$ and $(\alphapo+\betamu-\tbeta) \vee 0 \stackrel{\eqref{j982}}{\ge} (\kappa-\beta)\vee 0$. Therefore, also this contribution falls under the right hand side of \eqref{MS3}, which concludes the proof.
\end{proof}
\begin{proof}[Proof of Corollary \ref{lem:ExMult}]
We begin with the proof of \eqref{j772}, which we will obtain from the abstract reconstruction result, Proposition \ref{rec_lem},
via the multiplication results in Lemma \ref{lem:Multiplication}.
Towards this end, we first note that by \eqref{j770}, \eqref{j771} and by construction of $(\oAa{},\oTa{})$ and $\oGa{}{}$, condition \eqref{cw04} is fulfilled for the distribution $(u_-,\tf_+\diam\tf_-)\in \cald'(\R^2;\oTa{})$. Moreover, by Lemma \ref{lem:Multiplication}, $V_{u_{+} \diamond u_{-}}$ has the continuity and boundedness properties \eqref{MS2} and \eqref{MS3}, thus \eqref{cw51} in Remark \ref{lem:ff} implies
\begin{align*}
 [\cut{\kappa}V_{u_{+} \diamond u_{-}}]_{\cald^{\kappa}(\oTa{};\oGa{}{})}\le [V_{u_{+} \diamond u_{-}}]_{\cald^{\kappa}(\oTa{};\oGa{}{})} + (|\oAa{}|-1)N^{\min_{\beta\ge \kappa}\ap{\beta}}\|\oGa{}{}\|_{sk}M_{u_+\diamond u_-}^b.
\end{align*}
With these observations at hand, we are now justified in applying Proposition \ref{rec_lem} to obtain a distribution
\begin{equation*}
u_{+} \diamond u_{-}:=\mathcal{R}\cut{\kappa}V_{u_{+} \diamond u_{-}}
\end{equation*} 
such that \eqref{cw40} holds, which translates directly to \eqref{j772}.

\medskip

We now turn our attention to the sub-optimal bound \eqref{S870}.  First observe that since $\kappa>\kappa-\alpha$, the definition \eqref{cutind} ensures that for all $(\mathsf{u}_{-},\tta) \in \Ta{}$, the following identity holds:
\begin{equation}\label{S1102}
C_{\kappa}V_{u_{+} \diamond u_{-}}.
\begin{pmatrix}
\mathsf{u}_{-} \\
\tta\\ 
\end{pmatrix}
=u_{+}\mathsf{u}_{-}+C_{\kappa}\big(\overline{V}_{u_{+}} \otimes V_{u_{-}}\big).\tta.
\end{equation}
Hence, we may recast \eqref{j772} as 
\begin{align*}
&\|(u_{+} \diamond u_{-})_{T}-u_{+}(u_{-})_{T}-C_{\kappa}( \overline{V}_{u_{+}} \otimes V_{u_{-}}).(\tf_{+} \diamond \tf_{-})_{T} \| \nonumber\\
&\quad \lesssim M_{u_{+} \diamond u_{-},\kappa}\big (T^{\frac{1}{4}} \big )^{\kappa}.
\end{align*}
Moreover, observe that for each $\eta<\kappa$, we may appeal to \eqref{S8300} to deduce
\begin{align*}
\|(C_{\eta}-C_{\kappa})&\big(\overline{V}_{u_{+}}\otimes V_{u_{-}}\big).(v_{{+}} \diamond v_{{-}})_{T}\|  \nonumber\\
&\leq \|\overline{V}_{u_{+}} \otimes V_{u_{-}}\|_{\Ta{};\Ga{}{}}\sum_{\beta \in [\eta,\kappa)}N^{\ap{\beta}}\sup_{x\in\R^2}\|\Ga{\beta}{}(x)(v_{+} \diamond v_{-})_{T}\|_{\Ta{\beta}} \nonumber \\
& \stackrel{\eqref{neededlater1},\eqref{j770}}{\lesssim} N^{\min_{\beta\ge \eta}\ap{\beta}}\|\overline{V}_{u_{+}}\|_{\Ta{+};\Ga{+}{}}\|V_{u_{-}}\|_{\Ta{};\Ga{-}{}} (T^{\frac{1}{4}})^{\eta}.
\end{align*}
Using the triangle inequality, we may combine the two last estimates to obtain \eqref{S870}.
\end{proof}
\subsection{Proposition \ref{int lem}}

\begin{proof}[Proof of Proposition \ref{int lem}]
 \newcounter{PDE_P} 
\refstepcounter{PDE_P} 

The proof follows closely the one of Lemma 5 in \cite{OtW16}. For the convenience of the reader, we include the argument here.

\medskip

{\sc Step} \arabic{PDE_P}.\label{pde_scale_St}\refstepcounter{PDE_P} We claim that for all base points $x$ and scales $T^\frac14$, $R$ and $L$ with $R\ll L$ it holds
\begin{align}\label{p46}
 \inf_{\ell}\|U_{T}(x,\cdot)-\ell\|_{B_R(x)} &\lesssim \left(\frac{R}{L}\right)^2\inf_{\ell}\|U_{T}(x,\cdot) - \ell\|_{B_L(x)}\\ 
 &+ L^2 \overline{M}\sum_{\beta\in \Aa{}}\Tc{\beta-2}L^{\kappa-\beta}\nonumber,
\end{align}
where the infimum runs over all affine functions $\ell$, by which we mean functions of the form $\ell(y)=\nu y_{1}+c$ for some constants $\nu$ and $c$.  Towards this end, we define a decomposition $U_T(x,\cdot)=u_{<}(\cdot) + u_{>}(\cdot)$ by setting $u_{>}$ to be the (decaying and in particular non-periodic) solution to
\begin{align*}
 (\partial_2 - a(x)\partial_1^2)u_{>} = I(B_L(x)) \left((\partial_2-a(x)\partial_1^2)U_T(x,\cdot)-c_>\right),
\end{align*}
where $I(B_L(x))$ is the characteristic function of $B_L(x)$, and where $c_{>}$ is (near) optimal in the estimate \eqref{KS1}.  Observe that on $B_L(x)$ it holds
\begin{align}\label{KS131}
 (\partial_2 - a(x)\partial_1^2)u_{<} =c_>.
\end{align}
By standard estimates for the heat equation and \eqref{KS1} we have
\begin{align}\label{KS400}
 \|u_{>}\|_{B_L(x)}&\lesssim L^2 \|(\partial_2 - a(x)\partial_1^2)U_T(x,\cdot)-c_>\|_{B_L(x)}\\
 &\le L^2  \overline{M}\sum_{\beta\in\Aa{}}(T^{\frac{1}{4}})^{\beta-2}L^{\kappa-\beta}\nonumber, 
\end{align} 
together with
\begin{equation}\label{KS5}
 \|\{\partial_1^2,\partial_2\}u_{<}\|_{B_R(x)} \lesssim  L^{-2} \|u_{<}- \ell_{>}\|_{B_L(x)}
\end{equation}
for any affine $\ell_{>}$, where we used that $R\le L$. In fact, \eqref{KS5} is slightly non-standard due to the presence of a constant $c_{>}$ on the right-hand side of \eqref{KS131}. However, as observed in \cite{OtW16}, this can be reduced to the case $c_{>}=0$: First of all, we note that replacing $u_{<}$ by $u_{<}-\ell_{<}$, we may assume that $\ell_{<}=0$. Testing \eqref{KS131} with a cut-off function on $B_{L}$ that is smooth on scale $L$, we learn that $|c_{>}|\le L^2\|u_{<}\|_{B_L}$. We then may replace $u_{<}$ by $u_{<}+c_{>}y_2$ which reduces the further estimate to the standard case $c_{>}=0$. Next we define a concrete affine function $\ell_{<}$ via $\ell_{<}(y):=u_{<}(x)+\partial_1 u_{<}(x)(y-x)_{1}$ and observe that using Taylor's formula and \eqref{KS5} gives for any $\ell_{>}$
\begin{align*}
 \|u_{<}-\ell_{<}\|_{B_R(x)}&\lesssim R^2\|\{\partial_1^2,\partial_2\} u_{<}\|_{B_R(x)} \\
 & \stackrel{\mathclap{\eqref{KS5}}}{\lesssim} \left(\frac{R}{L}\right)^2 \|u_{<}-\ell_{>}\|_{B_L(x)} \\
 & \le \left(\frac{R}{L}\right)^2 \|U_T(x,\cdot)-\ell_{>}\|_{B_L(x)} + \|u_{>}\|_{B_L(x)}.
\end{align*}
Combining this observation with \eqref{KS400} gives
\begin{align*}
 \|U_T(x,\cdot)-&\ell_{<}\|_{B_R(x)} \le \|u_{>}\|_{B_R(x)} + \|u_{<}-\ell_{<}\|_{B_R(x)} \\
 & \lesssim \left(\frac{R}{L}\right)^2 \|U_T(x,\cdot)-\ell_{>}\|_{B_L(x)} + \|u_{>}\|_{B_L(x)} \\
 & \lesssim \left(\frac{R}{L}\right)^2 \|U_T(x,\cdot)-\ell_{>}\|_{B_L(x)} + L^2 \overline{M}\sum_{\beta\in\Aa{}}(T^{\frac{1}{4}})^{\beta-2}L^{\kappa-\beta},
\end{align*}
which implies \eqref{p46}.

\medskip

{\sc Step} \arabic{PDE_P}.\label{pde_conv_St}\refstepcounter{PDE_P} We claim that for all base points $x$ and all scales $T^\frac14$ and $L$ it holds
\begin{align}\label{p49}
 \|U_T(x,\cdot)-U(x,\cdot)\|_{B_L(x)} \lesssim M_{U} \Tc{\kappa}+\overline{\overline{M}}\sum_{\beta} L^{\beta}\Tc{\kappa-\beta}.
\end{align}
We give an argument that is also suitable for Step 5, when we approximate by semi-group convolution $(\cdot)_\tau$. 
We fix a point $x,y$ and convolution parameters $t>\tau\ge0$ and write
\begin{align*}
\lefteqn{\partial_tU_t(x,y)=\int U(x,z)\partial_t\psi_t(y-z){\rm d}z}\nonumber\\
&=\int U(y,z)\partial_t\psi_t(y-z){\rm d}z+\int (U(x,z)-U(y,z))\partial_t\psi_t(y-z){\rm d}z\nonumber\\
&\stackrel{\eqref{1.10}}{=}\int U_\tau(y,z)\partial_t\psi_{t-\tau}(y-z){\rm d}z
+\int (U(x,z)-U(y,z))\partial_t\psi_t(y-z){\rm d}z.
\end{align*}
Since by \eqref{wg50} we have in particular $\int\ell(z)\partial_t\psi_t(y-z){\rm d}z=0$
for any affine function $\ell$, we may rewrite the above identity as
\begin{align*}
\lefteqn{\partial_tU_t(x,y)
=\int (U_\tau(y,z)-U_\tau(y,y)-\nu^{\tau}(y)(z-y)_1)\partial_t\psi_{t-\tau}(y-z){\rm d}z}\nonumber\\
&+\int (U(x,z)-U(y,z)-U(x,y)+U(y,y)-\gamma(x,y)(z-y)_1)\partial_t\psi_t(y-z){\rm d}z,
\end{align*}
(where the reader should think of $\tau=0$ till Step 5)
so that by \eqref{KS2} and the definition of $M_{U_\tau}$ we obtain the inequality
\begin{align*}
|\partial_tU_t(x,y)|
&\le M_{U_\tau}\int d^{\kappa}(z,y)|\partial_t\psi_{t-\tau}(y-z)|{\rm d}z\nonumber\\
&+\overline{\overline{M}}\sum_{\beta\in\Aa{}}d^{\beta}(x,y)\int d^{\kappa-\beta}(z,y)|\partial_t\psi_t(y-z)|{\rm d}z.
\end{align*}
Hence in view of \eqref{wg50} and \eqref{1.13}, we obtain the estimate
\begin{align*}
|\partial_tU_t(x,y)|
\lesssim M_{U_\tau}((t-\tau)^\frac{1}{4})^{\kappa-4}
+\overline{\overline{M}}\sum_{\beta\in\Aa{}}d^{\beta}(y,x)(t^\frac{1}{4})^{\kappa-\beta-4}.
\end{align*}
Integrating over $t\in(\tau,T+\tau)$, using once more \eqref{1.10},
this yields because of $\kappa-\beta>0$ (and $\kappa>0$) that
\begin{align*}
\lefteqn{|((U_\tau)_T-U_\tau)(x,y)|}\nonumber\\
&\lesssim M_{U_\tau}(T^\frac{1}{4})^{\kappa}
+\overline{\overline{M}}\sum_{\beta\in\Aa{}}d^{\beta}(y,x)(((T+\tau)^\frac{1}{4})^{\kappa-\beta}-(\tau^\frac{1}{4})^{\kappa-\beta}).
\end{align*}
Since $((T+\tau)^\frac{1}{4})^{\kappa-\beta}-(\tau^\frac{1}{4})^{\kappa-\beta}$ $\lesssim (T^\frac{1}{4})^{\kappa-\beta}$, this
simplifies to the desired
\begin{align}\label{wg51}
|((U_\tau)_T-U_\tau)(x,y)|
\lesssim M_{U_\tau}(T^\frac{1}{4})^{\kappa}
+\overline{\overline{M}}\sum_{\beta\in\Aa{}}d^{\beta}(y,x)(T^\frac{1}{4})^{\kappa-\beta},
\end{align}
and implies in particular \eqref{p49} in the special case $\tau=0$.

\medskip

{\sc Step} \arabic{PDE_P}.\label{pde_normeq_St}\refstepcounter{PDE_P} We claim the norm equivalence
\begin{align}\label{3.8}
 M_{U} \sim M_{U}',
\end{align}
where we have set
\begin{align}\label{3.6}
 M_{U}':=\sup_{x\in\R^2} \sup_{R\leq 4} R^{-\kappa}\inf_\ell \|U(x,\cdot)-\ell\|_{B_R(x)},
\end{align}
and where $\sim$ means that both inequalities with $\lesssim$ and $\gtrsim$ are true.
We first argue that the $\ell$ in \eqref{3.6} may be chosen to be independent of $R$, that is,
\begin{align}\label{3.7}
\sup_{x}\inf_{\ell}\sup_{R\le 4}R^{-\kappa} \|U(x,\cdot)-\ell\|_{B_R(x)}\lesssim M'_{U}.
\end{align}
Indeed, fix $x$, say $x=0$, and let $\ell_R(y)=\nu_Ry_1+c_R$ be (near) optimal in \eqref{3.6}. Then by
definition of $M'_{U}$ and the triangle inequality, 
\begin{equation*}
R^{-\kappa}\|\ell_{2R}-\ell_{R}\|_{B_R(0)}\lesssim M'_{U}.
\end{equation*}
This implies $R^{-(\kappa-1)}|\nu_{2R}-\nu_{R}|+R^{-\kappa}|c_{2R}-c_R|\lesssim M'_{U}$.  Since $\kappa > 1$, telescoping gives $R^{-(\kappa-1)}|\nu_{R}-\nu_{R'}|+R^{-\kappa}|c_{R}-c_{R'}|\lesssim M'_{U}$
for all $R'\le R$ and thus the existence of $\nu,c\in\mathbb{R}$ such that 
\begin{equation*}
R^{-(\kappa-1)}|\nu_{R}-\nu|+R^{-\kappa}|c_{R}-c|\lesssim M'_{U},
\end{equation*}
so that $\ell(y):=\nu y_1+c$ satisfies
\begin{align}\label{1.21}
R^{-\kappa}\|\ell_R-\ell\|_{B_R(0)}\lesssim M'_{U}.
\end{align}
Hence we may pass from \eqref{3.6}
to \eqref{3.7} by the triangle inequality.

\medskip

It is clear from \eqref{3.7} that necessarily for any $x$,
the optimal $\ell$ must be of the form $\ell(y) =U(x,x)+\nu(x)(y-x)_1$.  This establishes the
main part of \eqref{3.8}, namely the modelledness 
\begin{align}
\left |U(x,y)-U(x,x)-\nu(x)(y-x)_{1}\right | \lesssim M'_{U} d^{\kappa}(y,x)\label{p43} 
\end{align} 
for any base point $x$
and any $y$ of distance at most $4$.
Since $B_4(x)$ covers a periodic cell, we may use
\eqref{p43} for $y=x+(1,0)$ so that by periodicity of $y\mapsto U(x,y)$
we extract $|\nu (x)|\lesssim M'_{U}$, which by $\kappa > 1$ implies $|\nu(x)(x-y)_1|\lesssim M'_{U}d^{\kappa}(x,y)$ for all $y\not\in B_4(x)$. Hence
once again by periodicity of
$y\mapsto U(x,y)$, \eqref{p43} holds also for $y\not\in B_4(x)$.

\medskip

{\sc Step} \arabic{PDE_P}.\label{pde_conc_St}\refstepcounter{PDE_P} We now give the argument for \eqref{KS3} and \eqref{wg52} under the additional 
(purely qualitative hypothesis) that there exists $\nu$ such that $M_U<\infty$, of which we free ourselves in Step \ref{intLemQH}.  Combining Steps \ref{pde_scale_St} and \ref{pde_conv_St}, we obtain by the triangle inequality for each base point $x$ and all scales $R$, $L$ and $T^\frac14$ with $R\ll L$
\begin{align*}
 \inf_\ell \|U(x,\cdot)-\ell\|_{B_R(x)}&\lesssim  \bigg(\frac{R}{L}\bigg)^2 \inf_\ell \|U(x,\cdot)-\ell\|_{B_L(x)}+M_{U}(T^{\frac{1}{4}})^{\kappa}  \\
  \quad &+ L^2 \overline{M}\sum_{\beta\in\Aa{}}\Tc{\beta-2}L^{\kappa-\beta}+\overline{\overline{M}}\sum_{\beta\in\Aa{}}L^{\beta}(T^{\frac{1}{4}})^{\kappa-\beta}.
\end{align*}
Multiplying by $R^{-\kappa}$ and using the definition of $M_{U}$ gives
\begin{align*}
 R^{-\kappa}&\inf_\ell \|U(x,\cdot)-\ell\|_{B_R(x)}\lesssim M_{U} \left(\frac{L}{R}\right)^{\kappa-2} +M_{U}\left (\frac{T^{\frac{1}{4}}}{R} \right)^{\kappa}\\
 &+ R^{-\kappa}L^2 \overline{M}\sum_{\beta\in\Aa{}}(T^{\frac{1}{4}})^{\beta-2}L^{\kappa-\beta}+R^{-\kappa}\overline{\overline{M}}\sum_{\beta\in\Aa{}}L^{\beta}(T^{\frac{1}{4}})^{\kappa-\beta}.
\end{align*}
Now we link the scales $L$ and $T^\frac14$ to $R$ by introducing a small $\eps>0$ and choosing $L=\frac{1}{\eps}R$ and $T^\frac14 = \eps R$.  Using Step 3, we find
\begin{align*}
M_{U} \lesssim M_{U}(\eps^{-(\kappa-2)}+\eps^{\kappa}) + \sum_{\beta\in\Aa{}} \left (\overline{M}\eps^{-\kappa+2\beta-4}+\overline{\overline{M}}\eps^{\kappa-2\beta} \right ). 
\end{align*} 
Taking into account $\kappa \in (1,2)$, we may choose $\eps$ small enough to ensure \eqref{KS3}, where we used the qualitative assumption that $M_U<\infty$.

The first inequality in \eqref{wg52} instantly follows from choosing $y=x+(1,0)$ and appealing to the periodicity
of $U$ in the $y$-variable.

\medskip

{\sc Step} \arabic{PDE_P}.\label{intLemQH}\refstepcounter{PDE_P} In this step, we give the argument for \eqref{KS3} without the additional
qualitative hypothesis of Step \ref{pde_conc_St} (which was not needed for \eqref{wg52}). 
To this purpose, we consider $U_\tau$ for $0<\tau\le\frac{1}{2}$; 
since $U_\tau$ is bounded and smooth in the $y$-variable, $\nu^{\tau}(x):=\frac{\partial}{\partial y_2}U_\tau(x,x)$ 
is such that $M_{U_\tau}<\infty$. By the semi-group property \eqref{1.10}, and because of $\beta-2\le 0$, 
our hypothesis \eqref{KS1} also holds with $U$ replaced by $U_\tau$ with the same constant $\overline{M}$, 
albeit in the restricted range $0<T\le\frac{1}{2}$. The latter is no problem for the argument in Step \ref{pde_conc_St}, since 
we set $T^\frac{1}{4}=\epsilon R$ with $R\le 4$ and $\epsilon\ll 1$. 
By estimate \eqref{wg51} in Step \ref{pde_conv_St} we have for all $T\ge 0$
\begin{align*}
\|(U_\tau)_T(x,\cdot)-U_\tau(x,\cdot)\|_{B_L(x)}\lesssim M_{U_\tau}(T^\frac{1}{4})^{\kappa}
+\overline{\overline{M}}\sum_{\beta\in\Aa{}}L^{\beta}(T^\frac{1}{4})^{\kappa-\beta}.
\end{align*}
By the buckling argument of Step \ref{pde_conc_St} we thus obtain the uniform estimate
\begin{align*}
M_{U_\tau}+\|\nu^\tau\|\lesssim\overline{M}+\overline{\overline{M}}.
\end{align*}
We now fix $x\in\mathbb{R}^2$; for any sequence of $\tau$'s there exists a subsequence $\tau_n\downarrow 0$ 
such that the bounded sequence $\nu^{\tau_n}(x)$ converges to some $\nu(x)\in\mathbb{R}$. 
By the continuity of $U$ in the $y$-variable, we have
$U_{\tau_n}(x,y)$ $\rightarrow U(x,y)$ for any $y\in\mathbb{R}^2$. Hence we may pass to the limit in
$|U_{\tau_n}(x,y)$ $-U_{\tau_n}(x,x)-\nu^{\tau_n}(x)(x-y)_1|$ $\lesssim(\overline{M}+\overline{\overline{M}})$
$d^{\kappa}(y,x)$, thus obtaining \eqref{KS3}.

\medskip

{\sc Step} \arabic{PDE_P}.\label{nu}\refstepcounter{PDE_P}
Finally, we turn to the estimate \eqref{KS4}.  By the definition of $M_{U}$, for any $x$ and $y$ we have
\begin{align*}
\left|U(x,y)-U(x,x)-\nu(x)(y-x)_{1}\right | \leq M_{U}d^{\kappa}(y,x).
\end{align*}
Using $z$ in place of $y$ gives
\begin{align*}
\left|U(x,z)-U(x,x)-\nu(x)(z-x)_{1}\right | \leq M_{U}d^{\kappa}(z,x).
\end{align*}
Combining these two estimates and using the triangle inequality for the absolute value, we obtain for all $x,y,z$
\begin{align*}
\left|U(x,z)-U(x,y)-\nu(x)(z-y)_{1}\right | \leq M_{U}\big (d^{\kappa}(y,x)+d^{\kappa}(z,x)\big ).
\end{align*}
In particular, setting $x=y$ gives
\begin{align*}
\left|U(y,z)-U(y,y)-\nu(y)(z-y)_{1}\right | \leq M_{U}d^{\kappa}(z,y).
\end{align*}
Combining the last two estimates by the triangle inequality for both the absolute value and the parabolic metric yields
\begin{align}\label{nu5}
&\left|(U(y,z)-U(y,y))-(U(x,z)-U(x,y))-(\nu(y)-\nu(x))(z-y)_{1}\right | \\
&\quad \lesssim M_{U}\big (d^{\kappa}(y,x)+d^{\kappa}(z,y)\big ).\nonumber
\end{align}
We now combine \eqref{nu5} with the three-point continuity condition \eqref{KS2} and the triangle inequality to obtain
\begin{align*}
&\left|((\nu(y)-\nu(x)+\gamma(x,y))(z-y)_{1}\right | \\
&\quad \lesssim M_{U}\big (d^{\kappa}(y,x)+d^{\kappa}(z,y)\big )+\overline{\overline{M}}\sum_{\beta\in\Aa{}} d^{\beta}(y,x)d^{\kappa-\beta}(z,y).
\end{align*}
Hence, given any two points $x,y$, setting $z=(y_{1}+d(y,x),y_{2})$ and observing that $(y-z)_{1}=d(z,y)=d(y,x)$, we obtain 
\begin{align*}
&\left|\nu(y)-\nu(x)+\gamma(x,y)\right | \\
&\quad \lesssim (M_{U}+\overline{\overline{M}})d^{\kappa-1}(y,x), \nonumber
\end{align*}
which implies \eqref{KS4}.

\medskip

{\sc Step} \arabic{PDE_P}.\label{intLemRemPf}\refstepcounter{PDE_P} We finally give the argument for Remark \ref{rem-wg53}.
We first address uniqueness: If $\tilde\nu$ were another function for which the expression
on the l.~h.~s.~of \eqref{KS3} is finite, we would have for every point $x$ that
$\sup_{y\not=x}d^{-\kappa}(y,x)|(\tilde\nu-\nu)(x)(y-x)_1|<\infty$.
Choosing $y=x+(M,0)$ and letting $M\downarrow 0$ we obtain $(\tilde\nu-\nu)(x)=0$
thanks to $\kappa > 1$. Uniqueness now implies that periodicity of $U$ in $x$
is transmitted to $\nu$.

\medskip

We now turn to H\"older continuity. By the periodicity of $U$ in the $y$-variable we obtain
from setting $z=y+(1,0)$ in \eqref{KS2} that
\begin{align*}
|\gamma(x,y)|\le\overline{\overline{M}}\sum_{\beta\in\Aa{}} d^{\beta}(y,x).
\end{align*}
Using this in conjunction with \eqref{KS4} yields the desired
\begin{align*}
|\nu(y)-\nu(x)|\lesssim(\overline{M}+\overline{\overline{M}})
\big(d^{\kappa-1}(y,x)+\sum_{\beta\in\Aa{}} d^{\beta}(y,x)\big).
\end{align*}

\end{proof}
\section{Proofs of Concrete Results}\label{S:concr}
\subsection{ Lemma \ref{tensor_lemma} and Lemma \ref{est_wwf_lemma}}

\begin{proof}[Proof of Lemma \ref{tensor_lemma}]
  Let us verify the crossnorm property \eqref{j501a}. By definition \eqref{j552} we have
  \begin{align*}
   \|g\otimes h\|_{C^{k',k}}&=\max_{a_0',a_0,i',i}\left|\partial^{i'}_{a_0'}\partial^{i}_{a_0}g(a_0')h(a_0)\right| =\max_{a_0',a_0,i',i}\left|\partial^{i'}_{a_0'}g(a_0')\right|\left|\partial^{i}_{a_0}h(a_0)\right| \\
   &=\max_{a_0',i'}\left|\partial^{i'}_{a_0'}g(a_0')\right| \max_{a_0,i}\left|\partial^{i}_{a_0}h(a_0)\right|= \|g\|_{C^{k'}}\|h\|_{C^{k}}.
  \end{align*}
  Let now $u=\sum_{n=1}^N g_n\otimes h_n$ be a general element of the algebraic tensor space of $C^{k'}(\interval{})$ and $C^{k}(\interval{})$, thus in particular a $C^{k',k}(\interval{2})$-function of the two variables $(a_0',a_0)\in \interval{2}$. Assume that
  \begin{align*}
   \aop&\in\call(C^{k'}(\interval{}); C^{l'}(\interval{})), \quad \bop\in\call(C^{k}(\interval{}); C^{l}(\interval{})), \\
   \taop&\in\call(C^{k'}(\interval{}); C^{m'}(\interval{})), \quad \tbop\in\call(C^{k}(\interval{}); C^{m}(\interval{})),
  \end{align*}
  and admit the estimates \eqref{j525} and \eqref{j526}.
  We observe that for all $i'\le l'$, $i\le l$, $a_0'\in\interval{}$ and $a_0\in\interval{}$ it holds
  \begin{align*}
   \partial_{a_0'}^{i'}\partial_{a_0}^{i}(\aop\otimes \bop)u(a_0',a_0)&\stackrel{\eqref{j509}}{=}\partial_{a_0'}^{i'}\partial_{a_0}^{i}\sum_{n=1}^N \aop g_n(a_0') \bop h_n(a_0)\\
   &=\partial_{a_0'}^{i'}\aop \left( \sum_{n=1}^N  g_n(a_0') \partial_{a_0}^{i} \bop  h_n(a_0)\right) 
  \end{align*}
  with the understanding that $\aop$ acts on the $a_0'$-variable. Hence, 
  \begin{align*}
   \|(\aop\otimes \bop)u\|_{C^{l',l}}&\stackrel{\eqref{j552}}{=}\max_{i\le l,a_0\in \interval{}}\max_{i'\le l',a_0'\in \interval{}}\left|\partial_{a_0'}^{i'}\aop \left( \sum_{n=1}^N  g_n(a_0') \partial_{a_0}^{i} \bop h_n(a_0)\right)\right| \\
   &\stackrel{\eqref{j525}}{\le} \max_{i\le l,a_0\in \interval{}} \max_{i'\le m',a_0'\in \interval{}}\left|\partial_{a_0'}^{i'}\taop \left( \sum_{n=1}^N  g_n(a_0') \partial_{a_0}^{i} \bop  h_n(a_0)\right)\right| \\
   &=\|(\taop\otimes \bop)u\|_{C^{m', l}}.
  \end{align*}
  By an analogous procedure we obtain from \eqref{j526}
  \begin{align*}
   \|(\taop\otimes \bop)u\|_{C^{m', l}}\le \|(\taop\otimes \tbop)u\|_{C^{m', m}},
  \end{align*}
  and thus the desired \eqref{j522} for $u$ in the algebraic tensor space.
  
 \medskip

  In particular, in the special case $\taop=\|\aop\|_{\call(C^{k'}; C^{l'})}\id_{C^{k'}}$ and $\tbop=\|\bop\|_{\call(C^{k}; C^{l})}\id_{C^{k}}$ we can now conclude $\aop \otimes \bop\in \call(C^{k',k}(\interval{2});C^{l',l}(\interval{2}))$ by a standard density argument based on the observation that any element $u\in C^{k',k}(\interval{2})$ may be approximated by elements in the algebraic tensor space, since in particular polynomials are dense in $C^{k',k}( \interval{2})$. In the general case with arbitrary $\taop$ and $\tbop$ we likewise have $\taop \otimes \tbop\in \call(C^{k',k}(\interval{2});C^{m',m}(\interval{2}))$. Hence, using the continuity of $\aop\otimes \bop$ and $\taop \otimes \tbop$, we may pass to the limit in \eqref{j522}.  For the triple tensor product spaces $C^{k'',k',k}(\interval{3})$, similar arguments apply.
  \qedhere
\end{proof}
We now give the proof of the properties of the model.
\begin{proof}[Proof of Lemma \ref{est_wwf_lemma}]
For each $a_{0}',a_{0} \in I$, define the function $U(\cdot,a_{0}',a_{0}): \R^{2} \times \R^{2} \to \R$ by
\begin{equation}
U(x,y,a_{0}',a_{0}):=w_{2\alpha}(y,a_{0}',a_{0})-\vf(x,a_{0}') \partial_{a_{0}}\vf(y,a_{0}).\label{ad0}
\end{equation}
Our plan is to apply Proposition \ref{int lem} to the function $(x,y) \mapsto \partial_{a_0'}^{k'}\partial_{a_0}^{k}U(x,y,a_{0}',a_{0})$ for each $a_{0}',a_{0} \in I$, $k'=0,1,2,3$ and $k=0,1,2$.  In Steps \ref{Mdprf_1} and \ref{Mdprf_2} we establish the three-point continuity condition \eqref{KS2} and the local splitting condition \eqref{KS1}.  With these inputs at hand, in Step \ref{Mdprf_5} we construct $\omega$ and obtain the estimates \eqref{j781} and \eqref{s2011}.  Finally, in Step \ref{Mdprf_4}, we turn to the estimates \eqref{j790} for the skeleton. 

\medskip

\newcounter{Mdprf} 
\refstepcounter{Mdprf} 

\medskip

{\sc Step} \arabic{Mdprf}.\label{Mdprf_1}\refstepcounter{Mdprf}[Three-point continuity] 
In this step, we claim that for all $x,y,z$ in $\R^{2}$, $k'=0,1,2,3$, $k=0,1,2$, $a_{0}',a_{0} \in I$ and $\tau \geq 0$ it holds 
 \begin{align}
&\big | \partial_{a_{0}'}^{k'}\partial_{a_0}^{k} \big ( U_{\tau}(x,z,a_{0}',a_{0})-U_{\tau}(x,y,a_{0}',a_{0})-U_{\tau}(y,z,a_{0}',a_{0}) \nonumber\\
& \, \,\quad \qquad \qquad \qquad \qquad +U_{\tau}(y,y,a_{0}',a_{0}) \big ) \big |\lesssim N^{2}d^{\alpha}(x,y)d^{\alpha}(z,y), \label{S34}
\end{align}
where $U_{\tau}$ denotes convolution in the second argument only.  Indeed, we note the identity
\begin{align*}
&\partial_{a_{0}'}^{k'}\partial_{a_0}^{k} \big ( U_{\tau}(x,z,a_{0}',a_{0})-U_{\tau}(x,y,a_{0}',a_{0})-U_{\tau}(y,z,a_{0}',a_{0})\\
& \, \,\quad \qquad \qquad \qquad \qquad +U_{\tau}(y,y,a_{0}',a_{0}) \big )\\ 
&= \partial_{a_0'}^{k'}\big (\vf(y,a_{0}')-\vf(x,a_{0}') \big)\partial_{a_0}^{k+1}\big ((\vf)_{\tau}(y,a_{0})-(\vf)_{\tau}(z,a_{0}) \big )
,
\end{align*}
which follows directly from the definition \eqref{ad0}.  For $\tau=0$, the identity above implies \eqref{S34} as a consequence of \eqref{s203}.  Note that it is here that we use the control in $C^{3}$.  For $\tau>0$, we again use \eqref{s203} together with the fact that H\"{o}lder estimates are stable with respect to the convolution operator $(\cdot)_{\tau}$.

\medskip

{\sc Step} \arabic{Mdprf}.\label{Mdprf_2}\refstepcounter{Mdprf}[Local splitting condition] 
We claim that for all base points $x \in \R^{2}$ and scales $L,T^{\frac{1}{4}}\leq 1$
\begin{align}
\inf_{c \in \R}&\|(\partial_{2}-a_{0}\partial_{1}^{2})(\cdot)_{T}\partial_{a_0'}^{k'}\partial_{a_0}^{k}U(x,\cdot,a_{0},a_{0}')-c\|_{B_{L}(x)}\nonumber\\
& \lesssim N^{2}\big ( \Tc{2\alpha-2}+L^{\alpha}\Tc{\alpha-2} \big)\label{SFin65},
\end{align} 
where $k'=0,1,2,3$, $k=0,1,2$, and $a_{0}',a_{0} \in I$.  The claim will be established by induction in $k$, starting with the anchoring $k=0$ which we now address.  Let us fix an $x \in \R^{2}$, an integer $k' \in \{0,1,2,3\}$, and ellipticities $a_{0}',a_{0} \in I$.  In light of definition \eqref{ad0}, together with \eqref{f22} and \eqref{se40}, it follows that 
\begin{align}
&(\partial_{2}-a_{0}\partial_{1}^{2})\partial_{a_0'}^{k'}U(x,\cdot,a_{0}',a_{0})\nonumber\\
&\quad =\partial_{a_0'}^{k'}P(\vvf)(\cdot,a_{0}',a_{0})-\partial_{a_{0}'}^{k'}\vf(x,a_{0}')(\partial_{1}^{2}\vf)(\cdot,a_{0})
\end{align}
in the sense of distributions.  Hence, applying $(\cdot)_{T}$ on both sides of the above and re-arranging, we find
\begin{align}
&(\partial_{2}-a_{0}\partial_{1}^{2})(\cdot)_{T}\partial_{a_0'}^{k'}U(x,\cdot,a_{0}',a_{0})+\partial_{a_0'}^{k'}c(a_{0}',a_{0})\nonumber\\
&\quad=\partial_{a_0'}^{k'}(\vvf)_{T}(\cdot,a_{0}',a_{0})-\partial_{a_0'}^{k'}\vf(\cdot,a_{0}')(\partial_{1}^{2}\vf)_{T}(\cdot,a_{0})\nonumber\\
&\quad+\partial_{a_0'}^{k'}\big (\vf(\cdot,a_{0}')-\vf(x,a_{0}') \big)(\partial_{1}^{2}\vf)_{T}(\cdot,a_{0}),\label{ad5}
\end{align}
where $c(a_{0}',a_{0})=\int_{[0,1)^{2}}(\vvf)(y,a_{0}',a_{0})\dd y$.  Hence, \eqref{SFin65} for $k=0$ is a consequence of \eqref{off1}, \eqref{s203}, and \eqref{p13}. 

\medskip

We now turn to the inductive step.  Assuming that the local splitting condition \eqref{SFin65} holds for some $k \in \{0,1,2\}$ and each $k'=0,1,2,3$, $a_{0}',a_{0} \in I$, we claim that \eqref{SFin65} continues to hold with $k+1$ in place of $k$.  Towards this end, we differentiate the identity \eqref{ad5} $k+1$ times in $a_{0}$ to obtain
\begin{align*}
&(\partial_{2}-a_{0}\partial_{1}^{2})(\cdot)_{T}\partial_{a_0'}^{k'}\partial_{a_0}^{k+1}U(x,\cdot,a_{0}',a_{0})-\partial_{a_0'}^{k'}\partial_{a_0}^{k+1}c(a_{0}',a_{0})\nonumber\\
&\quad=\partial_{a_0'}^{k'}\partial_{a_0}^{k+1}(\vvf)_{T}(\cdot,a_{0}',a_{0})-\partial_{a_0'}^{k'}\vf(\cdot,a_{0}')(\partial_{a_0}^{k+1}\partial_{1}^{2}\vf)_{T}(\cdot,a_{0})\nonumber\\
&\quad\,+\partial_{a_0'}^{k'}\big (\vf(\cdot,a_{0}')-\vf(x,a_{0}') \big)(\partial_{a_0}^{k+1}\partial_{1}^{2}\vf)_{T}(\cdot,a_{0}) \nonumber \\
&\quad\,+(k+1)\partial_{1}^{2}(\cdot)_{T}\partial_{a_0'}^{k'}\partial_{a_0}^{k}U(x,\cdot,a_{0}',a_{0}).
\end{align*}
The first two lines on the r.h.s. of the equality above are estimated by the r.h.s. of \eqref{SFin65} by the same argument as in the base case $k=0$, in light of \eqref{off1}, \eqref{s203}, and \eqref{p13}.  To estimate the last line, we write for each $y \in B_{L}(x)$
\begin{align*}
&\partial_{1}^{2}(\cdot)_{T}\partial_{a_0'}^{k'}\partial_{a_0}^{k}U(x,y,a_{0}',a_{0}) \nonumber\\
& \,=\int \partial_{a_0'}^{k'}\partial_{a_0}^{k}U(x,z,a_{0}',a_{0})\partial_{1}^{2}\psi_{T}(y-z) {\rm d}z \nonumber\\
&\, =\int \partial_{a_0'}^{k'}\partial_{a_0}^{k}\big (U(x,z,a_{0}',a_{0})-U(x,y,a_{0}',a_{0})-U(y,z,a_{0}',a_{0})\\
&\qquad \qquad \qquad \qquad \qquad \quad \, \, \, \, +U(y,y,a_{0}',a_{0}) \big) \partial_{1}^{2}\psi_{T}(y-z) {\rm d}z \nonumber\\
&\, +\int \partial_{a_0'}^{k'}\partial_{a_0}^{k}\big (  U(y,z,a_{0}',a_{0})-U(y,y,a_{0}',a_{0})\big )\partial_{1}^{2}\psi_{T}(y-z) {\rm d}z.
\end{align*}
The first integral on the r.h.s. is controlled by $N^{2}L^{\alpha}(T^{\frac{1}{4}})^{\alpha-2}$ as a consequence of \eqref{S34} (with $\tau=0$) and \eqref{1.13}.  Moreover, we claim that the second integral is estimated by $N^{2}(T^{\frac{1}{4}})^{2\alpha-2}$.  First observe that by \eqref{1.13} and \eqref{he1}, it suffices to find an affine function $\ell(z)$ (depending on $y, a_{0}',a_{0},k',k$) such that for all $z\in \R^{2}$
\begin{align*}
&\big |\partial_{a_0'}^{k'}\partial_{a_0}^{k} \big ( U(y,z,a_{0}',a_{0})-U(y,y,a_{0}',a_{0}) \big )-\ell(z) \big | \\
&\quad \lesssim N^{2}d^{2\alpha}(z,y).
\end{align*}
For this, we apply Proposition \ref{int lem} to the function 
\begin{equation*}
(y,z) \mapsto \partial_{a_0'}^{k'}\partial_{a_0}^{k}U(y,z,a_{0}',a_{0}),
\end{equation*}
with $\alphao:=2\alpha-2$, $\gamma \equiv 0$, and $\overline{M}=\overline{\overline{M}} \lesssim N^2$.  The hypotheses \eqref{KS1} and \eqref{KS2}, follow as a consequence of our inductive assumption \eqref{SFin65} and the three-point continuity condition \eqref{S34} for $\tau=0$.  The output $\nu$ determines the affine function as $\ell(z)=\nu(y)(z-y)_{1}$. 

\medskip

{\sc Step} \arabic{Mdprf}.\label{Mdprf_5}\refstepcounter{Mdprf}[Identification Step]
We now construct $\omega: \R^{2} \to C^{2,1}(I^{2})$ such that \eqref{j781} and \eqref{s2011} hold.  Towards this end, we first claim that defining $\omega^{\tau}: \R^{2} \to C^{3,2}(I^{2})$ by $\omega^{\tau}(x,a_{0}',a_{0})=\frac{\partial}{\partial y_2}U_{\tau}(x,x,a_{0}',a_{0})$, the following estimates hold uniformly in $0<\tau \leq \frac{1}{2}$: for all $x,y \in \R^{2}$
\begin{align}
&\|U_{\tau}(x,y)-U_{\tau}(x,x)-\omega^{\tau}(x)(y-x)_{1}\|_{C^{3,2}}
\lesssim N^{2}d^{2\alpha}(y,x),\label{ad10}\\
&\|\omega^{\tau}(y)-\omega^{\tau}(x)\|_{C^{3,2}} \lesssim N^{2}d^{2\alpha-1}(y,x).\label{ad11}
\end{align}
To see this, we apply Proposition \ref{int lem} as in the end of Step \ref{Mdprf_2} above.  The hypotheses \eqref{KS1} and \eqref{KS2}, follow as a consequence of the local splitting condition \eqref{SFin65} established in Step \ref{Mdprf_2} and the three-point continuity condition \eqref{S34} established in Step \ref{Mdprf_1}, applied with $\tau=0$.  For the output \eqref{ad10}, we appeal to Remark \ref{rem-wg53} in the form of estimate \eqref{KS6}.  Here, we are taking advantage of exchanging the order of differentiation in the sense of $\partial_{a_0'}^{k'}\partial_{a_0}^{k}\omega^{\tau}(x,a_{0}',a_{0})=\frac{\partial}{\partial y_2}\partial_{a_0'}^{k'}\partial_{a_0}^{k}U_{\tau}(x,x,a_{0}',a_{0})$, which enforces a consistency among the functions $\nu^{\tau}$ in Remark \ref{rem-wg53} as we vary the parameters $k',k,a_{0}',a_{0}$.   The estimate \eqref{ad11} now follows from \eqref{ad10} and \eqref{S34} by the same arguments that lead from \eqref{KS2} and \eqref{KS3} to \eqref{KS4} in the proof of Proposition \ref{int lem}, cf. Step \ref{nu}.

\medskip

We now apply the Arzel\`{a}-Ascoli theorem for functions with values in Banach spaces to the sequence $\{\omega^{\tau}\}_{\tau \leq \frac{1}{2}}$.  Note that $\omega^{\tau}$ inherits periodicity from $U_{\tau}$ which allows us to restrict to a bounded domain and also implies via \eqref{ad11} that $\{\omega^{\tau}\}_{\tau \leq \frac{1}{2}}$ is bounded in $C^{0}(\R^{2}; C^{3,2}(I^{2}))$.  Note that the embedding $C^{3,2}(I^{2}) \hookrightarrow C^{2,1}(I^{2})$ is compact.  The uniform estimate \eqref{ad11} also provides the necessary equi-continuity, so there exists $\omega \in C^{2\alpha-1}(\R^{2}; C^{2,1}(I^{2}))$ such that, along a subsequence, $\omega_{\tau} \to \omega$ in $C^{0}(\R^{2}; C^{2,1}(I^{2}))$.  Moreover, by the continuity of the convolution operator $(\cdot)_{\tau}$, it follows that $U_{\tau}(x,y) \to U(x,y)$ in $C^{2,1}(I^{2})$ pointwise in $(x,y)\in \R^{2} \times \R^{2}$.  Thus, sending $\tau \to 0$ in \eqref{ad10} and \eqref{ad11} leads to \eqref{s2011} and \eqref{j781} respectively.

\medskip

{\sc Step} \arabic{Mdprf}.\label{Mdprf_4}\refstepcounter{Mdprf}[Estimate of the Skeleton Norms] 

We now prove the bound \eqref{j790} for the skeleton $\Gamma$.  Namely, we treat separately $\Ga{+}{}$ and $\Ga{-}{}$ and then appeal to Lemma \ref{lem:prod_skel} to complete the estimate.  Starting with $\Ga{+}{}$, cf.\@ \eqref{cw12}, we will analyze each of the increments $\Gamma_{\beta}(y)-\Gamma_{\beta}(x)$ for $\beta \in \Aa{+}$ and show \eqref{ssss2} with a constant which is bounded (in the sense of $\lesssim$) by $1$.  There is nothing to show for the lowest homogeneity $0$.  For homogeneity $\alpha$, we note that $(\Ga{\alpha}{}(y)-\Ga{\alpha}{}(x))\ta_{+} \stackrel{\eqref{cw12}}{=}-(\vf(y)-\vf(x))\na$, which thanks to \eqref{s203} is estimated in $\Ta{\alpha}$ by $d^{\alpha}(y,x)|\na|\stackrel{\eqref{cw12}}{=}d^{\alpha}(y,x)\|\Ga{0}{}(x)\ta_{+}\|_{\Ta{0}}$.  For homogeneity $1$, we note that $(\Ga{1}{}(y)-\Ga{1}{}(x))\ta_{+}\stackrel{\eqref{cw12}}{=}-(v_{1}(y)-v_{1}(x))\na$, which is estimated by $d(y,x) \|\Ga{0}{}(x)\ta_{+}\|_{\Ta{0}}$.  Finally, we turn to homogeneity $2\alpha$. Rearranging terms, we obtain from \eqref{cw12}
\begin{align*}
& (\Ga{2\alpha}{}(y)-\Ga{2\alpha}{}(x) \big )\ta_{+}=-(\vf(y)-\vf(x)) \otimes \big ( \partial_{a_0}\va-\partial_{a_{0}}\vf(x)\na \big )\\ 
&\quad -(\omega(y)-\omega(x) ) \big ( \mathsf{v}_{1}-v_{1}(x)\na \big )\\ 
&\quad -\big (w_{2\alpha}(y)-w_{2\alpha}(x)-\vf(y) \otimes (\partial_{a_0}\vf(y)-\partial_{a_0}\vf(x) )\\ 
&\qquad - \omega(y) (v_{1}(y)-v_{1}(x)) \big ) \na. \nonumber
\end{align*}
We now estimate each line on the right hand side separately in the space $\Ta{2\alpha}$. By the uniform cross-norm property \eqref{j525}-\eqref{j522} for function spaces, cf. Lemma \ref{tensor_lemma}, the scaling in $N$, the grading in terms of order of derivatives of the norms of $\Ta{2\alpha}$ and $\Ta{\alpha}$, cf.\@ \eqref{S8030} leading to the estimate $\|g \otimes \partial_{a_0}h\|_{\Ta{2\alpha}} \leq \|g\|_{\Ta{\alpha}}\|h\|_{\Ta{\alpha}}$, and \eqref{s203}, the first line is bounded in $\Ta{2\alpha}$ by
\begin{equation*}
 \lesssim d^{\alpha}(y,x)\|\va-\vf(x)\na\|_{\Ta{\alpha}} \stackrel{\eqref{cw12}}{=} d^{\alpha}(y,x)\|\Ga{\alpha}{}(x)\ta_{+}\|_{\Ta{\alpha}{}}.
\end{equation*}
For the second line, we appeal to \eqref{j781} and bound in $\Ta{2\alpha}$ by
\begin{equation*}
\lesssim d^{2\alpha-1}(y,x)|\mathsf{v}_{1}-v_{1}(x)\na | \stackrel{\eqref{cw12}}{=} d^{2\alpha-1}(y,x)\|\Ga{1}{}(x)\ta_{+}\|_{\Ta{1}}.
\end{equation*}
Thanks to \eqref{s2011} with roles of $x$ and $y$ exchanged, the final term is then estimated by
\begin{align*}
 \lesssim d^{2\alpha}(y,x)| \na | \stackrel{\eqref{cw12}}{=} d^{2\alpha}(y,x)\|\Ga{0}{}(x)\ta_{+}\|_{\Ta{0}}.
\end{align*}
In summary, we find that \eqref{ssss2} holds for homogeneity $2\alpha$ in the form of
\begin{align*}
\|(\Ga{2\alpha}{}(y)-\Ga{2\alpha}{}(x))\ta_{+}\|_{\Ta{2\alpha}} &\lesssim \sum_{\beta\in \set{0,\alpha,1}}d^{2\alpha-\beta}(x,y)\|\Ga{\beta}{}(x)\ta_{+}\|_{\Ta{\beta}{}},
\end{align*}
so that indeed $\|\Ga{+}{}\|_{sk}\lesssim 1$.

\medskip

Now we turn our attention to $\Gamma_{-}$ and seek to prove $\|\Ga{-}{}\|_{sk}\lesssim 1$.  Observe that $\big ( \Ga{\alpha-2}{}(y)-\Ga{\alpha-2}{}(x)\big )\ta_{-}=0$ by \eqref{cw12a},
while $\big ( \Ga{2\alpha-2}{}(y)-\Ga{2\alpha-2}{}(x) \big )\ta_{-}=-(\vf(y)-\vf(x))\otimes\partial_{a_0}\pva$, which is controlled in $\Ta{2\alpha-2}$ by
\begin{equation}
 d^{\alpha}(y,x)\|\pva\|_{\Ta{\alpha-2}} \leq d^{\alpha}(y,x)\|\Ga{\alpha-2}{}(x)\ta_{-} \|_{\Ta{\alpha-2}}.
\end{equation}
Hence, we find also that $\|\Ga{-}{}\|_{sk}\lesssim 1$.  We now combine this with $\|\Ga{+}{}\|_{sk}\lesssim 1$ to obtain by Lemma \ref{lem:prod_skel}
\begin{equation*}
\|\Ga{}{}\|_{sk} \leq \|\Ga{+}{}\|_{sk}+\|\Ga{-}{}\|_{sk}+\|\Ga{-}{}\|_{sk}\|\Ga{+}{}\|_{sk} \lesssim 1,
\end{equation*}
which completes the proof of \eqref{j790}.
\end{proof}
\subsection{Lemma \ref{cor:prod1}}
\begin{proof}[Proof of Lemma \ref{cor:prod1}]
We begin in Step \ref{Algprf_1} with the estimate \eqref{SFin1}, and will in fact prove a sharper estimate which will be used to establish the inequality \eqref{IOR1} in Step \ref{Algprf_2} below, en route to \eqref{S565}, which is established in Step \ref{Algprf_3}.

\medskip

\newcounter{Algprf} 
\refstepcounter{Algprf} 
{\sc Step} \arabic{Algprf}.\label{Algprf_1}\refstepcounter{Algprf}[Analysis of $V_{\partial_{1}^{2}u}^{\eta-\alpha-2}$] 
Our first claim is that $V_{\partial_{1}^{2}u}^{\eta-\alpha-2}$ is a modelled distribution of order $\eta-\alpha-2$ on the abstract space with only one homogeneity, $(\{\alpha-2\},\frac{1}{N}C^{1}(I))$, endowed with the skeleton $\id$, and more specifically, the following estimates hold
\begin{equation}
\|V_{\partial_{1}^{2}u}^{\eta-\alpha-2}\|_{\frac{1}{N}C^{1}(I);\id}\leq 1, \quad [V_{\partial_{1}^{2}u}^{\eta-\alpha-2} ]_{\mathcal{D}^{\eta-\alpha-2}(\frac{1}{N}C^{1}(I);\id)}\leq N[a]_{\alpha}.\label{S561}
\end{equation}
In particular, \eqref{S561} implies \eqref{SFin1} since the norm on $\Ta{\alpha-2}$ is stronger than the $\frac{1}{N}C^{1}(I)$ norm, cf.\@ \eqref{S8001} and \eqref{wg90}.  For the proof of \eqref{S561}, we recall from \eqref{s115} that $V_{\puf}^{\thresh-\alpha-2}(x)=\delta_{a(x)}$, and clearly we have 
\begin{equation}
|\delta_{a(x)}.\pva| \leq \|\pva\|_{C^{0}} \leq N\|\pva\|_{\frac{1}{N}C^1}\stackrel{\eqref{wg90}}{=}N^{\langle \alpha-2 \rangle}\|\pva\|_{\frac{1}{N}C^1},\label{SFin2}
\end{equation}
so that the boundedness property \eqref{prod4} is fulfilled with a constant $1$, resulting in the first estimate in \eqref{S561}.  To establish the second estimate, we show that the continuity property \eqref{prod5} holds with a constant $N[a]_{\alpha}$.  Namely, for $x,y \in \R^{2}$ with $d(y,x) \leq 1$
\begin{align*}
&\left | (V_{\partial_{1}^{2}u}^{\eta-\alpha-2}(y)-V_{\partial_{1}^{2}u}^{\eta-\alpha-2}(x) ).
\partial_{1}^{2}\va
\right |
\stackrel{\eqref{s115}}{=}\left |(\delta_{a(y)}-\delta_{a(x)}).\partial_{1}^{2}\va \right| \\
&\, \leq [a]_{\alpha}d^{\alpha}(y,x)\|\partial_{1}^{2}\va\|_{C^{1}} \leq N[a]_{\alpha}d^{\eta-2\alpha}(y,x)\|\partial_{1}^{2}\va\|_{\frac{1}{N}C^{1}},
\end{align*}
where in the last step we used $\alpha \geq \eta-2\alpha$ as a consequence of $\eta<1+\alpha<3\alpha$ (using $\alpha>\frac{1}{2}$).  

\medskip

{\sc Step} \arabic{Algprf}.\label{Algprf_2}\refstepcounter{Algprf}[Analysis of $V_{a \diamond \partial_{1}^{2}u}^{\eta-2}$] 
In this step, we claim that the form $V_{a \diamond \partial_{1}^{2}u}^{\eta-2}$ built according to \eqref{s116} has the following continuity property: for all $\pua\in\R$,
$$
\ta=(\pva,\vva) \in C^1(\interval{})\oplus C^{2,1}(\interval{2}),
$$
and all $x,y \in \R^{2}$ with $d(y,x) \leq 1$ it holds
\begin{align}
&\left |(V_{a \diamond \puf}^{\thresh-2}(y)-V_{a \diamond \puf}^{\thresh-2}(x)).
\begin{pmatrix}
\pua \\
\pva\\
\vva
\end{pmatrix} 
\right |\label{IOR1}\\
&\leq [a]_{\alpha}d^{\alpha}(y,x)\big |\pua-V_{\puf}^{\thresh-\alpha-2}(x).\pva \big |\nonumber\\
&\quad +M_{a,\eta-\alpha}d^{\thresh-\alpha}(y,x)\|\Ga{\alpha-2}{}(x)\ta\|_{\frac{1}{N}C^1}\nonumber\\
&\quad+M_{a,\eta-\alpha}d^{\thresh-2\alpha}(y,x)\|
\Ga{2\alpha-2}{}(x)\ta\|_{\frac{1}{N^2}C^{2,1}},\nonumber
\end{align}
where $M_{a,\eta-\alpha}$ is defined by \eqref{IOR30}.  We now turn to the proof, the main ingredient being Lemma \ref{lem:Multiplication}.  Since $V_{a}\in \mathcal{D}^{\eta}(\Ta{+};\Ga{+}{})$ and by definition $V_{a}^{\eta-\alpha}=C_{\eta-\alpha}V_{a}^{\eta}$, Lemma \ref{cutting_lemma} implies that $V_{a}^{\eta-\alpha}$ is a modelled distribution of order $\eta-\alpha$ on the abstract space 
$(\{0, \alpha\}, \R \oplus \frac{1}{N}C^2(\interval{}))$ obtained from $(\Aa{+},\Ta{+})$ by removing homogeneities (and their associated Banach spaces) greater or equal to $1>\eta-\alpha$.  According to \eqref{j961}, the skeleton reduces to 
\begin{align*}
 \begin{pmatrix}
  \id_{\R} & 0 \\
  -\vf & \id_{C^2}
 \end{pmatrix},
\end{align*}
cf.\@ \eqref{cw12}. According to the construction in Section \ref{S:mult}, cf.\@ \eqref{s116}, the form $V_{a \diamond \partial_{1}^{2}u}^{\eta-2}$ acts on the tensor spaces $\R \otimes \frac{1}{N}C^{1}(I)=\frac{1}{N}C^{1}(I)$ and $\frac{1}{N}C^{2}(I) \otimes \frac{1}{N}C^{1}(I)=\frac{1}{N^2}C^{2,1}(\interval{2})$, cf.\@ Lemma \ref{tensor_lemma}, associated to homogeneities $\alpha-2$ and $2\alpha-2$ respectively, extended with a copy of $\R$ for the placeholder $\pua$ attached to homogeneity $\eta-\alpha-2$.  The product skeleton defined via \eqref{Prod9} takes the form
\begin{align*}
 \begin{pmatrix}
  \id_{C^1} & 0 \\
  -\vf\otimes & \id_{C^{2,1}}
 \end{pmatrix}.
\end{align*}
Hence, we may appeal to Lemma \ref{lem:Multiplication}, where $V_{a}^{\eta-\alpha}$ plays the role of $V_{u_{+}}$, $N_*$ is set to $1$ and $V_{\partial_{1}^{2}u}^{\eta-\alpha-2}$ plays the role of $V_{u_{-}}$.  Observe that \eqref{MS3} turns into \eqref{IOR1}.  Indeed, the constant $M_{a,\eta-\alpha}$ defined by \eqref{IOR30} corresponds exactly to $M_{a \diamond \partial_{1}^{2}u}^{c}$ defined by \eqref{S301}, taking into account that $\langle \alpha\rangle=\langle \alpha-2 \rangle=1$, cf. \eqref{wg90}.  

\medskip

{\sc Step} \arabic{Algprf}.\label{Algprf_3}\refstepcounter{Algprf}[Analysis of $V_{\partial_{1}^{2}u}^{\eta-2}$] In the final step, we claim that the estimate \eqref{S565} holds, and we now begin with the proof of the boundedness estimate.  We bound each of the two terms in the definition \eqref{algebraicRel} separately, the first being estimated exactly as in \eqref{SFin2} in Step \ref{Algprf_1} above.  The second term is estimated via \eqref{neededlater1} in Lemma \ref{lem:prod_skel} and \eqref{new} in Lemma \ref{cutting_lemma}, taking into consideration \eqref{S561}, so that for all $x\in\R^2$
\begin{align*}
&\left | \overline{V}_{a}^{\thresh-\alpha}(x) \otimes \delta_{a(x)}.
\begin{pmatrix}
\partial_{a_0}\pva \\
\partial_{1}^{2}\mathsf{w}_{2\alpha}
\end{pmatrix}
 \right | \\
& \quad \leq \|\overline{V}_{a}^{\eta-\alpha}\| N^{2}\left \|\partial_{1}^{2}\mathsf{w}_{2\alpha}-\vf(x) \otimes \partial_{a_0}\pva \right \|_{\frac{1}{N^2}C^{2,1}(\interval{2})}.
\end{align*}
Hence, we find that for each $\mathsf{v}_-=(\partial_1^2\mathsf{v}_{\alpha},\partial_1^2\mathsf{w}_{2\alpha}) \in\mathsf{T}_-$
\begin{equation*}
\big | V_{\partial_{1}^{2}u}^{\eta-2}.\mathsf{v}_- \big | \leq N^{\langle \alpha-2 \rangle}\|\pva\|_{\Ta{\alpha-2}}+\|\overline{V}_{a}^{\eta-\alpha}\|N^{\langle 2\alpha-2 \rangle}\|\Ga{2\alpha-2}{}\mathsf{v}_-\|_{\Ta{2\alpha-2}},
\end{equation*}
cf.\@ \eqref{wg90},\eqref{cw12a},\eqref{S8001}, which yields the first estimate in \eqref{S565}.  

\medskip

To establish the second estimate in \eqref{S565}, we use the following identity, which follows from \eqref{s300} in light of \eqref{functDef} and \eqref{j41}:  for all 
$
\ta_{-}:=(\pva,\partial_{1}^{2}\mathsf{w}_{2\alpha}) \in C^1(\interval{})\oplus C^{2,1}(\interval{2}),
$
it holds
\begin{equation}
V_{\partial_{1}^{2}u}^{\eta-2}.
\begin{pmatrix}
\pva \\
\partial_{1}^{2}\mathsf{w}_{2\alpha}
\end{pmatrix}
=\delta_{a}.\pva-a\pua +
V_{a \diamond \partial_{1}^{2}u}^{\eta-2}.
\begin{pmatrix}
\pua \\
\partial_{a_0}\pva \\
\partial_{1}^{2}\mathsf{w}_{2\alpha}
\end{pmatrix}
,
\end{equation}
where $\pua \in \R$ is free to be chosen.  Given $x,y \in \R^{2}$ with $d(x,y) \leq 1$ we find that
\begin{align*}
&\big ( V_{\partial_{1}^{2}u}^{\eta-2}(y)-V_{\partial_{1}^{2}u}^{\eta-2}(x) \big ).
\begin{pmatrix}
\pva \\
\partial_{1}^{2}\mathsf{w}_{2\alpha}
\end{pmatrix}
\\
&=(\delta_{a(y)}-\delta_{a(x)}).\pva-\big ( a(y)-a(x) \big )\pua \\
&+\big ( V_{a \diamond \partial_{1}^{2}u}^{\eta-2}(y)-V_{a \diamond \partial_{1}^{2}u}^{\eta-2}(x) \big ).
\begin{pmatrix}
\pua \\
\partial_{a_0}\pva \\
\partial_{1}^{2}\mathsf{w}_{2\alpha}
\end{pmatrix}.
\end{align*}
Choosing $\pua:=\partial_{a_0}\delta_{a(x)}.\pva$, we find that the first term is bounded by
\begin{align*}
 [a]_{\alpha}^{2}d^{2\alpha}(y,x)\|\pva\|_{C^{2}(I)}& \le N [a]_{\alpha}^2 d^{\thresh-\alpha}(y,x) \|\pva\|_{\Ta{\alpha-2}},
\end{align*}
where we have used that $\thresh-\alpha\le 2\alpha$ and \eqref{S8001}.  For the second term we appeal to \eqref{IOR1} noting that the first term on the right hand side of \eqref{IOR1} is annihilated by our choice of $\pua$, cf.\@ \eqref{s115}.  Hence, the second term is controlled by
\begin{align*}
&M_{a,\eta-\alpha}d^{\eta-\alpha}(y,x)\|\partial_{a_0}\pva\|_{\frac{1}{N}C^{1}(\interval{})}\\
&+M_{a,\eta-\alpha}d^{\eta-2\alpha}(y,x)\|\partial_{1}^{2}\mathsf{w}_{2\alpha}- \vf(x)\otimes \partial_{a_0}\pva  \|_{\frac{1}{N^2}C^{2,1}(\interval{2})} \\
&\stackrel{\eqref{S8001}, \eqref{cw12a}}{\le} M_{a,\eta-\alpha} \sum_{\beta\in \Aa{-}} d^{\thresh-2-\beta}(y,x) \|\Ga{\beta}{}(x)\ta_{-}\|_{\Ta{\beta}}.
\end{align*}
Combining these two observations gives the second inequality in \eqref{S565}.  
\end{proof}
We now give the proof of the reconstruction output.
\subsection{Lemma \ref{rec_lem_concr}}
\begin{proof} [Proof of Lemma \ref{rec_lem_concr}]

\newcounter{ROprf} 
\refstepcounter{ROprf} 
The main ingredient for the proof of the lemma is Corollary \ref{lem:ExMult}, after an appropriate translation of the hypotheses, which in turn requires two preliminary steps.  Namely, in Step \ref{ROprf_1}, we establish a local description \eqref{cw60} of $\partial_{1}^{2}u$, which will translate to the hypothesis \eqref{j771}.  In Step \ref{ROprf_2}, we prove a consistency of multiplication and reduction, which (effectively) means that the reduction of $V_{a \diamond \partial_{1}^{2}u}^{\eta+\alpha-2}$ on level $\eta-2$ is $V_{a \diamond \partial_{1}^{2}u}^{\eta-2}$, and allows to translate the sub-optimal output \eqref{S870} into \eqref{s30}.  Finally, in Step \ref{ROprf_3} we detail how to apply Corollary \ref{lem:ExMult} to obtain $a \diamond \partial_{1}^{2}u$, together with the estimates \eqref{eq:QualRO} and \eqref{s30}.

\medskip

{\sc Step} \arabic{ROprf}.\label{ROprf_1}\refstepcounter{ROprf}[The local description of $\partial_{1}^{2}u$] 

In this step, we claim that the following estimate holds: for all $T\le 1$,
\begin{align}\label{cw60}
\|\partial_1^2u_T-V_{\partial_1^2u}^{\eta-2}.(v_-)_T\|
\lesssim [V_u^\eta](T^\frac{1}{4})^{\eta-2}.
\end{align}
Here comes the argument: we first claim that for {\it all} points $x,y$ 
\begin{align}\label{cw61}
|(V_u^\eta(y)-V_u^\eta(x)).v_+(y)|\le [V_u^\eta]d^{\eta}(x,y).
\end{align}
Indeed, for points with $d(x,y)\le 1$ this is an immediate consequence of form continuity \eqref{prod5}, applied with exchanged
roles of $x$ and $y$, and of compatibility \eqref{j700}. We now turn to points with $d(x,y)\ge 1$, which as in
the proof of Proposition \ref{rec_lem} we need to consider due to the infinite tails of $\psi_T$. Clearly,
there exists $k\in\mathbb{Z}^2$ such that $d(x-k,y)\le 1$, so that by $\eta>2-\alpha>1$, it is
enough to show
\begin{align}\label{cw63}
|(V_u^\eta(x-k)-V_u^\eta(x)).v_+(y)|\le [V_u^\eta]|k_1|.
\end{align}
By periodicity of $V_u^\eta$ we have
\begin{align}\label{cw62}
(V_u^\eta(x-k)-V_u^\eta(x)).v_+(y)&\stackrel{\eqref{cw14}}{=}V_u^\eta(x).((k+v_+(y))-v_+(y))\nonumber\\
&\stackrel{\eqref{cw10}}{=}V_u^\eta(x).(0,0,k_1,0).
\end{align}
From \eqref{cw62} we learn in particular that $V_u^\eta(x).(0,0,1,0)$ $=(V_u^\eta(y)-V_u^\eta(x)).v_+(y)$ with $y=x-(1,0)$ 
and thus $d(x,y)=1$, so that from the already established part of \eqref{cw61} we obtain
\begin{align*}
|V_u^\eta(x).(0,0,1,0)|\le[V_u^\eta].
\end{align*}
Inserting this into \eqref{cw62} yields \eqref{cw63}, which completes the proof of \eqref{cw61}.

\medskip

We now argue how \eqref{cw61} entails \eqref{cw60}. On the one hand, 
by the functorial definitions \eqref{functDef} and \eqref{S6} (together with
the $\partial_1^2v_0=\partial_1^2v_1=0$) we have
$V_{\partial_1^2u}^{\eta-2}.(v_-)_T$ $=V_u^\eta.(0,v_{\alpha-2},0,v_{2\alpha-2,1})_T$ $=V_u^\eta.\partial_1^2(v_+)_T$.
On the other hand, we have by \eqref{j1003} in Remark \ref{wg03} that $u=V_u^\eta.v_+$ (since our skeleton and model are related by \eqref{j700}) and therefore $\partial_1^2u_T=\partial_1^2(V_u^\eta.v_+)_T$.
Combining both observations and the definition of $(\cdot)_T$, we obtain the representation
\begin{align*}
(\partial_1^2u_T-V_{\partial_1^2u}^{\eta-2}.(v_-)_T)(x)
=\int(V_u^\eta(y)-V_u^\eta(x)).v_+(y)\partial_1^2\psi_T(x-y){\rm d}y.
\end{align*}
Now \eqref{cw60} follows from \eqref{cw61} using \eqref{1.13}.

\medskip

{\sc Step} \arabic{ROprf}.\label{ROprf_2}\refstepcounter{ROprf}[Consistency of multiplication and reduction] 
In this step, we argue that the product $V_{\aufi}^{\eta-2}$ of the reductions of $V_{a}^{\eta-\alpha}=C_{\eta-\alpha}V_{a}^{\eta}$ and $V_{\partial_{1}^{2}u}^{\eta-\alpha-2}$ is the reduction $C_{\eta-2}V_{\aufi}^{\eta+\alpha-2}$ of the product $V_{a \diamond \partial_{1}^{2}u}^{\eta-2}$ of $V_{a}^{\eta}$ and $V_{a}^{\eta}$ to level $\eta-2$.  That is, we claim
\begin{align}
C_{\eta-2}(\overline{V}_a^\eta\otimes V_{\partial_1^2u}^{\eta-2})
=\overline{V}_a^{\eta-\alpha}\otimes V_{\partial_1^2u}^{\eta-\alpha-2}.
\end{align}
Here comes the argument.  By definition of $\overline{V}_{a}^{\eta-\alpha}$ and $V_{\partial_{1}^{2}u}^{\eta-\alpha-2}$, it suffices to show that $C_{\eta-2}(\overline{V}_a^{\eta}\otimes V_{\partial_1^2u}^{\eta-2})$
$=(C_{\eta-\alpha}\overline{V}_a^{\eta})\otimes(C_{\eta-\alpha-2}V_{\partial_1^2u}^{\eta-2})$. Appealing
to the orderings for $\mathsf{A}$, $\mathsf{A}_+$, and $\mathsf{A}_-$, respectively,
\begin{align*}\begin{array}{rcl}
\alpha-2<2\alpha-2&<\eta-2<&\alpha-1<3\alpha-2<2\alpha-1<4\alpha-2,\\
0<\alpha&<\eta-\alpha<&1<2\alpha,\\
\alpha-2&<\eta-\alpha-2<&2\alpha-2,\end{array}
\end{align*}
%
this instance of compatibility of product and reduction can be rephrased in terms of the
single-order reduction operators
\begin{align}\label{cw76}
C_{\alpha-1}C_{3\alpha-2}C_{2\alpha-1}C_{4\alpha-2}(\overline{V}_{+}\otimes V_-)
=(C_{1}C_{2\alpha}\overline{V}_{+})\otimes(C_{2\alpha-2}V_-),
\end{align}
where we set for abbreviation $V_+:=V_a^\eta$ and $V_-:=V_{\partial_1^2u}^{\eta-2}$.

\medskip

We start with the l.~h.~s. of \eqref{cw76}.
In view of $C_{\alpha-1}C_{3\alpha-2}C_{2\alpha-1}C_{4\alpha-2}V.\mathsf{v}$
$=V.({\rm id}-\Gamma_{4\alpha-2})$ $({\rm id}-\Gamma_{2\alpha-1})$
$({\rm id}-\Gamma_{3\alpha-2})$ $({\rm id}-\Gamma_{\alpha-1})\mathsf{v}$, which we obtain from iterating
\eqref{cut2}, our first step
towards \eqref{cw76} is the following string of identities. From the definition $\Gamma=\Gamma_+\otimes\Gamma_-$,
cf.~Lemma \ref{lem:prod_skel},
we obtain for the four factors
\begin{align*} 
{\rm id}-\Gamma_{\alpha-1}&=\id_{+}\otimes\id_{-}-\Ga{1}{}\otimes\Ga{\alpha-2}{} \\
&=({\rm id}_{+}-\Gamma_1)\otimes \Ga{\alpha-2}{} + \id_{+}\otimes (\id_{-}-\Ga{\alpha-2}{}),\\
{\rm id}-\Gamma_{3\alpha-2}&=\id_{+}\otimes\id_{-}-\Ga{\alpha}{}\otimes\Ga{2\alpha-2}{}-\Ga{2\alpha}{}\otimes\Ga{\alpha-2}{} \\
&=({\rm id}_{+}-\Gamma_{2\alpha}) \otimes \Ga{\alpha-2}{} - \Gamma_\alpha\otimes\Gamma_{2\alpha-2} + \id_{+}\otimes (\id_{-}-\Ga{\alpha-2}{}),\\
{\rm id}-\Gamma_{2\alpha-1}&
={\rm id}_{+}\otimes{\rm id}_{-}-\Gamma_{1}\otimes\Gamma_{2\alpha-2}\\
&=(\id_{+}-\Ga{1}{})\otimes\Ga{2\alpha-2}{} + \id_{+}\otimes (\id_{-}-\Ga{2\alpha-2}{}),\\
{\rm id}-\Gamma_{4\alpha-2}&={\rm id}_{+}\otimes{\rm id}_{-}-\Gamma_{2\alpha}\otimes\Gamma_{2\alpha-2} \\
&=(\id_{+}-\Ga{2\alpha}{})\otimes\Ga{2\alpha-2}{} + \id_{+}\otimes (\id_{-}-\Ga{2\alpha-2}{}),
\end{align*}
where ${\rm id}_{\pm}$ denotes the identity on $\mathsf{T}_\pm$ as opposed to the identity ${\rm id}$
on $\mathsf{T}$.
Using the triangular structure of $\Gamma_\pm$, cf.~\eqref{cw12} and \eqref{cw12a}, we obtain successively
\begin{align*}
({\rm id}-\Gamma_{3\alpha-2})({\rm id}-\Gamma_{\alpha-1})
&=({\rm id}_{+}-\Gamma_{2\alpha})({\rm id}_{+}-\Gamma_{1})
\otimes \Ga{\alpha-2}{} - \Gamma_\alpha\otimes\Gamma_{2\alpha-2} \\
&\quad + \id_{+}\otimes(\id_{-}-\Ga{\alpha-2}{}),
\end{align*}
and
\begin{align*}
\lefteqn{({\rm id}-\Gamma_{4\alpha-2})({\rm id}-\Gamma_{2\alpha-1})}\nonumber\\
&=({\rm id}_{+}-\Gamma_{2\alpha})({\rm id}_{+}-\Gamma_{1})\otimes \Ga{2\alpha-2}{} + \id_{+}\otimes (\id_{-}-\Ga{2\alpha-2}{}).
\end{align*}
%
Note that $(\id_{-}-\Ga{2\alpha-2}{})\Ga{\alpha-2}{}=(\id_{-}-\Ga{2\alpha-2}{})$ and on the same token $(\id_{-}-\Ga{2\alpha-2}{})(\id_{-}-\Ga{\alpha-2}{})=0$. Moreover, $({\rm id}_{+}-\Gamma_{2\alpha})({\rm id}_{+}-\Gamma_{1})$ is idempotent, so that
\begin{align}\label{cw77}
\lefteqn{({\rm id}-\Gamma_{4\alpha-2})({\rm id}-\Gamma_{2\alpha-1})
({\rm id}-\Gamma_{3\alpha-2})({\rm id}-\Gamma_{\alpha-1})}\nonumber\\
&=({\rm id}_{+}-\Gamma_{2\alpha})({\rm id}_{+}-\Gamma_{1})\otimes \Ga{2\alpha-2}{}\Ga{\alpha-2}{}\nonumber\\
&\quad -({\rm id}_+-\Gamma_{2\alpha})({\rm id}_+-\Gamma_1)\Gamma_\alpha\otimes\Gamma_{2\alpha-2}\nonumber\\
&\quad +({\rm id}_{+}-\Gamma_{2\alpha})({\rm id}_{+}-\Gamma_{1})\otimes \Ga{2\alpha-2}{}(\id_{-}-\Ga{\alpha-2}{})\nonumber\\
&\quad +({\rm id}_{+}-\Gamma_{2\alpha})({\rm id}_{+}-\Gamma_{1})\otimes (\id_{-}-\Ga{2\alpha-2}{})\nonumber\\
&=({\rm id}_{+}-\Gamma_{2\alpha})({\rm id}_{+}-\Gamma_{1})({\rm id}_{+}-\Gamma_{\alpha})\otimes\Gamma_{2\alpha-2}\nonumber\\
&\quad +({\rm id}_+-\Gamma_{2\alpha})({\rm id}_+-\Gamma_1)\otimes({\rm id}_--\Gamma_{2\alpha-2}).
\end{align}

\medskip

We now turn to the r.~h.~s. of \eqref{cw76}. 
In view of \eqref{j509} it is enough to check \eqref{cw76} for a fixed point $x$ (which we suppress in our notation)
on $\mathsf{v}_+\otimes\mathsf{v}_-$
$\in\mathsf{T}=\mathsf{T}_+\otimes\mathsf{T}_-$. We have from iterating \eqref{cut2}
\begin{align*}
\lefteqn{(C_{1}C_{2\alpha}\overline{V}_+)\otimes(C_{2\alpha-2}V_-).(\mathsf{v}_+\otimes\mathsf{v}_-)}\nonumber\\
&=\big(\overline{V}_+.({\rm id}_+-\Gamma_{2\alpha})({\rm id}_+-\Gamma_{1})\mathsf{v}_+\big)
\big(V_-.({\rm id}_--\Gamma_{2\alpha-2})\mathsf{v}_-\big).
\end{align*}
From \eqref{cw77} we obtain for the l.~h.~s. of \eqref{cw76}
\begin{align*}
\lefteqn{C_{\alpha-1}C_{3\alpha-2}C_{2\alpha-1}C_{4\alpha-2}(\overline{V_+}\otimes V_-).
(\mathsf{v}_+\otimes\mathsf{v}_-)}\nonumber\\
&=\big(\overline{V}_+.({\rm id}_{+}-\Gamma_{2\alpha})({\rm id}_{+}-\Gamma_{1})({\rm id}_{+}-\Gamma_{\alpha})\mathsf{v}_+\big)
\big(V_-.\Gamma_{2\alpha-2}\mathsf{v}_-\big)\nonumber\\
&+\big(\overline{V}_+.({\rm id}_+-\Gamma_{2\alpha})({\rm id}_+-\Gamma_{1})\mathsf{v}_+\big)
\big(V_-.({\rm id}_--\Gamma_{2\alpha-2})\mathsf{v}_-\big).
\end{align*}
Hence \eqref{cw76} follows once we convince ourselves of
$0$ 
$=\overline{V}_{+}.$ $({\rm id}_{+}-\Gamma_{2\alpha})$ $({\rm id}_{+}-\Gamma_{1})$ 
$({\rm id}_{+}-\Gamma_{\alpha})\mathsf{v}_+$
$=C_\alpha C_1 C_{2\alpha}\overline{V}_{+}.\mathsf{v}_+$, which follows from $\overline{V_+}$
$=V_+-C_\alpha C_1 C_{2\alpha}V_+$ and the idempotence of $C_\alpha C_1 C_{2\alpha}$.

\medskip

{\sc Step} \arabic{ROprf}.\label{ROprf_3}\refstepcounter{ROprf}[Application of Corollary \ref{lem:ExMult}] 
In the final step, we argue for the existence of $a \diamond \partial_{1}^{2}u$, together with the estimates \eqref{eq:QualRO} and \eqref{s30}.  With Steps \ref{ROprf_1} and \ref{ROprf_2} now complete, this is essentially a consequence of Corollary \ref{lem:ExMult}, with an appropriate translation of the hypotheses, which we now detail.  For positive homogeneity objects, we use the abstract spaces $(\mathsf{A}_{+},\mathsf{T}_{+})$ specified by \eqref{S1} and \eqref{S8030}, along with the function $v_{+}$ defined by \eqref{j532}.  For negative homogeneity objects, we use the abstract spaces $(\mathsf{A}_{-},\mathsf{T}_{-})$ specified by \eqref{S4} and \eqref{S8001}, along with the distribution $v_{-}$ defined by \eqref{S6}.   The role of $V_{u_{+}}$ is played by $V_{a}^{\eta}$, which belongs to $\mathcal{D}^{\eta}(\Ta{+};\Ga{+}{})$ by assumption.  The role of $V_{u_{-}}$ is played by $V_{\partial_{1}^{2}u}^{\eta-2}$, which belongs to $\mathcal{D}^{\eta-2}(\Ta{-};\Ga{-}{})$ as a consequence of our assumption $V_{u}^{\eta}(\Ta{+};\Ga{+}{}) \in \mathcal{D}^{\eta}(\Ta{+};\Ga{+}{})$.  In fact, one has the inequality
\begin{equation}
[V_{\partial_{1}^{2}u}^{\eta-2}] \leq [V_{u}^{\eta}],\label{fnctIneq}
\end{equation} 
which can be established using the definition \eqref{functDef} together with the functorial derivation (upon application of $\partial_{1}^{2}$) of $\Aa{-},\Ta{-},\Ga{-}{},\langle \cdot \rangle_{-}$ cf. \eqref{S4}, \eqref{S8001}, \eqref{cw12}, \eqref{wg90} from $\Aa{+},\Ta{+},\Ga{+}{},\langle \cdot \rangle_{+}$ cf. \eqref{S1}, \eqref{S8030}, \eqref{cw12}, \eqref{wg12}.  Noting that $\alpha$ is the smallest positive homogeneity in $\Aa{+}$, while $\alpha-2$ is the most negative homogeneity in $\Aa{-}$, the exponent $\kappa$ defined by \eqref{j982} is equal to $\eta+\alpha-2$.  We note that by \eqref{j1220}, $\kappa$ separates the homogeneities of $\Aa{}$ as follows
\begin{align*}
 \alpha-2<2\alpha-2<\alpha-1<3\alpha-1<0<\kappa<2\alpha-1<4\alpha-2.
\end{align*}
The model input $v_{+} \diamond v_{-}$ is played by \eqref{j531}, and we note that \eqref{j770} is a consequence of \eqref{lo01}.  We also note that $u_-=\puf$ satisfies \eqref{j771} with $(\tf_-)_{\betam}\stackrel{\eqref{S6},\eqref{j531}}{=}(\tf_+\diam \tf_+)_{0,\betam}$, provided we set $N_*:=[V_u^\thresh]$ in the definition of $\oTa{\kappa-\alpha}$. Indeed, in view of $\kappa-\alpha=\thresh-2$, which by \eqref{j1220} satisfies \eqref{jj3}, the estimate \eqref{j771} corresponds to \eqref{cw60} established in Step \ref{ROprf_1}.  

\medskip

Hence, we may appeal to Corollary \ref{lem:ExMult} to obtain a unique distribution $a \diamond \partial_{1}^{2}u$ satisfying \eqref{j772}, which we now claim translates exactly into \eqref{eq:QualRO}.  We only need to argue that the constant $M_{\aufi,\eta+\alpha-2}$ in \eqref{S600} corresponds to $M_{u_{+}\diamond u_{-},\kappa}$ of Corollary \ref{lem:ExMult}.  First we note that $\|\oGa{}{}\|_{sk}\lesssim 1$ as a consequence of
\begin{align*}
 \|\oGa{}{}\|_{sk}\stackrel{\eqref{j1004}}{\le} \|\Ga{}{}\|_{sk} + \frac{[V_{\puf}^{\thresh-2}]}{[V_u^\thresh]}\stackrel{\eqref{j790},\eqref{fnctIneq}}{\lesssim} 1.
\end{align*}
Next we note that $\min_{\beta \in \Aa{-}} \langle \beta\rangle= \min_{\beta \in \Aa{+}'}\langle \beta\rangle =1$, cf. \eqref{wg90} and \eqref{wg13}.  Thus, $M_{u_{+} \diamond u_{-}}^{c}$ of \eqref{S301} corresponds to $M_{a,\eta}$ defined by \eqref{SFin35}.  Also, it follows that $M_{u_{+} \diamond u_{-}}^{b}$ corresponds to $\|V_{\partial_{1}^{2}u}^{\eta-2}\| \|V_{a}^{\eta}\|$, which is multiplied by $N^{4}$ in \eqref{S600} because $\min_{\beta\ge \eta+\alpha-2}\ap{\beta}=\min\set{\ap{2\alpha-1}, \ap{4\alpha-2}}$ $=4$, cf.\@ \eqref{wg13}.    From these observations, we see that the constants $M_{\aufi,\eta+\alpha-2}$ and $M_{u_{+}\diamond u_{-},\kappa}$ coincide.

\medskip

The last step is to obtain the sub-optimal bound \eqref{s30}, which will follow from \eqref{S870} with $\eta$ in Corollary \ref{lem:ExMult} playing the role $\eta-2$ in the current lemma.  There are two points to justify here. The first is a consistency condition, which guarantees that the left hand side of \eqref{S870} turns into the left hand side \eqref{s30} upon cutting at level $\eta-2$.  Indeed, this follows from Step \ref{ROprf_2}, cf.\@ \eqref{MS1}.
The second point is to verify that the constant denoted $M_{a \diamond \partial_{1}^{2}u,\eta}$ in Corollary \ref{lem:ExMult} translates to $M_{a \diamond \partial_{1}^{2}u, \eta - 2}$ in the current lemma.
For this, it suffices to note that $\min_{\beta\ge \eta-2}\ap{\beta}=\min\set{\ap{\alpha-1},\ap{3\alpha-2}, \ap{2\alpha-1}, \ap{4\alpha-2}}=3$, cf.\@ \eqref{wg13}.
\end{proof}
\subsection{Lemma \ref{int_lem_concr}}
\begin{proof}[Proof of Lemma \ref{int_lem_concr}]
The proof consists of several steps.  In Step \ref{Intprf_000}, we record in advance a useful auxiliary result, which will be crucial in Step \ref{Intprf_00}.  In Step \ref{Intprf_0}, we introduce the language necessary to apply our abstract integration result, Proposition \ref{int lem}. We verify the local splitting condition \eqref{KS1} in Steps \ref{Intprf_00} and \ref{Intprf_i}, followed by the three-point continuity condition \eqref{KS2} in Step \ref{Intprf_ii}. The output of Proposition \ref{int lem} asserts a modelledness stated in Step \ref{Intprf_iii}, which is used in Step \ref{Intprf_iv} in conjunction with Lemma \ref{cor:prod1} to establish the form continuity \eqref{s16} for $V_{u}^{\thresh}$ in the $\mathcal{D}^{\thresh}(\Tplust; \Ga{+}{})$ norm, as well as \eqref{SFin3}. The form boundedness \eqref{SFin5} is obtained in Step \ref{Intprf_v}.

\medskip

\newcounter{Intprf} 
\refstepcounter{Intprf} 
{\sc Step} \arabic{Intprf}.\label{Intprf_000}\refstepcounter{Intprf}[Structural property] 
We claim that for all $x \in \R^{2}$ and $(\pva,\va \diamond \pva) \in C^{1}(I) \oplus C^{2,1}(I^{2})$ satisfying
\begin{equation}
\delta_{a(x)}.\pva=\delta_{a_0'} \otimes \delta_{a(x)}.\va \diamond \pva=0 \quad \text{for all } \, a_{0}' \in I, \label{SFin20}
\end{equation}
it holds
\begin{equation}
\big ( \overline{V}_{a}^{\eta-\alpha} \otimes V_{\partial_{1}^{2}u}^{\eta-\alpha-2}\big )(x) .
\begin{pmatrix}
\pva \\
\va \diamond \pva
\end{pmatrix}
=0.\label{SFin21}
\end{equation}
Towards this end, we claim that the following boundedness property holds.
\begin{align}
\Big | & \big ( \overline{V}_{a}^{\eta-\alpha}  \otimes V_{\partial_{1}^{2}u}^{\eta-\alpha-2} \big )(x).
\begin{pmatrix}
\pva \\
\va \diamond \pva
\end{pmatrix}
\Big |
\nonumber\\
& \leq N^{\langle \alpha\rangle}\|\overline{V}_{a}^{\eta-\alpha}\| \|\big ( \id_{\Ta{\alpha}} \otimes \delta_{a(x)} \big ).(\va \diamond \pva-\vf(x) \otimes \pva)\|_{\Ta{\alpha}}.\label{bdProp}
\end{align}
To see this, note first that trivially, since $V_{\partial_{1}^{2}u}^{\eta-\alpha-2}=\delta_{a}$, we have for all $\pva \in \Ta{\alpha-2}$
\begin{equation*}
\left |V_{\partial_{1}^{2}u}^{\eta-\alpha-2}(x).\pva \right | \leq \left |\delta_{a(x)}.\pva \right |.
\end{equation*}
Moreover, the inequality \eqref{new} of Lemma \ref{cutting_lemma} implies that for all $(\na, \va) \in \Ta{0} \otimes \Ta{\alpha}$ it holds
\begin{align*}
\left |\overline{V}_{a}^{\eta-\alpha}
\begin{pmatrix}
\na \\
\va
\end{pmatrix}
\right |
\leq \|\overline{V}_{a}^{\eta-\alpha} \|N^{\langle \alpha \rangle}\|\va-\vf(x)\na\|_{\Ta{\alpha}}.
\end{align*}
Thus, \eqref{bdProp} now follows from the cross norm property \eqref{j522} of the spaces $\Ta{\alpha}$ and $\R$.  The vanishing property \eqref{SFin21} now follows from \eqref{bdProp} by noting that for $(\pva,\vva)$ with property \eqref{SFin20}, the r.h.s. of \eqref{bdProp} vanishes.

\medskip

{\sc Step} \arabic{Intprf}.\label{Intprf_0}\refstepcounter{Intprf}[Set up] 

Our plan is to apply Proposition \ref{int lem} with $\eta$ playing the role of $\kappa$, and $U,\gamma \in C^{0}(\R^{2}\times \R^{2})$ given by
\begin{align}
U(x,y)&:=u(y)-
V_{\partial_{1}^{2}u}^{\thresh-2}(x). 
\begin{pmatrix}
\vf(y)\\
\wwf(y)
\end{pmatrix},
\label{I1}\\
\gamma(x,y)&:=
 \big ( V_{\partial_{1}^{2}u}^{\thresh-2}(y)-V_{\partial_{1}^{2}u}^{\thresh-2}(x) \big ).
\begin{pmatrix}
0\\
\of(y)
\end{pmatrix}. \label{j800}
\end{align}

{\sc Step} \arabic{Intprf}.\label{Intprf_00}\refstepcounter{Intprf}[Preparation for Local Splitting]

In order to motivate the definition \eqref{I1} of $U$ and in preparation for Step \ref{Intprf_i}, where we verify the local splitting condition \eqref{KS1} for this choice of $U$, we claim the following identity
\begin{align}\label{s25}
&(\partial_{2}-a(x)\partial_{1}^{2})U_{T}(x,\cdot)\nonumber\\
&\quad =(\cdot)_{T}P(a \diam \partial_{1}^{2}u)-V_{\aufi}^{\thresh-2}(x).
\begin{pmatrix}
(\puf)_{T}\\
(\pvf)_{T}\\
P(\vvf)_{T}
\end{pmatrix}
,
\end{align}
where we recall that $\partial_1$ and $\partial_2$ act on the implicit variable $y$.  Here comes the argument: in light of the definition \eqref{algebraicRel}, applied with the abstract vectors
$
\partial_{1}^{2}\va:=\vf(y), \, \,
\partial_{1}^{2}\mathsf{w}_{2\alpha}:=\wwf(y),
$
we may write $U$ explicitly as
\begin{align*}
U(x,y)=u(y)- \delta_{a(x)}.\vf(y)
-\overline{V}_{a}^{\thresh-\alpha}(x)\otimes V_{\puf}^{\thresh-\alpha-2}(x). 
\begin{pmatrix}
\partial_{a_0}\vf(y)\\
\wwf(y)
\end{pmatrix}.
\end{align*}
To establish \eqref{s25}, we apply the operator $(\partial_{2}-a(x)\partial_{1}^{2})(\cdot)_{T}$ to each of the three contributions of $y \to U(x,y)$ separately, then combine the terms.  For the first contribution, we observe that applying $(\cdot)_{T}$ to the PDE \eqref{ss11} gives
\begin{equation}
(\partial_{2}-a(x)\partial_{1}^{2})(\cdot)_{T}u=(\cdot)_{T}P(a \diamond \partial_{1}^{2}u)-a(x)(\puf)_{T}+Pf_{T}.\label{eq:IOR6}
\end{equation}
For the second contribution, we observe that
\begin{equation}
(\partial_{2}-a(x)\partial_{1}^{2})(\cdot)_T\delta_{a(x)}
\vf(\cdot)
=Pf_{T},\label{eq:IOR7}
\end{equation}
which is a consequence of the PDE \eqref{f21} defining $\vf$.
For the third contribution, we observe that
\begin{align}
&(\partial_{2}-a(x)\partial_{1}^{2})(\cdot)_T \overline{V}_{a}^{\thresh-\alpha}(x)\otimes V_{\puf}^{\thresh-\alpha-2}(x).
\begin{pmatrix}
\partial_{a_0}\vf \\
\wwf
\end{pmatrix}\nonumber\\
&\quad \, =\overline{V}_{a}^{\thresh-\alpha}(x)\otimes V_{\puf}^{\thresh-\alpha-2}(x).
\begin{pmatrix}
(\pvf)_{T} \\
P(\vvf)_{T}
\end{pmatrix}.\label{s27}
\end{align}
Indeed, \eqref{s27} follows from the defining PDEs \eqref{se40} and \eqref{f22} together with the structural property \eqref{SFin21} established in Step \ref{Intprf_000}.  More specifically, we define for each $y$ the abstract elements of $\Ta{\alpha-2}$ and $\Ta{2\alpha-2}$, respectively,
\begin{align*}
\pva(a_0)&:=(a_0-a(x))(\partial_{a_0}\pvf)_{T}(y,a_0),\\
\vva(a_0',a_0)&:=(a_0-a(x))(\pwf)_{T}(y,a_0',a_0),
\end{align*}
and note that \eqref{SFin20} certainly holds, which allows us to replace $\partial_{2}-a(x)\partial_{1}^{2}$ by $\partial_{2}-a_0\partial_{1}^{2}$ when applied to $((\partial_{a_0}\vf)_{T}, (\wwf)_{T})$ in the argument of $\overline{V}_{a}^{\thresh-\alpha}(x)\otimes V_{\puf}^{\thresh-\alpha-2}(x)$.  It remains to combine the three identities \eqref{eq:IOR6}, \eqref{eq:IOR7}, and \eqref{s27} to obtain \eqref{s25}.

\medskip

{\sc Step} \arabic{Intprf}.\label{Intprf_i}\refstepcounter{Intprf}[Local Splitting] 

In this step we verify the local splitting condition \eqref{KS1} for $U$ defined in \eqref{I1}.  More precisely, we claim that for all $x \in \R^{2}$ and all $L,T\le 1$ it holds
\begin{align}
&\inf_{c\in\R}\|(\partial_{2}-a(x)\partial_{1}^{2})U_{T}(x,\cdot)-c\|_{B_{L}(x)} \nonumber \\
& \lesssim M_{\aufi,\eta-2}\Tc{\thresh-2}  +[a]_{\alpha} [V_{u}^{\thresh-\alpha}]L^{\alpha}\Tc{\thresh-\alpha-2}\nonumber\\
& \quad + M_{a,\eta-\alpha}\big (L^{\thresh-\alpha}\Tc{\alpha-2}+ L^{\thresh-2\alpha}\Tc{2\alpha-2} \big), \label{s24}
\end{align}
where we recall the definition of $M_{a,\eta-\alpha}$ in \eqref{IOR30}.

\medskip

For this purpose, we re-write identity \eqref{s25} obtained in Step \ref{Intprf_00} as
\begin{align}
&(\partial_{2}-a(x)\partial_{1}^{2})U_{T}(x,y)-V_{\aufi}^{\thresh-2}(x).
\begin{pmatrix}
 0 \\ 0 \\ c_0
\end{pmatrix}
+c_1\nonumber \\
&\quad =(a \diam \puf)_{T}(y)
-V_{\aufi}^{\thresh-2}(y).
\begin{pmatrix}
(\puf)_{T}(y)\\
(\pvf)_{T}(y)\\
(\vvf)_{T}(y)
\end{pmatrix}
\nonumber \\
&\quad  +\big ( V_{\aufi}^{\thresh-2}(y)-V_{\aufi}^{\thresh-2}(x) \big ).
\begin{pmatrix}
(\puf)_{T}(y)\\
(\pvf)_{T}(y)\\
(\vvf)_{T}(y)
\end{pmatrix},
\nonumber
\end{align}
where $c_0(a_{0}',a_{0}):=\int_{[0,1)^2}(\vvf)(y,a_{0}',a_{0}) \dd y$ and $c_1:=\int_{[0,1)^2}(\auf)(y) \dd y$.
By assumption \eqref{s30}
, the first right-hand side term is estimated by $M_{\aufi,\eta-2}\Tc{\eta-2}$.  To estimate the second right-hand side term, we appeal to the form continuity \eqref{IOR1} of $V_{a \diamond \partial_{1}^{2}u}^{\eta-2}$ contained in \eqref{IOR1} in Step \ref{Algprf_2} of the proof of Lemma \ref{cor:prod1} (with the roles of $x$ and $y$ exchanged) and find the following upper bound for all $y \in B_{L}(x)$  
\begin{align}
&[a]_{\alpha}L^{\alpha}\big |(\puf)_{T}(y)-V_{\puf}^{\thresh-\alpha-2}(y).(\pvf)_{T}(y) \big |\label{IOR8}\\
&+ M_{a,\eta-\alpha} L^{\thresh-\alpha}\|(\pvf)_{T}(y)\|_{\frac{1}{N} C^1} \nonumber \\ 
&+ M_{a,\eta-\alpha}L^{\thresh-2\alpha}\left\|
(\vvf)_{T}(y) - \vf(y)\otimes (\pvf)_T(y)
\right\|_{\frac{1}{N^2}C^{2,1}}. \nonumber 
\end{align}
Recalling that $\Ta{\alpha-2}\hookrightarrow \frac{1}{N}C^1(\interval{})$ and $\Ta{2\alpha-2,2}=\frac{1}{N^2}C^{2,1}(\interval{2})$ cf. \eqref{S8001}, we can appeal to \eqref{p13} and \eqref{off1} to estimate the the last two lines by 
\begin{align*}
M_{a,\eta-\alpha} \big (L^{\thresh-\alpha}\Tc{\alpha-2}+ L^{\thresh-2\alpha}\Tc{2\alpha-2} \big).
\end{align*}
It remains to estimate the first line \eqref{IOR8}, which is a lower-order version of \eqref{cw60}. By even simpler arguments than in \eqref{cw60}, this line is hence estimated by
$$
[a]_{\alpha}[V_{u}^{\thresh-\alpha}]L^{\alpha}(T^{\frac{1}{4}})^{\thresh-\alpha-2}.
$$

\medskip

{\sc Step} \arabic{Intprf}.\label{Intprf_ii}\refstepcounter{Intprf}[Three-Point Continuity] 

The next step is to verify the three-point continuity condition \eqref{KS2} in our setting of \eqref{I1} and \eqref{j800}.  More precisely, we claim that for any three points $x,y,z$ it holds
\begin{align}
&\big |U(x,z)-U(x,y)-U(y,z)+U(y,y)
-\gamma(x,y) (z-y)_1
\big |\nonumber \\
&\quad \lesssim [V_{\partial_{1}^{2}u}^{\eta-2}]\big (d^{\thresh-\alpha}(x,y)d^{\alpha}(z,y)+ d^{\thresh-2\alpha}(x,y)d^{2\alpha}(z,y) \big)\label{s34}.
\end{align}
First observe that by definitions \eqref{I1} and \eqref{j800} of $U$ and $\gamma$ it holds
\begin{align*}
&U(x,z)-U(x,y)-U(y,z)+U(y,y)
-\gamma(x,y)(z-y)_1 
\\
&=\big ( V_{\partial_{1}^{2}u}^{\thresh-2}(y)-V_{\partial_{1}^{2}u}^{\thresh-2}(x) \big ).
\begin{pmatrix}
\vf(z)-\vf(y)\\
\wwf(z)-\wwf(y)
-\of(y)(\xf(z)-\xf(y))
\end{pmatrix}
.
\end{align*}
We now appeal to the form continuity of $V_{\partial_{1}^{2}u}^{\eta-2}$. Using the following choices of abstract vectors
\begin{align*}
\partial_{1}^{2}\va & := \vf(z)-\vf(y), \\
\partial_{1}^{2}\mathsf{w}_{2\alpha} & :=\wwf(z)-\wwf(y)
-\of(y)(\xf(z)-\xf(y)),
\end{align*}
and recalling the definition of $\Ga{-}{}$ in \eqref{cw12a}, the quantity above is estimated by
\begin{align*}
&[V_{\partial_{1}^{2}u}^{\eta-2}]d^{\eta-\alpha}(y,x)\|\vf(z)-\vf(y)\|_{\Ta{\alpha-2}}\\
&+ [V_{\partial_{1}^{2}u}^{\eta-2}]d^{\eta-2\alpha}(y,x)\|\wwf(z)-\wwf(y)
-\of(y)(\xf(z)-\xf(y))\\
&\qquad \qquad \qquad \qquad-\vf(y) \otimes (\partial_{a_0}\vf(z)-\partial_{a_0}\vf(y)) \|_{\Ta{2\alpha-2}},
\end{align*}
which implies \eqref{s34} as a consequence of \eqref{s203} and \eqref{s2011}.

\medskip

{\sc Step} \arabic{Intprf}.\label{Intprf_iii}\refstepcounter{Intprf}[Output of Proposition \ref{int lem} and Remark \ref{rem-wg53}] 

We claim that there exists a H\"older continuous and periodic function $\nu$ such that for all $x,y \in \R^{2}$
\begin{align}
&\left |U(x,y)-U(x,x)-\nu(x)(\xf(y)-\xf(x))\right |\lesssim M d^{\thresh}(y,x),\label{s20}\\
&\left |\nu(y)-\nu(x)
+\gamma(x,y)
\right | \lesssim M d^{\thresh-1}(y,x)\label{s21},
\end{align}
where
\begin{align}
M&:=M_{\aufi,\thresh-2}+[a]_{\alpha}[V_{u}^{\thresh-\alpha}]+ M_{a,\eta-\alpha}+[V_{\partial_{1}^{2}u}^{\eta-2}].
\label{j786}
\end{align}
Indeed, the estimates \eqref{s20} and \eqref{s21} correspond to the outputs \eqref{KS3} and \eqref{KS4}, since the sum of the constants gathered in \eqref{s24} and \eqref{s34} is controlled by $M$.


\medskip

{\sc Step} \arabic{Intprf}.\label{Intprf_iv}\refstepcounter{Intprf}[Form Continuity] 

Using the function $\nu$ obtained in Step \ref{Intprf_iii}, the definition \eqref{s300bis} of $V_{u}^{\eta}$ is now complete, and we may turn to the proof of \eqref{s16}, which in view of definition \eqref{j786} amounts to showing that for all $x,y \in \R^{2}$ with $d(x,y) \leq 1$ and all $\mathsf{v}_{+} \in \Ta{+}$
\begin{align} \label{S91}
\big | \left (V_{u}^{\thresh}(y)-V_{u}^{\thresh}(x)\right ).\ta_{+}\big| &\lesssim
M \sum_{\beta\in\Aa{+}} d^{\eta-\beta}(y,x)\|\Ga{\beta}{}(y)\ta_{+}\|_{\Ta{\beta}}.
\end{align}

For this we first note that for all $\mathsf{v}_{+}=(\na,\va,\mathsf{v}_{1},\wa) \in \Ta{+}$
\begin{align}
\big(V_{u}^{\thresh}&(y)-V_{u}^{\thresh}(x)\big).\mathsf{v}_{+}\nonumber\\
&=\big (U(x,y)-U(x,x)-\nu(x)(v_{1}(y)-v_{1}(x)) \big)\na \nonumber\\
&\, +\left (\nu(y)-\nu(x)+\gamma(x,y) \right) \left (\mathsf{v}_{1}-v_{1}(y)\na \right)\nonumber\\
&\, +\left (V_{\partial_{1}^{2}u}^{\eta-2}(y)-V_{\partial_{1}^{2}u}^{\eta-2}(x) \right).\nonumber\\
& \quad
\begin{pmatrix}
\va-\vf(y)\na\\
\wa-\wwf(y)\na-\omega(y)(\mathsf{v}_{1}-v_{1}(y)\na)
\end{pmatrix}
.\label{jj6}
\end{align}
Indeed, note that
\begin{align}
V_{u}^{\thresh}(x).\mathsf{v}_{+}
=u(x)\na&+\nu(x)\left (\mathsf{v}_{1}-v_{1}(x)\na \right)\nonumber\\
&+V_{\partial_{1}^{2}u}^{\eta-2}(x).
\begin{pmatrix}
\va-\vf(x)\na\\
\wa-\wwf(x)\na
\end{pmatrix}
,\label{jj8}
\end{align}
which follows from comparing the definitions \eqref{algebraicRel} and \eqref{s300bis}.  Applying the discrete product rule and noting that
\begin{equation*}
\gamma(x,y)\left (\mathsf{v}_{1}-v_{1}(y)\na \right)+\left (V_{\partial_{1}^{2}u}^{\eta-2}(y)-V_{\partial_{1}^{2}u}^{\eta-2}(x) \right).
\begin{pmatrix}
0\\
-\omega(y) \left (\mathsf{v}_{1}-v_{1}(y)\na \right)
\end{pmatrix}
\end{equation*}
vanishes by \eqref{j800}, gives \eqref{jj6}. By \eqref{s20} and \eqref{s21} in Step \ref{Intprf_iii}, the first two contributions to \eqref{jj6} are estimated by
\begin{equation*}
M \big (d^{\eta}(y,x)|\na|+ d^{\eta-1}(y,x)|\mathsf{v}_{1}-v_{1}(y)\na |\big ),
\end{equation*}
which falls under the right-hand side of \eqref{S91} in view of \eqref{S8030} and \eqref{cw12}.  The last contribution to \eqref{jj6} is estimated as in Step \ref{Intprf_ii}, using the form continuity of $V_{\partial_{1}^{2}u}^{\eta-2}$ with the abstract vectors
\begin{align*}
\partial_{1}^{2}\va:=\va-\vf(y)\na, \, \,
\partial_{1}^{2}\mathsf{w}_{2\alpha}:=\wa-\wwf(y)\na-\omega(y)(\mathsf{v}_{1}-v_{1}(y)\na),
\end{align*}
which yields
\begin{align*}
& [V_{\partial_{1}^{2}u}^{\eta-2}]d^{\eta-\alpha}(y,x)\|\va-\vf(y)\na\|_{\Ta{\alpha-2}}\\
&+ [V_{\partial_{1}^{2}u}^{\eta-2}]d^{\eta-2\alpha}(y,x)\|\wa-\wwf(y)\na
-\omega(y)(\mathsf{v}_{1}-v_{1}(y)\na)\\
&\qquad \qquad \qquad \qquad-\vf(y) \otimes \partial_{a_0}(\va-\vf(y)\na)\|_{\Ta{2\alpha-2}},
\end{align*}
which again falls under the right-hand side of \eqref{S91} in light of \eqref{cw12} and \eqref{j786}.  

\medskip

We give now the argument for \eqref{SFin3}. Since $V_u^\thresh \in \cald^\thresh(\Tplus;\Ga{+}{})$, we obtain by definition \eqref{prod5} of form continuity and the property \eqref{j700}
\begin{align*}
 |(V_u^\thresh(y)-V_u^\thresh(x)).\tf_+(y)|\le [V_u^\thresh] d^{\thresh}(y,x).
\end{align*}
Inserting $\tf_+(y)=(1,\vf(y),\xf(y),\wwf(y))$ into \eqref{jj6}, the last two lines of \eqref{jj6} vanish and we obtain $(V_u^\thresh(y)-V_u^\thresh(x)).\tf_+(y) = U(x,y)-U(x,x)-\nu(x)(y-x)_1$. Choosing $y:=x+(1,0)$ and appealing to the periodicity of $y\to U(x,y)$, we obtain \eqref{SFin3}.

\medskip

{\sc Step} \arabic{Intprf}.\label{Intprf_v}\refstepcounter{Intprf}[Form Boundedness] 
To show \eqref{SFin5},
observe that \eqref{jj8} implies
\begin{align*}
 |V_{u}^{\thresh}(x).\ta_{+}|&\le \|u\||\na| + \|\nu\||\xa-\xf(x)\na| + \left|V_{\puf}^{\thresh-2}(x)\begin{pmatrix} \va-\vf(x)\na \\ \wa-\wwf(x)\na\end{pmatrix} \right|  \\
 &\le \|u\||\Ga{0}{}(x)\ta_+| + N^{-2}\left (\|\nu\|+\left | V_{\puf}^{\eta-2}(x).\begin{pmatrix} 0 \\ \of(x) \end{pmatrix}\right | \right )N^2|\Ga{1}{}(x)\ta_+| \\
 &\quad + \left|V_{\puf}^{\thresh-2}(x)\begin{pmatrix} \va-\vf(x)\na \\ \wa-\wwf(x)\na-\of(x)(\xa-\xf(x)\na)\end{pmatrix} \right|.
\end{align*}
By definition \eqref{wg16} of $\|V_{\partial_{1}^{2}u}^{\eta-2}\|$, the prescription $\langle 2\alpha-2 \rangle=2$, cf. \eqref{wg13}, and the bound $\|\of\|_{C^0(\R^2,\Ta{2\alpha})}\lesssim 1$, cf. \eqref{s2011}, we obtain
\begin{align*}
 &\left | V_{\puf}^{\eta-2}(x).\begin{pmatrix} 0 \\ \of(x) \end{pmatrix}\right |\lesssim \|V_{\puf}^{\eta-2}\| N^2,\nonumber\\
 &\left|V_{\puf}^{\thresh-2}(x).\begin{pmatrix} \va-\vf(x)\na \\ \wa-\wwf(x)\na-\of(x)(\xa-\xf(x)\na)\end{pmatrix} \right| \\
 &\quad \, \, \le \|V_{\puf}^{\eta-2}\|\big(N^{\langle \alpha-2 \rangle } \|\va-\vf(x)\na\|_{\Ta{\alpha-2}} \\
 &\qquad + N^{\langle 2\alpha-2 \rangle} \|\wa-\wwf(x)\na-\of(x)(\xa-\xf(x)\na) \\
 &\qquad \qquad \quad - \vf(x)\otimes \partial_{a_0}(\va-\vf(x)\na)\|_{\Ta{2\alpha-2}} \big)\\
 &\stackrel{\eqref{wg90},\eqref{cw12},\eqref{S8001}}{=} \|V_{\puf}^{\eta-2}\| \big(N^{\langle \alpha \rangle} \|\Ga{\alpha}{}(x)\ta_+\|_{\Ta{\alpha}} + N^{\langle 2 \alpha \rangle} \|\Ga{2\alpha}{}(x)\ta_+\|_{\Ta{2\alpha}} \big).
\end{align*}
Combining the above estimates we find
\begin{align*}
 |V_{u}^{\thresh}(x).\ta_{+}|&\lesssim \|u\||\Ga{0}{}(x)\ta_+| + (N^{-2}\|\nu\|+\|V_{\puf}^{\eta-2}\|)N^{2}|\Ga{1}{}(x)\ta_+| \\
 &\quad + \|V_{\puf}^{\thresh-2}\|\big(N^{\langle \alpha \rangle} \|\Ga{\alpha}{}(x)\ta_+\|_{\Ta{\alpha}} + N^{\langle 2\alpha \rangle} \|\Ga{2\alpha}{}(x)\ta_+\|_{\Ta{2\alpha}} \big),
\end{align*}
which yields \eqref{SFin5} in view of \eqref{prod4} and \eqref{wg12}.
\end{proof}
\subsection{Theorem \ref{Prop:aPriori} and Corollary \ref{cor:wg}}
\begin{proof}[Proof of Theorem \ref{Prop:aPriori}]

\medskip

\newcounter{MTprf} 
\refstepcounter{MTprf} 
The proof involves several steps, with Steps \ref{MTprf_2}-\ref{MTprfNew_2} being devoted to the existence of $V_{u}^{\eta} \in \mathcal{D}^{\eta}(\Ta{+};\Ga{+}{})$ with the desired three properties as well as the estimates \eqref{S306} and \eqref{SFin4}, and the final Step \ref{MTprf_9} being devoted to the uniqueness.  The existence proof proceeds by an approximation argument.  To start, we establish in Step \ref{MTprf_2} that for each $\tau \leq 1$ there exists a periodic mean-free solution $u^{\tau} \in C^{2+\eta-2\alpha}$ to the following regularized PDE
\begin{equation}\label{MT17}
\partial_{2}u^{\tau}-P(a\partial_{1}^{2}u^{\tau})=Pf_{\tau}+PC_{\eta+\alpha-2}\big ( \overline{V}_{a}^{\eta}\otimes V_{\partial_{1}^{2}u}^{\eta-2} \big ).v_{\tau},
\end{equation}
where $V_{\partial_{1}^{2}u}^{\eta-2}$ is given by \eqref{algebraicRel} and the modified right hand side may be viewed as renormalization, cf.\@ \eqref{MT13} for further motivation.  We also introduce a regularized skeleton
\begin{align}
\Gamma_{+}^{\tau}:=\left(\begin{array}{lccc}
{\rm id}_{\R}                                                           &0                            &0       &0\\
-(v_\alpha)_{\tau}                                                       &{\rm id}_{C^2}                     &0       &0\\
-v_1                                                               &0                            &{\rm id}_{\R} &0\\
-(w_{2\alpha})_{\tau}+v_\alpha\otimes\partial_{a_0}(v_\alpha)_{\tau}+\omega^{\tau} v_1&-v_\alpha\otimes\partial_{a_0}&-\omega^{\tau} &{\rm id}_{C^{2,1}}
\end{array}\right),\label{MT14}
\end{align}
cf. \eqref{cw12}.  In the above definition,
\begin{align}
\omega^{\tau}(x)&:=\partial_{y_{1}}(U_{2\alpha})_{\tau}(x,x),\nonumber\\
U_{2\alpha}(x,y)&:=\wwf(y)-\vf(x)\otimes \partial_{a_{0}}\vf(y),\label{MT35}
\end{align}
where $(U_{2\alpha})_{\tau}$ denotes convolution in the $y$ argument only.  Note that these quantities were used in the proof of Lemma \ref{est_wwf_lemma}, cf. \eqref{ad0} to establish \eqref{j781} and \eqref{s2011} and the specific form of $\omega^{\tau}$ is in line with Remark \ref{rem-wg53}.  We emphasize that the regularized skeleton $\Gamma_{+}^{\tau}$ agrees with $\Gamma_{+}$ away from homogeneities $0$ and $1$, in that the functorial relation between $\Ga{+}{\tau}$ and $\Ga{-}{}$ is preserved
\begin{equation}\label{MT45}
\Gamma_{+}^{\tau}\ta_{+}=\Gamma_{+}\ta_{+} \quad \text{for} \quad \ta_{+}=(0,\va,0,\wa) \in \Ta{+}.
\end{equation}
Step \ref{MTprfNew_3} is devoted to establishing uniform estimates on the regularized model $v_{\tau}$.  In Step \ref{MTprf_3}, we slightly modify the definition \eqref{s300bis} (to account for the regularization parameter $\tau$) and define a form $V_{u^{\tau}}^{\eta}$ with $C_{\alpha}V_{u^{\tau}}^{\eta}=u^{\tau}$ which will serve as an approximation to $V_{u}^{\eta}$.  For notational convenience, we will use $[V_{u^{\tau}}^{\eta}]_{\tau},\|V_{u^{\tau}}^{\eta}\|_{\tau}$ to abbreviate the associated form continuity semi-norm and boundedness norms.  We will then establish the necessary qualitative ingredient $V_{u^{\tau}}^{\eta} \in \mathcal{D}^{\eta}(\Ta{+};\Gamma_{+}^{\tau})$.   In Steps \ref{MTprf_4}-\ref{MTprf_6}, we improve upon the qualitative Step \ref{MTprf_3} and obtain uniform estimates on $V_{u^{\tau}}^{\eta} \in \mathcal{D}^{\eta}(\Ta{+};\Gamma_{+}^{\tau})$.  Steps \ref{MTprf_7}-\ref{MTprfNew_2} are concerned with the limit $\tau \to 0$: Step \ref{MTprf_7} establishes the convergence of the skeleton, Step \ref{MTprf_8} convergence of $V_{u^{\tau}}^{\eta}$, and Step \ref{MTprfNew_2} convergence for $a \diamond \partial_{1}^{2}u^{\tau}$.

\medskip

{\sc Step} \arabic{MTprf}.\label{MTprf_2}\refstepcounter{MTprf}[Solvability of the regularized PDE] 
To this purpose, we rewrite \eqref{MT17} as a perturbation
of a constant-coefficient equation, i.~e.
\begin{align*}
(\partial_2-a_0\partial_1^2)u^\tau
&=P\big(f^\tau
+(a-a_0)\partial_1^2u^\tau\big)
\end{align*}
with periodic $f^\tau$ $:=f_\tau+C_{\eta+\alpha-2}(\overline{V}_a^\eta\otimes V_{\partial_1^2u}^{\eta-2})
.v_\tau$.  Note that the periodicity of $C_{\eta+\alpha-2}\big (\overline{V}_a^\eta\otimes V_{\partial_1^2u}^{\eta-2}\big )$ was established in Section \ref{ss:periodicity}.
As can be seen from Fourier series, $\partial_2-a_0\partial_1^2$ is (formally) invertible on the space of
periodic functions that are mean-free (thus the projection $P$); 
by standard Schauder theory we have the \emph{a priori} estimate
$\|u^\tau\|_{C^{2+\eta-2\alpha}}$ $\lesssim [f^\tau+(a-a_0)\partial_1^2u^\tau]_{\eta-2\alpha}$.
Since $[a]_\alpha\ll 1$ and thus $\|a-a_0\|\ll 1$ for suitable $a_0\in I$ together with $\alpha\ge\eta-2\alpha>0$, 
cf.~\eqref{j1220}, we learn from a standard fixed-point argument
that there exists a periodic mean-free solution $u^\tau\in C^{\eta-2\alpha+2}(\mathbb{R}^2)$ 
provided $f^\tau\in C^{\eta-2\alpha}(\mathbb{R}^2)$.
Rewriting $f^\tau$ $=f_\tau+C_{\eta+\alpha-2}V_{a\diamond\partial_1^2u}^{\eta+\alpha-2}.{0\choose v_\tau}$, 
and using that thanks to $\tau>0$ we have in particular $f_\tau\in C^{\eta-2\alpha}(\mathbb{R}^2)$ and 
$v_\tau\in C^{\eta-2\alpha}(\mathbb{R}^2;\mathsf{T})$, we see that this follows from
$C_{\eta+\alpha-2}V_{a\diamond\partial_1^2u}^{\eta+\alpha-2}\in C^{\eta-2\alpha}(\mathbb{R}^2 ;\overline{\mathsf{T}}^*)$.
This in turn is a consequence of the boundedness and
continuity of $C_{\eta+\alpha-2}V_{a\diamond\partial_1^2u}^{\eta+\alpha-2}$
to order $\eta+\alpha-2$, cf.~\eqref{prod5}, together with $\max_{\beta\in\mathsf{A},\beta<0}\beta=3\alpha-2$ 
and the boundedness (on a periodic cell) of the range of $\overline{\Gamma}$ in ${\mathcal L}(\mathsf{T},\mathsf{T})$.

\medskip

{\sc Step} \arabic{MTprf}.\label{MTprfNew_3}\refstepcounter{MTprf}[Uniform bounds on the regularized model] 
We now claim that the crucial boundedness properties \eqref{lo02} and \eqref{lo01} are preserved in the following sense: for any $0<\tau\le 1$ it holds
\begin{align}
\|\Gamma_{\beta}^{\tau}(x)(v_{+})_{\tau}(y)\|_{\Ta{\beta}} &\lesssim d^{\beta}(y,x) \quad\mbox{for all}\;\beta\in\Aa{+},\;x,y\in\mathbb{R}^2,\label{MT40}\\
\|\Gamma_\beta(x)(v_\tau)_T(x)\|_{\mathsf{T}_\beta}&\lesssim(T^\frac{1}{4})^\beta
\quad\mbox{for all}\;\beta\in\mathsf{A}_{<0},\;T\le 1,\;x\in\mathbb{R}^2.\label{MT12}
\end{align}
Moreover, in analogy with \eqref{j781} and \eqref{j790}, we claim that
\begin{equation}
\|\omega^{\tau}\|_{C^{2\alpha-1}(\R^{2};\Ta{2\alpha})} \lesssim 1, \quad \quad \|\Gamma_{+}^{\tau}\|_{sk} \lesssim 1.\label{MT70}
\end{equation}
Here comes the argument for \eqref{MT12}.  By the semi-group property \eqref{1.10}, this is immediate once we show that \eqref{lo01} implies
the same estimate in the extended range of $T\le 2$. The latter can be seen as follows: Writing
$\Gamma_\beta(x) v_{2T}(x)$ $=\int\psi_T(x-y)\Gamma_\beta(x)v_T(y) {\dd}y$ and
$\Gamma_\beta(x)v_T(y)=\Gamma_\beta(y) v_T(y)+(\Gamma_\beta(x)-\Gamma_\beta(y))v_T(y)$,
we obtain from the continuity \eqref{ssss2} of the skeleton that
\begin{align*}
\lefteqn{\|\Gamma_\beta(x) v_{2T}(x)\|_{\mathsf{T}_\beta}}\nonumber\\
&\le
(1 + \|\Gamma\|_{sk})\int|\psi_T(x-y)|\sum_{\gamma\le\beta}
d^{\beta-\gamma}(x,y)\|\Gamma_\gamma(y) v_T(y)\|_{\mathsf{T}_\gamma} {\dd}y,
\end{align*}
and thus, provided $T\le 1$, by \eqref{lo01} and \eqref{1.13} that
$\|\Gamma_\beta(x)v_{2T}(x)\|_{\mathsf{T}_\beta}$ 
$\lesssim(1 + \|\Gamma\|_{sk})(T^\frac{1}{4})^\beta$. In conjunction with
\eqref{j790} this yields the desired upgrade from $T\le 1$ to $T\le 2$.

\medskip

We now argue for \eqref{MT40}, and note that in light of the definition \eqref{MT14}, for $\beta=0,1$ this is immediate and for $\beta=\alpha$ it suffices to note that the H\"{o}lder estimate \eqref{s203} is stable under convolution.  For $\beta=2\alpha$ as well as the first estimate in \eqref{MT70}, we refer the reader to \eqref{ad10} and \eqref{ad11} in Step \ref{Mdprf_5} of the proof of Lemma \ref{est_wwf_lemma}.  Finally, the uniform bounds on the regularized skeleton follow from adapting the arguments in Step \ref{Mdprf_4} of Lemma \ref{est_wwf_lemma}.

\medskip

{\sc Step} \arabic{MTprf}.\label{MTprf_3}\refstepcounter{MTprf}[Verifying the qualitative hypothesis] 

In this step, for each $\tau \leq 1$, we define a form $V_{u^{\tau}}^{\eta}$ such that $C_{\alpha}V_{u^{\tau}}^{\eta}=u^{\tau}$  and verify the qualitative hypothesis $V_{u^{\tau}}^{\eta} \in \mathcal{D}^{\eta}(\Ta{+};\Gamma_{+}^{\tau})$ (required for Lemma \ref{rec_lem_concr}).  In line with Step \ref{Intprf_0} in the proof of Lemma \ref{int_lem_concr}, we define  
\begin{align}
U^{\tau}(x,y)&:=u^{\tau}(y)-
V_{\partial_{1}^{2}u}^{\thresh-2}(x). 
\begin{pmatrix}
(\vf)_{\tau}(y)\\
(\wwf)_{\tau}(y)
\end{pmatrix},
\label{MT19}\\
\gamma^{\tau}(x,y)&:=
 \big ( V_{\partial_{1}^{2}u}^{\thresh-2}(y)-V_{\partial_{1}^{2}u}^{\thresh-2}(x) \big ).
\begin{pmatrix}
0\\
\of^{\tau}(y)
\end{pmatrix},
\label{MT20}
\end{align}
and in line with Remark \ref{rem-wg53}
\begin{equation}
\nu^{\tau}(x):=\partial_{y_{1}}U^{\tau}(x,x)\label{MT61}. 
\end{equation}
We now modify \eqref{s300bis} to define $V_{u^{\tau}}^{\eta}$.  Namely, for $\ta_{+}:=(\na,\va,\xa,\wa)$ we set
\begin{align}
V_{u^{\tau}}^{\thresh}.\ta_{+}
&:=
u^{\tau}\na+\delta_{a}.\big (\va-(\vf)_{\tau}\na \big ) +\nu^{\tau}(\xa-\xf\na)
\nonumber\\
&\quad +\overline{V}_{a}^{\thresh-\alpha} \otimes \delta_{a}. 
\begin{pmatrix}
\partial_{a_0}\va-\partial_{a_0}(\vf)_{\tau}\na \\
\wa-(\wwf)_{\tau}\na
\end{pmatrix}, \label{MT44}\\
V_{u^{\tau}}^{\eta-\alpha}&:=C_{\eta-\alpha}V_{u^{\tau}}^{\eta}.\label{MT43}
\end{align}
For later purpose, we note that the definitions above together with \eqref{algebraicRel} entail the following consistency properties,
\begin{align}
u^{\tau}&=C_{\alpha}V_{u^{\tau}}^{\eta}=V_{u^{\tau}}.(v_{+})_{\tau}, \label{MT41}\\
V_{\partial_{1}^{2}u}^{\eta-2}.(\va,\wa)
&=V_{u^{\tau}}^{\eta}.(0,\va,0,\wa)\label{MT42},
\end{align}
cf.\@ \eqref{j1003} and \eqref{functDef}.  

\medskip

We now turn to the proof that $V_{u^{\tau}} \in \mathcal{D}^{\eta}(\Ta{+};\Gamma_{+}^{\tau})$ and begin with the following claims
\begin{align}
&\sup_{x \neq y}d^{-\eta}(y,x)\left | U^{\tau}(x,y)-U^{\tau}(x,x)-\nu^{\tau}(x)(y-x)_{1} \right |<\infty,\label{MT15}\\
&\sup_{x \neq y} d^{-(\eta-1)}(y,x) \left |\nu^{\tau}(y)-\nu^{\tau}(x)+\gamma^{\tau}(x,y)\right |<\infty.\label{MT16}
\end{align}
In light of our regularization in $y$, together with $u^{\tau} \in C^{\eta-2\alpha+2}$ by Step \ref{MTprf_2}, the first claim holds as a consequence of Taylor's formula and the periodicity of $U^{\tau}$ in both arguments, taking into account $1 \leq \eta \leq 2$.  The second claim will be deduced from the first, taking into account the following three-point continuity condition: for all $x,y,z \in \R^{2}$ it holds
\begin{align}
&\big |U^{\tau}(x,z)-U^{\tau}(x,y)-U^{\tau}(y,z)+U^{\tau}(y,y)
-\gamma^{\tau}(x,y) (z-y)_1
\big |\nonumber \\
&\quad \lesssim [V_{\partial_{1}^{2}u}^{\eta-2}]\big (d^{\thresh-\alpha}(y,x)d^{\alpha}(y,z)+ d^{\thresh-2\alpha}(y,x)d^{2\alpha}(z,y) \big)\label{MT22}.
\end{align}
Indeed, as in Step \ref{Intprf_ii} of the proof of Lemma \ref{int_lem_concr}, \eqref{MT22} is based on the straightforward identity
\begin{align*}
&U^{\tau}(x,z)-U^{\tau}(x,y)-U^{\tau}(y,z)+U^{\tau}(y,y)
-\gamma^{\tau}(x,y)(z-y)_1 
\\
&\stackrel{\eqref{MT19},\eqref{MT20}}{=}\big ( V_{\partial_{1}^{2}u}^{\thresh-2}(y)-V_{\partial_{1}^{2}u}^{\thresh-2}(x) \big ).\\
&\quad \begin{pmatrix}
(\vf)_{\tau}(z)-(\vf)_{\tau}(y)\\
(\wwf)_{\tau}(z)-(\wwf)_{\tau}(y)
-\of^{\tau}(y)(\xf(z)-\xf(y))
\end{pmatrix}
.
\end{align*}
By definition of $[V_{\partial_{1}^{2}u}^{\eta-2}]$, cf. \eqref{prod5}, \eqref{cw12}, and \eqref{S8001}, the quantity above is estimated by
\begin{align*}
&[V_{\partial_{1}^{2}u}^{\eta-2}] \big ( d^{\eta-\alpha}(y,x)\|(\vf)_{\tau}(z)-(\vf)_{\tau}(y)\|_{\Ta{\alpha}}\\
&\qquad  +d^{\eta-2\alpha}(y,x)\|(\wwf)_{\tau}(z)-(\wwf)_{\tau}(y)-\of^{\tau}(y)(\xf(z)-\xf(y))\\
&\qquad \qquad \qquad \qquad -\vf(y) \otimes \big ( (\partial_{a_0}\vf)_{\tau}(z)
-(\partial_{a_0}\vf)_{\tau}(y)) \|_{\Ta{2\alpha}} \big )\\
&\stackrel{\eqref{MT14}}{=}[V_{\partial_{1}^{2}u}^{\eta-2}] \big (d^{\eta-\alpha}(y,x)\|\Gamma_{\alpha}^{\tau}(y)(v_{+})_{\tau}(z)\|_{\Ta{\alpha}}\\
&\qquad  \qquad+d^{\eta-2\alpha}(y,x)\|\Gamma_{2 \alpha}^{\tau}(y)(v_{+})_{\tau}(z)\|_{\Ta{2\alpha}} \big).
\end{align*}
The three-point continuity \eqref{MT22} now follows from \eqref{MT40}.  Using \eqref{MT15} and \eqref{MT22}, one can deduce the relation \eqref{MT16} by the arguments leading from \eqref{KS2} and \eqref{KS3} to \eqref{KS4}, cf. Step \ref{nu} of the proof of Proposition \ref{int lem}. 

\medskip

With the above ingredients, the remaining arguments for $V_{u^{\tau}}^{\eta} \in \mathcal{D}^{\eta}(\Ta{+};\Gamma_{+}^{\tau})$ are similar to those given in Step \ref{Intprf_iv} of Lemma \ref{int_lem_concr}.  Indeed, the analogue of \eqref{jj6} takes the form
\begin{align*}
&\big(V_{u^{\tau}}^{\thresh}(y)-V_{u^{\tau}}^{\thresh}(x)\big).\mathsf{v}_{+} \nonumber
\\
&=\bigg (U^{\tau}(x,y)-U^{\tau}(x,x)-\nu^{\tau}(x)(v_{1}(y)-v_{1}(x)) \bigg ) \na\nonumber\\
&+\left (\nu^{\tau}(y)-\nu^{\tau}(x)+\gamma^{\tau}(x,y) \right) \left (\mathsf{v}_{1}-v_{1}(y)\na \right)\nonumber\\
&+\left (V_{\partial_{1}^{2}u}^{\eta-2}(y)-V_{\partial_{1}^{2}u}^{\eta-2}(x) \right).
\begin{pmatrix}
\va-(\vf)_{\tau}(y)\na\\
\wa-(\wwf)_{\tau}(y)\na-\omega^{\tau}(y)(\mathsf{v}_{1}-v_{1}(y)\na)
\end{pmatrix}
.  
\end{align*}
Appealing to \eqref{MT15} and \eqref{MT16}, the first two lines are estimated by $d^{\eta}(y,x)|\na| + d^{\eta-1}(y,x)|\Gamma_{1}^{\tau}(y)\mathsf{v}_{+}|$, where it is used that $\Gamma_{1}^{\tau}(y).\ta_{+}\stackrel{\eqref{MT14}}{=}\mathsf{v}_{1}-v_{1}(x)\na$.  As in Step \ref{Intprf_iv} of Lemma \ref{int_lem_concr}, the estimate for the third line is entirely analogous to the arguments above leading to \eqref{MT22}: using the form continuity of $V_{\partial_{1}^{2}u}^{\eta-2}$ it is controlled, up to a factor of $[V_{\partial_{1}^{2}u}^{\eta-2}]$, by the quantity $d^{\eta-\alpha}(y,x)\|\Gamma_{\alpha}^{\tau}(y)\ta_{+}\|_{\Ta{\alpha}}+d^{\eta-2\alpha}(y,x)\|\Gamma_{2\alpha}^{\tau}(y)\ta_{+}\|_{\Ta{2\alpha}}$, which completes the proof that $u^{\tau} \in V_{u^{\tau}}^{\eta}(\Ta{+};\Gamma_{+}^{\tau})$.

\medskip

{\sc Step} \arabic{MTprf}.\label{MTprf_4}\refstepcounter{MTprf}[Application of Lemma \ref{rec_lem_concr}] 
In this step, we claim there exists $a \diamond \partial_{1}^{2}u^{\tau} \in \mathcal{S}'(\R^{2})$ such that for all $T \leq 1$
\begin{align}
&\left\|(a \diamond \partial_{1}^{2}u^{\tau})_T -C_{\eta+\alpha-2}V_{\aufi}^{\thresh+\alpha-2}.
\begin{pmatrix}
(\partial_{1}^{2}u^{\tau})_{T}\\
(v_{\tau})_{T}
\end{pmatrix} \right\| \nonumber\\
&\quad\lesssim M_{a \diamond \partial_{1}^{2}u^{\tau},\eta+\alpha-2}\Tc{\eta+\alpha-2},\label{MT2}
\end{align}
where, in line with \eqref{S600},
\begin{align}
M_{a \diamond \partial_{1}^{2}u^{\tau},\eta+\alpha-2}&:=[a]_{\alpha}[V_{u^{\tau}}^{\eta}]_{\tau}+M_{a,\eta}+N^{4}\|V_{a}^{\eta}\| \|V_{\partial_{1}^{2}u}^{\eta-2}\|,\label{MT7}
\end{align}
and $M_{a,\eta}$ is defined by \eqref{SFin35}.  Moreover, we claim the following sub-optimal estimate: for all $T \leq 1$
\begin{align}
&\left\| (a \diamond \partial_{1}^{2}u^{\tau})_T -V_{\aufi}^{\thresh-2}.
\begin{pmatrix}
(\partial_{1}^{2}u^{\tau})_{T}\\
 ( (\pvf)_{\tau}  )_{T}\\
 ( (\vvf)_{\tau} )_{T}
\end{pmatrix} \right\|
\nonumber\\
& \quad \lesssim M_{a \diamond \partial_{1}^{2}u^{\tau},\eta-2}\Tc{\eta-2}, \label{MT3}
\end{align}
where, in line with \eqref{SFin6},
\begin{align}
M_{a \diamond \partial_{1}^{2}u^{\tau},\eta-2}&:=M_{a \diamond \partial_{1}^{2}u^{\tau},\eta+\alpha-2}+N^{3}\|\overline{V}_{a}^{\eta} \| \|V_{\partial_{1}^{2}u}^{\eta-2}\|.\label{MT8}
\end{align}
The claim essentially follows from Lemma \ref{rec_lem_concr}, with the estimates \eqref{MT2} and \eqref{MT3} corresponding to \eqref{eq:QualRO} and \eqref{s30}.  The minor difference is that $u^{\tau},V_{u^{\tau}}^{\eta}$ play the role of $u,V_{u}^{\eta}$ and in place of the hypothesis $V_{u}^{\eta} \in \mathcal{D}^{\eta}(\Ta{+};\Ga{+}{})$ we use that $V_{u^{\tau}}^{\eta} \in \mathcal{D}^{\eta}(\Ta{+};\Gamma_{+}^{\tau})$, cf. Step \ref{MTprf_3} above.  Note that this is reflected in the definitions \eqref{MT7} and \eqref{MT8} in comparison with \eqref{S600} and \eqref{SFin6}.  

\medskip

We now outline the minor changes to the proof of Lemma \ref{rec_lem_concr} that are required.  Recall that Lemma \ref{rec_lem_concr} essentially follows from Corollary \ref{lem:ExMult}, modulo an appropriate translation of the hypotheses.  In this regard, we continue to use the ($\tau$-independent) forms $V_{a}^{\eta}$ and $V_{\partial_{1}^{2}u}^{\eta-2}$ to play the role of $V_{u_{+}}$ and $V_{u_{-}}$.  However, the roles of $\partial_{1}^{2}u$ and $v$ (corresponding to $u_{-}$ and $v_{+} \diamond v_{-}$, cf. \eqref{j771},\eqref{j770}) are replaced by $\partial_{1}^{2}u^{\tau}$  and $v_{\tau}$.  The hypothesis \eqref{j770} continues to hold in light of \eqref{MT12}.
Hypothesis \eqref{j771} now takes the form
\begin{align}
&\|(\partial_1^2u^{\tau})_T-V_{\partial_{1}^{2}u}^{\eta-2}.\big ((v_{-})_{\tau}\big)_{T} \|
\lesssim [V_{u^{\tau}}^\eta]_{\tau}(T^\frac{1}{4})^{\eta-2}.\label{MT1}
\end{align}
The proof is essentially the same as that of \eqref{cw60}, but we give the details for the convenience of the reader.  We remark in advance that the most essential ingredients are the compatibility conditions \eqref{MT41} and the functorial relations \eqref{MT42} and \eqref{MT45}, and now we turn to the argument.  First note that a periodicity argument identical to the one given in Lemma \ref{rec_lem_concr} yields that for all $x,y \in \R^{2}$ it holds
\begin{align*}
|(V_{u^{\tau}}^\eta(y)-V_{u^{\tau}}^\eta(x)).(v_{+})_{\tau}(y)|\le [V_{u^{\tau}}^\eta]_{\tau}d^{\eta}(x,y).
\end{align*}
Moreover, one verifies that for $x \in \R^{2}$
\begin{align*}
&(\partial^{2}_{1}u^{\tau})_{T}(x)-V_{\partial_{1}^{2}u}^{\eta-2}.\big ( (v_{-})_{\tau} \big )_{T}(x)\nonumber\\
&\, =\int \big (V_{u^{\tau}}^{\eta}(y)-V_{u^{\tau}}^{\eta}(x)\big).(v_{+})_{\tau}(y)\partial_{1}^{2}\psi_{T}(x-y){\rm d}y
\end{align*}
as a consequence of \eqref{MT41}, \eqref{MT42} and of the equality 
$$
V_{\partial_{1}^{2}u}^{\eta-2}.\big ( (v_{-})_{\tau} \big )_{T}=V_{u^{\tau}}^{\eta}.(0,(v_{\alpha-2})_{\tau},0,(v_{2\alpha-2})_{\tau})_{T} =V_{u^{\tau}}^{\eta}.\partial_{1}^{2}(\cdot)_{T}(v_{+})_{\tau}.
$$ 
These two points, together with \eqref{1.13}, imply \eqref{MT1} as desired.  There is one other change to the proof of Lemma \ref{rec_lem_concr} that should be mentioned, regarding the extraction of the sub-optimal bounds \eqref{MT3}.  Namely, the constant $N_{*}$ defining the normed space $\overline{\Ta{}}_{\eta-2}$, cf. \eqref{j1102} is now played by $[V_{u^{\tau}}^{\eta}]_{\tau}$, and for the purpose of estimating the extended skeleton $\overline{\Gamma}$, cf. \eqref{j1004}, one should note that $[V_{\partial_{1}^{2}u}^{\eta-2}]\leq [V_{u^{\tau}}^{\eta}]_{\tau}$, cf. \eqref{fnctIneq}.  Indeed, this follows from the functorial relations \eqref{MT42} and \eqref{MT45}.

\medskip

{\sc Step} \arabic{MTprf}.\label{MTprfNew_1}\refstepcounter{MTprf}[Identification of product distribution] We now claim that
\begin{align}\label{MT13}
a\diam\partial_1^2u^\tau=a\partial_1^2u^\tau+C_{\eta+\alpha-2}(\overline{V}_a^\eta\otimes V_{\partial_1^2u}^{\eta-2}).v_\tau.
\end{align}
By definition \eqref{S880} 
the r.~h.~s. of \eqref{MT13} can be written as $C_{\eta+\alpha-2}V_{a\diamond\partial_1^2u}^{\eta+\alpha-2}
.{\partial_1^2u^\tau\choose v_\tau}$, since the reduction operator $C_{\eta+\alpha-2}$ does not affect
the $(\eta-2)$-component into which $\partial_1^2u^\tau$ is substituted. Hence \eqref{MT13} assumes
the form
\begin{align}\label{MT13bis}
a\diam\partial_1^2u^\tau=C_{\eta+\alpha-2}V_{a\diamond\partial_1^2u}^{\eta+\alpha-2}
.{\partial_1^2u^\tau\choose v_\tau}
\end{align}
and follows immediately from \eqref{MT2} in the limit $T\downarrow 0$, at least formally. 
For the sake of completeness, we now give the details. Because of the qualitative properties of
$C_{\eta+\alpha-2}V_{a\diamond\partial_1^2u}^{\eta+\alpha-2}$ 
$\in C^0(\mathbb{R}^2;\mathsf{\overline{T}}^*)$,
${\partial_1^2u^\tau\choose v_\tau}$ $\in C^0(\mathbb{R}^2;\mathsf{\overline{T}^{*} })$, and thus
$C_{\eta+\alpha-2}V_{a\diamond\partial_1^2u}^{\eta+\alpha-2}.{\partial_1^2u^\tau\choose v_\tau}$
$\in C^0(\mathbb{R}^2)$, we have 
\begin{align*}
\lim_{T\downarrow 0}C_{\eta+\alpha-2}V_{a\diamond\partial_1^2u}^{\eta+\alpha-2}
.{(\partial_1^2u^\tau)_T\choose (v_\tau)_T}&=C_{\eta+\alpha-2}V_{a\diamond\partial_1^2u}^{\eta+\alpha-2}
.{\partial_1^2u^\tau\choose v_\tau},\\
\lim_{T\downarrow 0}\Big(C_{\eta+\alpha-2}V_{a\diamond\partial_1^2u}^{\eta+\alpha-2}
.{\partial_1^2u^\tau\choose v_\tau}\Big)_T&=C_{\eta+\alpha-2}V_{a\diamond\partial_1^2u}^{\eta+\alpha-2}
.{\partial_1^2u^\tau\choose v_\tau},
\end{align*}
as limits in $C^0(\mathbb{R}^2)$, also appealing to (generalized) periodicity for uniformity. 
By the triangle inequality w.~r.~t. $\|\cdot\|$,
these limits combine with \eqref{MT2} to
\begin{align*}
\lim_{T\downarrow 0}\Big(a\diam\partial_1^2u^\tau-C_{\eta+\alpha-2}V_{a\diamond\partial_1^2u}^{\eta+\alpha-2}
.{\partial_1^2u^\tau\choose v_\tau}\Big)_T&=0.
\end{align*}
By \eqref{j760} this strengthens to
\begin{align*}
\Big(a\diam\partial_1^2u^\tau-C_{\eta+\alpha-2}V_{a\diamond\partial_1^2u}^{\eta+\alpha-2}
.{\partial_1^2u^\tau\choose v_\tau}\Big)_T&=0\quad\mbox{for any}\;T>0,
\end{align*}
which in the limit $T\downarrow 0$ yields \eqref{MT13bis} as an identity of distributions .

\medskip

{\sc Step} \arabic{MTprf}.\label{MTprf_5}\refstepcounter{MTprf}[Application of Lemma \ref{int_lem_concr}]
We now claim that the following estimates hold:
\begin{align}
[V_{u^{\tau}}^{\thresh}]_{\tau} &\lesssim M_{a \diamond \partial_{1}^{2}u^{\tau},\eta-2} +[a]_{\alpha}[V_{u^{\tau}}^{\eta-\alpha}]_{\tau}+M_{a,\eta-\alpha}+[V_{\partial_{1}^{2}u}^{\eta-2}]\label{MT4},\\ 
\|\nu^{\tau} \| &\lesssim [V_{u^{\tau}}^{\thresh}]_{\tau},\label{MT5}\\
\|V_{u^{\tau}}^{\eta}\|_{\tau} &\lesssim \|u^{\tau} \|+N^{-2}\|\nu^{\tau}\|+\|V_{\partial_{1}^{2}u}^{\eta-2}\|\label{MT6}. 
\end{align}
The above estimates essentially follow from Lemma \ref{int_lem_concr} and correspond to \eqref{s16},\eqref{SFin3}, and \eqref{SFin5} with $u^{\tau},\nu^{\tau},V_{u^{\tau}}^{\eta-\alpha},V_{u^{\tau}}^{\eta}$ in place of $u,\nu,V_{u}^{\eta-\alpha},V_{u}^{\eta}$ as well as $\Gamma_{+}^{\tau}$ in place of $\Gamma_{+}$ in the continuity and boundedness semi-norms measuring the forms describing $u$.  We now outline the modifications to the proof of Lemma \ref{int_lem_concr} that are required to deduce these estimates.  First observe that in light of the identification \eqref{MT13}, we may re-write \eqref{MT17} in line with \eqref{ss11} as
\begin{equation}
\partial_{2}u^{\tau}-P(a \diamond \partial_{1}^{2}u^{\tau})=Pf_{\tau}.\label{MT23}
\end{equation}

\medskip

Also note that by the periodicity of $u^{\tau}$ and $V_{a}^{\eta}$ (and thus $V^{\eta-2}_{\partial_{1}^{2}u}$ and $a$) in the sense of \eqref{cw14}, the identification \eqref{MT13} implies that $a \diam \partial_{1}^{2}u^{\tau}$ is periodic.  In modifying Step \ref{Intprf_0} of Lemma \ref{int_lem_concr}, we use the functions $U^{\tau},\gamma^{\tau}$ defined by \eqref{MT19} and \eqref{MT20} above.  In Step \ref{Intprf_00}, the identity \eqref{s25} is replaced by
\begin{align*}
&(\partial_{2}-a(x)\partial_{1}^{2})(U^{\tau})_{T}(x,\cdot)
\\
&\quad=(\cdot)_{T}P(a \diam \partial_{1}^{2}u^{\tau})-V_{\aufi}^{\thresh-2}(x).
\begin{pmatrix}
(\partial_{1}^{2}u^{\tau})_{T}\\
\big ( (\pvf)_{\tau} \big )_{T}\\
\big ( P(\vvf)_{\tau} \big )_{T}
\end{pmatrix}
.
\end{align*}
The proof is the same, but with $f_{\tau},(v_{\alpha})_{\tau},(w_{2\alpha})_{\tau}$ in place of $f,v_{\alpha},w_{2\alpha}$.  In Step \ref{Intprf_i}, one deduces from the identity above, as a replacement for \eqref{s24}, that for all $x \in \R^{2}$ and all $L,T\le 1$ it holds
\begin{align*}
&\inf_{c\in\R}\|(\partial_{2}-a(x)\partial_{1}^{2})(U^{\tau})_{T}(x,\cdot)-c\|_{B_{L}(x)} \nonumber \\
& \lesssim M_{a \diamond \partial_{1}^{2}u^{\tau},\eta-2}\Tc{\thresh-2}  +[a]_{\alpha} [V_{u^{\tau}}^{\thresh-\alpha}]_{\tau}L^{\alpha}\Tc{\thresh-\alpha-2}\nonumber\\
& \quad + M_{a,\eta-\alpha}\big (L^{\thresh-\alpha}\Tc{\alpha-2}+ L^{\thresh-2\alpha}\Tc{2\alpha-2} \big).
\end{align*}
The proof is identical, but uses \eqref{MT3} as a replacement for \eqref{s30}, \eqref{MT12} in place of \eqref{p13} \& \eqref{off1}, and also the analogue of \eqref{MT1} on level $\eta-\alpha$.  In Step \ref{Intprf_ii}, in place of \eqref{s34}, we use \eqref{MT22} established above.  In Step \ref{Intprf_iii}, one obtains that for all $x,y \in \R^{2}$
\begin{align}
&\left |U^{\tau}(x,y)-U^{\tau}(x,x)-\nu^{\tau}(x)(\xf(y)-\xf(x))\right |\lesssim M_{\tau} d^{\thresh}(y,x),\label{MT60}\\
&\left |\nu^{\tau}(y)-\nu^{\tau}(x)
+\gamma^{\tau}(x,y)
\right | \lesssim M_{\tau} d^{\thresh-1}(y,x),\label{MT26}
\end{align}
where
\begin{align}
M_{\tau}&:=M_{a \diamond \partial_{1}^{2}u^{\tau},\thresh-2}+[a]_{\alpha}[V_{u^{\tau}}^{\thresh-\alpha}]_{\tau}+ M_{a,\eta-\alpha}+[V_{\partial_{1}^{2}u}^{\eta-2}].\label{MT27}
\end{align}
This is again an application of Proposition \ref{int lem}, keeping in mind that the resulting function $\nu$ (depending on $\tau$) is determined through \eqref{MT60} and hence must agree with $\nu^{\tau}$ defined by \eqref{MT61}.  As in Step \ref{Intprf_iv} of Lemma \ref{int_lem_concr}, one then establishes \eqref{MT4} and \eqref{MT5}. Finally, one follows the same proof in Step \ref{Intprf_v} in order to establish \eqref{MT6}, using the functorial relation between the skeletons, cf.\@ \eqref{MT45}, and also the uniform estimate on $\omega^{\tau}$, cf.\@ \eqref{MT70}.
\medskip

{\sc Step} \arabic{MTprf}.\label{MTprf_6}\refstepcounter{MTprf}[Buckling argument] 
In this step, we claim that the estimates \eqref{S306} and \eqref{SFin4} hold provided the quantities $[V_{u^{\tau}}^{\thresh}]_{\tau}$ and $\|V_{u^{\tau}}^{\thresh}\|_{\tau}$ are written in place of $[V_{u}^{\thresh}]$ and $\|V_{u}^{\thresh}\|$ on the left hand side.  As a first step, we will show that
\begin{align}
\|V_{u^{\tau}}^{\eta}\|_{\tau} \lesssim N^{-2}[V_{u^{\tau}}^{\eta}]_{\tau}+1+\|\overline{V}_{a}^{\eta-\alpha}\| .\label{MT24}
\end{align}
Using that $u^{\tau}$ is periodic and mean-free together with \eqref{MT41}, applying Lemma \ref{cutting_lemma} to $V_{u^{\tau}}^{\eta}$ in the form of inequality \eqref{eq:cutLemmaBd1} and identity \eqref{wg31}, taking note of \eqref{wg12} and $\|\Gamma_{+}^{\tau}\|_{sk} \lesssim 1$, cf. \eqref{MT70}, we find
\begin{equation}
\|u^{\tau}\| \lesssim [u^{\tau}]_{\alpha} \lesssim [V_{u^{\tau}}^{\eta}]_{\tau}+N\|V_{u^{\tau}}^{\eta}\|_{\tau}.\label{MT25}
\end{equation}  
Inserting this into \eqref{MT6} and using $N \ll 1$ we obtain
\begin{align*}
\|V_{u^{\tau}}^{\eta}&\|_{\tau} \lesssim [V_{u^{\tau}}^{\eta}]_{\tau}+N^{-2}\|\nu^{\tau}\|+\|V_{\partial_{1}^{2}u}^{\eta-2}\|\nonumber\\
&\, \, \, \stackrel{\eqref{MT5}}{\lesssim} N^{-2}[V_{u^{\tau}}^{\eta}]_{\tau}+\|V_{\partial_{1}^{2}u}^{\eta-2}\|.
\end{align*}
Hence, inserting \eqref{S565} into the estimate above yields \eqref{MT24}.

\medskip

Next we claim that \eqref{S306} holds with $[V_{u^{\tau}}^{\eta}]_{\tau}$ in place of $[V_{u}^{\eta}]$ on the left hand side.  First observe that applying once more Lemma \ref{cutting_lemma} to $V_{u}^{\eta}$ in the form of estimate \eqref{eq:cutLemmaBd1}, recalling \eqref{MT43} and $\langle 1\rangle=\langle 2\alpha \rangle =2$, cf. \eqref{wg12}, we obtain
\begin{align}
[V_{u^{\tau}}^{\eta-\alpha}]_{\tau} \lesssim [V_{u^{\tau}}^{\eta}]_{\tau}+N^{2}\|V_{u^{\tau}}^{\eta}\|_{\tau}
\, \,\stackrel{\eqref{MT24}}{\lesssim} [V_{u^{\tau}}^{\eta}]_{\tau}+N^{2}\big ( 1 +\|\overline{V}_{a}^{\eta-\alpha}\| \big ).\label{MT75}
\end{align}
Moreover, we note that inserting \eqref{MT7} into \eqref{MT8} and appealing again to \eqref{S565} to estimate the contribution of $\|V_{\partial_{1}^{2}u}^{\eta-2}\|$ gives
\begin{equation*}
M_{a \diamond \partial_{1}^{2}u^{\tau},\eta-2} \lesssim [a]_{\alpha}[V_{u^{\tau}}^{\eta}]_{\tau}+M_{a,\eta}+\big(1 + \|\overline{V}_{a}^{\eta-\alpha}\| \big)\left (N^{3}\|\overline{V}_{a}^{\eta}\|+N^{4}\|V_{a}^{\eta}\| \right).
\end{equation*}
Inserting the previous two inequalities into \eqref{MT4}, appealing to \eqref{S565} to estimate $[V_{\partial_{1}^{2}u}^{\eta-2}]$, and using $[a]_{\alpha} \ll 1$ to buckle yields
\begin{align}
[V_{u^{\tau}}^{\eta}]_{\tau}&\lesssim M_{a,\eta-\alpha}+M_{a,\eta} + N[a]_{\alpha}^2 \nonumber\\
&\, \,+N^{2} \left ([a]_{\alpha}+N\|\overline{V}_{a}^{\eta}\|+N^{2}\|V_{a}^{\eta}\|\right)\big (1 + \|\overline{V}_{a}^{\eta-\alpha}\| \big)\label{SFin14}.
\end{align}
It remains to estimate $M_{a,\eta-\alpha}+M_{a,\eta}$, for which we appeal to Lemma \ref{cor:prod1} to obtain the following inequalities
\begin{align*}
M_{a,\eta-\alpha} &\lesssim N \big ( [a]_{\alpha}([a]_{\alpha}+N\|\overline{V}_{a}^{\eta-\alpha}\| )+[V_{a}^{\eta-\alpha}] \big),\\
M_{a,\eta} &\lesssim (1+N\|\overline{V}_{a}^{\eta}\|+[V_{a}^{\eta}] )(N[a]_{\alpha}^{2} + M_{a,\eta-\alpha})
\nonumber\\
&\quad \quad +N[V_{a}^{\eta}](1 + \|\overline{V}_{a}^{\eta-\alpha}\|) .
\end{align*}
Indeed, the estimates above follow from inserting \eqref{SFin1} into definition \eqref{IOR30}, \eqref{S565} into definition \eqref{SFin35}, then appealing to $[a]_{\alpha} \leq 1$.  Note in particular that the right hand side of the second estimate dominates $M_{a,\eta-\alpha}$.  Hence, inserting the first estimate into the second, then grouping the contributions of $[V_{a}^{\eta}]$ and using $[a]_{\alpha},N \leq 1$ gives
\begin{align*}
M_{a,\eta-\alpha}+M_{a,\eta} &\lesssim  N \big ( [a]_{\alpha}([a]_{\alpha}+N\|\overline{V}_{a}^{\eta-\alpha}\|)+[V_{a}^{\eta-\alpha}]  \big )(1+N\|\overline{V}_{a}^{\eta}\|+[V_{a}^{\eta}] ) \\
&+N[V_{a}^{\eta}]\big(1 + \|\overline{V}_{a}^{\eta-\alpha}\| \big) \\
&\lesssim N \big ([a]_{\alpha}([a]_{\alpha}+N\|\overline{V}_{a}^{\eta-\alpha}\|)+[V_{a}^{\eta-\alpha}] \big)\big (1+N\|\overline{V}_{a}^{\eta}\| \big)\\
&+N[V_{a}^{\eta}]\big (1 + \|\overline{V}_{a}^{\eta-\alpha}\|+[V_{a}^{\eta-\alpha}] \big).
\end{align*}
Inserting this into \eqref{SFin14} completes the proof of \eqref{S306} with $[V_{u^{\tau}}^{\eta}]_{\tau}$ in place of $[V_{u}^{\eta}]$, noting that $N[a]_{\alpha}^{2}$ already appears on the right hand side of the estimate above.

\medskip

Finally, we claim that \eqref{SFin4} holds with $\|V_{u^{\tau}}^{\eta}\|_{\tau}$ in place of $\|V_{u}^{\eta}\|$ on the left hand side.  Indeed, since we have already established that \eqref{S306} holds with $[V_{u^{\tau}}^{\eta}]_{\tau}$ in place of $[V_{u}^{\eta}]$, inserting this information into the estimate
\eqref{MT24} yields the desired estimate.
\medskip

{\sc Step} \arabic{MTprf}.\label{MTprf_7}\refstepcounter{MTprf}[Sending $\tau \to 0$: analysis of the skeleton]
In this step, we establish the following estimates: for each $0<\tau\leq 1$ it holds
\begin{align}
\sup_{x \in \R^{2}}\|(\vf)_{\tau}(x)-\vf(x)\|_{\Ta{\alpha}}&\lesssim (\tau^{\frac{1}{4}})^{\alpha}, \label{MT29}\\
\sup_{x \in \R^{2}}\|(w_{2\alpha})_{\tau}(x)-w_{2\alpha}(x)\|_{\Ta{2\alpha}}& \lesssim (\tau^{\frac{1}{4}})^{\alpha}, \label{MT28}\\
\sup_{x \in \R^{2}}\|\omega^{\tau}(x)-\omega(x) \|_{\Ta{2\alpha}} &\lesssim (\tau^{\frac{1}{4}})^{2\alpha-1}.\label{MT30}
\end{align}
The estimate \eqref{MT29} follows by writing for each $x \in \R^{2}$
\begin{equation}
(\vf)_{\tau}(x)-\vf(x)=\int \psi_{\tau}(x-y)(\vf(y)-\vf(x) )\dd y,
\end{equation}
then estimating the integrand in $\Ta{\alpha}$ by $d^{\alpha}(x,y)$, cf. \eqref{s203}, then using \eqref{1.13}.  To establish the remaining two estimates, we will first establish that
\begin{align}
&\sup_{x \in \R^{2}}\|(w_{2\alpha})_{\tau}(x)-w_{2\alpha}(x)
-\vf(x) \otimes \partial_{a_0}((\vf)_{\tau}-\vf)(x)\|_{\Ta{2\alpha}}\nonumber\\
& \qquad \,  \lesssim (\tau^{\frac{1}{4}})^{2\alpha}.\label{MT31}
\end{align}
Indeed, for each $x \in \R^{2}$ the quantity inside the norm can be written as
\begin{align*}
&\int \psi_{\tau}(x-y) \big (w_{2\alpha}(y)-w_{2\alpha}(x)-\vf(x) \otimes \partial_{a_0}(\vf(y)-\vf(x)) \big ) \dd y\\
&\stackrel{\eqref{he1}}{=}\int \psi_{\tau}(x-y) \big (w_{2\alpha}(y)-w_{2\alpha}(x)-\vf(x) \otimes \partial_{a_0}(\vf(y)-\vf(x))\\
&\qquad \qquad \qquad \qquad-\omega(x)(y-x)_{1}\big ) \dd y.
\end{align*}
Since the integrand is estimated in $\Ta{2\alpha}$ by $d^{2\alpha}(x,y)$, cf.~\eqref{s2011}, \eqref{MT31} follows by \eqref{1.13}.  Note that \eqref{MT31} implies \eqref{MT28} as a consequence of \eqref{MT29} and the definitions \eqref{S8030} of the norms $\Ta{\alpha}$ and $\Ta{2\alpha}$ together with the cross-norm property \eqref{j501a} in the form of 
\begin{equation*}
\|\vf(x) \otimes \partial_{a_0}\big ( (\vf)_{\tau}-\vf \big)(x) \|_{\Ta{2\alpha}} \leq \|\vf(x)\|_{\Ta{\alpha}} \|((\vf)_{\tau}-\vf)(x)\|_{\Ta{\alpha}}.
\end{equation*}
We also use that since $\vf$ has vanishing mean value, \eqref{s203} implies $\|\vf(x)\|_{\Ta{\alpha}} \lesssim 1$.

\medskip

Finally, we turn our attention to \eqref{MT30}.  Indeed, recalling the definition \eqref{MT35} of $U_{2\alpha}$ and of $\omega^{\tau}$, we note that for all $x,y \in \R^{2}$ it holds
\begin{align*}
\|(U_{2\alpha})_{\tau}(x,y)-(U_{2\alpha})_{\tau}(x,x)-\omega^{\tau}(x)(y-x)_{1}\|_{\Ta{2\alpha}} &\stackrel{\eqref{MT40}}{\lesssim}d^{2\alpha}(y,x),\\
\|U_{2\alpha}(x,y)-U_{2\alpha}(x,x)-\omega(x)(y-x)_{1}\|_{\Ta{2\alpha}} &\stackrel{\eqref{s2011}}{\lesssim} d^{2\alpha}(y,x).
\end{align*}
By triangle inequality this yields
\begin{align*}
\|\omega^{\tau}(x)-\omega(x)\|_{\Ta{2\alpha}}|(y-x)_{1}| &\lesssim
\|((U_{2\alpha})_{\tau}-U_{2\alpha})(x,y)\|_{\Ta{2\alpha}}+d^{2\alpha}(y,x)\\
&\,+\|( (U_{2\alpha})_{\tau}-U_{2\alpha})(x,x)\|_{\Ta{2\alpha}}.
\end{align*}
Moreover, note that
\begin{align*}
&(U_{2\alpha})_{\tau}(x,y)-U_{2\alpha}(x,y)\\
&\stackrel{\eqref{MT35}}{=}\big ( (w_{2\alpha})_{\tau}-w_{2\alpha}-\vf \otimes ((\partial_{a_0}\vf)_{\tau}-\partial_{a_0}\vf) \big )(y)\\
&\, \,+(\vf(y)-\vf(x)) \otimes \big((\partial_{a_0}\vf)_{\tau}-\partial_{a_0}\vf \big)(y).
\end{align*}
In light of \eqref{MT31}, \eqref{MT29}, and \eqref{s203} we find that
\begin{equation}
\|((U_{2\alpha})_{\tau}-U_{2\alpha})(x,y)\|_{\Ta{2\alpha}} \lesssim (\tau^{\frac{1}{4}})^{2\alpha}+d^{\alpha}(x,y)(\tau^{\frac{1}{4}})^{\alpha}.
\end{equation}
Choosing $y=x+\tau (1,0)$ in the estimate above yields \eqref{MT30}.

\medskip

{\sc Step} \arabic{MTprf}.\label{MTprf_8}\refstepcounter{MTprf}[Sending $\tau \to 0$: the limiting modelled distribution]
In this step, we construct the limiting solution $u$, its associated form $V_{u}^{\eta}$, and establish the \emph{a priori} bounds \eqref{S306} and \eqref{SFin4}.  We start by establishing uniform bounds for the sequences $\{u^{\tau}\}_{\tau}$ and $\{\nu^{\tau}\}_{\tau}$.
First we claim that $\{u^{\tau}\}_{\tau}$ is uniformly bounded in $C^{\alpha}$.  Indeed, as a consequence of Step \ref{MTprf_6}, we have uniform bounds on $[V_{u^{\tau}}^{\eta}]_{\tau}$ and $\|V_{u^{\tau}}^{\eta}\|_{\tau}$, so the claim follows from \eqref{MT25}.  Next we claim that $\{\nu^{\tau}\}_{\tau}$ is uniformly bounded in $C^{\eta-2\alpha}$.  Towards this end, first observe that
\begin{equation*}
\sup_{\tau}\sup_{x \neq y}d^{-(\eta-1)}(x,y)\left |\nu^{\tau}(y)-\nu^{\tau}(x)-\gamma^{\tau}(x,y)\right |<\infty,
\end{equation*}
as a consequence of the estimate \eqref{MT26} from Step \ref{MTprf_5}. Indeed, the quantity \eqref{MT27} is bounded uniformly in $\tau$ as a result of the uniform bounds on $[V_{u^{\tau}}^{\eta}]_{\tau}$, $\|V_{u^{\tau}}^{\eta}\|_{\tau}$ obtained in Step \ref{MTprf_6}, taking into account the estimate \eqref{MT75} and the definition \eqref{MT61}.  Moreover, using the definition \eqref{MT20}, the form continuity of $V_{\partial_{1}^{2}u}^{\eta-2}$, and the fact that $\{\omega^{\tau}\}_{\tau}$ is bounded uniformly in $C^{0}(\R^{2};\Ta{2\alpha})$, cf.\@ \eqref{MT70}, it follows that 
\begin{equation*}
\sup_{\tau}\sup_{x \neq y}d^{-(\eta-2\alpha)}(x,y)\left | \gamma_{\tau}(x,y) \right |<\infty.
\end{equation*}
These estimates, together with periodicity, cf.\@ \eqref{MT61}, yield that $\{\nu^{\tau}\}_{\tau}$ is uniformly bounded in $C^{\eta-2\alpha}$ as desired.  

\medskip

In light of the discussion in the previous paragraph, together with periodicity to have a compact domain, we may apply Arzel\`{a}-Ascoli to obtain uniform limits $u \in C^{\alpha}$ and $\nu \in C^{\eta-2\alpha}$ of the sequences $\{u^{\tau}\}_{\tau}$ and $\{\nu^{\tau}\}_{\tau}$, along a subsequence which we omit for notational convenience.  Together, these inputs allow to define the form $V_{u}^{\eta}$ according to \eqref{s300bis}.  In particular, using the pointwise convergences $u_{\tau} \to u, \nu_{\tau} \to \nu$ together with \eqref{MT29} and \eqref{MT28}, we obtain the convergence $V_{u_{\tau}}^{\eta}(x).\ta_{+} \to V_{u}^{\eta}(x).\ta_{+}$ holds for $x \in \R^{2}$ and $\ta_{+} \in \Ta{+}$.  Combining this pointwise convergence of the approximate forms with the pointwise convergence of the skeleton $\Gamma_{+}^{\tau}$, cf.\@ \eqref{MT14}, to $\Gamma_{+}$, cf.\@ \eqref{cw12}, as a consequence of \eqref{MT29}-\eqref{MT30}, we find that
\begin{equation*}
 [V_{u}^{\eta}]\leq \liminf_{\tau \to 0}[V_{u_{\tau}}^{\eta}]_{\tau}, \quad \quad \|V_{u}^{\eta}\| \leq \liminf_{\tau \to 0}\|V_{u_{\tau}}^{\eta}\|_{\tau}.
\end{equation*} 
These inequalities, together with Step \ref{MTprf_6}, lead to the estimates \eqref{S306} and \eqref{SFin4}.

\medskip

{\sc Step} \arabic{MTprf}.\label{MTprfNew_2}\refstepcounter{MTprf}[Sending $\tau \to 0$: Limiting product and equation] We postpone the argument that, along a further subsequence, 
$a\diam\partial_1^2u^\tau$ has a distributional limit we name $a\diam\partial_1^2u$ to the end of this
step and first derive the limiting equation \eqref{MT33}. 
Indeed, \eqref{MT23} turns into \eqref{MT33} in the limit $\tau\downarrow0$ (along the subsequence) 
by the (uniform) convergence
of $u^\tau$ to $u$ and the distributional convergence of $(a\diamond\partial_1^2u^\tau,f_\tau)$ to
$(a\diamond\partial_1^2u,f)$, respectively.  Moreover, note that $a \diamond \partial_{1}^{2}u$ inherits periodicity from $a\diam\partial_1^2u^\tau$ in the limit $\tau\downarrow0$.

\medskip

We now argue that \eqref{MT34} holds with $\delta=\eta+\alpha-2$. 
Indeed, this is a consequence of  \eqref{MT2} together with \eqref{S306}
ensuring that the r.~h.~s. constant 
$M_{a\diamond\partial_1^2u^\tau,\eta+\alpha-2}$ defined in \eqref{MT7} is bounded for $\tau\downarrow0$. 
On the one hand, for fixed $T>0$ as $\tau\downarrow 0$,
$(a\diam\partial_1^2u^\tau)_T$ converges to $(a\diam\partial_1^2u)_T$ in $C^0(\mathbb{R}^2)$.
On the other hand, $((\partial_1^2u^\tau)_T,(v_\tau)_T=(v_T)_\tau)$ converges 
in $C^0(\mathbb{R}^2;\overline{\mathsf{T}})$ to $((\partial_1^2u)_T,v_T)$, and thus because of
$C_{\eta+\alpha-2}V_{a\diamond\partial_1^2u}^{\eta+\alpha-2}\in C^0(\mathbb{R}^2;\overline{\mathsf{T}}^*)$
also
$C_{\eta+\alpha-2}V_{a\diamond\partial_1^2u}^{\eta+\alpha-2}.{(\partial_1^2u^\tau)_T\choose(v_\tau)_T}$
converges to
$C_{\eta+\alpha-2}V_{a\diamond\partial_1^2u}^{\eta+\alpha-2}.{(\partial_1^2u)_T\choose v_T}$ in $C^0(\mathbb{R}^2)$.

\medskip

We finally argue that, along a subsequence, $a\diam\partial_1^2u^\tau$ indeed has a distributional limit. We do so by an
application of the compactness argument in Step \ref{rec_comp_St} in the proof of Proposition \ref{rec_lem}.
As an input for the latter, we need to establish
$\limsup_{\tau\downarrow0}\sup_{T\le 1}(T^\frac{1}{4})^{2-\alpha}\|(a\diam\partial_1^2u^\tau)_T\|<\infty$,
which by \eqref{MT2}, and again the boundedness of $M_{a\diamond\partial_1^2u^\tau,\eta+\alpha-2}$, will be 
a consequence of $\limsup_{\tau\downarrow0}\sup_{T\le 1}(T^\frac{1}{4})^{2-\alpha}
\|C_{\eta+\alpha-2}V_{a\diamond\partial_1^2u}^{\eta+\alpha-2}.{\partial_1^2u^\tau\choose v_\tau}_T\|<\infty$. 
For the latter we observe that for any $x$,
\begin{align*}
\lefteqn{\big|C_{\eta+\alpha-2}V_{a\diamond\partial_1^2u}^{\eta+\alpha-2}(x).
{ \big ( \partial_1^2u^\tau \big )_{T}(x)\choose \big ( v_{\tau} \big )_{T}(x) }\big|}\nonumber\\
&\stackrel{\eqref{S8300}}{\le}
\|V_{a\diamond\partial_1^2u}^{\eta+\alpha-2}\|_{\overline{\mathsf{T}};\overline{\mathsf{\Gamma}}}
\Big(\frac{1}{N_*}\|(\partial_1^2u^\tau)_T-V_{\partial_1^2u}^{\eta-2}.((v_-)_\tau)_T\|\\
&+\sum_{\beta<\eta+\alpha-2}
N^{\langle\beta\rangle}\|\Gamma_\beta(x)(v_\tau)_T(x)\|_{\mathsf{T}_\beta}\Big).
\end{align*}
The prefactor $\|V_{a\diamond\partial_1^2u}^{\eta+\alpha-2}\|_{\overline{\mathsf{T}};\overline{\mathsf{\Gamma}}}$
is finite, cf.~Lemma \ref{lem:Multiplication}, and obviously independent of $\tau$. Estimate \eqref{MT1}
together with the definition $N_*=[V_{u^\tau}^\eta]_\tau$ implies that the first r.~h.~s. term is 
$\lesssim (T^\frac{1}{4})^{\eta-2}\le (T^\frac{1}{4})^{\alpha-2}$. By \eqref{MT12}, the remaining r.~h.~s. terms are 
$\lesssim N^{\langle\beta\rangle}(T^\frac{1}{4})^{\beta}$ $\le (T^\frac{1}{4})^{\alpha-2}$.

\medskip

{\sc Step} \arabic{MTprf}.\label{MTprf_9}\refstepcounter{MTprf}[Uniqueness]
In the final step, we argue that the three properties in the statement of the theorem uniquely determine the modelled distribution $V_{u}^{\eta}$.  Indeed, suppose we are given two such modelled distributions $V_{u_{i}}^{\eta}$ for $i=1,2$ with the three listed properties and let $a \diamond \partial_{1}^{2}u_{i}$ be the two associated distributions.  We will first argue that $u_{1}=u_{2}$, where $u_{i}:=C_{\alpha}V_{u_{i}}^{\eta}$.  Towards this end, define $u:=u_{1}-u_{2}$, $a \diamond \partial_{1}^{2}u:=a \diamond \partial_{1}^{2}u_{1}-a \diamond \partial_{1}^{2}u_{2}$, and $R^{T}:=P(a \diamond \partial_{1}^{2}u)_{T}-P(a\partial_{1}^{2}u_{T})$.  Applying $(\cdot)_{T}$ on both sides of \eqref{MT33} for each $i=1,2$ and taking the difference leads to 
\begin{equation}
\partial_{2}u_{T}-P(a\partial_{1}^{2}u_{T})=R^{T}.\label{MT62}
\end{equation}
First note that since the modelled distribution $V_{a}^{\eta}$ is fixed, the algebraic property \eqref{algebraicRel} ensures that $V_{\partial_{1}^{2}u_{1}}^{\eta-2}=V_{\partial_{1}^{2}u_{2}}^{\eta-2}$ and therefore \eqref{MT34} implies that for $T \ll 1$ it holds $\|R^{T}\|=O\big ((T^{\frac{1}{4}})^{\delta} \big )$.  

\medskip

We argue further that for a sufficiently small $\gamma>0$, it holds that $[R^{T}]_{\gamma}=o(1)$.  Using the semi-group property of the kernel, we now argue in favor of the following Lipschitz estimates
\begin{equation}
[(a \diamond \partial_{1}^{2}u)_{T}]_{1}=O \big ((T^{\frac{1}{4}})^{\alpha-4} \big ), \quad [(\partial_{1}^{2}u)_{T}]_{1}=O \big ((T^{\frac{1}{4}})^{\alpha-4} \big ).\label{MT50}
\end{equation}
We give the details only for the first estimate, a similar but easier argument gives the second estimate.  By periodicity of $(a \diamond \partial_{1}^{2}u)_{T}$, it suffices to consider $x \neq y$ with $d(x,y) \leq 1$, and we write
\begin{align*}
&(a \diamond \partial_{1}^{2}u)_{T}(y)-(a \diamond \partial_{1}^{2}u)_{T}(x)\\
&\stackrel{\eqref{1.10}}{=}\int (a \diamond \partial_{1}^{2}u)_{\frac{T}{2}}(z)\big ( \psi_{\frac{T}{2}}(x-z)-\psi_{\frac{T}{2}}(y-z)\big  )\dd z.\\
&=(y-x) \cdot \int \int_{0}^{1} (a \diamond \partial_{1}^{2}u)_{\frac{T}{2}}(z)\big (\partial_{1}\psi_{\frac{T}{2}},\partial_{2}\psi_{\frac{T}{2}})(\theta x + (1-\theta)y -z) \dd \theta \dd z.
\end{align*}
The claim \eqref{MT50} now follows from $\sup_{T \leq 1}(T^{\frac{1}{4}})^{\alpha-2}\|(a \diamond \partial_{1}^{2}u)_{T}\|<\infty$, cf.\@ Step \ref{MTprfNew_2}, the definition \eqref{j1001}, property \eqref{1.13} of the kernel, and the fact that the parabolic distance is larger than the Euclidean distance for $d(x,y) \leq 1$.

\medskip

By \eqref{MT50}, $a \in C^{\alpha}$, and the discrete product rule, we obtain $[R^{T}]_{\alpha}=O \big ((T^{\frac{1}{4}})^{\alpha-4} \big )$.  Combining this with $\|R^{T}\|=O \big ((T^{\frac{1}{4}})^{\delta} \big )$, we find by elementary interpolation that for each $ \theta \in [0,1]$ it holds 
\begin{equation*}
[R^{T}]_{\alpha \theta} \leq \big (2\|R^{T}\| \big )^{1-\theta}[R^{T}]_{\alpha}^{\theta}=O \big ( (T^{\frac{1}{4}})^{(\alpha-4)\theta}(T^{\frac{1}{4}})^{\delta (1-\theta)} \big ). 
\end{equation*}
Choosing $\theta$ sufficiently small and setting $\gamma=\theta \alpha$ implies the desired $[R^{T}]_{\gamma}=o(1)$.

\medskip

We now claim that $\|u_{T}\| \lesssim [R^{T}]_{\gamma}$.  Setting $a_{0}=a(0)$ and freezing-in the coefficients we rewrite \eqref{MT62} as
\begin{equation*}
\partial_{2}u_{T}-a_{0}\partial_{1}^{2}u_{T}=R^{T}+P \big ( (a-a_{0})\partial_{1}^{2}u_{T} \big ).
\end{equation*}
Since $u$ is mean free, we obtain that $\|u_{T}\|+\|\partial_{1}^{2}u_{T}\| \lesssim [u_{T}]_{2+\gamma}$, which may be combined with constant-coefficient Schauder theory to obtain
\begin{align*}
[u_{T}]_{2+\gamma} &\lesssim [R^{T}]_{\gamma}+[a]_{\gamma}\|\partial_{1}^{2}u_{T}\|+\|a-a_{0}\|[\partial_{1}^{2}u]_{\gamma}.\\
& \lesssim [R^{T}]_{\gamma}+[a]_{\gamma}[u_{T}]_{2+\gamma}
\end{align*}
Since by assumption $[a]_{\gamma} \leq [a]_{\alpha} \ll 1$, we may buckle and obtain that $\|u_{T}\| \lesssim [u_{T}]_{2+\gamma} \lesssim [R^{T}]_{\gamma}$. Hence, from the above established $[R^{T}]_{\gamma}=o(1)$, sending $T \to 0$ we find that $\|u\|=0$, which ensures that $u_{1}=u_{2}$.\\

We finally argue in favor of the stronger identification $V_{u_{1}}^{\eta}=V_{u_{2}}^{\eta}$ as modelled distributions.  Indeed, first observe that for each $i=1,2$ and $\ta_{+}=(\na,\va,\mathsf{v}_{1},\mathsf{w}_{2\alpha})$ it holds that
\begin{align}
&V_{u_{i}}^{\eta}(x).\ta_{+}=V_{u_{i}}^{\eta}(x).v_{+}(x) \na + V_{u_{i}}^{\eta}(x).
\begin{pmatrix}
0\\
\va-\vf(x)\na \\
0 \\
\mathsf{w}_{2\alpha}-w_{2\alpha}(x)\na
\end{pmatrix}
\nonumber\\
&\qquad \qquad \qquad \qquad \qquad \, \,\quad  +V_{u_{i}}^{\eta}(x).
\begin{pmatrix}
0\\
0\\
\mathsf{v}_{1}-v_{1}(x)\na \\
0
\end{pmatrix}
\nonumber\\
\stackrel{\eqref{functDef}, \eqref{j1003}}{=}
&u_{i}(x)+V_{\partial_{1}^{2}u_{i}}^{\eta-2}(x).
\begin{pmatrix}
\va-\vf(x)\na \\
\wa-\wwf(x)\na
\end{pmatrix}
+V_{u_{i}}(x).
\begin{pmatrix}
0\\
0\\
1\\
0
\end{pmatrix}
(\mathsf{v}_{1}-v_{1}(x)\na)
.\label{MT65}
\end{align}
By the uniqueness of $u_{i}$ established above and the algebraic relation \eqref{algebraicRel}, it only remains to argue that the quantity $V_{u_{i}}^{\eta}(x).(0,0,1,0)$ is unique.  For this, we note that by the arguments leading to Lemma \ref{int_lem_concr}, there exists a unique function $\nu$ such that
\begin{equation}
\sup_{x \neq y} d^{-\eta}(y,x)\left |U(x,y)-U(x,x)-\nu(x)(y-x)_{1} \right | = O(1),\label{MT63}
\end{equation}
where $U(x,y)=u_{i}(y)-V_{\partial_{1}^{2}u_{i}}^{\eta-2}(x).(\vf(y),w_{2\alpha}(y))$.  By the uniqueness of $u_{i}$ and \eqref{algebraicRel}, we see that $U$ does not depend on $i$.  Hence, to complete the proof, it suffices to show that $\nu(x)=V_{u_{i}}^{\eta}(x).(0,0,1,0)$.  For this, we note that by using \eqref{MT65} with $\ta_{+}=v_{+}(y)$ in the form of
\begin{align*}
V_{u_{i}}^{\eta}(x).\tf_{+}(y)=u_{i}(x)&+V_{\partial_{1}^{2}u}^{\eta-2}(x).
\begin{pmatrix}
\vf(y)-\vf(x)\\
\wwf(y)-\wwf(x)
\end{pmatrix}
\\
&+V_{u_{i}}(x).
\begin{pmatrix}
0\\
0\\
1\\
0
\end{pmatrix}
(v_{1}(y)-v_{1}(x)),
\end{align*}
from which we subtract the same identity with $x$ replaced by $y$, we obtain
\begin{align*}
&\left ( V_{u_{i}}^{\eta}(x).
\begin{pmatrix}
0\\
0\\
1\\
0
\end{pmatrix}
-\nu(x) \right ) (y-x)_{1}=U(x,y)-U(x,x)-\nu(x)(y-x)_{1}\\
& \qquad \qquad \quad \qquad \qquad \qquad \qquad \qquad -(V_{u_{i}}^{\eta}(y)-V_{u_{i}}^{\eta}(x)).v_{+}(y).
\end{align*}
By \eqref{MT63}, \eqref{j700} and the form continuity of $V_{u_{i}}^{\eta}$, the r.h.s is of the order $d^{\eta}(x,y)$, so letting $y$ approach $x$ we complete the claim.
\end{proof}
{\sc Proof of Corollary \ref{cor:wg}}. 
First we claim that
\begin{equation}\label{CorCut}
[a]_{\alpha} \lesssim \epsilon^{-1}N, \quad [V_{a}^{\eta-\alpha}] \lesssim \epsilon^{-1}N^{2}, \quad \|\overline{V}_{a}^{\eta}\| \leq \epsilon^{-1}.
\end{equation}
Indeed, by definition $a=C_{\alpha}V_{a}^{\eta}$, $V_{a}^{\eta-\alpha}=C_{\eta-\alpha}V_{a}^{\eta}$, so applying Lemma \ref{cutting_lemma} in the form of \eqref{eq:cutLemmaBd1}, noting that $\langle 2\alpha \rangle=2$, $\langle 1\rangle=2$, $\langle \alpha\rangle=1$, cf. \eqref{wg12} and \eqref{j790}, we find
\begin{align*}
[a]_{\alpha} &\stackrel{\eqref{wg31}}{\lesssim} [V_{a}^{\eta}]+N\|V_{a}^{\eta}\| \stackrel{\eqref{wg32}}{\leq} \epsilon^{-4}N^{3}+\epsilon^{-1}N \stackrel{N \leq \epsilon^{3},N \leq 1}{\lesssim} \epsilon^{-1}N,\\
[V_{a}^{\eta-\alpha}] &\stackrel{\eqref{wg31}}{\lesssim} [V_{a}^{\eta}]+N^{2}\|V_{a}^{\eta}\|\stackrel{\eqref{wg32}}{\leq} \epsilon^{-4}N^{3}+\epsilon^{-1}N^{2} \stackrel{N \leq \epsilon^{3}}{\lesssim} \epsilon^{-1}N^{2}.
\end{align*}
Also, note that $\|\overline{V}_{a}^{\eta}\| \leq \|V_{a}^{\eta}\| \stackrel{\eqref{wg32}}{\leq} \epsilon^{-1}$ by \eqref{new} of Lemma \ref{cutting_lemma}.

\medskip  

Inserting \eqref{CorCut} and \eqref{wg32} into the estimate \eqref{S306} we find
\begin{align*}
[V_{u}^{\eta}]&\lesssim N\big (\epsilon^{-1}N( \epsilon^{-1}N+N)+\epsilon^{-1}N^{2} \big) \big (1+\epsilon^{-1}N \big)\nonumber\\
&+\epsilon^{-4}N^{4}\big(1+\epsilon^{-1}N^{2} \big )\nonumber\\
&+N^{2}\left (\epsilon^{-1}N+\epsilon^{-1}N^{2}\right) \stackrel{N \leq \epsilon^{2},\epsilon \leq 1}{\lesssim} \epsilon^{-2}N^{3}.
\end{align*}
Similarly, from \eqref{SFin4} we obtain
\begin{align*}
\|V_{u}^{\eta} \|&\lesssim N^{-1}\big (\epsilon^{-1}N(\epsilon^{-1}N+N)+\epsilon^{-1}N^{2} \big) \big (1+\epsilon^{-1}N \big)\nonumber\\
&+\epsilon^{-4}N^{2}\big(1 +\epsilon^{-1}N \big )\nonumber\\
&+\left (1+\epsilon^{-1}N^{}+\epsilon^{-1}N^{2}\right) \stackrel{N \leq \epsilon^{2},\epsilon \leq 1}{\lesssim} 1.
\end{align*}
This completes the proof of the first two estimates in \eqref{wg34}, it only remains to argue that $\|\overline{V}_{u}^{\eta-\alpha}\| \leq 1$.  Indeed, inspecting the proof of Theorem \ref{Prop:aPriori}, cf.\@ Step \ref{MTprf_8}, we find that $V_{u}^{\eta}$ is defined by \eqref{s300bis}.  

From $u=C_{\alpha}V_{u}^{\eta}$ and \eqref{j41} we learn from $\overline{V}_{u}^{\eta-\alpha}=V_{u}^{\eta-\alpha}-C_{\alpha}V_{u}^{\eta}$ that
\begin{equation*}
\overline{V}_{u}^{\eta-\alpha}.\ta_{+}=\delta_{a}.(\va-\vf\na),
\end{equation*}
from which the claim follows by definition of $\|\cdot\|_{\Ta{+};\Ga{+}{}}$, cf. Definition \ref{prodDef1}, and by \eqref{S8030}, \eqref{wg12},\eqref{cw12}.
\bibliographystyle{plain}
\bibliography{bibliography.bib}

\end{document}